\documentclass[12pt]{scrartcl}  
\usepackage{a4}   
\usepackage{amsmath}    
\usepackage{paralist}
\usepackage{amssymb}
\usepackage{amsfonts} 
\usepackage{mathrsfs}  
\usepackage{dsfont}
\usepackage{latexsym} 
\usepackage{xcolor}
\usepackage{bbm,exscale}
\definecolor{Myblue}{rgb}{0,0,0.6}  
\usepackage[colorlinks,citecolor=Myblue,linkcolor=Myblue,urlcolor=Myblue,pdfpagemode=None]{hyperref}
\usepackage{amsthm}
\usepackage{accents}
\usepackage[square,numbers,sort&compress]{natbib} 
\usepackage[all,cmtip]{xy}
\usepackage{ifthen} 
\usepackage{bbding}
\usepackage{stmaryrd}  
\usepackage{wasysym}
\usepackage{tikz} 
\usepackage{tikz-cd}
\usepackage{tikz-3dplot}
\usepackage{pgfplots}
	\pgfplotsset{width=7cm,compat=1.8}
	\usetikzlibrary{decorations.markings}
	\usetikzlibrary{patterns}
\tikzset{
    string/.style={draw=#1, postaction={decorate}, decoration={markings,mark=at position .51 with {\arrow[color=#1]{>}}}},
    costring/.style={draw=#1, postaction={decorate}, decoration={markings,mark=at position .51 with {\arrow[draw=#1]{<}}}},
    ostring/.style={draw=#1, postaction={decorate}, decoration={markings,mark=at position .47 with {\arrow[draw=#1]{>}}}},
    ustring/.style={draw=#1, postaction={decorate}, decoration={markings,mark=at position .56 with {\arrow[draw=#1]{>}}}},
    oostring/.style={draw=#1, postaction={decorate}, decoration={markings,mark=at position .43 with {\arrow[draw=#1]{>}}}},
    uustring/.style={draw=#1, postaction={decorate}, decoration={markings,mark=at position .59 with {\arrow[draw=#1]{>}}}},
    directed/.style={string=blue!50!black}, 
    odirected/.style={ostring=blue!50!black}, 
    udirected/.style={ustring=blue!50!black}, 
    oodirected/.style={oostring=blue!50!black}, 
    uudirected/.style={uustring=blue!50!black},     
    redirected/.style={costring= blue!50!black},
    redirectedgreen/.style={costring= green!50!black},
    directedgreen/.style={string= green!50!black},
    redirectedlightgreen/.style={costring= green!65!black},
    directedlightgreen/.style={string= green!65!black},
}

\tikzset{-dot-/.style={decoration={
  markings,
  mark=at position 0.5 with {\fill circle (1.875pt);}},postaction={decorate}}}

\tikzset{
	Fdot/.style={circle, draw, fill, inner sep=0pt}, 
	Odot/.style={circle, draw, inner sep=0.1pt, minimum size=0.1cm}
	}

  \vfuzz \hfuzz
\makeatletter
\newcommand{\raisemath}[1]{\mathpalette{\raisem@th{#1}}}
\newcommand{\raisem@th}[3]{\raisebox{#1}{$#2#3$}}
\makeatother

\newcommand{\E}{\text{e}}
\newcommand{\I}{\text{i}}
\newcommand{\B}{\mathcal{B}}

\newcommand{\C}{\mathds{C}}
\newcommand{\D}{\mathds{D}}
\newcommand{\K}{\mathds{K}}

\newcommand{\N}{\mathds{N}}

\newcommand{\R}{\mathds{R}}
\newcommand{\Z}{\mathds{Z}}

\def\1{\ifmmode\mathrm{1\!l}\else\mbox{\(\mathrm{1\!l}\)}\fi}

\newcommand{\be}{\begin{equation}}
\newcommand{\ee}{\end{equation}}
\newcommand{\bes}{\begin{equation*}}
\newcommand{\ees}{\end{equation*}}
\newcommand{\cc}[1] {\overline{#1}}
\newcommand{\inv}[0]{{-1}}

\newcommand{\id}{\text{id}}

\newcommand{\Hom}{\operatorname{Hom}}
\newcommand{\End}{\operatorname{End}}

\newcommand{\ev}{\operatorname{ev}}
\newcommand{\eval}{\operatorname{eval}}

\newcommand{\coev}{\operatorname{coev}}

\def\lra{\longrightarrow}

\newcommand{\Bord}{\operatorname{Bord}}

\newcommand{\Bords}{\operatorname{Bord}_{3}^{\APLstar}}
\newcommand{\Cube}{\operatorname{Cube}}

\newcommand{\Bordstrat}{\operatorname{Bord}^{\mathrm{strat}}}

\newcommand{\Bordd}{\operatorname{Bord}_{3}^{\mathrm{def}}(\mathds{D})}

\newcommand{\G}{\mathcal{G}}
\newcommand{\tz}{\mathcal T_\zz}
\newcommand{\tzp}{\mathcal T_{\mathcal Z'}}

\newcommand{\Obj}{\mathrm{Obj}}
\newcommand{\zz}{\mathcal{Z}}
\newcommand{\Fdp}{\operatorname{\mathcal F}_{\textrm{d}}^{\textrm{p}}}
\newcommand{\zzd}{\mathcal{Z}^{\mathrm{def}}}
\newcommand{\zztriv}{\mathcal{Z}^{\mathrm{triv}}} 
\newcommand{\zzAtriv}{\mathcal{Z}_{\mathrm{triv}}^{\Cat{A}}} 
\newcommand{\unit}{\operatorname{\mathbf{1}}}

\newcommand{\Borddef}{\operatorname{Bord}_{3}^{\mathrm{def}}}

\newcommand{\eps}{\varepsilon}

\newcommand{\sta}{\boxempty}
\newcommand{\fus}{\otimes}

\newcommand{\dagg}{^{\dagger}}
\newcommand{\hash}{^{\#}}

\newcommand{\Cat}[1]         {\operatorname{\mathcal{#1}}}
\newcommand{\Catpre}[1]         {\operatorname{\mathcal{#1}^{\mathrm{pre}}}}
\newcommand{\Vect}{\operatorname{Vect}}

\newcommand\arxiv[2]      {\href{http://arXiv.org/abs/#1}{#2}}
\newcommand\doi[2]        {\href{http://dx.doi.org/#1}{#2}}

\allowdisplaybreaks

\deffootnote[1em]{1em}{1em}
{\textsuperscript{\thefootnotemark}}

\theoremstyle{definition}
\newtheorem{definition}{Definition}
\newtheorem{proposition}[definition]{Proposition}
\newtheorem{theorem}[definition]{Theorem}
\newtheorem{lemma}[definition]{Lemma}
\newtheorem{corollary}[definition]{Corollary}
\newtheorem{remark}[definition]{Remark}

\numberwithin{equation}{section}
\numberwithin{definition}{section}
\numberwithin{figure}{section}

\newcommand\void[1]{}

\begin{document}

\title{3-dimensional defect TQFTs \\ and their tricategories}

\author{%
\!\!\!\!\!\!\!Nils Carqueville$^*$ \quad
Catherine Meusburger$^\#$ \quad
Gregor Schaumann$^*$%
\\[0.5cm]
   \hspace{-1.8cm}  \normalsize{\texttt{\href{mailto:nils.carqueville@univie.ac.at}{nils.carqueville@univie.ac.at}}} \\  %
   \hspace{-1.8cm}  \normalsize{\texttt{\href{mailto:catherine.meusburger@math.uni-erlangen.de}{catherine.meusburger@math.uni-erlangen.de}}} \\
   \hspace{-1.8cm}  \normalsize{\texttt{\href{mailto:gregor.schaumann@univie.ac.at}{gregor.schaumann@univie.ac.at}}}\\[0.1cm]
   \hspace{-1.2cm} {\normalsize\slshape $^*$Fakult\"at f\"ur Mathematik, Universit\"at Wien, Austria}\\[-0.1cm]
   \hspace{-1.2cm} {\normalsize\slshape $^\#$Department Mathematik, Friedrich-Alexander-Universit\"at Erlangen-N\"urnberg, Germany}\\[-0.1cm]
}

\date{}
\maketitle

\begin{abstract}
We initiate a systematic study of 3-dimensional `defect' topological quantum field theories, that we introduce as symmetric monoidal functors on stratified and decorated bordisms. 
For every such functor we construct a tricategory with duals, which is the natural categorification of a pivotal bicategory. 
This captures the algebraic essence of defect TQFTs, and it gives precise meaning to the fusion of line and surface defects as well as their duality operations. 
As examples, we discuss how Reshetikhin-Turaev and Turaev-Viro theories embed into our framework, and how they can be extended to defect TQFTs. 
\end{abstract}

\newpage

\tableofcontents

\section{Introduction and summary}
\label{sec:intro}

In the approach of Atiyah and Segal, an $n$-dimensional topological quantum field theory (TQFT) is a symmetric monoidal functor from $\Bord_n$ to vector spaces. 
Such functors in particular provide topological invariants of manifolds. 
Even without leaving the realm of topology, one can obtain finer invariants, and a much richer theory, by augmenting the bordism category with decorated lower-dimensional subspaces. 
These so-called `defects' appear naturally in both physics and pure mathematics. 
The present paper is dedicated to unearthing the general algebraic structure of the 3-dimensional case of such `defect TQFTs'. 

Already in the 2-dimensional case, the above process of enriching topological bordisms with defect structure is one of iterated categorification: 
By a classical result, 2-dimensional \textsl{closed} TQFTs (where objects in $\Bord_2$ are closed 1-manifolds as in the original definition of \cite{AtiyahTQFT}, and all manifolds that we consider are oriented) are equivalent to commutative Frobenius algebras \cite{Kockbook}. 
By allowing labelled boundaries for objects and morphisms, one ascends to \textsl{open/closed} TQFTs $\zz^{\textrm{oc}}$ \cite{l0010269, ms0609042, lp0510664}. 
These are known to be equivalent to a commutative Frobenius algebra (what $\zz^{\textrm{oc}}$ assigns to $S^1$) together with a Calabi-Yau category (whose Hom sets are what $\zz^{\textrm{oc}}$ does to intervals with labelled endpoints) and certain relations between the two.
In fact for all known examples it is only the Calabi-Yau category that matters, and the Frobenius algebras can be recovered as its Hochschild cohomology \cite{c0412149}. 
Hence the transition from 2-dimensional closed to open/closed TQFTs is one from algebras to categories. 

The natural conclusion of this line of thought is to consider 2-dimensional \textsl{defect} TQFTs, of which the closed and open/closed flavours are special cases: now bordisms may be decomposed into decorated components which are separated by labelled 1-dimensional submanifolds (called `defects'), for example: 
%\vspace{-1.1cm}
%\belowdisplayskip=-12pt
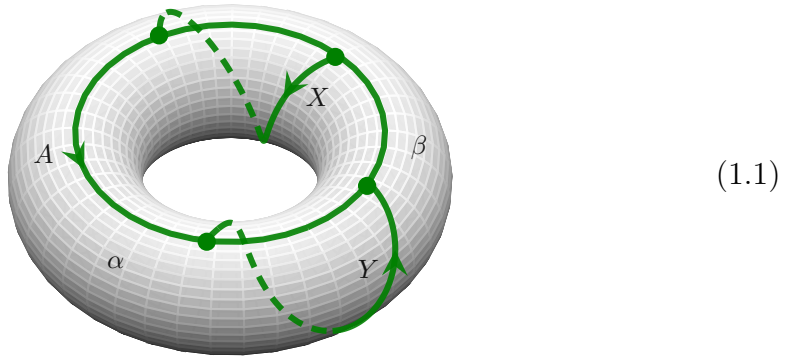
\be\label{eq:decoratedtorus}
%%%%%%%%%%%%%%%%%%%%%% 
\begin{tikzpicture}[scale=1.9, fill opacity=0.9, baseline=3.5cm]
  \begin{axis}[
      hide axis,
      view={60}{45},
      axis equal image,
    ]
    % torus: 
    \addplot3 [
      surf,
      shader=flat,
%      point meta=x, % comment out
      colormap/blackwhite, %violet, blackwhite
      samples=42,
      samples y=42,
      z buffer=sort,
      domain=0:360,
      y domain=0:360
    ] (
              {(3.5 + 1.5*cos(y))*cos(x)},
              {(3.5 + 1.5*cos(y))*sin(x)},
              {1.5*sin(y)});
    % horizontal circle: 
    \addplot3 [
     samples=42,
     samples y=1,
     domain=0:360,
      line width=1.2pt, 
      color=green!50!black,
      decoration={markings, mark=at position 0.7 with {\arrow{>}}}, 
      postaction={decorate},
      >=stealth
    ] (
              {(3.5 + 1.5*cos(90))*cos(x)},
              {(3.5 + 1.5*cos(90))*sin(x)},
              {1.5*sin(90)});
    % segment 1.9 
    \addplot3 [
      samples=72,
      samples y=1,
      domain=1530:1550,%-47:90,
      line width=1.2pt, 
      color=green!50!black
    ] (
              {(3.5 + 1.5*cos(x))*cos(15 + 0.2 * x)},
              {(3.5 + 1.5*cos(x))*sin(15 + 0.2 * x)},
              {1.5*sin(x)});
    % segment 1.10 
    \addplot3 [
      samples=72,
      samples y=1,
      domain=1550:1740,%-47:90,
      line width=1.2pt, 
      densely dashed,
      color=green!50!black
    ] (
              {(3.5 + 1.5*cos(x))*cos(15 + 0.2 * x)},
              {(3.5 + 1.5*cos(x))*sin(15 + 0.2 * x)},
              {1.5*sin(x)});
    % segment 1.10.5 
    \addplot3 [
      samples=72,
      samples y=1,
      domain=1740:1890,%-47:90,
      line width=1.2pt, 
      decoration={markings, mark=at position 0.6 with {\arrow{>}}}, postaction={decorate},
      >=stealth,
      color=green!50!black
    ] (
              {(3.5 + 1.5*cos(x))*cos(15 + 0.2 * x)},
              {(3.5 + 1.5*cos(x))*sin(15 + 0.2 * x)},
              {1.5*sin(x)});
              
    % segment 1.11 
    \addplot3 [
      samples=72,
      samples y=1,
      domain=450:590,%-47:90,
      line width=1.2pt, 
      decoration={markings, mark=at position 0.6 with {\arrow{>}}}, postaction={decorate},
      >=stealth,
      color=green!50!black
    ] (
              {(3.5 + 1.5*cos(x))*cos(15 + 0.2 * x)},
              {(3.5 + 1.5*cos(x))*sin(15 + 0.2 * x)},
              {1.5*sin(x)});
    % segment 1.12 
    \addplot3 [
      samples=72,
      samples y=1,
      domain=590:780,%-47:90,
      line width=1.2pt, 
      densely dashed,
      color=green!50!black
    ] (
              {(3.5 + 1.5*cos(x))*cos(15 + 0.2 * x)},
              {(3.5 + 1.5*cos(x))*sin(15 + 0.2 * x)},
              {1.5*sin(x)});
    % segment 1.13 
    \addplot3 [
      samples=72,
      samples y=1,
      domain=780:810,%-47:90,
      line width=1.2pt, 
      color=green!50!black
    ] (
              {(3.5 + 1.5*cos(x))*cos(15 + 0.2 * x)},
              {(3.5 + 1.5*cos(x))*sin(15 + 0.2 * x)},
              {1.5*sin(x)});
\end{axis}
%
% vertices: 
%\fill[color=green!50!black, opacity=1] (1.93,1.44) circle (1.8pt) node[black, below] (0up) {{\footnotesize $\Phi$}};
%\fill[color=green!50!black, opacity=1] (1.60,2.88) circle (1.8pt) node[black, below] (0up) {{\footnotesize $\Psi$}};
%\fill[color=green!50!black, opacity=1] (2.82,2.73) circle (1.8pt) node[black, above] (0up) {{\footnotesize $\Xi$}};
%\fill[color=green!50!black, opacity=1] (3.04,1.83) circle (1.8pt) node[black, left] (0up) {{\footnotesize $\Gamma$}};
\fill[color=green!50!black, opacity=1] (1.93,1.44) circle (1.8pt) node[black, below] (0up) {};
\fill[color=green!50!black, opacity=1] (1.60,2.88) circle (1.8pt) node[black, below] (0up) {};
\fill[color=green!50!black, opacity=1] (2.82,2.73) circle (1.8pt) node[black, above] (0up) {};
\fill[color=green!50!black, opacity=1] (3.04,1.83) circle (1.8pt) node[black, left] (0up) {};
%
% without pgfsettings: 
%\fill[color=green!50!black, opacity=1] (2.535, 1.895) circle (1.8pt) node[black, below] (0up) {};
%\fill[color=green!50!black, opacity=1] (4.04, 2.42) circle (1.8pt) node[black, below] (0up) {};
%\fill[color=green!50!black, opacity=1] (2.12, 3.78) circle (1.8pt) node[black, above] (0up) {};
%\fill[color=green!50!black, opacity=1] (3.74, 3.62) circle (1.8pt) node[black, left] (0up) {};
%
% labels: 
\draw[line width=1] (1.3,1.3) node[line width=0pt] (beta) {{\footnotesize $\alpha$}};
\draw[line width=1] (3.4,2.1) node[line width=0pt] (beta) {{\footnotesize $\beta$}};
\draw[line width=1] (0.8,2.05) node[line width=0pt] (beta) {{\footnotesize $A$}};
\draw[line width=1] (2.7,2.45) node[line width=0pt] (beta) {{\footnotesize $X$}};
\draw[line width=1] (3.05,1.25) node[line width=0pt] (beta) {{\footnotesize $Y$}};
\end{tikzpicture}
%%%%%%%%%%%%%%%%%%%%%% 
\ee
One finds that every 2-dimensional defect TQFT naturally gives rise to a `pivotal 2-category' \cite{dkr1107.0495}, which is a 2-category with very well-behaved duals. 
Examples of such 2-categories derived from TQFTs include those of smooth projective varieties and Fourier-Mukai kernels (also known as `B-twisted sigma models') \cite{cw1007.2679}, 
symplectic manifolds and Lagrangian correspondences (`A-twisted sigma models') \cite{WehrheimFFP}, 
or isolated singularities and matrix factorisations (`affine Landau-Ginzburg models') \cite{cm1208.1481}. 

\medskip

In principle, the notion of defect TQFT generalises to arbitrary dimension: an $n$-dimensional defect TQFT should be a symmetric monoidal functor
$$
\Bord_n^{\textrm{def}} \lra \Vect_\Bbbk
$$
where $\Bord_n^{\textrm{def}}$ is some suitably augmented version of $\Bord_n$ that also allows decorated submanifolds of various codimensions. 
One generally expects such TQFTs to be described by $n$-categories with additional structure \cite{k1004.2307}, but the details have not been worked out for $n>2$. 
Despite the naturality of the notion, not even a precise definition of defect TQFT has appeared in the literature for $n>2$. 
In the present paper we set out to remedy both issues for $n=3$. 

Higher categories also feature prominently in `extended' TQFT, in which decompositions along lower-dimensional subspaces promote $\Bord_n$ itself to a higher category. 
This approach has received increased focus following Lurie's influential work on the cobordism hypothesis \cite{l0905.0465}. 
Among the recent results here are that fully extended 3-dimensional TQFTs valued in the $(\infty,3)$-category of monoidal categories are classified by spherical fusion categories \cite{dsps1312.7188}, that 3-2-1 extended TQFTs valued in 2-Vect are classified by modular tensor categories \cite{bdspv1509.06811}, and the work on boundary conditions and surface defects in \cite{ks1012.0911, fsv1203.4568}. 

The crucial difference to our approach is that for us a defect TQFT is an ordinary symmetric monoidal functor, and everything that leads to higher-categorical structures is contained in an augmented (yet otherwise ordinary) bordism category. 
On the other hand, the approach of extended TQFT is to assume higher categories and higher functors from the outset. 
In this sense (which we further discuss in the bulk of the paper) our approach is more minimalistic, and possibly broader. 

We find that the most conceptual and economical way of systematically studying 3-dimensional defect TQFT is through the notion of `defect bordisms' which we will introduce in Section~\ref{sec:bordisms}.  
Roughly, defect bordisms are bordisms that come with a \textsl{stratification} whose 1-, 2-, and 3-strata\footnote{A priori we do not allow decorated 0-strata in the interior of a defect bordism, as we will see in Section~\ref{sec:tricatfromTQFT} how such data can be obtained from a defect TQFT.} are decorated by a choice of label sets~$\D$, as for example the cubes depicted below.  
Defect bordisms form the morphisms in a symmetric monoidal category $\Bordd$, whose objects are compatibly decorated stratified closed surfaces such as~\eqref{eq:decoratedtorus}. 
In Section~\ref{sec:tricatfromTQFT} we shall define a 3-dimensional defect TQFT to be a symmetric monoidal functor
\be\label{eq:ZdefectTQFT1}
\zz: \Bordd \lra \Vect_\Bbbk \, . 
\ee
 
\medskip

Our main result is the categorification of the result that a 2-dimensional defect TQFT gives rise to a pivotal 2-category. 
To state it precisely we recall that while every weak 2-category is equivalent to a strict 2-category, the analogous statement is not true in the 3-dimensional case \cite{GPS}. 
As we will review in Section~\ref{subsec:free2andGray}, the generically strictest version of a 3-category is a so-called `Gray category'. 
It should hence come as no surprise that in Sections~\ref{subsec:tricatfromZ} and~\ref{subsec:GraycatwithdualsfromZ} we will prove: 

\begin{theorem}
Every 3-dimensional defect TQFT $\zz: \Bordd \to \Vect_\Bbbk$ gives rise to a $\Bbbk$-linear Gray category with duals $\tz$. 
\end{theorem}

The Gray category $\tz$ is an invariant associated to $\zz$ which measures among other things how closed 3-dimensional TQFTs associated to unstratified bordisms are `glued together'. 
If there are no labels for 2-dimensional defects in the chosen decoration data~$\D$, then $\tz$ reduces to a product of ordinary categories. 

Very roughly, the idea behind the construction of $\tz$ is to directly transport the decorated structure of bordisms in $\Bordd$ into the graphical calculus for Gray categories with duals, developed in \cite{BMS}. 
This leads to objects~$u$, 1-morphisms $\alpha: u \to v$, and 2-morphisms $X: \beta \to \gamma$ being decorated cubes such as
$$
%%%%%%%%%%%%%%%%%%%%%% 
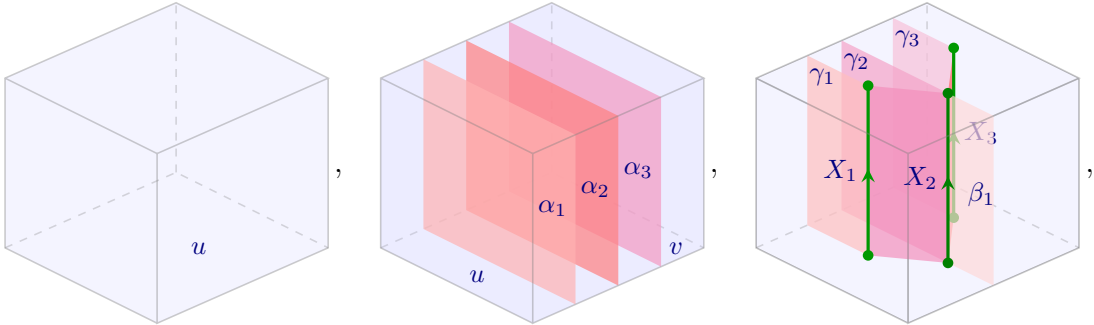
\begin{tikzpicture}[thick,scale=2.5,color=blue!50!black, baseline=0.0cm, >=stealth, 
				style={x={(-0.9cm,-0.4cm)},y={(0.8cm,-0.4cm)},z={(0cm,0.9cm)}}]
% 3-stratum: 
\fill [blue!20,opacity=0.2] (1,0,0) -- (1,1,0) -- (0,1,0) -- (0,1,1) -- (0,0,1) -- (1,0,1);
% invisible edges of cube: 
\draw[
	 color=gray, 
	 opacity=0.3, 
	 semithick,
	 dashed
	 ] 
	 (1,0,0) -- (0,0,0) -- (0,1,0)
	 (0,0,0) -- (0,0,1);
% object label: 
\draw[line width=1] (0.4, 0.6, 0) node[line width=0pt] (beta) {{\footnotesize $u$}};
% visible edges of cube
\draw[
	 color=gray, 
	 opacity=0.4, 
	 semithick
	 ] 
	 (0,1,1) -- (0,1,0) -- (1,1,0) -- (1,1,1) -- (0,1,1) -- (0,0,1) -- (1,0,1) -- (1,0,0) -- (1,1,0)
	 (1,0,1) -- (1,1,1);
\end{tikzpicture}
%%%%%%%%%%%%%%%%%%%%%% 
\, , \quad
%%%%%%%%%%%%%%%%%%%%%% 
\begin{tikzpicture}[thick,scale=2.5,color=blue!50!black, baseline=0.0cm, >=stealth, 
				style={x={(-0.9cm,-0.4cm)},y={(0.8cm,-0.4cm)},z={(0cm,0.9cm)}}]
% 3-stratum: 
\fill [blue!20,opacity=0.2] (1,0,0) -- (1,1,0) -- (0,1,0) -- (0,1,1) -- (0,0,1) -- (1,0,1);
% invisible edges of cube: 
\draw[
	 color=gray, 
	 opacity=0.3, 
	 semithick,
	 dashed
	 ] 
	 (1,0,0) -- (0,0,0) -- (0,1,0)
	 (0,0,0) -- (0,0,1);
%%%
% 3-stratum: 
\fill [blue!20,opacity=0.2] (1,0,0) -- (1,1,0) -- (0,1,0) -- (0,1,1) -- (0,0,1) -- (1,0,1);
\coordinate (b1) at (0.25, 0, 0);
\coordinate (b2) at (0.5, 0, 0);
\coordinate (b3) at (0.75, 0, 0);
\coordinate (b4) at (0.25, 1, 0);
\coordinate (b5) at (0.5, 1, 0);
\coordinate (b6) at (0.75, 1, 0);
\coordinate (t1) at (0.25, 0, 1);
\coordinate (t2) at (0.5, 0, 1);
\coordinate (t3) at (0.75, 0, 1);
\coordinate (t4) at (0.25, 1, 1);
\coordinate (t5) at (0.5, 1, 1);
\coordinate (t6) at (0.75, 1, 1);
%
% planes: 
\fill [magenta!50,opacity=0.7] (b1) -- (b4) -- (t4) -- (t1);
\fill [red!50,opacity=0.7] (b2) -- (b5) -- (t5) -- (t2);
\fill [red!30,opacity=0.7] (b3) -- (b6) -- (t6) -- (t3);
%
% labels: 
\draw[line width=1] (0.88, 0.5, 0) node[line width=0pt] (beta) {{\footnotesize $u$}};
\draw[line width=1] (0.09, 0.92, 0) node[line width=0pt] (beta) {{\footnotesize $v$}};
\draw[line width=1] (0.75, 0.85, 0.5) node[line width=0pt] (beta) {{\footnotesize $\alpha_1$}};
\draw[line width=1] (0.5, 0.85, 0.5) node[line width=0pt] (beta) {{\footnotesize $\alpha_2$}};
\draw[line width=1] (0.25, 0.85, 0.5) node[line width=0pt] (beta) {{\footnotesize $\alpha_3$}};
%
% visible edges of cube
\draw[
	 color=gray, 
	 opacity=0.4, 
	 semithick
	 ] 
	 (0,1,1) -- (0,1,0) -- (1,1,0) -- (1,1,1) -- (0,1,1) -- (0,0,1) -- (1,0,1) -- (1,0,0) -- (1,1,0)
	 (1,0,1) -- (1,1,1);
\end{tikzpicture}
%%%%%%%%%%%%%%%%%%%%%% 
\, , \quad
%%%%%%%%%%%%%%%%%%%%%% 
\begin{tikzpicture}[thick,scale=2.5,color=blue!50!black, baseline=0.0cm, >=stealth, 
				style={x={(-0.9cm,-0.4cm)},y={(0.8cm,-0.4cm)},z={(0cm,0.9cm)}}]
% visible edges of cube
\draw[
	 color=gray, 
	 opacity=0.4, 
	 semithick
	 ] 
	 (0,1,1) -- (0,1,0) -- (1,1,0) -- (1,1,1) -- (0,1,1) -- (0,0,1) -- (1,0,1) -- (1,0,0) -- (1,1,0)
	 (1,0,1) -- (1,1,1);
% to cut off unwanted edges: 
\clip (1,0,0) -- (1,1,0) -- (0,1,0) -- (0,1,1) -- (0,0,1) -- (1,0,1);
% 3-stratum: 
\fill [blue!20,opacity=0.2] (1,0,0) -- (1,1,0) -- (0,1,0) -- (0,1,1) -- (0,0,1) -- (1,0,1);
% invisible edges of cube: 
\draw[
	 color=gray, 
	 opacity=0.3, 
	 semithick,
	 dashed
	 ] 
	 (1,0,0) -- (0,0,0) -- (0,1,0)
	 (0,0,0) -- (0,0,1);
%%%
% top vertices: 
\coordinate (topleft2) at (0.5, 1, 1);
\coordinate (bottomleft2) at (0.5, 1, 0);
\coordinate (topright1) at (0.2, 0, 1);
\coordinate (topright2) at (0.5, 0, 1);
\coordinate (topright3) at (0.7, 0, 1);
\coordinate (bottomright1) at (0.2, 0, 0);
\coordinate (bottomright2) at (0.5, 0, 0);
\coordinate (bottomright3) at (0.7, 0, 0);
\coordinate (b1s) at (0.2, 0.4, 0);
\coordinate (b2s) at (0.5, 0.7, 0);
\coordinate (b3s) at (0.7, 0.4, 0);
\coordinate (t1s) at (0.2, 0.4, 1);
\coordinate (t2s) at (0.5, 0.7, 1);
\coordinate (t3s) at (0.7, 0.4, 1);
\draw[line width=1] (0.2, 0.4, 0.5) node[line width=0pt, right] (beta) {{\footnotesize $X_3$}};
%
% rear plane: 
\fill [magenta!30,opacity=0.7] 
(topright1) -- (bottomright1) -- (b1s) -- (t1s);
%
% back tilted plane: 
\fill [red!60,opacity=0.7] 
(t1s) -- (b1s) -- (b2s) -- (t2s);
%
% X1:
\draw[color=green!60!black, 
	postaction={decorate}, 
	decoration={markings,mark=at position .51 with {\arrow[color=green!60!black]{>}}}, 
	very thick]%
(b1s) -- (t1s);
\fill[color=green!60!black] (b1s) circle (0.8pt) node[left] (0up) {};
%
% middle plane: 
\fill [magenta!45,opacity=0.7] 
(topright2) -- (bottomright2) -- (b2s) -- (t2s);
%
% front tilted plane: 
\fill [magenta!55,opacity=0.7] 
(t3s) -- (b3s) -- (b2s) -- (t2s);
%
% middle left plane: 
\fill [red!15,opacity=0.7] 
(topleft2) -- (bottomleft2) -- (b2s) -- (t2s);
%
% X2:
\draw[color=green!50!black, 
	postaction={decorate}, 
	decoration={markings,mark=at position .51 with {\arrow[color=green!50!black]{>}}}, 
	very thick]%
(b2s) -- (t2s);
\fill[color=green!50!black] (b2s) circle (0.8pt) node[left] (0up) {};
%
% front plane: 
\fill [red!25,opacity=0.7] 
(topright3) -- (bottomright3) -- (b3s) -- (t3s);
%
% X3:
\draw[color=green!60!black, 
	postaction={decorate}, 
	decoration={markings,mark=at position .51 with {\arrow[color=green!60!black]{>}}}, 
	very thick]%
(b3s) -- (t3s);
\draw[line width=1] (0.7, 0.4, 0.5) node[line width=0pt, left] (beta) {{\footnotesize $X_1$}};
\draw[line width=1] (0.5, 0.7, 0.5) node[line width=0pt, left] (beta) {{\footnotesize $X_2$}};
\draw[line width=1] (0.2, 0.1, 0.92) node[line width=0pt] (beta) {{\footnotesize $\gamma_3$}};
\draw[line width=1] (0.5, 0.1, 0.92) node[line width=0pt] (beta) {{\footnotesize $\gamma_2$}};
\draw[line width=1] (0.7, 0.1, 0.92) node[line width=0pt] (beta) {{\footnotesize $\gamma_1$}};
\draw[line width=1] (0.5, 0.92, 0.5) node[line width=0pt] (beta) {{\footnotesize $\beta_1$}};
% top 0-strata: 
\fill[color=green!60!black] (t1s) circle (0.8pt) node[left] (0up) {};
\fill[color=green!50!black] (t2s) circle (0.8pt) node[left] (0up) {};
\fill[color=green!60!black] (t3s) circle (0.8pt) node[left] (0up) {};
\fill[color=green!60!black] (b3s) circle (0.8pt) node[left] (0up) {};
%
% visible edges of cube
\draw[
	 color=gray, 
	 opacity=0.4, 
	 semithick
	 ] 
	 (0,1,1) -- (0,1,0) -- (1,1,0) -- (1,1,1) -- (0,1,1) -- (0,0,1) -- (1,0,1) -- (1,0,0) -- (1,1,0)
	 (1,0,1) -- (1,1,1);
\end{tikzpicture}
%%%%%%%%%%%%%%%%%%%%%% 
\, , 
$$
respectively, where we will extract the labels $u,v,\alpha_i, \beta_j, \gamma_k, X_l$ from the given decoration data~$\D$. 
It is only in the definition of 3-morphisms where the functor~$\zz$ is used, namely by evaluating it on suitably decorated spheres, cf.~Section~\ref{subsec:tricatfromZ}. 
(Accordingly, if one replaces $\Vect_\Bbbk$ by another symmetric monoidal category~$\mathcal C$ in~\eqref{eq:ZdefectTQFT1}, the Gray category $\tz$ is not $\Bbbk$-linear, but $\mathcal C$-enriched.) 
The three types of composition in $\tz$ correspond to stacking cubes along the $x$-, $y$-, and $z$-direction. 

There are two types of duals in $\tz$, corresponding to orientation reversal of surfaces (for 1-morphisms~$\alpha$) and lines (for 2-morphisms~$X$). 
In Section~\ref{subsec:GraycatwithdualsfromZ} we will see how these duals are exhibited by coevaluation 2- and 3-morphisms such as 
$$
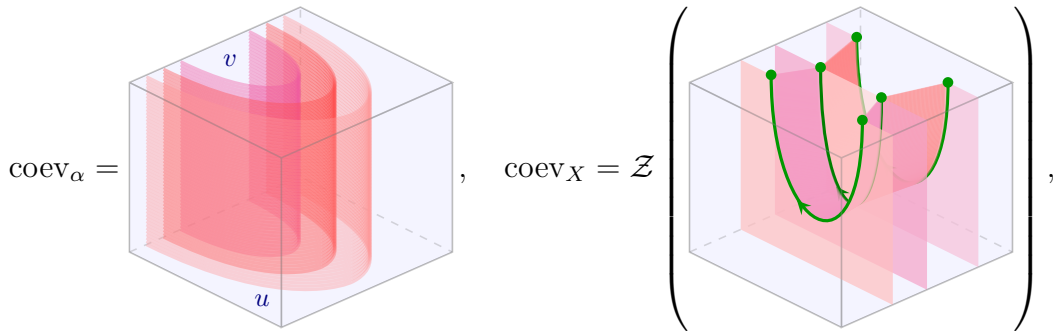
\coev_\alpha = 
%%%%%%%%%%%%%%%%%%%%%% 
\begin{tikzpicture}[thick,scale=2.5,color=blue!50!black, baseline=0.0cm, >=stealth, 
				style={x={(-0.9cm,-0.4cm)},y={(0.8cm,-0.4cm)},z={(0cm,0.9cm)}}]
% visible edges of cube
\draw[
	 color=gray, 
	 opacity=0.4, 
	 semithick
	 ] 
	 (0,1,1) -- (0,1,0) -- (1,1,0) -- (1,1,1) -- (0,1,1) -- (0,0,1) -- (1,0,1) -- (1,0,0) -- (1,1,0)
	 (1,0,1) -- (1,1,1);
% to cut off unwanted edges: 
\clip (1,0,0) -- (1,1,0) -- (0,1,0) -- (0,1,1) -- (0,0,1) -- (1,0,1);
% 3-stratum: 
\fill [blue!20,opacity=0.2] (1,0,0) -- (1,1,0) -- (0,1,0) -- (0,1,1) -- (0,0,1) -- (1,0,1);
% invisible edges of cube: 
\draw[
	 color=gray, 
	 opacity=0.3, 
	 semithick,
	 dashed
	 ] 
	 (1,0,0) -- (0,0,0) -- (0,1,0)
	 (0,0,0) -- (0,0,1);
%%%
% alpha1-coordinates: 
\coordinate (p1) at (0.1, 0, 0);
\coordinate (p2) at (0.9, 0, 0);
\coordinate (q1) at (0.1, 0, 1);
\coordinate (q2) at (0.9, 0, 1);
% alpha2-coordinates: 
\coordinate (p12) at (0.2, 0, 0);
\coordinate (p22) at (0.8, 0, 0);
\coordinate (q12) at (0.2, 0, 1);
\coordinate (q22) at (0.8, 0, 1);
% alpha3-coordinates: 
\coordinate (p13) at (0.3, 0, 0);
\coordinate (p23) at (0.7, 0, 0);
\coordinate (q13) at (0.3, 0, 1);
\coordinate (q23) at (0.7, 0, 1);
\coordinate (S1) at (0.5, 0.9, 0);
\coordinate (S2) at (0.5, 0.7, 0);
\coordinate (S3) at (0.5, 0.5, 0);
%
% alpha1-fold back:
\foreach \z in {0, 0.0194, ..., 1.0}
{
	\draw[color=red!40, very thick, opacity=0.4]%
	($(S1) + (0,0,\z)$) .. controls +(-0.02,0,0) and +(0,0.9,0) .. ($(p1) + (0,0,\z)$);
}
% alpha2-fold back:
\foreach \z in {0, 0.0194, ..., 1.0}
{
	\draw[color=red!70, very thick, opacity=0.4]%
	($(S2) + (0,0,\z)$) .. controls +(-0.02,0,0) and +(0,0.7,0) .. ($(p12) + (0,0,\z)$);
}
% alpha3-fold:
\foreach \z in {0, 0.0194, ..., 1.0}
{
	\draw[color=magenta!80, very thick, opacity=0.4]%
	($(S3) + (0,0,\z)$) .. controls +(-0.01,0,0) and +(0,0.5,0) .. ($(p13) + (0,0,\z)$);
	\draw[color=magenta!80, very thick, opacity=0.4]%
	($(p23) + (0,0,\z)$) .. controls +(0,0.5,0) and +(0.01,0,0) .. ($(S3) + (0,0,\z)$);
}
%
% alpha2-fold front:
\foreach \z in {0, 0.0194, ..., 1.0}
{
	\draw[color=red!70, very thick, opacity=0.4]%
	($(p22) + (0,0,\z)$) .. controls +(0,0.7,0) and +(0.02,0,0) .. ($(S2) + (0,0,\z)$);
}
%
% alpha1-fold front:
\foreach \z in {0, 0.0194, ..., 1.0}
{
	\draw[color=red!40, very thick, opacity=0.4]%
	($(p2) + (0,0,\z)$) .. controls +(0,0.9,0) and +(0.02,0,0) .. ($(S1) + (0,0,\z)$);
}
%
% object labels: 
\draw[line width=1] (0.89, 0.75, 0) node[line width=0pt] (beta) {{\footnotesize $u$}};
\draw[line width=1] (0.5, 0.1, 0.95) node[line width=0pt] (beta) {{\footnotesize $v$}};
%
% visible edges of cube
\draw[
	 color=gray, 
	 opacity=0.4, 
	 semithick
	 ] 
	 (0,1,1) -- (0,1,0) -- (1,1,0) -- (1,1,1) -- (0,1,1) -- (0,0,1) -- (1,0,1) -- (1,0,0) -- (1,1,0)
	 (1,0,1) -- (1,1,1);
\end{tikzpicture}
%%%%%%%%%%%%%%%%%%%%%% 
\, , \quad 
\coev_X = 
\zz\left(
%%%%%%%%%%%%%%%%%%%%%% 
\begin{tikzpicture}[thick,scale=2.5,color=blue!50!black, baseline=0.0cm, >=stealth, 
				style={x={(-0.9cm,-0.4cm)},y={(0.8cm,-0.4cm)},z={(0cm,0.9cm)}}]
% visible edges of cube
\draw[
	 color=gray, 
	 opacity=0.4, 
	 semithick
	 ] 
	 (0,1,1) -- (0,1,0) -- (1,1,0) -- (1,1,1) -- (0,1,1) -- (0,0,1) -- (1,0,1) -- (1,0,0) -- (1,1,0)
	 (1,0,1) -- (1,1,1);
% to cut off unwanted edges: 
\clip (1,0,0) -- (1,1,0) -- (0,1,0) -- (0,1,1) -- (0,0,1) -- (1,0,1);
% 3-stratum: 
\fill [blue!20,opacity=0.2] (1,0,0) -- (1,1,0) -- (0,1,0) -- (0,1,1) -- (0,0,1) -- (1,0,1);
% invisible edges of cube: 
\draw[
	 color=gray, 
	 opacity=0.3, 
	 semithick,
	 dashed
	 ] 
	 (1,0,0) -- (0,0,0) -- (0,1,0)
	 (0,0,0) -- (0,0,1);
%%%
% top vertices: 
\coordinate (topleft1) at (0.2, 0, 1);
\coordinate (topleft2) at (0.5, 0, 1);
\coordinate (topleft3) at (0.7, 0, 1);
\coordinate (bottomleft1) at (0.2, 0, 0);
\coordinate (bottomleft2) at (0.5, 0, 0);
\coordinate (bottomleft3) at (0.7, 0, 0);
\coordinate (topright1) at (0.2, 1, 1);
\coordinate (topright2) at (0.5, 1, 1);
\coordinate (topright3) at (0.7, 1, 1);
\coordinate (bottomright1) at (0.2, 1, 0);
\coordinate (bottomright2) at (0.5, 1, 0);
\coordinate (bottomright3) at (0.7, 1, 0);
\coordinate (t1) at (0.2, 0.2, 1);
\coordinate (t2) at (0.5, 0.3, 1);
\coordinate (t3) at (0.7, 0.2, 1);
\coordinate (t1s) at (0.2, 0.8, 1);
\coordinate (t2s) at (0.5, 0.7, 1);
\coordinate (t3s) at (0.7, 0.8, 1);
%
% rear plane: 
\fill [magenta!30,opacity=0.7] 
(topleft1) -- (t1) .. controls +(0,0,-0.95) and +(0,0,-0.95) .. (t1s) -- (topright1) -- (bottomright1) -- (bottomleft1);
% U in rear: 
\draw[color=green!60!black, 
	postaction={decorate}, 
	decoration={markings,mark=at position .51 with {\arrow[color=green!60!black]{<}}}, 
	very thick]%
(t1) .. controls +(0,0,-0.95) and +(0,0,-0.95) .. (t1s);
%
% curved surface in back:
\foreach \x in {0, 0.02, ..., 0.98}
{
	\draw[color=red!50, very thick, opacity=0.4]%
	($(t2) + \x*(t1) - \x*(t2)$) .. controls +(0,0,-0.95) and +(0,0,-0.95) .. ($(t2s) + \x*(t1s) - \x*(t2s)$);
}
%
% middle plane: 
\fill [magenta!45,opacity=0.7] 
(topleft2) -- (t2) .. controls +(0,0,-0.95) and +(0,0,-0.95) .. (t2s) -- (topright2) -- (bottomright2) -- (bottomleft2);
% middle U-plane: 
\fill [red!15,opacity=0.7] 
(t2) .. controls +(0,0,-0.95) and +(0,0,-0.95) .. (t2s);
% U in middle: 
\draw[color=green!50!black, 
	postaction={decorate}, 
	decoration={markings,mark=at position .51 with {\arrow[color=green!50!black]{<}}}, 
	very thick]%
(t2) .. controls +(0,0,-0.95) and +(0,0,-0.95) .. (t2s);
%
% curved surface in front:
\foreach \x in {0.06, 0.08, ..., 1.0}
{
	\draw[color=magenta!50, very thick, opacity=0.4]%
	($(t2) + \x*(t3) - \x*(t2)$) .. controls +(0,0,-0.95) and +(0,0,-0.95) .. ($(t2s) + \x*(t3s) - \x*(t2s)$);
}
%
% front plane: 
\fill [red!25,opacity=0.7] 
(topleft3) -- (t3) .. controls +(0,0,-0.95) and +(0,0,-0.95) .. (t3s) -- (topright3) -- (bottomright3) -- (bottomleft3);
% U in front: 
\draw[color=green!60!black, 
	postaction={decorate}, 
	decoration={markings,mark=at position .51 with {\arrow[color=green!60!black]{<}}}, 
	very thick]%
(t3) .. controls +(0,0,-0.95) and +(0,0,-0.95) .. (t3s);
%
% top 0-strata: 
\fill[color=green!60!black] (t1) circle (0.8pt) node[left] (0up) {};
\fill[color=green!60!black] (t2) circle (0.8pt) node[left] (0up) {};
\fill[color=green!60!black] (t3) circle (0.8pt) node[left] (0up) {};
\fill[color=green!60!black] (t1s) circle (0.8pt) node[left] (0up) {};
\fill[color=green!60!black] (t2s) circle (0.8pt) node[left] (0up) {};
\fill[color=green!60!black] (t3s) circle (0.8pt) node[left] (0up) {};
%
% visible edges of cube
\draw[
	 color=gray, 
	 opacity=0.4, 
	 semithick
	 ] 
	 (0,1,1) -- (0,1,0) -- (1,1,0) -- (1,1,1) -- (0,1,1) -- (0,0,1) -- (1,0,1) -- (1,0,0) -- (1,1,0)
	 (1,0,1) -- (1,1,1);
\end{tikzpicture}
%%%%%%%%%%%%%%%%%%%%%% 
\right)
, 
$$
respectively. 
These are the two distinct analogues of the coevaluation maps $%
%%%%%%%%%%%%%%%%%%%%%%
\begin{tikzpicture}[very thick, scale=0.39, color=blue!50!black, baseline=-0.2cm, rotate=180, >=stealth]
\draw[-{stealth[length=1.6mm, scale width=1.1]}]  
(2,0) .. controls +(0,1) and +(0,1) .. (3,0);
\end{tikzpicture} 
%%%%%%%%%%%%%%%%%%%%%%  
$ 
for duals in 2-dimensional string diagrams. 

\medskip

After having established the natural 3-categorical structure $\tz$ associated to a defect TQFT~$\zz$, we will present several examples of such~$\zz$ and~$\tz$ in Section~\ref{sec:examples}. 
Even the trivial TQFT (which assigns~$\Bbbk$ to everything in sight) leads to a non-trivial Gray category as we show in Section~\ref{subsec:case-trivial-defect}. 
In Section~\ref{subsec:resh-tura-as} we explain how the Reshetikhin-Turaev construction for a modular tensor category~$\mathcal C$ fits into our framework as a special class of defect TQFTs $\zz^{\mathcal C}$ with only 1-dimensional defects. 
We also prove that $\mathcal T_{\zz^{\mathcal C}}$ recovers~$\mathcal C$ viewed as a 3-category. 
Finally, in Section~\ref{subsec:htqfts-as-defect}, we shall utilise the work on homotopy quantum field theory in \cite{Hqft1} to construct two classes of 3-dimensional defect TQFTs  for every $G$-graded spherical fusion category. 
In particular we obtain interesting surface defects which allow us to extend Turaev-Viro theories to defect TQFTs. 

\medskip

We end this introduction by mentioning very briefly a number of applications and further developments of the work presented in this paper. 
(1) It would be interesting to compare our approach to the correspondence between TQFTs and higher categories developed in \cite{blobcomplex}; in particular one could expect that our construction of $\tz$ from~$\zz$ factors through a `disc-like 3-category'. 
(2) In our framework one can extend Reshetikhin-Turaev theory to produce not only invariants for 3-dimensional manifolds with embedded ribbons, but also with embedded surfaces. 
(3) In \cite{KR0909.3643} it was argued that Rozansky-Witten theory gives rise to a 3-category $\mathcal T^{\textrm{RW}}$ which brings together algebraic and symplectic geometry; it would be interesting to understand $\mathcal T^{\textrm{RW}}$ as a Gray category with duals. 
(4) There is a theory of `orbifold completion' of pivotal 2-categories \cite{cr1210.6363} which generalises certain group actions on 2-dimensional TQFTs via the algebraic language of defects, leading to new equivalences of categories \cite{CRCR, cqv2015}. The idea of orbifold completion applies to TQFTs of any dimension~$n$ and should be developed for $n=3$. It is expected that this will also give rise to new TQFTs. 
(5) 3-dimensional TQFT has applications in quantum computing, where orbifoldable group actions have received much attention, see e.\,g.~\cite{BJQ, fs1310.1329}. Our approach offers a robust and conceptual framework for generalisations of these constructions.

\subsubsection*{Acknowledgements} 

It is a pleasure to thank Ingo Runkel for helpful discussions and ideas. 
The work of N.\,C.~is partially supported by a grant from the Simons Foundation. 
N.\,C.~and G.\,S.~are partially supported by the stand-alone project P\,27513-N27 of the Austrian Science Fund.

\section{Bordisms}
\label{sec:bordisms}

For any $n\in\Z_+$, there is a category $\Bord_n$ of closed oriented $(n-1)$-dimensional manifolds and their bordisms \cite{Hirsch, Kockbook}. 
The objective of this section is to introduce a much bigger category of stratified and decorated bordisms, into which $\Bord_n$ embeds non-fully. 
For this purpose we briefly review and decide on a certain flavour of stratified manifolds with boundary, define stratified bordism categories $\Bordstrat_n$, and finally arrive at the stratified decorated (or `defect') bordism category $\Bordd$.

\subsection{Stratified bordisms}
\label{subsec:stratbord}

By a \textsl{filtration} of a topological space~$\Sigma$ we mean a finite sequence $\mathcal F = (\Sigma \supset F_m \supset F_{m-1} \supset \dots \supset F_0 \supset F_{-1} = \emptyset)$ of not necessarily closed subspaces $F_j \subset \Sigma$. We sometimes use the notation $\mathcal F=(F_j)$.
The only topological spaces we will consider are 
%arXiv_v2: 
	%smooth 
oriented manifolds (possibly with boundary). 
Hence from hereon we obey the following 

\medskip

\noindent
\textbf{Convention. }
By a ``manifold'' we mean 
%arXiv_v2: 
	%a ``smooth oriented real manifold''. 
	an ``oriented real (topological) manifold''. 
%arXiv_v2: 
	The boundary of a manifold with boundary is outward-oriented. 

\medskip

There are many different notions of stratification and stratified spaces in the literature,  such as  Whitney stratifications \cite{Whit1,Whit2},  stratified sets by Thom and Mather \cite{Thom, Mather}, and variants of these definitions in the piecewise-linear or topological context. 
In addition to their setting, these notions differ  in their choice of regularity conditions imposed on the strata, which then  ensure that the strata have neighbourhoods with certain properties, and that different strata can be glued smoothly. 

As we only use stratifications as a substrate to carry defect data, we work with a minimal definition  of stratification as a filtration by smooth manifolds that satisfies the `frontier condition' and the `finiteness condition' defined below. 
We circumvent the question about appropriate regularity conditions and gluing of strata by requiring certain normal forms around strata, see Section~\ref{subsec:defectbord}. 

\begin{definition}
\label{def:stratman}
An $m$-dimensional 
%arXiv_v2: 
	%\textsl{stratified manifold} 
	\textsl{stratified manifold (without boundary)}
is an $m$-dimensional 
%arXiv_v2: 
	%manifold~$\Sigma$ 
	manifold (without boundary)~$\Sigma$ 
together with a filtration $\mathcal F = (\Sigma = F_m \supset F_{m-1} \supset \dots \supset F_0 \supset F_{-1} = \emptyset)$ subject to the following conditions: 
\begin{enumerate}
\item
$\Sigma_j := F_j \setminus F_{j-1}$ is a $j$-dimensional 
%arXiv_v2: 
	smooth 
manifold (which may be empty) for all 
%arXiv_v2: 
	%$j \in \{1,2,\dots,m\}$; 
	$j \in \{0,1,\dots,m\}$; 
connected components of $\Sigma_j$ are denoted  $\Sigma_j^\alpha$ and are referred to as \textsl{$j$-strata}; each $j$-stratum is 
%arXiv_v2: 
	%equipped with a choice of orientation. 
	equipped with a choice of orientation (which for $j=m$ may differ from the orientation induced from~$\Sigma$). 
\item 
\textsl{Frontier condition:}
for all strata $\Sigma_i^\alpha, \Sigma_j^\beta$ with $\Sigma_i^\alpha \cap \overline{\Sigma}_j^\beta \neq \emptyset$, we have $\Sigma_i^\alpha \subset \overline{\Sigma}_j^\beta$ 
%arXiv_v2: 
	(where $\overline{\Sigma}_j^\beta$ denotes the closure of $\Sigma_j^\beta$). 
\item 
\textsl{Finiteness condition:} the total number of strata is finite. 
\end{enumerate}
\end{definition}

We often denote a stratified manifold $(\Sigma, \mathcal F)$ simply by~$\Sigma$. 
For the set of $j$-strata we write $S_j$, so we have the topological decomposition 
$$
\Sigma = 
%arXiv_v2: 
	%\bigcup_{j=1}^m 
	\bigcup_{j=0}^m
\, \bigcup_{\Sigma_j^\alpha \in S_j} \Sigma_j^\alpha \, . 
$$

Every 
%arXiv_v2: 
	smooth 
$m$-dimensional manifold~$\Sigma$ is trivially a stratified manifold when equipped with the filtration $\Sigma = F_m \supset \emptyset \supset \dots \supset \emptyset$. 
In this case there is a single stratum, namely $\Sigma = F_m = \Sigma_m$. 
A less trivial example is the unit disc $\Sigma = D_2 := \{ x \in \R^2 \,|\, |x| < 1 \}$ with three 2-strata, three 1-strata, and one (positively oriented) 0-stratum as follows: 
$$
%%%%%%%%%%%%%%%%%%%%%%
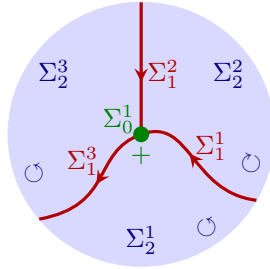
\begin{tikzpicture}[very thick,scale=1.0,color=blue!50!black, baseline=0.38cm]
\filldraw[blue!15, line width=0] 
(0,0) circle (1.75);
%1-strata:  
\draw[
	color=red!70!black, 
	>=stealth,
	decoration={markings, mark=at position 0.5 with {\arrow{>}},
					}, postaction={decorate}
	] 
 (0,0) .. controls (200:0.7) and (250:1) .. (220:1.75);
 \draw[
	color=red!70!black, 
	>=stealth,
	decoration={markings, mark=at position 0.5 with {\arrow{<}},
					}, postaction={decorate} 
	] 
 (0,0) .. controls (20:0.7) and (-50:1) .. (-30:1.75);
  \draw[
	color=red!70!black, 
	>=stealth,
	decoration={markings, mark=at position 0.5 with {\arrow{<}},
					}, postaction={decorate}
	] 
 (0,0) -- (90:1.75); 
\draw[line width=1, color=red!70!black] 
(0.29,0.8) node[line width=0pt] (Xbottom) {{\footnotesize $\Sigma_1^2$}}
(-10:0.93) node[line width=0pt] (Xbottom) {{\footnotesize $\Sigma_1^1$}}
(205:0.85) node[line width=0pt] (Xbottom) {{\footnotesize $\Sigma_1^3$}};
%
%0-stratum: 
\fill[color=green!50!black] (0,0) circle (3.0pt) node[below] {{\footnotesize $+$}};
\draw[line width=1] 
(-0.3,0.2) node[color=green!50!black, line width=0pt] (Xbottom) {{\footnotesize $\Sigma_0^1$}}; 
%
%2-strata: 
\draw[line width=1] 
(0,-1.4) node[line width=0pt] (Xbottom) {{\footnotesize $\Sigma_2^1$}}
(35:1.4) node[line width=0pt] (Xbottom) {{\footnotesize $\Sigma_2^2$}}
(145:1.4) node[line width=0pt] (Xbottom) {{\footnotesize $\Sigma_2^3$}}
(-55:1.5) node[line width=0pt] (Xbottom) {{\footnotesize $\circlearrowleft$}}
(200:1.5) node[line width=0pt] (Xbottom) {{\footnotesize $\circlearrowleft$}}
(-15:1.5) node[line width=0pt] (Xbottom) {{\footnotesize $\circlearrowright$}}
;
\end{tikzpicture}%
%%%%%%%%%%%%%%%%%%%%%% 
$$
Below we will sometimes suppress orientations in pictures of stratified manifolds. 

\begin{definition}
\label{def:stratmorph}
A \textsl{morphism} from an $m$-dimensional stratified manifold~$\Sigma$ to an $m'$-dimensional stratified manifold~$\Sigma'$ is a 
%arXiv_v2: 
	%smooth 
	continuous 
map $f: \Sigma \rightarrow \Sigma'$ such that $f(\Sigma_j) \subset \Sigma'_j$ and $f|_{\Sigma_j^\alpha}:\Sigma_j^\alpha\to {\Sigma'}_{j}^{\alpha'}$ are orientation-preserving 
%arXiv_v2: 
	smooth 
maps between strata for all 
%arXiv_v2: 
	%$j \in \{1,2,\dots,m\}$. 
	$j \in \{0,1,\dots,m\}$. 
The morphism is called an \textsl{embedding} if all the $f|_{\Sigma_j^\alpha}$ are diffeomorphisms onto their images. 
\end{definition}

We illustrate this definition with an example. Take for $\Sigma=S^2$ the trivially stratified 2-sphere and for $\Sigma'$ the 2-sphere $S^2$ stratified by two meridians and the north and the south pole. 
Then any map $f: \Sigma \to \Sigma'$ that sends the unit sphere to a point $p\in \Sigma'_2$ is a morphism of stratified manifolds. 

\medskip

Stratified manifolds and their morphisms form a category, however not the one we are ultimately interested in. 
Instead, for a fixed $n\in\Z_+$, we wish to define a category whose objects are $(n-1)$-dimensional stratified manifolds, and whose morphisms are `stratified bordisms'. 
To arrive at this notion we first consider manifolds with boundary~$M$ and their neat submanifolds as an intermediate step: a submanifold (with boundary) $N \subset M$ is called \textsl{neat} if 
%arXiv_v2: 
	%$\partial N = N \cap \partial M$, $N\cap \accentset{\circ}{M}\neq \emptyset$, and~$N$ meets the boundary $\partial M$ transversally (if $\partial N \neq \emptyset$), see for instance \cite{Hirsch}. Neatness will ensure that stratified bordisms can be glued together. 
	$\partial N = N \cap \partial M$ and $N\cap \accentset{\circ}{M}\neq \emptyset$. (In the smooth case, neatness typically also involves a transversality condition, see for instance \cite{Hirsch}; in the case of our stratified bordisms, we replace this condition by requiring appropriate collars for boundary parametrisations in the category $\Bordstrat_n$ to be defined on the next page.)

\begin{definition}
\label{def:stratmanbdry}
A $n$-dimensional \textsl{stratified manifold with boundary} is an $n$-dimensional manifold with boundary~$M$ together with a filtration $M = F_n \supset F_{n-1} \supset \dots \supset F_0 \supset F_{-1} = \emptyset$ such that: 
\begin{enumerate}
\item 
The interior~$\accentset{\circ}{M}$ together with the filtration 
%arXiv_v2: 
	%$(\accentset{\circ}{F}_j)$
	$(\accentset{\circ}{M} \cap F_j)$ 
is a stratified manifold. 
\item 
$M_j := F_j \setminus F_{j-1}$ is a neat $j$-dimensional 
%arXiv_v2: 
	smooth
submanifold for all 
%arXiv_v2: 
	%$j \in \{1,2,\dots,n\}$. 
	$j \in \{0,1,\dots,n\}$. 
%arXiv_v2: 
	Connected components of $M_j$ are denoted $M_j^\alpha$ and are referred to as \textsl{$j$-strata}; each $j$-stratum is equipped with a choice of orientation (which may differ from the orientation induced from~$M$). 
\item 
$\partial M$ together with the filtration 
%arXiv_v2: 
	%$(\partial M \cap F_j)$ 
	$(\partial M \cap F_{j+1})$ 
is an $(n-1)$-dimensional stratified manifold (with orientations induced by those of the strata in the $M_j$). 
\end{enumerate}
\end{definition}

As in the case without boundary, every manifold with boundary can be viewed as (trivially) stratified. 
A nontrivial example of a 2-dimensional stratified manifold with boundary is
$$
%%%%%%%%%%%%%%%%%%%%%% 
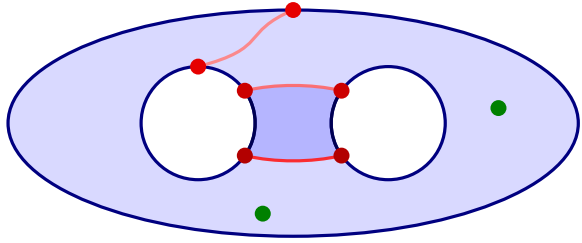
\begin{tikzpicture}[very thick,scale=1.0,color=blue!50!black]
%main/background colouring:
\fill[color=blue!15] (0,0) ellipse (3.75 and 1.5);
%tunnel colour: 
\fill [blue!30,opacity=0.9] ($(-1.25,0)+(-35:0.75)$) arc (-35:35:0.75) arc (120:60:1.25 and 0.5) -- ($(1.25,0)+(145:0.75)$) arc (145:215:0.75) arc (-60:-120:1.25 and 0.5); %magenta!30
%outer boundary:
\draw (0,0) ellipse (3.75 and 1.5);
%inner boundaries:
\fill[color=white] (1.25,0) circle (0.75);
\fill[color=white] (-1.25,0) circle (0.75);
\draw (-1.25,0) circle (0.75);
\draw (1.25,0) circle (0.75);
\draw [blue!30!black,opacity=0.9] ($(-1.25,0)+(-35:0.75)$) arc (-35:35:0.75); %magenta!90!black
\draw [blue!30!black,opacity=0.9] ($(1.25,0)+(215:0.75)$) arc (215:145:0.75); %magenta!90!black
%1-strata:
\draw[color=red!50,
	 opacity=0.9] 
	 ($(-1.25,0)+(90:0.75)$) .. controls (-0.3,1) and (-0.7,1.2) ..  (90:3.75 and 1.5);
\draw[color=red!90,
	 opacity=0.9] 
	 (240:1.25 and 0.5) arc (240:300:1.25 and 0.5);
\draw[color=red!60,
	 opacity=0.9] 
	 (120:1.25 and 0.5) arc (120:60:1.25 and 0.5);
%
%0-stratum on ellipse: 
\fill[color=red!90!black] (90:3.75 and 1.5) circle (3.0pt) node[above] {};
%0-strata on left circle: 
\fill[color=red!90!black] ($(-1.25,0)+(90:0.75)$) circle (3.0pt) node[above] {};
\fill[color=red!80!black] ($(-1.25,0)+(35:0.75)$) circle (3.0pt) node[above] {};
\fill[color=red!70!black] ($(-1.25,0)+(-35:0.75)$) circle (3.0pt) node[above] {};
%
%0-strata on right circle: 
\fill[color=red!80!black] ($(1.25,0)+(145:0.75)$) circle (3.0pt) node[above] {};
\fill[color=red!70!black] ($(1.25,0)+(215:0.75)$) circle (3.0pt) node[above] {};
%0-strata in interior: 
\fill[color=green!50!black] (-0.4,-1.2) circle (3.0pt) node[above] (0up) {};
\fill[color=green!50!black] (2.7,0.2) circle (3.0pt) node[above] (0up) {};
\end{tikzpicture}%
%%%%%%%%%%%%%%%%%%%%%% 
$$
where matching colours of interior strata and adjacent boundary strata indicate the induced filtration and orientations. 

\begin{definition}
A \textsl{morphism} from an $n$-dimensional stratified manifold with boundary~$M$ to an $n'$-dimensional stratified manifold with boundary~$M'$ is a 
%arXiv_v2: 
	%smooth 
	continuous 
map $f: M \rightarrow M'$ with $f(\partial M) \subset \partial M'$ such that
\begin{enumerate}
\item
$f(M_j) \subset M'_j$ and $f|_{M_j^\alpha}$ are 
%arXiv_v2: 
	smooth 
orientation-preserving maps between strata for all 
%arXiv_v2: 
	%$j \in \{1,2,\dots,n\}$, 
	$j \in \{0,1,\dots,n\}$,  
and 
\item 
$f|_{\partial M}: \partial M \rightarrow \partial M'$ is a morphism of stratified manifolds. 
\end{enumerate}
\end{definition}

\medskip

We may now define the category $\Bordstrat_n$ of stratified bordisms in analogy to the standard bordism category $\Bord_n$: Objects of $\Bordstrat_n$ are closed $(n-1)$-dimensional stratified manifolds~$\Sigma$. 
%arXiv_v2: 
	%A morphism $\Sigma \rightarrow \Sigma'$ is an equivalence class of compact $n$-dimensional stratified manifolds with boundary~$M$ together with orientation-preserving isomorphisms  of stratified manifolds  $\beta: \partial M \rightarrow \Sigma^{\textrm{rev}} \sqcup \Sigma'$, where $(-)^{\textrm{rev}}$ denotes orientation reversal. Here, $(M, \beta)$ is equivalent to $(\widetilde M, \widetilde \beta)$ iff there is an isomorphism of stratified manifolds with boundary $f: M \rightarrow \widetilde M$ such that $\widetilde\beta \circ f|_{\partial M} = \beta$. 
	A bordism $\Sigma \rightarrow \Sigma'$ is a compact stratified $n$-manifold $M$, together with an isomorphism $\Sigma^{\mathrm{rev}} \sqcup \Sigma' \to \partial M$ of (germs of collars around) stratified $(n-1)$-manifolds, where $\Sigma^{\mathrm{rev}}$ is~$\Sigma$ with reversed orientations for all strata. Two bordisms $M,N \colon \Sigma \to \Sigma'$ are equivalent if there is an isomorphism (of stratified manifolds with boundary, i.\,e.~a homeomorphism whose restrictions to strata are orientation-preserving diffeomorphisms) $M \to N$ compatible with the boundary parameterisations. The set of morphisms $\Hom_{\Bordstrat_n}(\Sigma,\Sigma')$ consists of the equivalence classes of bordisms $[M] \colon \Sigma \to \Sigma'$. We will often write~$M$ for~$[M]$. 

%arXiv_v2: 
	%Due to the transversality condition on the boundary, composition of morphisms in $\Bordstrat_n$ is defined by applying the standard construction in terms of collars (see for instance \cite{Hirsch}) in $\Bord_n$ finitely many times to the boundary strata. 
	The composition of morphisms in $\Bordstrat_n$ is defined by applying the standard construction in terms of collars (see for instance \cite{Hirsch}) in $\Bord_n$ finitely many times to the boundary strata.
Together with disjoint union this makes $\Bordstrat_n$ into a symmetric monoidal category, of which $\Bord_n$ is a  subcategory. 
However, it is not a full subcategory, due to the presence of strata that do not reach the boundary.

\subsection{Standard stratified bordisms} 
\label{subsec:defectbord}

We could now proceed to study a rich class of `stratified TQFTs' as symmetric monoidal functors $\Bordstrat_n \rightarrow \Vect_\Bbbk$. 
However for our purposes this notion is simultaneously too narrow and too broad. 
It is too narrow in the sense that we will consider additional structure in the form of `defect decorations' as discussed below. 
And it is too broad because from now on we shall make the following restrictive assumptions: 
\begin{enumerate}
\item
$n=3$. 
\item 
There are no 0-strata in the interior of bordisms. 
\item 
Every 1-stratum in a bordism has a neighbourhood which is of `standard form' (as explained below, compare~\eqref{eq:starlikecylinder}), and similarly for 0-strata of boundaries (compare~\eqref{eq:starlikedisc}). 
\end{enumerate}

\begin{remark}
The reason for condition (ii) is one of perspective: for us the fundamental entity will be a `defect TQFT' (cf.~Definition~\ref{def:defectTQFT}), which is a functor~$\zz$ on a certain 3-dimensional decorated bordism category $\Bordd$ (which is defined at the end of the present section). 
By removing small balls around potential 0-strata in the interior of a stratified bordism, we can always transform a bordism with inner 0-strata into one without such points by adding new sphere-shaped boundary components.  
We will see that the state spaces of the TQFT on these decorated spheres are natural label sets for the removed 0-strata. 
In this sense a set of coherent defect labels for 0-strata is already contained in $\zz$, and requiring 0-dimensional data from the start would lead to unnecessarily complicated consistency issues. 
(Our constructions in Sections~\ref{subsec:tricatfromZ} and~\ref{subsec:GraycatwithdualsfromZ} will make these statements precise.)

On the other hand, condition (iii) is merely one of convenience. We do not address the question whether all 1-strata in $\Bordstrat_3$ have neighbourhoods which are of standard form. 
We are exclusively interested in the latter as they encode the relevant combinatorial data of `surface operators' meeting at `line operators'. 
\end{remark}

Next we detail the standard forms of condition (iii) above, starting with the case of neighbourhoods around 0-strata in 2-dimensional stratified manifolds. 
The underlying manifold is the open disc $D_2 = \{ x \in \C \,|\, |x| < 1 \}$, where we identify $\R^2$ with the complex plane. 
For every 
%arXiv_v2: 
	%$p\in\N$ 
	$p\in\N = \Z_{\geqslant 0}$ 
and for every $\varepsilon \in \{ \pm \}^{\times (p+2)}$ we endow $\Sigma = D_2$ with a stratification given by
\begin{align*}
\Sigma_0 & = \{ 0 \} \, , \\
\Sigma_1 & = \bigcup_{\alpha=1}^{p} \underbrace{\Big\{ 
%arXiv_v2: 
	%r \E^{2\pi\I(\alpha-1)}
	r \E^{2\pi\I(\alpha-1)/p} 
\,\Big|\, 0<r<1 \,  \Big\}}_{=: \Sigma_1^\alpha} \, , \\
\Sigma_2 & =  \big( D_2 \setminus \Sigma_1 \big) \setminus \Sigma_0 
\end{align*}
where the orientations of the strata are encoded in the tuple~$\varepsilon$: the 1-stratum $\Sigma_1^\alpha$ is oriented away from the origin iff $\varepsilon_{\alpha} = +$ for all $\alpha \in \{ 1,2,\dots,p \}$, the 0-stratum $\Sigma_0$ has orientation $\varepsilon_{p+1}$, and all 2-strata have the same orientation $\varepsilon_{p+2}$. 
We denote this `star-like' stratified disc as $D_2^{p,\varepsilon}$. 
For example, we have 
\be\label{eq:starlikedisc}
D_2^{3, (+,+,-,-,-)} = 
%%%%%%%%%%%%%%%%%%%%%%
\begin{tikzpicture}[very thick,scale=0.7,color=blue!50!black, baseline=-0.1cm]
\filldraw[blue!15, line width=0] 
(0,0) circle (1.75);
%1-strata:  
\draw[
	color=red!70!black, 
	>=stealth,
	decoration={markings, mark=at position 0.5 with {\arrow{>}},
					}, postaction={decorate}
	] 
 (0,0) -- (0:1.75);
 \draw[
	color=red!70!black, 
	>=stealth,
	decoration={markings, mark=at position 0.5 with {\arrow{>}},
					}, postaction={decorate} 
	] 
 (0,0) -- (120:1.75);
  \draw[
	color=red!70!black, 
	>=stealth,
	decoration={markings, mark=at position 0.5 with {\arrow{<}},
					}, postaction={decorate}
	] 
 (0,0) -- (240:1.75); 
%
%0-stratum: 
\fill[color=green!50!black] (0,0) circle (3.5pt) node[above] {{\footnotesize $-$}};
%
%2-strata: 
\foreach \x in {1,...,3}
	\draw[line width=1] (120*\x - 60:1.4) node[line width=0pt] (Xbottom) {{\footnotesize $\circlearrowright$}};
\end{tikzpicture}%
%%%%%%%%%%%%%%%%%%%%%% 
\, .
\ee

The standard neighbourhood around 1-strata in 3-dimensional stratified bordisms is set to be a cylinder over some star-like disc $D_2^{p,\varepsilon}$. 
Since these 1-strata may have nonempty boundary, there are actually four types of cylinders, namely $D_2^{p,\varepsilon} \times I$ for~$I$ one of the intervals $[0,1]$, $[0,1)$, $(0,1]$, or $(0,1)$. 
%arXiv_v2: 
	The $j$-strata of $D_2^{p,\varepsilon} \times I$ are $s \times I$, where~$s$ ranges over the $(j-1)$-strata of $D_2^{p,\varepsilon}$ for $j \in \{1,2,3\}$; 
%arXiv_v2: 
	%Note that the definition of stratified manifold with boundary dictates the orientations of these stratified cylinders. 
	the 1-stratum of $D_2^{p,\varepsilon} \times I$ is oriented upwards iff $\varepsilon_{p+1}=+$, a 2-stratum $\Sigma_1^\alpha \times I$ has the orientation induced from the standard orientation of~$\R^3$ iff $\varepsilon_\alpha=+$, and the 3-strata all have this standard orientation iff $\varepsilon_{p+2}=+$. 
%arXiv_v2: 
	%They 
	The standard neighbourhoods around 1-strata 
look like `symmetrically opened books with~$p$ pages': 
\be\label{eq:starlikecylinder}
D_2^{3, (+,+,-,-,-)} \times I = 
%%%%%%%%%%%%%%%%%%%%%% 
\begin{tikzpicture}[thick,scale=1.0,color=blue!50!black, baseline=-1.9cm]
%cylinder:
\fill [blue!15,
      opacity=0.5, 
      left color=blue!15, 
      right color=white] (-1.25,0) -- (-1.25,-3.5) arc (180:360:1.25 and 0.5) -- (1.25,0) arc (0:180:1.25 and -0.5);
\fill [blue!35,opacity=0.5] (-1.25,-3.5) arc (180:360:1.25 and 0.5) -- (1.25,-3.5) arc (0:180:1.25 and 0.5);
\fill [blue!25,opacity=0.5] (-1.25,0) arc (180:360:1.25 and 0.5) -- (1.25,0) arc (0:180:1.25 and 0.5);
%
%2-strata: 
\fill [red,opacity=0.4] (0,0) -- (0,-3.5) -- (1.25,-3.5) -- (1.25,0);
\fill [red,opacity=0.4] (0,0) -- ($(0,0)+(120:1.25 and 0.5)$) -- ($(0,-3.5)+(120:1.25 and 0.5)$) -- (0,-3.5) -- (0,0);
\fill [red,opacity=0.4] (0,0) -- ($(0,0)+(245:1.25 and 0.5)$) -- ($(0,-3.5)+(245:1.25 and 0.5)$) -- (0,-3.5) -- (0,0);
%Note that the last 2-stratum above should really be at 240 degrees, but I warped space here to make it easier on the eye... 
%
%1-stratum: 
\draw[color=green!50!black, very thick] (0,0) -- (0,-3.5);  
%
%0-strata: 
\fill[color=green!50!black] (0,0) circle (2.5pt) node[above] (0up) {};% {{\footnotesize $-$}};
\fill[color=green!50!black] (0,-3.5) circle (2.5pt) node[below] (0down) {};% {{\footnotesize $-$}};
\end{tikzpicture}
%%%%%%%%%%%%%%%%%%%%%% 
\, .
\ee

Now we can be clear about which stratified bordisms we restrict to. 
To wit, we define $\Bords$ to be the subcategory of $\Bordstrat_3$ as follows: 
For every object $\Sigma \in \Bords$ and every 0-stratum in~$\Sigma$ there exists a chart, i.\,e.~an isomorphism of stratified manifolds with boundary, from a neighbourhood of the 0-stratum  to some star-like disc $D_2^{p,\varepsilon}$; 
morphisms~$M$ in $\Bords$ do not have 0-strata, around every 1-stratum in~$M$ there exists a chart to some cylinder $D_2^{p,\varepsilon} \times I$, 
%arXiv_v2: 
	and all 3-strata have the same orientation, induced by the orientation on~$M$. 

\medskip

\subsection{Decorated bordisms}
\label{subsec:decoratedbordisms}

Geometrically, $\Bords$ is the bordism category on which we want to study defect TQFTs -- however not `combinatorially'. 
Indeed, we will allow decorations of all strata according to certain rules, which are encoded in the notion of `defect data' in Definition \ref{def:defectdata} below. 

To motivate this notion, let~$M$ be a stratified bordism in $\Bords$, and recall that we write $S_j$ for the set of $j$-dimensional strata of~$M$. 
To every 2-stratum $M_2^\alpha \in S_2$ we can associate the unique pair of 3-strata 
%arXiv_v2: 
	%$M_3^\beta, M_3^\gamma \in S_3$ 
	$(M_3^\beta, M_3^\gamma) \in S_3 \times S_3$ 
whose closures both contain $M_2^\alpha$. 
Furthermore, orientations induce a sense of direction: by convention, we think of $M_3^\beta$ as the `source' (and $M_3^\gamma$ as the `target') of $M_2^\alpha$ iff an arrow from $M_3^\beta$ to $M_3^\gamma$ together with a positive frame on $M_2^\alpha$ is a positive frame in~$M$.\footnote{Here we use the fact that all 3-strata have the same orientation.} 
This defines a map 
$$
m_2: S_2 \lra S_3 \times S_3 
$$
which we immediately extend to a map $m_2: S_2 \times \{ \pm \} \rightarrow S_3 \times S_3$ by setting 
%arXiv_v2: 
	%$m_2(M_2^\alpha,-) = m_2(M_2^\alpha)^{\textrm{rev}}$. 
	$m_2(M_2^\alpha,+) = m_2(M_2^\alpha)$ and $m_2(M_2^\alpha,-) = m_2((M_2^\alpha)^{\textrm{rev}})$. 

From the stratified bordism~$M$ we can also extract a map $m_1$ with domain $S_1$, which keeps track of the neighbourhoods of 1-strata. 
Indeed, by choosing a small positive loop around a given 1-stratum and collecting the incident 2-strata together with their orientations in their cyclic order produces a map 
\be\label{eq:m1map}
m_1 : 
S_1 \lra S_3 \sqcup \bigsqcup_{m\in\Z_{+}} \underbrace{\big( (S_2 \times \{ \pm \}) \times \dots \times (S_2 \times \{ \pm \}) \big)/C_m}_{\text{cyclically ordered list of $m$ elements}} 
\ee
%arXiv_v2: 
	%where
	where~$C_m$ is the cyclic group of order~$m$, the second argument in $S_2 \times \{ \pm \}$ is~$+$ iff the orientation of the respective 2-stratum together with the tangent vector of the loop is a positive frame for~$M$, and 
an image in $S_3$ means that the chosen 1-stratum does not touch any 2-strata (`$m=0$'), but only a single ambient 3-stratum. 
(We remark that in the case $m>0$ in \eqref{eq:m1map} the 3-strata can be recovered from the 2-strata with the help of the map $m_2$.)

Now we can codify the data and rules for decorated stratified bordisms. 
We introduce three sets $D_1,D_2,D_3$ of labels for the 1-, 2-, 3-strata of~$M$, respectively, source and target maps $s,t:D_2\to D_3$ that correspond to the map $m_2: S_2\to S_3\times S_3$, 
and a `folding map'~$f$ that corresponds to the map $m_1$ in \eqref{eq:m1map}. Of these structures, only the folding map is substantially new when compared to the 2-dimensional case of \cite{dkr1107.0495}: 

\begin{definition}
\label{def:defectdata}
By \textsl{defect data} $\mathds D$ we mean a choice of
\begin{itemize}
\item 
three sets $D_3,D_2,D_1$; 
\item 
\textsl{source} and \textsl{target maps} $s,t: D_2 \rightarrow D_3$, which are extended to maps $s,t: D_2 \times \{ \pm \} \rightarrow D_3$ by  
\begin{align}
\!\!
  s(x,+) = s(x) \, , \quad
  t(x,+) = t(x) \, , \quad 
  s(x,-) = t(x) \, , \quad
  t(x,-) = s(x) \, ;
\end{align}
\item 
a \textsl{folding map}
$$
f: D_1 \lra D_3 \sqcup \bigsqcup_{m\in\Z_{+}} \big( (D_2 \times \{ \pm \}) \times \dots \times (D_2 \times \{ \pm \}) \big)/C_m
$$
satisfying the condition that every element $( (\alpha_1, \varepsilon_1), \dots,  (\alpha_m, \varepsilon_m))$ in 
%arXiv_v2: 
	%$f(D_1)$ 
	$f(D_1) \setminus D_3$ 
has the property
$$
%arXiv_v2: 
	%s(\alpha_{i-1},\varepsilon_{i-1}) = t (\alpha_i, \varepsilon_i)
	s(\alpha_{i+1},\varepsilon_{i+1}) = t (\alpha_i, \varepsilon_i)
\quad\text{for all }
i\in \{1,2,\ldots,m\} \mod m \, . 
$$
\end{itemize}
\end{definition}

Finally we are in a position to define the \textsl{($\mathds D$-decorated) defect bordism category} $\Bordd$ for given defect data~$\mathds D$. 
%arXiv_v2: 
	%Objects 
	Roughly, objects 
and morphisms in $\Bordd$ are those of $\Bords$ together with a decoration consistent with~$\mathds D$. 
	To give the precise definition, we first define a \textsl{$\mathds D$-decorated bordism} to be bordism~$M$ in $\Bords$ together with three \textsl{label maps} $\ell_j: S_j \rightarrow D_j$ for $j \in \{ 1,2,3 \}$ which are required to be consistent with~$\mathds D$ in the sense that the diagrams
$$
%%%%%%%%%%%%%%%%%%%%%%%
\begin{tikzpicture}[
			     baseline=(current bounding box.base), 
			     >=stealth,
			     descr/.style={fill=white,inner sep=3.5pt}, 
			     normal line/.style={->}
			     ] 
\matrix (m) [matrix of math nodes, row sep=3.5em, column sep=2.0em, text height=1.1ex, text depth=0.1ex] {%
S_2 &&& D_2
\\
S_3 \times S_3 &&& D_{3} \times D_{3} 
\\
};
\path[font=\footnotesize] (m-1-1) edge[->] node[auto] {$ \ell_2 $} (m-1-4);
\path[font=\footnotesize] (m-1-1) edge[->] node[left] {$ m_2 $} (m-2-1);
\path[font=\footnotesize] (m-1-4) edge[->] node[right] {$ (s,t) $} (m-2-4);
\path[font=\footnotesize] (m-2-1) edge[->] node[below] {$ \ell_3 \times \ell_3 $} (m-2-4);
\end{tikzpicture}
%%%%%%%%%%%%%%%%%%%%%%% 
$$
and
$$
%%%%%%%%%%%%%%%%%%%%%%%
\begin{tikzpicture}[
			     baseline=(current bounding box.base), 
			     >=stealth,
			     descr/.style={fill=white,inner sep=3.5pt}, 
			     normal line/.style={->}
			     ] 
\matrix (m) [matrix of math nodes, row sep=3.5em, column sep=2.0em, text height=1.1ex, text depth=0.1ex] {%
S_1 &&& S_3 \sqcup  \bigsqcup_{m\in\Z_{+}} \big( (S_2 \times \{ \pm \}) \times \dots \times (S_2 \times \{ \pm \}) \big)/C_m
\\
D_1 &&& D_3 \sqcup  \bigsqcup_{m\in\Z_{+}} \big( (D_2 \times \{ \pm \}) \times \dots \times (D_2 \times \{ \pm \}) \big)/C_m
\\
};
\path[font=\footnotesize] (m-1-1) edge[->] node[auto] {$ m_1 $} (m-1-4);
\path[font=\footnotesize] (m-1-1) edge[->] node[left] {$ \ell_1 $} (m-2-1);
\path[font=\footnotesize] (m-1-4) edge[->] node[right] {$ \ell_3 \sqcup (\ell_{2} \times 1) \times\dots\times (\ell_{2} \times 1) $} (m-2-4);
\path[font=\small] (m-2-1) edge[->] node[below] {$ f $} (m-2-4);
\end{tikzpicture}
%%%%%%%%%%%%%%%%%%%%%%% 
$$
commute. We say that a $j$-stratum $M_j^\alpha$ \textsl{is decorated by} $x\in D_j$ if $\ell_j(M_j^\alpha)=x$. Similarly, a \textsl{$\mathds D$-decorated surface} is an object~$\Sigma$ in $\Bords$ together with the structure of a $\mathds D$-decorated bordism on the cylinder $\Sigma \times [0,1]$, and we say that a $(j-1)$-stratum of~$\Sigma$ \textsl{is decorated by} $x\in D_j$ if the corresponding $j$-stratum in the cylinder is decorated by~$x$. A morphism $M \to N$ of $\mathds D$-decorated bordisms is a morphism of stratified bordisms such that every stratum in~$M$ carries the same decoration as the stratum in~$N$ into which it is mapped. 

Objects in $\Bordd$ are $\mathds D$-decorated surfaces. A morphism $\Sigma \to \Sigma'$ between two such objects is an equivalence class of $\mathds D$-decorated bordisms whose underlying bordisms represent morphisms between the underlying stratified surfaces in $\Bords$ subject to the condition that the labels of strata with boundary in~$M$ match with the labels of the corresponding strata in~$\Sigma$ or~$\Sigma'$. Here, two $\mathds D$-decorated bordisms are equivalent if there is an isomorphism of $\mathds D$-decorated bordisms between them which is compatible with the boundary parameterisations. $\Bordd$ is a symmetric monoidal category with composition, identities, monoidal and symmetric structure inherited from $\Bords$. 

%arXiv_v2: [added one sentence with two pictures]
	An example of an object in $\Bordd$ is the decorated 2-sphere 
$$
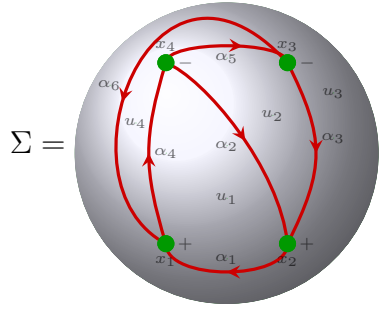
\Sigma = 
%%%%%%%%%%%%%%%%%%%%%%
\begin{tikzpicture}[very thick,scale=2,color=green!60!black=-0.1cm, >=stealth, baseline=0]
\clip (0,0) circle (1.0 cm);
\fill[ball color=white!95!blue] (0,0) circle (1.0 cm);
\coordinate (v1) at (-0.4,-0.6);
\coordinate (v2) at (0.4,-0.6);
\coordinate (v3) at (0.4,0.6);
\coordinate (v4) at (-0.4,0.6);
\draw[color=red!80!black, very thick, rounded corners=0.5mm, postaction={decorate}, decoration={markings,mark=at position .5 with {\arrow[draw=red!80!black]{>}}}] 
	(v2) .. controls +(0,-0.25) and +(0,-0.25) .. (v1);
\draw[color=red!80!black, very thick, rounded corners=0.5mm, postaction={decorate}, decoration={markings,mark=at position .62 with {\arrow[draw=red!80!black]{>}}}] 
	(v4) .. controls +(0,0.15) and +(0,0.15) .. (v3);
\draw[color=red!80!black, very thick, rounded corners=0.5mm, postaction={decorate}, decoration={markings,mark=at position .5 with {\arrow[draw=red!80!black]{>}}}] 
	(v4) .. controls +(0.25,-0.1) and +(-0.05,0.5) .. (v2);
\draw[color=red!80!black, very thick, rounded corners=0.5mm, postaction={decorate}, decoration={markings,mark=at position .58 with {\arrow[draw=red!80!black]{>}}}] 
	(v3) .. controls +(-0.9,0.99) and +(-0.75,0.4) .. (v1);
\draw[color=red!80!black, very thick, rounded corners=0.5mm, postaction={decorate}, decoration={markings,mark=at position .5 with {\arrow[draw=red!80!black]{>}}}] 
	(v1) .. controls +(-0.15,0.5) and +(-0.15,-0.5) .. (v4);
\draw[color=red!80!black, very thick, rounded corners=0.5mm, postaction={decorate}, decoration={markings,mark=at position .5 with {\arrow[draw=red!80!black]{>}}}] 
	(v3) .. controls +(0.25,-0.5) and +(0.25,0.5) .. (v2);
\fill (v1) circle (1.6pt) node[black, opacity=0.6, right, font=\tiny] { $+$ };
\fill (v1) circle (1.6pt) node[black, opacity=0.6, below, font=\tiny] { $x_1$ };
\fill (v2) circle (1.6pt) node[black, opacity=0.6, right, font=\tiny] { $+$ };
\fill (v2) circle (1.6pt) node[black, opacity=0.6, below, font=\tiny] { $x_2$ };
\fill (v3) circle (1.6pt) node[black, opacity=0.6, right, font=\tiny] { $-$ };
\fill (v3) circle (1.6pt) node[black, opacity=0.6, above, font=\tiny] { $x_3$ };
\fill (v4) circle (1.6pt) node[black, opacity=0.6, right, font=\tiny] { $-$ };
\fill (v4) circle (1.6pt) node[black, opacity=0.6, above, font=\tiny] { $x_4$ };
\fill (0,-0.7) circle (0pt) node[black, opacity=0.6, font=\tiny] { $\alpha_1$ };
\fill (0,0.05) circle (0pt) node[black, opacity=0.6, font=\tiny] { $\alpha_2$ };
\fill (0.7,0.1) circle (0pt) node[black, opacity=0.6, font=\tiny] { $\alpha_3$ };
\fill (-0.4,0) circle (0pt) node[black, opacity=0.6, font=\tiny] { $\alpha_4$ };
\fill (0,0.63) circle (0pt) node[black, opacity=0.6, font=\tiny] { $\alpha_5$ };
\fill (-0.77,0.45) circle (0pt) node[black, opacity=0.6, font=\tiny] { $\alpha_6$ };
\fill (0,-0.3) circle (0pt) node[black, opacity=0.6, font=\tiny] { $u_1$ };
\fill (0.3,0.25) circle (0pt) node[black, opacity=0.6, font=\tiny] { $u_2$ };
\fill (0.7,0.4) circle (0pt) node[black, opacity=0.6, font=\tiny] { $u_3$ };
\fill (-0.6,0.2) circle (0pt) node[black, opacity=0.6, font=\tiny] { $u_4$ };
\end{tikzpicture}
%%%%%%%%%%%%%%%%%%%%%% 
$$
where $u_i \in D_3$, $\alpha_j \in D_2$ and $x_k \in D_1$, and the decorated cylinder (with suppressed labels~$u_i$ for 3-strata) 
$$
%%%%%%%%%%%%%%%%%%%%%%
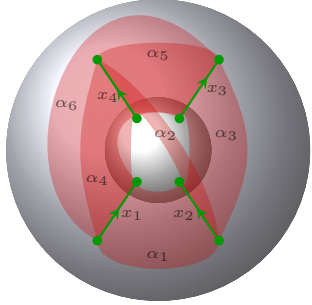
\begin{tikzpicture}[very thick,scale=2,color=green!60!black=-0.1cm, >=stealth, baseline=0]
\clip (0,0) circle (1.0 cm);
\fill[ball color=white!95!blue] (0,0) circle (1.0 cm);
\coordinate (v1) at (-0.4,-0.6);
\coordinate (v2) at (0.4,-0.6);
\coordinate (v3) at (0.4,0.6);
\coordinate (v4) at (-0.4,0.6);
\newcommand{\klei}{0.35}
\coordinate (w1) at (-0.4*\klei,-0.6*\klei);
\coordinate (w2) at (0.4*\klei,-0.6*\klei);
\coordinate (w3) at (0.4*\klei,0.6*\klei);
\coordinate (w4) at (-0.4*\klei,0.6*\klei);
	\fill[ball color=white] (0,0) circle (0.35 cm);
%
% 2-strata: 
\fill [red!80!black,opacity=0.3] (w2) -- (v2) .. controls +(0,-0.25) and +(0,-0.25) .. (v1) -- (w1) .. controls +(0,-0.25*\klei) and +(0,-0.25*\klei) .. (w2);
\fill [red!80!black,opacity=0.3] (w4) -- (v4) .. controls +(0,0.15) and +(0,0.15) .. (v3) -- (w3) .. controls +(0,0.15*\klei) and +(0,0.15*\klei) .. (w4);
\fill [red!80!black,opacity=0.3] (w4) -- (v4) .. controls +(0.25,-0.1) and +(-0.05,0.5) .. (v2) -- (w2) .. controls +(0.25*\klei,-0.1*\klei) and +(-0.05*\klei,0.5*\klei) .. (w4);
\fill [red!80!black,opacity=0.3] (w3) -- (v3) .. controls +(-0.9,0.99) and +(-0.75,0.4) .. (v1) -- (w1) .. controls +(-0.9*\klei,0.99*\klei) and +(-0.75*\klei,0.4*\klei) .. (w3);
\fill [red!80!black,opacity=0.3] (w1) -- (v1) .. controls +(-0.15,0.5) and +(-0.15,-0.5) .. (v4) -- (w4) .. controls +(-0.15*\klei,0.5*\klei) and +(-0.15*\klei,-0.5*\klei) .. (w1);
\fill [red!80!black,opacity=0.3] (w3) -- (v3) .. controls +(0.25,-0.5) and +(0.25,0.5) .. (v2) -- (w2) .. controls +(0.25*\klei,-0.5*\klei) and +(0.25*\klei,0.5*\klei) .. (w3);
%
% 1-strata: 
\draw[thick, opacity=1, postaction={decorate}, decoration={markings,mark=at position .6 with {\arrow[draw=green!60!black]{<}}}] (w1) -- (v1);
\draw[thick, opacity=1, postaction={decorate}, decoration={markings,mark=at position .6 with {\arrow[draw=green!60!black]{<}}}] (w2) -- (v2);
\draw[thick, opacity=1, postaction={decorate}, decoration={markings,mark=at position .7 with {\arrow[draw=green!60!black]{>}}}] (w3) -- (v3);
\draw[thick, opacity=1, postaction={decorate}, decoration={markings,mark=at position .5 with {\arrow[draw=green!60!black]{>}}}] (w4) -- (v4);
\fill (v1) circle (0.9pt) node[black, opacity=1.0, right, font=\tiny] {};
\fill (v2) circle (0.9pt) node[black, opacity=1.0, right, font=\tiny] {};
\fill (v3) circle (0.9pt) node[black, opacity=1.0, right, font=\tiny] {};
\fill (v4) circle (0.9pt) node[black, opacity=1.0, right, font=\tiny] {};
\fill (w1) circle (0.9pt) node[black, opacity=1.0, right, font=\tiny] {};
\fill (w2) circle (0.9pt) node[black, opacity=1.0, right, font=\tiny] {};
\fill (w3) circle (0.9pt) node[black, opacity=1.0, right, font=\tiny] {};
\fill (w4) circle (0.9pt) node[black, opacity=1.0, right, font=\tiny] {};
\fill (0,-0.7) circle (0pt) node[black, opacity=0.6, font=\tiny] { $\alpha_1$ };
\fill (0.05,0.1) circle (0pt) node[black, opacity=0.6, font=\tiny] { $\alpha_2$ };
\fill (0.45,0.1) circle (0pt) node[black, opacity=0.6, font=\tiny] { $\alpha_3$ };
\fill (-0.4,-0.2) circle (0pt) node[black, opacity=0.6, font=\tiny] { $\alpha_4$ };
\fill (0,0.63) circle (0pt) node[black, opacity=0.6, font=\tiny] { $\alpha_5$ };
\fill (-0.6,0.3) circle (0pt) node[black, opacity=0.6, font=\tiny] { $\alpha_6$ };
\fill (-0.17,-0.43) circle (0pt) node[black, opacity=0.6, font=\tiny] { $x_1$ };
\fill (0.17,-0.43) circle (0pt) node[black, opacity=0.6, font=\tiny] { $x_2$ };
\fill (0.39,0.4) circle (0pt) node[black, opacity=0.6, font=\tiny] { $x_3$ };
\fill (-0.33,0.35) circle (0pt) node[black, opacity=0.6, font=\tiny] { $x_4$ };
\end{tikzpicture}
%%%%%%%%%%%%%%%%%%%%%% 
$$
represents the identity morphism on~$\Sigma$.

\section{Tricategories from defect TQFTs}
\label{sec:tricatfromTQFT}

For any choice of defect data~$\D$ and field~$\Bbbk$, we have arrived at our central object of study: 
\begin{definition}
\label{def:defectTQFT}
A \textsl{3-dimensional defect TQFT} is a symmetric monoidal functor
$$
\zz: \Bordd \lra \Vect_\Bbbk \, .
$$
\end{definition}
The goal of this section is to extract a 3-categorical structure from any given defect TQFT~$\zz$ which captures as much of~$\zz$ as possible. 
We achieve this goal in Theorems~\ref{thm:TZGray} and~\ref{thm:TZGrayduals} by constructing a certain type of tricategory with extra structure from~$\zz$, namely a `$\Bbbk$-linear Gray category with duals' $\tz$. 

We shall perform this task in two main steps. 
The first step is to define objects, 1- and 2-morphisms for $\tz$, and this involves only  the choice of defect data~$\D$, with no need for the functor~$\zz$. 
In this sense we may think of the first three layers of $\tz$ as purely `kinematic'. 
On the other hand, the second step is about `dynamics': 3-morphisms and the dualities of $\tz$ are built by evaluating~$\zz$ on particular defect bordisms. 

The programme for this section is thus as follows. 
We will review all relevant 2- and 3-categorical structures -- in particular Gray categories -- in Section~\ref{subsec:free2andGray}, and discuss various notions of duality in Section~\ref{subsec:duals}. 
We also present the 2- and 3-dimensional graphical calculus which we shall employ extensively. 
Then in Section~\ref{subsec:tricatfromZ} we construct the Gray category $\tz$ associated to a defect TQFT~$\zz$, and in Section~\ref{subsec:GraycatwithdualsfromZ} we prove that $\tz$ in fact naturally carries the structure of a Gray category with duals.

\subsection{Free 2-categories and Gray categories}
\label{subsec:free2andGray}

This section serves as a brief review of free bicategories and Gray categories. We also recall their respective graphical calculi, state our conventions, and construct the free 2-category associated to a set of defect data~$\D$. 

\subsubsection{Bicategories}
\label{subsubsec:bicats}

We assume familiarity with the basics of bicategories and their 2-functors, as presented for example in \cite[Sect.\,1.5]{LeinsterBook1} and \cite[App.\,A.3]{GregorDiss}. 
By a \textsl{2-category} we mean a strict bicategory, i.\,e.~one whose associator and unitor 2-morphisms are all identities. 
Put differently, a 2-category is a category enriched over Cat. 

Explicitly, in a 2-category~$\B$ there are strictly associative and unital compositions $\fus$ and $\circ$ of 1- and 2-morphisms, respectively, together with strict functors 
$$
X \otimes (-): \B(\gamma,\alpha) \lra \B(\gamma,\beta)
\, , \quad 
(-) \otimes X : \B(\beta,\gamma) \lra \B(\alpha,\gamma) 
$$ 
for all $X \in \B(\alpha,\beta)$, where by convention $X \fus (-)$ sends a morphism~$\Phi$ to $1_X \fus \Phi$. 
The tensor product~$\fus$ is required to satisfy the \textsl{interchange law} 
\be\label{eq:interchangelaw}
\big(\Psi \otimes 1_{X'} \big) \circ \big(1_Y \otimes \Phi\big) 
=
\big( 1_{Y'} \otimes \Phi \big) \circ \big( \Psi \otimes 1_{X} \big) 
\ee
for all 2-morphisms $\Phi: X \rightarrow X'$ and $\Psi: Y \rightarrow Y'$. 
If one keeps the same data but does not require the interchange law, one arrives at the notion of a \textsl{pre-2-category}.  

By generalising Mac Lane's strictification and coherence results for monoidal categories \cite{Maclane} to bicategories, one obtains that every bicategory~$\B$ is equivalent to a 2-category $\B^{\textrm{str}}$, and that for every (weak) 2-functor  $F: \mathcal B \rightarrow \mathcal C$ there exists a strict 2-functor $F^{\textrm{str}} : \B^{\textrm{str}} \rightarrow \mathcal C^{\textrm{str}}$ such that the diagram
$$
%%%%%%%%%%%%%%%%%%%%%%%
\begin{tikzpicture}[
			     baseline=(current bounding box.base), 
			     >=stealth,
			     descr/.style={fill=white,inner sep=3.5pt}, 
			     normal line/.style={->}
			     ] 
\matrix (m) [matrix of math nodes, row sep=3.5em, column sep=2.0em, text height=1.5ex, text depth=0.1ex] {%
\B &&& \mathcal C
\\
\B^{\textrm{str}} &&& \mathcal C^{\textrm{str}}
\\
};
\path[font=\footnotesize] (m-1-1) edge[->] node[auto] {$ F $} (m-1-4);
\path[font=\footnotesize] (m-1-1) edge[->] node[left] {$ \cong $} (m-2-1);
\path[font=\footnotesize] (m-1-4) edge[->] node[right] {$ \cong $} (m-2-4);
\path[font=\footnotesize] (m-2-1) edge[->] node[below] {$ F^{\textrm{str}} $} (m-2-4);
\end{tikzpicture}
%%%%%%%%%%%%%%%%%%%%%%% 
$$
commutes. In this sense bicategories can be completely strictified. 

One class of bicategories which will be relevant for us are those `freely generated by a computad'. 
Let us first recall that by unpacking the definition of \cite[Sect.\,2]{StreetComputad}, a \textsl{computad}~$\K$ consists of three sets 
$\K_0, \K_1, \K_2$ together with four ($\sigma$ource and $\tau$arget) maps $\sigma_0, \tau_0, \sigma_1, \tau_1$ as follows: 
$\sigma_1$ and~$\tau_1$ are maps $\K_1 \rightarrow \K_2$, from which we define `sets of allowed $\K_1$-words' 
$$
\K_1^{(m)} = 
\Big\{ 
(k_1, k_2, \dots, k_m) \in \K_1^{\times m} \,\Big|\, \sigma_1(k_{j+1}) = \tau_1(k_j) \text{ for all } 
%arXiv_v2: 
	%j \in \{ 1,2,\dots,m\} 
	j \in \{ 1,2,\dots,m-1\} 
\Big\}
$$
for all $m\in \Z_+$, and for convenience we set $\K_1^{(0)} = \K_2$. 
Then by definition $\sigma_0$ and $\tau_0$ are maps $\K_0 \rightarrow \bigsqcup_{m\in\N} \K_1^{(m)}$ 
	compatible with $\sigma_1, \tau_1$ in the sense that 
	$$
	\sigma_1  \sigma_0 = \sigma_1  \tau_0
	\, , \quad
	\tau_1  \tau_0 = \tau_1  \sigma_0
	$$
	where for all $m\in \Z_+$ we set $\sigma_1(k_1, \dots, k_m) = \sigma_1(k_1)$ and $\tau_1(k_1, \dots, k_m) = \tau_1(k_m)$.\footnote{Two special cases of a computad~$\K$ are as follows: if $\K_2$ has only one element, then~$\K$ is a tensor scheme as in \cite{JoyalStreetGeo}; and if $\sigma_0, \tau_0$ have images only in $\K_1^{(1)}$, then~$\K$ is a 2-globular set.}

Computads form the objects of a category Comp. 
There is a forgetful functor from the category of 2-categories to Comp which has a left adjoint~$\mathcal F$. 
%(for ``free''). 
The \textsl{free 2-category} $\mathcal F\K$ associated to a computad~$\K$ has a geometric description in terms of string diagrams \cite{StreetComputad}. 
We briefly review this construction; contrasting it with our 3-categorical construction in Section~\ref{subsec:tricatfromZ} will prove useful. 
\begin{itemize}
\item 
Objects of $\mathcal F \K$ are unit squares labelled with elements $\alpha \in \K_2$: 
$$
%%%%%%%%%%%%%%%%%%%%%% 
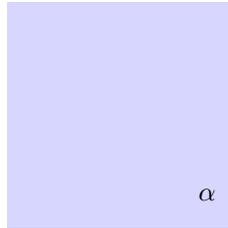
\begin{tikzpicture}[thick,scale=3.0,color=black, baseline=0cm]
\fill [blue!40,opacity=0.4] (0,0) -- (0,1) -- (1,1) -- (1,0);
\draw[line width=1] (0.875,0.15) node[line width=0pt] (alpha) {{\footnotesize $\alpha$}};
\end{tikzpicture}
%%%%%%%%%%%%%%%%%%%%%% 
$$
\item 
1-morphisms $X:\alpha \rightarrow \beta$ are isotopy classes of (unoriented) stratified unit squares with no 0-strata, the 1-strata are $m\in \N$ horizontal lines labelled by elements $(X_1,X_2,\dots,X_m) \in \K_1^{(m)}$ from bottom to top, and the $m+1$ 2-strata are labelled with elements 
%arXiv_v2: 
	%$\alpha_{1}, \alpha_{2},\dots, \alpha_m \in \K_2$ 
	$\alpha_{1}, \alpha_{2},\dots, \alpha_{m+1} \in \K_2$ 
such that $\sigma_1(X_1) = \alpha = \alpha_1$, $\sigma_1(X_{j+1}) = \tau_1 (X_{j})$ for all $j \in \{ 1,2,\dots,m-1 \}$, and $\tau_1(X_m) = \beta = \alpha_{m+1}$. For example: 
%arXiv_v2: [added name to equation]
\be
\label{eq:example1minFK}
%%%%%%%%%%%%%%%%%%%%%% 
\begin{tikzpicture}[thick,scale=3.0,color=black, baseline=0cm, rotate=90]
% 2-strata:
\fill [blue!40,opacity=0.4] (0,0) -- (0,1) -- (0.25,1) -- (0.25,0);
\fill [blue!50,opacity=0.5] (0.25,0) -- (0.25,1) -- (0.5,1) -- (0.5,0);
\fill [blue!30,opacity=0.4] (0.5,0) -- (0.5,1) -- (0.75,1) -- (0.75,0);
\fill [blue!50,opacity=0.4] (0.75,0) -- (0.75,1) -- (1,1) -- (1,0);
% 2-strata labels:
\draw[line width=1] (0.125,0.4) node[line width=0pt] (alpha) {{\footnotesize $\;\;\;\;\;\;\alpha_{1} = \alpha$}};
\draw[line width=1] (0.375,0.4) node[line width=0pt] (alpha) {{\footnotesize $\;\;\;\;\;\;\alpha_{2}\phantom{ = \beta}$}};
\draw[line width=1] (0.625,0.4) node[line width=0pt] (alpha) {{\footnotesize $\;\;\;\;\;\;\alpha_{3}\phantom{ = \beta}$}};
\draw[line width=1] (0.875,0.4) node[line width=0pt] (alpha) {{\footnotesize $\;\;\;\;\;\;\alpha_{4} = \beta$}};
%
% 1-strata_
\draw[
	color=red!60, 
	very thick
	] 
 (0.25,0) --  (0.25,1);
\draw[
	color=red!90!black, 
	very thick
	] 
 (0.5,0) --  (0.5,1);
\draw[
	color=red!90, 
	very thick
	] 
 (0.75,0) --  (0.75,1);
% 1-strata labels:
\draw[line width=1] (0.33,0.8) node[line width=0pt] (alpha) {{\footnotesize $X_1$}};
\draw[line width=1] (0.58,0.8) node[line width=0pt] (alpha) {{\footnotesize $X_2$}};
\draw[line width=1] (0.83,0.8) node[line width=0pt] (alpha) {{\footnotesize $X_3$}};
\end{tikzpicture}
%%%%%%%%%%%%%%%%%%%%%% 
\ee
Only isotopies which leave the lines straight and respect their order are considered. We call such isotopies \textsl{linear}. 
\item 
2-morphisms $\Phi: X \rightarrow Y$ are equivalence classes of  string diagrams which stratify the unit square 
and are  \textsl{progressive}, e.\,g.~such that the vertical projection map restricted to each 1-stratum is a diffeomorphism onto its image.
 Moreover, there must be a 
 neighbourhood of the right and left boundaries of the square in which the diagram looks like the (squeezed) squares~$X$ and~$Y$, respectively.\footnote{Our unusual convention of reading string diagrams from right to left, as opposed to vertically, will become advantageous once they will appear in 3-dimensional diagrams below.} 
Such a diagram consists of embedded progressive 
%(meaning that the projection to the horizontal axis is regular) 
lines (the 1-strata) which may meet at points (the 0-strata) in the interior, but not with the top and bottom boundaries. 
All $j$-strata are decorated with elements of $\K_j$ such that the source and target relations are satisfied. 
As in the example below, we will however often suppress such decorations: 
$$
%%%%%%%%%%%%%%%%%%%%%% 
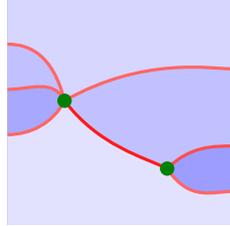
\begin{tikzpicture}[thick,scale=3.0,color=black, baseline=0cm]
% 2-strata: 
\fill [blue!60,opacity=0.4] (0,0) -- (0,1) -- (1,1) -- (1,0);
\fill [blue!10,opacity=0.9] 
(0,0) -- (1,0) -- (1,0.15) .. controls (0.9, 0.15) and (0.8, 0.1) .. (0.7, 0.25) -- (0.7, 0.25) .. controls (0.6, 0.3) and (0.4, 0.35) .. (0.25, 0.55) -- (0.25, 0.55) .. controls (0.2, 0.45) and (0.1, 0.4) .. (0, 0.4);
\fill [blue!40,opacity=0.9] 
(1,0.15) .. controls (0.9, 0.15) and (0.8, 0.1) .. (0.7, 0.25) -- (0.7, 0.25) .. controls (0.8, 0.35) and (0.9, 0.35) .. (1,0.35) -- (1,0.15);
\fill [blue!15,opacity=0.9] 
(1,1) -- (1, 0.69) .. controls (0.9, 0.69) and (0.6, 0.75) .. (0.25, 0.55) -- (0.25, 0.55) .. controls (0.2, 0.75) and (0.1, 0.8) .. (0, 0.8) -- (0,1) -- (1,1);
\fill [blue!35,opacity=0.9] 
(0.25, 0.55) .. controls (0.2, 0.65) and (0.1, 0.6) .. (0, 0.6) -- (0,0.4) .. controls (0.1, 0.4) and (0.2, 0.45) .. (0.25, 0.55); 
%
% 1-strata: 
\draw[color=red!60, 
	very thick,
	opacity=1.0] 
	 (1,0.15) .. controls (0.9, 0.15) and (0.8, 0.1) .. (0.7, 0.25);
\draw[color=red!70, 
	very thick,
	opacity=1.0] 
	 (1,0.35) .. controls (0.9, 0.35) and (0.8, 0.35) .. (0.7, 0.25);
\draw[color=red!90, 
	very thick,
	opacity=1.0] 
	 (0.7, 0.25) .. controls (0.6, 0.3) and (0.4, 0.35) .. (0.25, 0.55);
\draw[color=red!60, 
	very thick,
	opacity=1.0] 
	 (1, 0.69) .. controls (0.9, 0.69) and (0.6, 0.75) .. (0.25, 0.55);
\draw[color=red!60, 
	very thick,
	opacity=1.0] 
	 (0.25, 0.55) .. controls (0.2, 0.65) and (0.1, 0.6) .. (0, 0.6);
\draw[color=red!60, 
	very thick,
	opacity=1.0] 
	 (0.25, 0.55) .. controls (0.2, 0.75) and (0.1, 0.8) .. (0, 0.8);
\draw[color=red!60, 
	very thick,
	opacity=1.0] 
	 (0.25, 0.55) .. controls (0.2, 0.45) and (0.1, 0.4) .. (0, 0.4);
%
% 0-strata: 
\fill[color=green!50!black] (0.7, 0.25) circle (0.9pt) node[above] (0up) {};
\fill[color=green!50!black] (0.25, 0.55) circle (0.9pt) node[above] (0up) {};
\end{tikzpicture}
%%%%%%%%%%%%%%%%%%%%%% 
$$
Two diagrams are equivalent if there is an isotopy of progressive diagrams relating the two. 
In particular this means that the interchange law simply reads
\be\label{eq:2dinterchangelaw}
%%%%%%%%%%%%%%%%%%%%%% 
\begin{tikzpicture}[thick,scale=3.0,color=black, baseline=1.4cm]
% 2-strata:
\fill [blue!40,opacity=0.4] (0,0) -- (1,0) -- (1,0.33) -- (0,0.33);
\fill [blue!50,opacity=0.5] (0,0.33) -- (1,0.33) -- (1,0.66) -- (0,0.66);
\fill [blue!30,opacity=0.4] (0,0.66) -- (1,0.66) -- (1,1) -- (0,1);
%
% 1-strata
\draw[color=red!60, 
	very thick,
	opacity=1.0] 
	 (1,0.33) -- (0.66, 0.33);
\draw[color=red!90!black, 
	very thick,
	opacity=1.0] 
	 (0.66, 0.33) -- (0,0.33);
\draw[color=red!80!black, 
	very thick,
	opacity=1.0] 
	 (1,0.66) -- (0.33, 0.66);
\draw[color=red!80, 
	very thick,
	opacity=1.0] 
	 (0.33, 0.66) -- (0, 0.66);
%
% 0-strata: 
\fill[color=green!60!black] (0.66, 0.33) circle (0.9pt) node[above] (0up) {};
\fill[color=green!50!black] (0.33, 0.66) circle (0.9pt) node[above] (0up) {};
\end{tikzpicture}
%%%%%%%%%%%%%%%%%%%%%% 
=
%%%%%%%%%%%%%%%%%%%%%% 
\begin{tikzpicture}[thick,scale=3.0,color=black, baseline=1.4cm]
% 2-strata:
\fill [blue!40,opacity=0.4] (0,0) -- (1,0) -- (1,0.33) -- (0,0.33);
\fill [blue!50,opacity=0.5] (0,0.33) -- (1,0.33) -- (1,0.66) -- (0,0.66);
\fill [blue!30,opacity=0.4] (0,0.66) -- (1,0.66) -- (1,1) -- (0,1);
%
% 1-strata
\draw[color=red!60, 
	very thick,
	opacity=1.0] 
	 (1,0.33) -- (0.33, 0.33);
\draw[color=red!90!black, 
	very thick,
	opacity=1.0] 
	 (0.33, 0.33) -- (0,0.33);
\draw[color=red!80!black, 
	very thick,
	opacity=1.0] 
	 (1,0.66) -- (0.66, 0.66);
\draw[color=red!80, 
	very thick,
	opacity=1.0] 
	 (0.66, 0.66) -- (0, 0.66);
%
% 0-strata: 
\fill[color=green!60!black] (0.33, 0.33) circle (0.9pt) node[above] (0up) {};
\fill[color=green!50!black] (0.66, 0.66) circle (0.9pt) node[above] (0up) {};
\end{tikzpicture}
%%%%%%%%%%%%%%%%%%%%%% 
\ee
in the free 2-category $\mathcal F \K$. 
\item 
Horizontal composition of 1- and 2-morphisms is given by vertical(!) alignment of the respective squares, followed by a rescaling (to end up with a square again; the precise way in which the rescaling is performed is irrelevant, since different rescalings are related by isotopies). 
Similarly, vertical composition of 2-morphisms is defined by horizontal alignment. 
The unit 1-morphism $1_\alpha$ on an object~$\alpha$ is the diagram for~$\alpha$ (regarded as 1-morphism with no lines), and the unit 2-morphism $1_X$ on a 1-morphism~$X$ is the diagram for~$X$. 
\end{itemize}

Slight variations of the above construction lead to two closely related bicategorical structures. 
Firstly, the \textsl{free pre-2-category} $\mathcal F^{\textrm{p}} \K$ associated to~$\K$ has the same diagrams for its objects, 1- and 2-morphisms as $\mathcal F \K$, but the equivalence relation on 2-morphisms is different: 
two string diagrams are equivalent in $\mathcal F^{\textrm{p}} \K$ iff they are related by a \textsl{progressive isotopy}. 
This is an isotopy of progressive diagrams where no two horizontal coordinates of distinct vertices are allowed to coincide during the isotopy. 
Thus, the equality \eqref{eq:2dinterchangelaw} does not hold in $\mathcal F^{\textrm{p}} \K$.

Secondly, the \textsl{free bicategory} $\mathcal F^{\textrm{b}} \K$ associated to~$\K$ has again the same diagrams for its objects, 1- and 2-morphisms as $\mathcal F \K$, and 2-morphisms in $\mathcal F^{\textrm{b}} \K$ are isotopy classes of string diagrams as in $\mathcal F \K$. 
However, non-indentical (but possibly isotopic) diagrams for 1-morphisms are not identified in $\mathcal F^{\textrm{b}} \K$; (objects and) 1-morphisms simply \textsl{are} diagrams here. 
Identities between 1-morphisms in $\mathcal F \K$ generically amount only to 2-isomorphisms in $\mathcal F^{\textrm{b}} \K$, and horizontal composition is not strictly associative or unital in $\mathcal F^{\textrm{b}} \K$. 

\medskip

We shall now use the above construction to produce the free 2-category associated to a set of defect data $\D = (D_3, D_2, D_1, s, t, f)$ (cf.~Definition~\ref{def:defectdata}). 
The first step is to identify the computad $\K^{\D}$ that~$\D$ naturally gives rise to. 
For this we abbreviate $\K^{\D} = \K$ and set 
%arXiv_v2: [added name to equation]
\be
\label{eq:K2K1}
\K_2 = D_3 
\, , \quad
%arXiv_v2: 
	%\K_1 = D_2 \times \{ \pm \}
	\K_1 = D_3 \sqcup ( D_2 \times \{ \pm \} )
\, , \quad
%%%%%%%%%%%%%%%%%%%%%%%
\begin{tikzpicture}[
			     baseline=(current bounding box.base), 
			     >=stealth,
			     descr/.style={fill=white,inner sep=3.5pt}, 
			     normal line/.style={->},
			     baseline=-0.1cm
			     ] 
\matrix (m) [matrix of math nodes, row sep=3.5em, column sep=3.3em, text height=1.5ex, text depth=0.1ex] {%
%arXiv_v2: 
	%\K_1 
	\K_1 \supset D_2 \times \{ \pm \}
&& \K_2 \, ,
\\
};
\path[font=\footnotesize] 
	(m-1-1) edge[->, out = 10, in=160] node[above, sloped] {$ \sigma_1|_{D_2 \times \{ \pm \}} = s $} (m-1-3)
	(m-1-1) edge[->, out = -10, in=200] node[below, sloped] {$ \tau_1|_{D_2 \times \{ \pm \}} = t $} (m-1-3);
\end{tikzpicture}
%%%%%%%%%%%%%%%%%%%%%%% 
\ee
%arXiv_v2: 
	and $\sigma_1(u) = u = \tau_1(u)$ if $u\in D_3 \subset \K_1$, 
and we immediately extend $\sigma_1, \tau_1$ to $\K_1^{(m)}$ by setting $\sigma_1\big( (\alpha_1,\varepsilon_1), \dots, (\alpha_m,\varepsilon_m) \big) = \sigma_1(\alpha_1,\varepsilon_1)$ and $\tau_1\big( (\alpha_1,\varepsilon_1), \dots, (\alpha_m,\varepsilon_m) \big) = \tau_1(\alpha_m,\varepsilon_m)$ for all $\big((\alpha_1,\varepsilon_1), \dots, (\alpha_m,\varepsilon_m)\big) \in \K_1^{(m)}$ and $m\in \Z_+$. 
Writing 
$$
\big( (\alpha_1,\varepsilon_1), \dots, (\alpha_m,\varepsilon_m) \big)^\# 
=
\big( (\alpha_m, -\varepsilon_m), \dots, (\alpha_1, -\varepsilon_1) \big)
$$
for the reversed list, we set 
$$
%arXiv_v2: 
	%\K_0 = \bigsqcup_{m,m' \in \N} \; \bigsqcup_{A,A'} f^{-1} \big(A' \sta A^\# \big )
	\K_0 = D_3 \sqcup \bigsqcup_{m,m' \in \N, \, (m,m') \neq (0,0)} \; \bigsqcup_{A,A'} f^{-1} \big(A' \sta A^\# \big )
$$
where the second disjoint union is over all pairs $(A,A') \in \K_1^{(m)} \times \K_1^{(m')}$ with 
%arXiv_v2: 
	%$\sigma_1(A) = \sigma_1(A')$ and $\tau_1(A) = \tau_1(A')$, 
	$\sigma_1(A) = \sigma_1(A')$ and $\tau_1(A) = \tau_1(A')$ if $m,m' >0$, $\sigma_1(A) = \tau_1(A)$ if $m'=0$, $\sigma_1(A') = \tau_1(A')$ if $m=0$, 
and `$\sta$' denotes 
%arXiv_v2: 
	%concatenation of lists. 
	concatenation of lists (where we treat~$A^\#$ (resp.~$A'$) as the empty list if~$m'$ (resp.~$m$) is zero). 
Finally, the maps $\sigma_0, \tau_0$ are defined by
$$
X \in f^{-1} \big(A' \sta A^\# \big ) 
\quad
\implies
\quad
\sigma_0 (X) = A \, , 
\quad
\tau_0(X) = A' 
$$
%arXiv_v2: 
	for $X \in \K_0 \setminus D_3$ with $(A,A') \in \K_1^{(m)} \times \K_1^{(m')}$ with $(m,m') \neq (0,0)$, and $\sigma_0(X) = X = \tau_0(X)$ if $X\in D_3 \subset \K_0$. 

\begin{lemma}
$\K^{\D}$ is a computad. 
\end{lemma}
\begin{proof}
This follows immediately from the properties of~$\D$. 
\end{proof}

\begin{corollary}
Every set of defect data~$\D$ gives rise to 
%arXiv_v2: 
	%free bicategories $\mathcal F \K^{\D}$ and $\mathcal F^{\textrm{b}} \K^{\D}$, and to 
	the free 2-category $\mathcal F \K^{\D}$, the free bicategory $\mathcal F^{\textrm{b}} \K^{\D}$, and 
the free pre-2-category $\mathcal F^{\textrm{p}} \K^{\D}$. 
\end{corollary}

In Section~\ref{subsec:tricatfromZ} we will be guided to make a tricategory out of $\mathcal \K^{\D}$ by not necessarily identifying isotopic string diagrams, and by adding a third dimension to relate them. 
However, first we have to discuss the type of tricategories which we shall consider.

\subsubsection{Tricategories}
\label{subsubsec:tricats}

A (strict) 3-category is easily defined: it is a category enriched over (strict) 2-categories. 
Relaxing associativity and other relations of the various compositions to hold only up to higher morphisms, subject to certain coherence conditions, leads to the notion of a `tricategory' \cite{GPS, Gurskibook}. 
In this sense tricategories are to bicategories as bicategories are to categories. 
Yet this analogy is only superficial: contrary to the 2-dimensional situation, not every tricategory is equivalent to a strict 3-category.\footnote{Every tricategory with only one object and one 1-morphism whose braided monoidal category of 2- and 3-morphisms is non-symmetric is a counterexample.} 
This is one reason why 3-dimensional category theory (and accordingly 3-dimensional TQFT) is considerably richer than the 2-dimensional case. 

While a generic tricategory cannot be completely strictified, it was shown in \cite{GPS} that `the next best thing' is true: every tricategory is equivalent to a tricategory whose compositions are all strictly associative and unital, but for which the interchange law may only hold up to isomorphism. 
Such tricategories are called `Gray categories' as they are categories enriched over the symmetric monoidal category of 2-categories and strict 2-functors with the `Gray tensor product', cf.~\cite{GrayBook}. 
Another way to concisely define Gray categories is as `strict opcubical tricategories'. 
For our purposes though it is appropriate to unpack this into building blocks which are familiar from 1- and 2-category theory:\footnote{It was shown in \cite{GPS} that Definition~\ref{def:Graycat} is equivalent to the notion of strict opcubical tricategory.
} 

\begin{definition}
\label{def:Graycat}
A \textsl{Gray category}~$\G$ consists of the following data: 
\begin{enumerate}
\item
a set of \textsl{objects} denoted $\Obj(\G)$, or simply~$\G$; 
\item 
for all $u,v \in \G$, a 2-category $\G(u,v)$ of \textsl{1-}, \textsl{2-}, and \textsl{3-morphisms}, with horizontal and vertical composition denoted `$\fus$' and `$\circ$'; 
\item 
for all $u\in \G$, a 1-morphism $1_u \in \G(u,u)$; 
\item 
for all $u,v,w \in \G$ and 
%arXiv_v2: 
	%all $\alpha \in \G(u,v)$, 
	all 1-morphisms $\alpha \in \G(u,v)$, 
(strictly associative and strictly unital) strict 2-functors 
$$
\alpha \sta(-): \G(w,u) \lra \G(w,v)
\, , \quad 
(-) \sta \alpha : \G(v,w) \lra \G(u,w) 
\, ;
$$ 
\item
for all 2-morphisms $Y: \alpha \rightarrow \alpha'$ in $\G(u,v)$ and all 2-morphisms $X: \beta \rightarrow \beta'$ in $\G(v,w)$, an invertible 3-morphism  
\be\label{eq:tensorator}
\sigma_{X,Y} : 
\big( X \sta 1_{\alpha'} \big) \fus \big( 1_{\beta} \sta Y \big) 
\lra 
\big( 1_{\beta'} \sta Y \big) \fus \big( X \sta 1_\alpha \big)
\ee
called \textsl{tensorator}, natural in~$X$ and~$Y$. 
\end{enumerate}
The tensorator is subject to the following conditions: 
\begin{enumerate}
	\setcounter{enumi}{5}
\item 
$\sigma_{X, 1_{\alpha}} = 1_{X \sta 1_{\alpha}}$ 
and 
$\sigma_{1_{\beta}, Y} = 1_{1_{\beta} \sta Y}$; 
\item 
for all 2-morphisms $Y': \alpha' \to \alpha''$ and $X': \beta' \to \beta''$, 
\begin{align*}
\sigma_{X, Y' \fus Y} & = \Big( \big( 1_{1_{\beta'}} \sta 1_{Y'} \big) \fus \sigma_{X,Y} \Big) \circ \Big( \sigma_{X,Y'} \fus \big( 1_{1_{\beta}} \sta 1_Y \big) \Big) \, , \\
\sigma_{X' \fus X, Y} & = \Big( \sigma_{X',Y} \fus \big( 1_X \sta 1_{1_{\alpha}} \big) \Big) \circ \Big( \big( 1_{X'} \sta 1_{1_{\alpha'}} \big) \fus \sigma_{X,Y} \Big) \; ; 
\end{align*}
\item
for all 1-morphisms~$\alpha$ and 2-morphisms $X,Y$ such that the following expressions are defined: 
$$
\sigma_{X \sta 1_\alpha, Y} = \sigma_{X, 1_\alpha \sta Y}
\, , \quad
\sigma_{1_\alpha \sta X, Y} = 1_{1_{\alpha}} \sta \sigma_{X,Y}
\, , \quad
\sigma_{X, Y \sta 1_\alpha} = \sigma_{X,Y} \sta 1_{1_{\alpha}} \, . 
$$
\end{enumerate}
\end{definition}

We adopt the convention that the $\sta$-composite of two 2- or 3-morphisms $\Psi: \alpha \rightarrow \alpha'$ and $\Phi: \beta \rightarrow \beta'$ is 
$\Phi \sta \Psi= (\Phi \sta 1_{\alpha'}) \fus (1_\beta \sta \Psi)$. 
In the following we will often refer to $\sta$-composition as the \textsl{Gray product}.  
We say that a Gray category is \textsl{$\Bbbk$-linear} (or simply \textsl{linear} if the field~$\Bbbk$ is known from the context) if 
%CHANGED: 
%	the sets of 3-morphisms carry the structure of $\Bbbk$-vector spaces, 
	the categories of 2- and 3-morphisms are $\Bbbk$-linear, 
and the compositions~$\sta$ and~$\fus$ induce $\Bbbk$-linear functors. 
Furthermore, we call a 1-morphism $\alpha: u \rightarrow v$ in a Gray category $\Cat{G}$ 
 a \textsl{biequivalence} \cite{Gurskibook}, if there exists a 1-morphism
 $\beta: v \rightarrow u$ such that 
$\alpha \sta \beta$ is equivalent to $1_{v}$ in the bicategory $\Cat{G}(v,v)$, and $\beta \sta \alpha$ is equivalent to $1_{u}$ in $\Cat{G}(u,u)$. 
This allows us to formulate the following internal notion of equivalence:\footnote{We note that Definition~\ref{def:Grayequiv} gives a minimal notion of equivalence of Gray categories, adopted to the examples we will consider. In general, even a quasi-inverse of a strict equivalence may be non-strict.} 

\begin{definition}  \label{def:Grayequiv}
A \textsl{strict equivalence $\Cat{G} \rightarrow \Cat{G}'$} between Gray categories
 $\G, \G'$ consists of
\begin{enumerate}
\item
a \textsl{strict functor of Gray categories} $F: \Cat{G} \rightarrow \Cat{G}'$, i.\,e.~a function $F_{0}: \Obj(\Cat{G}) \rightarrow \Obj(\Cat{G}')$ and strict 2-functors $F_{u,v}: \Cat{G}(u,v)\rightarrow \Cat{G}'(F_{0}(u),F_{0}(v))$, such that $F_{u,u}(1_{u})=1_{F_{0}(u)}$ and $F_{w,u}(\Psi \sta \Phi)=F_{v,u}(\Psi) \sta F_{w,v}(\Phi)$ for all $u,v,w \in \Obj(\G)$ and 2- or 3-morphisms $\Phi$ and $\Psi$ for which this expression is defined. Moreover, we require $F_{u,v}(\sigma_{X,Y})=\sigma'_{F_{v,w}(X),F_{u,v}(Y)}$ for all~$X$ and~$Y$ as in \eqref{eq:tensorator}. 
If $\Cat{G}$ and $\Cat{G}'$ are linear Gray categories, the 2-functors $F_{u,v}$ are required to be linear.  
\item 
Each 2-functor $F_{u,v}$ is an \textsl{equivalence of 2-categories}, i.\,e.~$F_{u,v}$ is biessentially surjective (to wit, every object of $ \Cat{G}'(F_{0}(u),F_{0}(v))$ is equivalent to $F_{u,v}(\alpha)$ for some 
$\alpha \in \Cat{G}(u,v)$), and it induces an equivalence on the categories of 2- and 3-morphisms. 
\item
$F_{0}$ is \textsl{triessentially surjective}, i.\,e.~every object $u' \in \Obj(\Cat{G}')$ is biequivalent to an object $F_{0}(u)$ for some $u \in \Obj(\Cat{G})$.
\end{enumerate}
\end{definition} 

\medskip
Just like every calculation in a 2-category can be performed with the help of progressive 2-dimensional string diagrams \cite{JoyalStreetGeo}, there also exists a diagrammatic calculus for Gray categories. 
It was originally studied in \cite{TrimbleSurfaceDiagrams} and fully developed in \cite[Sect.\,2.5]{BMS}, to which we refer for a precise treatment. 
In the following we only present the basic picture. 

Given a Gray category~$\G$, any 3-morphism~$\Phi$ in~$\G$, i.\,e.~a 2-morphism in $\G(u,v)$ for some $u,v \in \G$, can be presented and evaluated as an (isotopy class of a) \textsl{progressive Gray category diagram} as follows. 
Such a diagram is a 
%arXiv_v2: 
	%progressive \textsl{3d diagram}, i.\,e.~a progressive stratification (with certain conditions on the boundary) of the unit cube in $\R^3$, 
	progressive 3d diagram\footnote{A \textsl{progressive 3d diagram} is a stratification of the unit cube $[0,1]^3 \subset \R^3$ such that for each $j$-stratum~$s$ we have $\partial ([0,1]^3) \cap s = s \cap ( (0,1)^{3-j} \times \partial ([0,1]^j) )$, the side faces $[0,1] \times \{0\} \times [0,1]$ and $[0,1] \times \{1\} \times [0,1]$ are progressive 2-dimensional diagrams, the projection $(x,y,z) \mapsto (y,z)$ is a regular map of each surface, and projection $(x,y,z) \mapsto z$ is a regular map for each line, cf.\ \cite[Def.\,2.8\,\&\,2.22]{BMS}.} 
together with a decoration of all $j$-strata by 
%arXiv_v2: 
	%$(3-j)$-morphisms in~$\G$. 
	$(3-j)$-morphisms in~$\G$ for all $j\in \{ 0,1,2,3 \}$ (where we use the convention that 0-morphisms are objects). 
The decoration must be compatible with the source and target relations, where $\sta$-composition is in negative $x$-direction, $\fus$-composition is in negative $y$-direction, and $\circ$-composition is in $z$-direction. 
The source and target 2-morphisms of~$\Phi$ are respectively identified with the bottom and top of the cube. 
It follows that the projection of a Gray category diagram for~$\Phi$ to the $x=0$ plane produces a string diagram in the 2-category $\G(u,v)$ whose evaluation is~$\Phi$. 

%arXiv_v2: [start new paragraph now]

For example, 
%arXiv_v2: 
	consider objects $u,v,w \in \G$, 
	1-morphisms $\alpha, \beta, \gamma \in \G(u,v)$, $\delta \in \G(u,w)$, $\varepsilon \in \G(w,v)$, $\zeta \in \G(w,u)$, 
	2-morphisms $X \colon \varepsilon \sta \delta \to \alpha$, $Y \colon \gamma \sta \zeta \to \varepsilon$, $Z \colon 1_u \to \zeta \sta \delta$, $X' \colon \beta \to \alpha$, and $Y' \colon \gamma \to \beta$. Then 
the Gray category diagram (whose origin is at the lower back corner)
$$
%%%%%%%%%%%%%%%%%%%%%% 
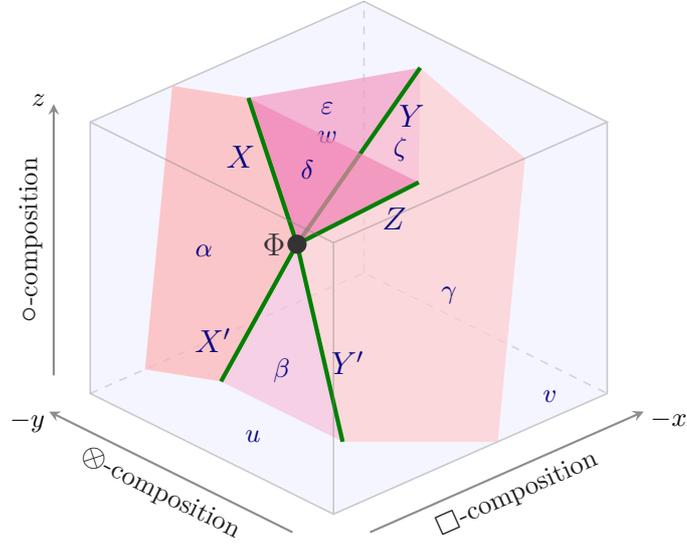
\begin{tikzpicture}[thick,scale=4.0,color=blue!50!black, baseline=0.0cm, >=stealth, 
				style={x={(-0.9cm,-0.4cm)},y={(0.8cm,-0.4cm)},z={(0cm,0.9cm)}}]
% invisible edges of cube: 
\draw[
	 color=gray, 
	 opacity=0.3, 
	 semithick,
	 dashed
	 ] 
	 (1,0,0) -- (0,0,0) -- (0,1,0)
	 (0,0,0) -- (0,0,1);
%%%
% 3-stratum v at back: 
\fill [blue!40,opacity=0.1] (0,0,0) -- (1,0,0) -- (1,1,0) -- (0,1,0) -- (0,1,1) -- (0,0,1) -- (1,0,1) -- (1,0,0);
%
% endpoint of rho: 
\coordinate (rho) at (0.6, 0.2, 1);
% endpoint of kappa: 
\coordinate (kappa) at (0.15, 0.4, 1);
% endpoint of eta: 
\coordinate (eta) at (0.6, 0.9, 1);
%
% Psi: 
\coordinate (Psi) at (0.6, 0.4, 0.55);
%
% endpoint of nu: 
\coordinate (nu) at (0.7, 0.2, 0);
% endpoint of mu: 
\coordinate (mu) at (0.7, 0.7, 0);
%
% L-plane: 
\fill [magenta!50,opacity=0.7] (Psi) -- (kappa) -- (rho);
% epsilon-label: 
\draw[line width=1] (0.4, 0.3, 0.92) node[line width=0pt] (beta) {{\footnotesize $\eps$}};
% F-plane: 
\fill [red!20,opacity=0.7] (Psi) -- (kappa) -- (0.3, 1, 1) -- (0.4, 1, 0) -- (mu);
%arXiv_v2: 
	% w-label: 
	\draw[line width=1] (0.4, 0.3, 0.81) node[line width=0pt] (beta) {{\footnotesize $w$}};
% gamma-label: 
\draw[line width=1] (0.4, 0.8, 0.45) node[line width=0pt] (beta) {{\footnotesize $\gamma$}};
% J-plane: 
\fill [magenta!30,opacity=0.7] (Psi) -- (kappa) -- (eta);
% zeta-label: 
\draw[line width=1] (0.4, 0.6, 0.9) node[line width=0pt] (beta) {{\footnotesize $\zeta$}};
% kappa-line:
\draw[color=green!50!black, ultra thick] (Psi) -- node[pos=0.86, color=blue!50!black, below] {$\footnotesize \;\;Y$} (kappa);  
% K-plane: 
\fill [magenta!60,opacity=0.7] (Psi) -- (rho) -- (eta);
% delta-label: 
\draw[line width=1] (0.5, 0.33, 0.75) node[line width=0pt] (beta) {{\footnotesize $\delta$}};
% eta-line:
\draw[color=green!50!black, ultra thick] (Psi) -- node[pos=0.8, color=blue!50!black, below] {$\footnotesize Z$} (eta);  
% G-plane: 
\fill [magenta!30,opacity=0.7] (Psi) -- (mu) -- (nu);
% beta-label: 
\draw[line width=1] (0.7, 0.45, 0.15) node[line width=0pt] (beta) {{\footnotesize $\beta$}};
% H-plane: 
\fill [red!30,opacity=0.7] (Psi) -- (nu) -- (0.8, 0, 0) -- (0.7, 0., 1) -- (rho);
% alpha-label: 
\draw[line width=1] (0.7, 0.13, 0.45) node[line width=0pt] (beta) {{\footnotesize $\alpha$}};
% rho-line:
\draw[color=green!50!black, ultra thick] (Psi) -- node[pos=0.6, color=blue!50!black, left] {$\footnotesize X$} (rho);  
% mu-line:
\draw[color=green!50!black, ultra thick] (Psi) -- node[pos=0.6, color=blue!50!black, right] {$\footnotesize \!Y'$} (mu);  
% nu-line:
\draw[color=green!50!black, ultra thick] (Psi) -- node[pos=0.7, color=blue!50!black, left] {$\footnotesize {}\;{}X'$} (nu);  
%
% Phi-label:
\fill[color=black!80] (Psi) circle (0.9pt) node[left] (0up) {$\footnotesize \Phi$};
%
% object labels: 
\draw[line width=1] (0.85, 0.5, 0) node[line width=0pt] (beta) {{\footnotesize $u$}};
\draw[line width=1] (0.12, 0.9, 0) node[line width=0pt] (beta) {{\footnotesize $v$}};
%
% visible edges of cube
\draw[
	 color=gray, 
	 opacity=0.4, 
	 semithick
	 ] 
	 (0,1,1) -- (0,1,0) -- (1,1,0) -- (1,1,1) -- (0,1,1) -- (0,0,1) -- (1,0,1) -- (1,0,0) -- (1,1,0)
	 (1,0,1) -- (1,1,1);
%
% composition arrows: 
\draw[
	 color=gray, 
	 opacity=0.9, 
	 thick, 
	 ->
	 ] 
	 (1, 1.15, 0) -- node[pos=0.5, color=black, below, sloped] {$\footnotesize \sta\text{-composition}$} (0, 1.15, 0); 
\draw[
	 color=gray, 
	 opacity=0.9, 
	 thick, 
	 ->
	 ] 
	 (1.15, 1, 0) -- node[pos=0.5, color=black, below, sloped] {$\footnotesize \fus\text{-composition}$} (1.15, 0, 0); 
\draw[
	 color=gray, 
	 opacity=0.9, 
	 thick, 
	 ->
	 ] 
	 (1, -0.15, 0) -- node[pos=0.5, color=black, above, sloped] {$\footnotesize \circ\text{-composition}$} (1, -0.15, 1); 
%
% directions: 
\draw[line width=1] (1.23, 0, 0) node[line width=0pt, black] (beta) {{\footnotesize $-y$}};
\draw[line width=1] (-0.02, 1.23, 0) node[line width=0pt, black] (beta) {{\footnotesize $-x$}};
\draw[line width=1] (1.1, -0.1, 1.08) node[line width=0pt, black] (beta) {{\footnotesize $z$}};
\end{tikzpicture}
%%%%%%%%%%%%%%%%%%%%%%
$$
evaluates to the string diagram 
$$
%%%%%%%%%%%%%%%%%%%%%% 
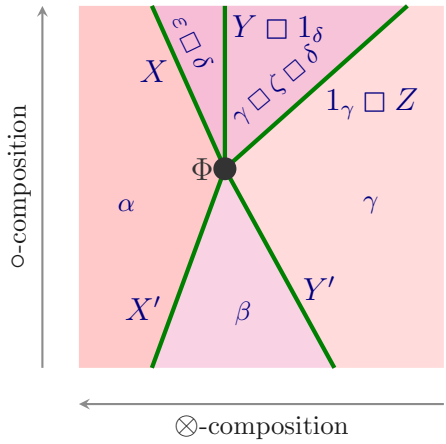
\begin{tikzpicture}[thick,scale=4.8,color=blue!50!black, baseline=0.0cm, >=stealth]
%
% endpoint of rho: 
\coordinate (rho) at (0.2, 1);
% endpoint of kappa: 
\coordinate (kappa) at (0.4, 1);
% endpoint of eta: 
\coordinate (eta) at (0.9, 1);
%
% Psi: 
\coordinate (Psi) at (0.4, 0.55);
%
% endpoint of nu: 
\coordinate (nu) at (0.2, 0);
% endpoint of mu: 
\coordinate (mu) at (0.7, 0);
%
% L-plane: 
\fill [magenta!50,opacity=0.7] (Psi) -- (kappa) -- (rho);
% F-plane: 
\fill [red!20,opacity=0.7] (Psi) -- (kappa) -- (1, 1) -- (1, 0) -- (mu);
% J-plane: 
\fill [magenta!30,opacity=0.7] (Psi) -- (kappa) -- (eta);
% K-plane: 
\fill [magenta!30,opacity=0.7] (Psi) -- (rho) -- (eta);
% eta-line:
\draw[color=green!50!black, ultra thick] (Psi) -- node[pos=0.8, color=blue!50!black, below] {} (eta);  
\draw[line width=1] (0.79, 0.73) node[line width=0pt] (beta) {$\footnotesize \;1_\gamma \sta Z$};
% K-plane: 
\fill [magenta!30,opacity=0.7] (Psi) -- (mu) -- (nu);
% H-plane: 
\fill [red!30,opacity=0.7] (Psi) -- (nu) -- (0, 0) -- (0, 1) -- (rho);
% rho-line:
\draw[color=green!50!black, ultra thick] (Psi) -- node[pos=0.6, color=blue!50!black, left] {$\footnotesize X$} (rho);  
% mu-line:
\draw[color=green!50!black, ultra thick] (Psi) -- node[pos=0.6, color=blue!50!black, right] {$\footnotesize Y'$} (mu);  
% nu-line:
\draw[color=green!50!black, ultra thick] (Psi) -- node[pos=0.7, color=blue!50!black, left] {$\footnotesize X'$} (nu);  
% kappa-line:
\draw[color=green!50!black, ultra thick] (Psi) -- node[pos=0.86, color=blue!50!black, right] {$\footnotesize \!Y \sta 1_\delta$} (kappa);  
\fill[color=black!80] (Psi) circle (0.9pt) node[left] (0up) {$\footnotesize \Phi$};
%
% composition arrows: 
\draw[
	 color=gray, 
	 opacity=0.9, 
	 thick, 
	 ->
	 ] 
	 (1, -0.1) -- node[pos=0.5, color=black, below, sloped] {$\footnotesize \fus\text{-composition}$} (0, -0.1); 
\draw[
	 color=gray, 
	 opacity=0.9, 
	 thick, 
	 ->
	 ] 
	 (-0.1, 0) -- node[pos=0.5, color=black, above, sloped] {$\footnotesize \circ\text{-composition}$} (-0.1, 1); 
%
% deltaepsilon-label: 
\draw[line width=1] (0.31, 0.9) node[line width=0pt] (beta) {{\footnotesize \rotatebox{-64}{$\eps\sta\delta$}}};
% gamma-label: 
\draw[line width=1] (0.8, 0.45) node[line width=0pt] (beta) {{\footnotesize $\gamma$}};
% deltazeta-label: 
\draw[line width=1] (0.54, 0.78) node[line width=0pt] (beta) {{\footnotesize \rotatebox{40}{$\gamma\sta\zeta\sta\delta$}}};
% beta-label: 
\draw[line width=1] (0.45, 0.15) node[line width=0pt] (beta) {{\footnotesize $\beta$}};
% alpha-label: 
\draw[line width=1] (0.13, 0.45) node[line width=0pt] (beta) {{\footnotesize $\alpha$}};
\end{tikzpicture}
%%%%%%%%%%%%%%%%%%%%%%
$$ 
of a map 
%arXiv_v2: 
	%$\Phi: X' \fus Y' \rightarrow X \fus Y \fus Z$. 
	$\Phi \colon X' \fus Y' \rightarrow X \fus (Y \sta 1_\delta) \fus (1_{\gamma} \sta Z)$. 

In this language, the tensorator $\sigma_{X,Y}$ in Definition~\ref{def:Graycat} is simply a cube with two $y = \text{const.}$ planes with embedded $X$- and $Y$-lines braiding past each other: 
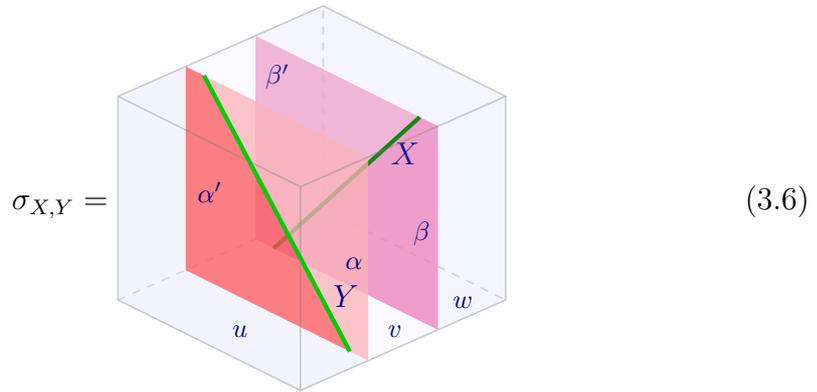
\be\label{eq:tensoratordiagram}
\sigma_{X,Y} = 
%%%%%%%%%%%%%%%%%%%%%% 
\begin{tikzpicture}[thick,scale=3.0,color=blue!50!black, baseline=0.0cm, >=stealth, 
				style={x={(-0.9cm,-0.4cm)},y={(0.8cm,-0.4cm)},z={(0cm,0.9cm)}}]
% invisible edges of cube: 
\draw[
	 color=gray, 
	 opacity=0.3, 
	 semithick,
	 dashed
	 ] 
	 (1,0,0) -- (0,0,0) -- (0,1,0)
	 (0,0,0) -- (0,0,1);
%%%
% for 3-strata: 
\coordinate (d) at (-0.34, 0, 0);
\coordinate (d2) at (-0.67, 0, 0);
% 3-stratum v at back: 
\fill [blue!40,opacity=0.1] ($(d2) + (1,0,0)$) -- ($(d2) + (1,1,0)$) -- ($(d2) + (0.67,1,0)$) -- ($(d2) + (0.67,1,1)$) -- ($(d2) + (0.67,0,1)$) -- ($(d2) + (1,0,1)$);
%
% lower endpoint of X: 
\coordinate (X0) at (0.33, 0.1, 0);
% upper endpoint of X: 
\coordinate (X1) at (0.33, 0.9, 1);
% lower endpoint of Y: 
\coordinate (Y0) at (0.67, 0.9, 0);
% upper endpoint of Y: 
\coordinate (Y1) at (0.67, 0.1, 1);
%
% beta'-plane: 
\fill [magenta!50,opacity=0.7] (0.33, 0, 0) -- (X0) -- (X1) -- (0.33, 1, 1) -- (0.33, 0, 1);
% beta-plane:
\fill [magenta!70,opacity=0.7] (X0) -- (X1) -- (0.33, 1, 1) -- (0.33, 1, 0) -- (0.33, 0.2, 0);
% X-line:
\draw[color=green!50!black, ultra thick] (X0) -- node[pos=0.9, color=blue!50!black, below] {$\footnotesize X$} (X1);  
% beta label: 
\draw[line width=1] (0.4, 0.99, 0.5) node[line width=0pt] (beta) {{\footnotesize $\beta$}};
% beta' label: 
\draw[line width=1] (0.4, 0.2, 0.92) node[line width=0pt] (beta) {{\footnotesize $\beta'$}};
%
% 3-stratum in middle: 
\fill [blue!20,opacity=0.1] ($(d) + (1,0,0)$) -- ($(d) + (1,1,0)$) -- ($(d) + (0.67,1,0)$) -- ($(d) + (0.67,1,1)$) -- ($(d) + (0.67,0,1)$) -- ($(d) + (1,0,1)$);
%
% 3-stratum u in front: 
\fill [blue!20,opacity=0.2] (1,0,0) -- (1,1,0) -- (0.67,1,0) -- (0.67,1,1) -- (0.67,0,1) -- (1,0,1);
%
% alpha'-plane: 
\fill [red!70,opacity=0.7] (0.67, 0, 0) -- (Y0) -- (Y1) -- (0.67, 0, 1);
% alpha-plane: 
\fill [red!30,opacity=0.7] (Y1) -- (0.67, 1, 1) -- (0.67, 1, 0) -- (0.67, 0.2, 0) -- (Y0);
% Y-line:
\draw[color=green!80!black, ultra thick] (Y0) -- node[pos=0.2, color=blue!50!black, right] {$\footnotesize Y$} (Y1);  
% alpha label: 
\draw[line width=1] (0.73, 0.99, 0.5) node[line width=0pt] (beta) {{\footnotesize $\alpha$}};
% alpha' label: 
\draw[line width=1] (0.73, 0.2, 0.5) node[line width=0pt] (beta) {{\footnotesize $\alpha'$}};
%
%
% object labels: 
\draw[line width=1] (0.85, 0.5, 0) node[line width=0pt] (beta) {{\footnotesize $u$}};
\draw[line width=1] (0.45, 0.9, 0) node[line width=0pt] (beta) {{\footnotesize $v$}};
\draw[line width=1] (0.12, 0.9, 0) node[line width=0pt] (beta) {{\footnotesize $w$}};
%
% visible edges of cube
\draw[
	 color=gray, 
	 opacity=0.4, 
	 semithick
	 ] 
	 (0,1,1) -- (0,1,0) -- (1,1,0) -- (1,1,1) -- (0,1,1) -- (0,0,1) -- (1,0,1) -- (1,0,0) -- (1,1,0)
	 (1,0,1) -- (1,1,1);
%%
%% composition arrows: 
%\draw[
%	 color=gray, 
%	 opacity=0.9, 
%	 thick, 
%	 ->
%	 ] 
%	 (1, 1.15, 0) -- node[pos=0.5, color=black, below, sloped] {$\footnotesize \sta\text{-composition}$} (0, 1.15, 0); 
%\draw[
%	 color=gray, 
%	 opacity=0.9, 
%	 thick, 
%	 ->
%	 ] 
%	 (1.13, 1, 0) -- node[pos=0.5, color=black, below, sloped] {$\footnotesize \fus\text{-composition}$} (1.15, 0, 0); 
%\draw[
%	 color=gray, 
%	 opacity=0.9, 
%	 thick, 
%	 ->
%	 ] 
%	 (1, -0.15, 0) -- node[pos=0.5, color=black, above, sloped] {$\footnotesize \circ\text{-composition}$} (1, -0.15, 1); 
\end{tikzpicture}
%%%%%%%%%%%%%%%%%%%%%%
\ee
This makes it diagrammatically clear how the interchange law~\eqref{eq:2dinterchangelaw} is weakened in a Gray category: the left and right sides of~\eqref{eq:2dinterchangelaw} precisely correspond to the bottom and top boundary of the cube in~\eqref{eq:tensoratordiagram}, respectively, i.\,e.~the source and target of $\sigma_{X,Y}$. 
Hence the interchange law holds in the strict sense iff the tensorator is the identity. 

From the diagram~\eqref{eq:tensoratordiagram} for the tensorator it is also apparent how braided monoidal categories are special cases of Gray categories. 
Indeed, if in the definition~\eqref{eq:tensorator} of $\sigma_{X,Y}$ we take $u=v=w$ and $\alpha = \alpha' = \beta = \beta' = 1_u$, then $(\sigma_{X,Y})$ gives a braided structure on the monoidal category of 2-endomorphisms of $1_u$. 

Finally, we observe that, analogously to the 2-dimensional case, the axioms of a Gray category are built into its diagrammatic calculus: 
Since the unit $1_\alpha$ on a 1-morphism~$\alpha$ corresponds to an unstratified $\alpha$-decorated plane, and since $\sta$-composition corresponds to stacking such planes together, conditions (vi) and (viii) in Definition~\ref{def:Graycat} are manifest in the diagrammatic calculus. 
Similarly, 
%arXiv_v2: 
	the first identity of 
condition (vii) is encoded in the isotopy invariance of Gray category diagrams: 
$$
%%%%%%%%%%%%%%%%%%%%%% 
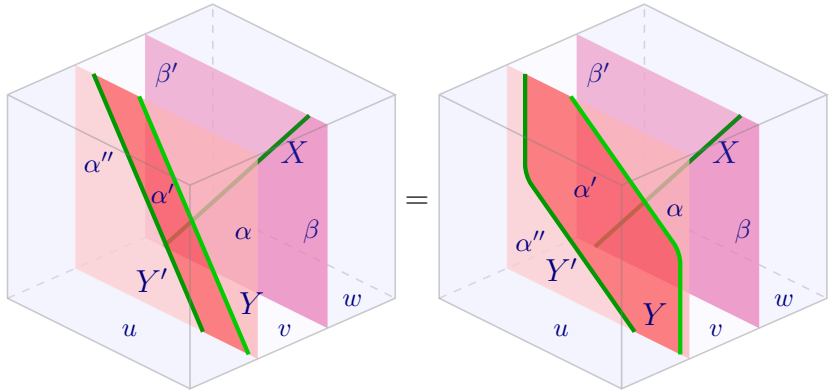
\begin{tikzpicture}[thick,scale=3.0,color=blue!50!black, baseline=0.0cm, 
				style={x={(-0.9cm,-0.4cm)},y={(0.8cm,-0.4cm)},
    z={(0cm,0.9cm)}}]
% invisible edges of cube: 
\draw[
	 color=gray, 
	 opacity=0.3, 
	 semithick,
	 dashed
	 ] 
	 (1,0,0) -- (0,0,0) -- (0,1,0)
	 (0,0,0) -- (0,0,1);
%%%
% for 3-strata: 
\coordinate (d) at (-0.34, 0, 0);
\coordinate (d2) at (-0.67, 0, 0);
% 3-stratum v at back: 
\fill [blue!40,opacity=0.1] ($(d2) + (1,0,0)$) -- ($(d2) + (1,1,0)$) -- ($(d2) + (0.67,1,0)$) -- ($(d2) + (0.67,1,1)$) -- ($(d2) + (0.67,0,1)$) -- ($(d2) + (1,0,1)$);
%
% lower endpoint of X: 
\coordinate (X0) at (0.33, 0.1, 0);
% upper endpoint of X: 
\coordinate (X1) at (0.33, 0.9, 1);
% lower endpoint of Y: 
\coordinate (Y0) at (0.67, 0.95, 0);
% upper endpoint of Y: 
\coordinate (Y1) at (0.67, 0.35, 1);
% lower endpoint of Y': 
\coordinate (Ys0) at (0.67, 0.7, 0);
% upper endpoint of Y': 
\coordinate (Ys1) at (0.67, 0.1, 1);
%
%
% beta'-plane: 
\fill [magenta!50,opacity=0.7] (0.33, 0, 0) -- (X0) -- (X1) -- (0.33, 1, 1) -- (0.33, 0, 1);
% beta-plane:
\fill [magenta!70,opacity=0.7] (X0) -- (X1) -- (0.33, 1, 1) -- (0.33, 1, 0) -- (0.33, 0.2, 0);
% X-line:
\draw[color=green!50!black, ultra thick] (X0) -- node[pos=0.9, color=blue!50!black, below] {$\footnotesize X$} (X1);  
% beta label: 
\draw[line width=1] (0.4, 0.99, 0.5) node[line width=0pt] (beta) {{\footnotesize $\beta$}};
% beta' label: 
\draw[line width=1] (0.4, 0.2, 0.92) node[line width=0pt] (beta) {{\footnotesize $\beta'$}};
%
% 3-stratum in middle: 
\fill [blue!20,opacity=0.1] ($(d) + (1,0,0)$) -- ($(d) + (1,1,0)$) -- ($(d) + (0.67,1,0)$) -- ($(d) + (0.67,1,1)$) -- ($(d) + (0.67,0,1)$) -- ($(d) + (1,0,1)$);
%
% 3-stratum u in front: 
\fill [blue!20,opacity=0.2] (1,0,0) -- (1,1,0) -- (0.67,1,0) -- (0.67,1,1) -- (0.67,0,1) -- (1,0,1);
%
% alpha'-plane: 
\fill [red!20,opacity=0.7] (0.67, 0, 0) -- (Ys0) -- (Ys1) -- (0.67, 0, 1);
% alpha'-plane: 
\fill [red!70,opacity=0.7] (Y0) -- (Ys0) -- (Ys1) -- (Y1);
% alpha-plane: 
\fill [red!30,opacity=0.7] (Y1) -- (0.67, 1, 1) -- (0.67, 1, 0) -- (0.67, 0.2, 0) -- (Y0);
% Y-line:
\draw[color=green!80!black, ultra thick] (Y0) -- node[pos=0.2, color=blue!50!black, right] {$\footnotesize Y$} (Y1);  
% Y'-line:
\draw[color=green!60!black, ultra thick] (Ys0) -- node[pos=0.2, color=blue!50!black, left] {$\footnotesize Y'$} (Ys1);  
% alpha label: 
\draw[line width=1] (0.73, 0.99, 0.64) node[line width=0pt] (beta) {{\footnotesize $\alpha$}};
% alpha'' label: 
\draw[line width=1] (0.73, 0.55, 0.64) node[line width=0pt] (beta) {{\footnotesize $\alpha'$}};
% alpha'' label: 
\draw[line width=1] (0.73, 0.2, 0.64) node[line width=0pt] (beta) {{\footnotesize $\alpha''$}};
%
%
% object labels: 
\draw[line width=1] (0.85, 0.5, 0) node[line width=0pt] (beta) {{\footnotesize $u$}};
\draw[line width=1] (0.45, 0.9, 0) node[line width=0pt] (beta) {{\footnotesize $v$}};
\draw[line width=1] (0.12, 0.9, 0) node[line width=0pt] (beta) {{\footnotesize $w$}};
%
% visible edges of cube
\draw[
	 color=gray, 
	 opacity=0.4, 
	 semithick
	 ] 
	 (0,1,1) -- (0,1,0) -- (1,1,0) -- (1,1,1) -- (0,1,1) -- (0,0,1) -- (1,0,1) -- (1,0,0) -- (1,1,0)
	 (1,0,1) -- (1,1,1);
\end{tikzpicture}
%%%%%%%%%%%%%%%%%%%%%%
=
%%%%%%%%%%%%%%%%%%%%%% 
\begin{tikzpicture}[thick,scale=3.0,color=blue!50!black, baseline=0.0cm, 
				style={x={(-0.9cm,-0.4cm)},y={(0.8cm,-0.4cm)},
    z={(0cm,0.9cm)}}]
% invisible edges of cube: 
\draw[
	 color=gray, 
	 opacity=0.3, 
	 semithick,
	 dashed
	 ] 
	 (1,0,0) -- (0,0,0) -- (0,1,0)
	 (0,0,0) -- (0,0,1);
%%%
% for 3-strata: 
\coordinate (d) at (-0.34, 0, 0);
\coordinate (d2) at (-0.67, 0, 0);
% 3-stratum v at back: 
\fill [blue!40,opacity=0.1] ($(d2) + (1,0,0)$) -- ($(d2) + (1,1,0)$) -- ($(d2) + (0.67,1,0)$) -- ($(d2) + (0.67,1,1)$) -- ($(d2) + (0.67,0,1)$) -- ($(d2) + (1,0,1)$);
%
% lower endpoint of X: 
\coordinate (X0) at (0.33, 0.1, 0);
% upper endpoint of X: 
\coordinate (X1) at (0.33, 0.9, 1);
% lower endpoint of Y: 
\coordinate (Y0) at (0.67, 0.95, 0);
% upper endpoint of Y: 
\coordinate (Y1) at (0.67, 0.35, 1);
% lower endpoint of Y': 
\coordinate (Ys0) at (0.67, 0.7, 0);
% upper endpoint of Y': 
\coordinate (Ys1) at (0.67, 0.1, 1);
%
%
% beta'-plane: 
\fill [magenta!50,opacity=0.7] (0.33, 0, 0) -- (X0) -- (X1) -- (0.33, 1, 1) -- (0.33, 0, 1);
% beta-plane:
\fill [magenta!70,opacity=0.7] (X0) -- (X1) -- (0.33, 1, 1) -- (0.33, 1, 0) -- (0.33, 0.2, 0);
% X-line:
\draw[color=green!50!black, ultra thick] (X0) -- node[pos=0.9, color=blue!50!black, below] {$\footnotesize X$} (X1);  
% beta label: 
\draw[line width=1] (0.4, 0.99, 0.5) node[line width=0pt] (beta) {{\footnotesize $\beta$}};
% beta' label: 
\draw[line width=1] (0.4, 0.2, 0.92) node[line width=0pt] (beta) {{\footnotesize $\beta'$}};
%
% 3-stratum in middle: 
\fill [blue!20,opacity=0.1] ($(d) + (1,0,0)$) -- ($(d) + (1,1,0)$) -- ($(d) + (0.67,1,0)$) -- ($(d) + (0.67,1,1)$) -- ($(d) + (0.67,0,1)$) -- ($(d) + (1,0,1)$);
%
% 3-stratum u in front: 
\fill [blue!20,opacity=0.2] (1,0,0) -- (1,1,0) -- (0.67,1,0) -- (0.67,1,1) -- (0.67,0,1) -- (1,0,1);
%
% alpha''-plane: 
\fill [red!20,opacity=0.7] (0.67, 0, 0) -- (Ys0) -- (0.67, 0.1, 0.5) -- (Ys1) -- (0.67, 0, 1);
% alpha'-plane: 
\fill [red!70,opacity=0.7] (0.67, 0.95, 0.5)-- (Y0) -- (Ys0) -- (0.67, 0.1, 0.5) -- (Ys1) -- (Y1);
% alpha-plane: 
\fill [red!30,opacity=0.7] (Y1) -- (0.67, 1, 1) -- (0.67, 1, 0) -- (0.67, 0.2, 0) -- (Y0) -- (0.67, 0.95, 0.5);
% Y-line:
\draw[color=green!80!black, ultra thick, rounded corners] (Y0) -- node[pos=0.4, color=blue!50!black, left] {$\footnotesize Y$} (0.67, 0.95, 0.5) -- (Y1);  
% Y'-line:
\draw[color=green!60!black, ultra thick, rounded corners] (Ys0) -- node[pos=0.4, color=blue!50!black, left] {$\footnotesize Y'$} (0.67, 0.1, 0.5) -- (Ys1);  
% alpha label: 
\draw[line width=1] (0.73, 0.99, 0.75) node[line width=0pt] (beta) {{\footnotesize $\alpha$}};
% alpha'' label: 
\draw[line width=1] (0.73, 0.5, 0.65) node[line width=0pt] (beta) {{\footnotesize $\alpha'$}};
% alpha'' label: 
\draw[line width=1] (0.73, 0.2, 0.25) node[line width=0pt] (beta) {{\footnotesize $\alpha''$}};
%
%
% object labels: 
\draw[line width=1] (0.85, 0.5, 0) node[line width=0pt] (beta) {{\footnotesize $u$}};
\draw[line width=1] (0.45, 0.9, 0) node[line width=0pt] (beta) {{\footnotesize $v$}};
\draw[line width=1] (0.12, 0.9, 0) node[line width=0pt] (beta) {{\footnotesize $w$}};
%
% visible edges of cube
\draw[
	 color=gray, 
	 opacity=0.4, 
	 semithick
	 ] 
	 (0,1,1) -- (0,1,0) -- (1,1,0) -- (1,1,1) -- (0,1,1) -- (0,0,1) -- (1,0,1) -- (1,0,0) -- (1,1,0)
	 (1,0,1) -- (1,1,1);
\end{tikzpicture}
%%%%%%%%%%%%%%%%%%%%%%
$$
and analogously for the 
%arXiv_v2: 
	%second identity Definition~\ref{def:Graycat}\,(vii). 
	second identity.

\subsection{Duals}
\label{subsec:duals}

In this section we recall duality and pivotal structures for 2-categories, and we explain how to enhance the free 2-category associated to a computad to a strictly pivotal 2-category. 
Then we define `Gray categories with duals' and discuss their 3-dimensional graphical calculus. 

\subsubsection{Bicategories}
\label{subsubsec:bicatsduals}

Given a 2-category~$\B$, 
%arXiv_v2: 
	%the 
	a 
\textsl{(right) dual} of a 1-morphism $X\in \B(\alpha,\beta)$ is a 1-morphism $X\dagg \in \B(\beta,\alpha)$ together with 2-morphisms $\ev_X : X\dagg \fus X \rightarrow 1_\alpha$ and $\coev_X: 1_\beta \rightarrow X \fus X\dagg$ such that the \textsl{Zorro moves} are satisfied. 
Our convention for the diagrammatic calculus is that upward oriented lines labelled by~$X$ represent $1_X$, so the (co)evaluation maps are
%arXiv_v2: [added name to equation]
\be
\label{eq:evcoevin2d}
\ev_X = 
%%%%%%%%%%%%%%%%%%%%%% 
\begin{tikzpicture}[thick,scale=3.0,color=black, >=stealth, color=blue!50!black, baseline=1.4cm]
% alpha-colour: 
\fill [red!16,opacity=1.0] (0,0) -- (0,1) -- (1,1) -- (1,0);
% beta-colour: 
\fill [magenta!30,opacity=1.0] (0.8, 0) .. controls +(0,0.5) and +(0,0.5) .. (0.2, 0);
% alpha-label
\draw[line width=1] (0.5,0.7) node[line width=0pt] (alpha) {{\footnotesize $\alpha$}};
% beta-label
\draw[line width=1] (0.5,0.15) node[line width=0pt] (alpha) {{\footnotesize $\beta$}};
%
%
%: X-line: 
\draw[string=green!60!black, very thick] (0.8, 0) .. controls +(0,0.5) and +(0,0.5) .. (0.2, 0);
% X-label: 
\draw[line width=1] (0.8,0.32) node[line width=0pt] (alpha) {{\footnotesize $X$}};
\end{tikzpicture}
%%%%%%%%%%%%%%%%%%%%%% 
\, , \quad
\coev_X = 
%%%%%%%%%%%%%%%%%%%%%% 
\begin{tikzpicture}[thick,scale=3.0,color=black, >=stealth, color=blue!50!black, baseline=1.4cm]
% beta-colour: 
\fill [magenta!30,opacity=1.0] (0,0) -- (0,1) -- (1,1) -- (1,0);
% alpha-colour: 
\fill [red!16,opacity=1.0] (0.8, 1) .. controls +(0,-0.5) and +(0,-0.5) .. (0.2, 1);
% alpha-label
\draw[line width=1] (0.5,0.85) node[line width=0pt] (alpha) {{\footnotesize $\alpha$}};
% beta-label
\draw[line width=1] (0.5,0.3) node[line width=0pt] (alpha) {{\footnotesize $\beta$}};
%
%
%: X-line: 
\draw[string=green!60!black, very thick] (0.8, 1) .. controls +(0,-0.5) and +(0,-0.5) .. (0.2, 1);
% X-label: 
\draw[line width=1] (0.2,0.68) node[line width=0pt] (alpha) {{\footnotesize $X$}};
\end{tikzpicture}
%%%%%%%%%%%%%%%%%%%%%% 
\ee
and the Zorro moves are the identities
$$
%%%%%%%%%%%%%%%%%%%%%% 
\begin{tikzpicture}[thick,scale=3.0,color=black, >=stealth, color=green!60!black, baseline=1.4cm]
% alpha-colour: 
\fill [red!16,opacity=1.0] (0,0) -- (0,1) -- (1,1) -- (1,0);
% beta-colour: 
\fill [magenta!30,opacity=1.0] (0.8, 1) -- (0.8, 0.5) .. controls +(0,-0.25) and +(0,-0.25) .. (0.5, 0.5) -- (0.5, 0.5) .. controls +(0,0.25) and +(0,0.25) .. (0.2, 0.5) -- (0.2, 0) -- (1,0) -- (1,1);
%
%: X-line: 
\draw[color=green!60!black, very thick] (0.8, 1) -- (0.8, 0.5);
\draw[string=green!60!black, very thick] (0.8, 0.5) .. controls +(0,-0.25) and +(0,-0.25) .. (0.5, 0.5);
\draw[string=green!60!black, very thick] (0.5, 0.5) .. controls +(0,0.25) and +(0,0.25) .. (0.2, 0.5);
\draw[color=green!60!black, very thick] (0.2, 0.5) -- (0.2, 0);
% X-label: 
%\draw[line width=1] (0.8,0.32) node[line width=0pt] (alpha) {{\footnotesize $X$}};
\end{tikzpicture}
%%%%%%%%%%%%%%%%%%%%%% 
=
%%%%%%%%%%%%%%%%%%%%%% 
\begin{tikzpicture}[thick,scale=3.0,color=black, >=stealth, color=green!60!black, baseline=1.4cm]
% alpha-colour: 
\fill [red!16,opacity=1.0] (0,0) -- (0,1) -- (1,1) -- (1,0);
% beta-colour: 
\fill [magenta!30,opacity=1.0] (0.5, 1) -- (0.5,0) -- (1,0) -- (1,1);
%
%: X-line: 
\draw[color=green!60!black, postaction={decorate}, decoration={markings,mark=at position .51 with {\arrow[draw]{>}}}, very thick] (0.5, 1) -- (0.5,0); 
% X-label: 
%\draw[line width=1] (0.8,0.32) node[line width=0pt] (alpha) {{\footnotesize $X$}};
\end{tikzpicture}
%%%%%%%%%%%%%%%%%%%%%% 
\, , \quad
%%%%%%%%%%%%%%%%%%%%%% 
\begin{tikzpicture}[thick,scale=3.0,color=black, >=stealth, color=green!60!black, baseline=1.4cm]
% beta-colour: 
\fill [magenta!30,opacity=1.0] (0,0) -- (0,1) -- (1,1) -- (1,0);
% alpha-colour: 
\fill [red!16,opacity=1.0] (0.8, 0) -- (0.8, 0.5) .. controls +(0,0.25) and +(0,0.25) .. (0.5, 0.5) -- (0.5, 0.5) .. controls +(0,-0.25) and +(0,-0.25) .. (0.2, 0.5) -- (0.2, 1) -- (1,1) -- (1,0);
%
%: X-line: 
\draw[color=green!60!black, very thick] (0.8, 0) -- (0.8, 0.5);
\draw[string=green!60!black, very thick] (0.8, 0.5) .. controls +(0,0.25) and +(0,0.25) .. (0.5, 0.5);
\draw[string=green!60!black, very thick] (0.5, 0.5) .. controls +(0,-0.25) and +(0,-0.25) .. (0.2, 0.5);
\draw[color=green!60!black, very thick] (0.2, 0.5) -- (0.2, 1);
% X-label: 
%\draw[line width=1] (0.8,0.32) node[line width=0pt] (alpha) {{\footnotesize $X$}};
\end{tikzpicture}
%%%%%%%%%%%%%%%%%%%%%% 
=
%%%%%%%%%%%%%%%%%%%%%% 
\begin{tikzpicture}[thick,scale=3.0,color=black, >=stealth, color=green!60!black, baseline=1.4cm]
% beta-colour: 
\fill [magenta!30,opacity=1.0] (0,0) -- (0,1) -- (1,1) -- (1,0);
% alpha-colour: 
\fill [red!16,opacity=1.0] (0.5, 1) -- (0.5,0) -- (1,0) -- (1,1);
%
%: X-line: 
\draw[color=green!60!black, postaction={decorate}, decoration={markings,mark=at position .51 with {\arrow[draw]{>}}}, very thick] (0.5, 0) -- (0.5,1); 
% X-label: 
%\draw[line width=1] (0.8,0.32) node[line width=0pt] (alpha) {{\footnotesize $X$}};
\end{tikzpicture}
%%%%%%%%%%%%%%%%%%%%%% 
\, . 
$$

If every 1-morphism in~$\B$ has a right dual we say that~$\B$ \textsl{has duals}. 
%arXiv_v2: 
	%In this case $(-)\dagg: \B \rightarrow \B^{\textrm{op}}$ is a 2-functor, where
	In this case choosing a right dual for every object makes $(-)\dagg: \B \rightarrow \B^{\textrm{op}}$ into a 2-functor with $\phi^\dagger = (\ev_Y \otimes 1_{X^\dagger}) \circ (1_{Y^\dagger} \otimes \phi \otimes 1_{X^\dagger}) \circ (1_{Y^\dagger} \otimes \coev_X)$ for $\phi \in \B(X,Y)$. Here, 
the \textsl{opposite} 2-category $\B^{\textrm{op}}$ is given by $\B^{\textrm{op}}(\alpha, \beta) = \B(\beta, \alpha)^{\textrm{op}}$, and the horizontal composite $X\fus Y$ in $\B^{\textrm{op}}$ is $Y\fus X$ in~$\B$, viewed as a 1-morphism in $\B^{\textrm{op}}$. 
It follows that $(\B^{\textrm{op}})^{\textrm{op}} = \B$. 

\begin{definition}
\begin{enumerate}
\item
A \textsl{strictly pivotal 2-category} is a 2-category~$\B$ which has duals such that
%arXiv_v2: 
	(for a prescribed choice of duals) 
$(-)\dagg: \B \rightarrow \B^{\textrm{op}}$ is a strict 2-functor with 
%arXiv_v2: 
	%$(-)^{\dagger\dagger} = 1_\B$. 
	$(-)^{\dagger\dagger} = 1_\B$, where $(-)^{\dagger\dagger} = ((-)^{\dagger})^{\textrm{op}} \circ (-)^\dagger$. 
\item 
A \textsl{strictly pivotal 2-functor} is a 2-functor $F: \B \rightarrow \mathcal C$ between strictly pivotal 2-categories such that $F(X\dagg) = F(X)\dagg$ for all 1-morphism~$X$ in~$\B$. 
\end{enumerate}
\end{definition}

As observed e.\,g.~in \cite[Sect.\,2]{BarWes}, the evaluation maps of a strictly pivotal 2-category can be recovered from the coevaluation maps and the functor $(-)\dagg$ as $\ev_X = (\coev_{X\dagg})\dagg$, if certain conditions are satisfied. 
As we will encounter an analogue 
%arXiv_v2: 
	%if
	of 
this fact in our discussion of duals in Gray categories we state this precisely: 
a strictly pivotal 2-category is the same as a 2-category~$\B$ with a 2-functor $(-)\dagg: \B \rightarrow \B^{\textrm{op}}$ which acts as $1_\B$ on objects, and for all $X \in \B(\alpha, \beta)$ a 2-morphism $\coev_X : 1_\beta \rightarrow X \fus X\dagg$ such that $(-)^{\dagger\dagger} = 1_\B$ and 
\begin{align}
& \big( \Phi \fus 1_{X\dagg} \big) \circ \coev_X = \big( 1_Y \fus \Phi\dagg \big) \circ \coev_Y
 \, , \nonumber \\
& \big( 1_X \fus ( \coev_{X\dagg})\dagg \big) \circ \big( \!\coev_X \fus 1_X \big) = 1_X \, , \label{eq:evfromcoev}\\
& \big( 1_X \fus \coev_Z \fus 1_{X\dagg} \big) \circ \coev_X = \coev_{X\fus Z} \nonumber
\end{align}
for all $\Phi: X \rightarrow Y$ and suitably composable 1-morphisms $X,Y,Z$. 

\medskip

The free bicategories associated to a computad~$\K$ can be made into strictly pivotal bicategories by `freely adding duals'. 
We explain this for the free 2-category $\mathcal F \K$ and its `pivotalisation' $\mathcal F_{\textrm{d}} \K$; the discussion for $\mathcal F^{\textrm{p}} \K$ and $\mathcal F^{\textrm{b}} \K$ is very similar. 
Paralleling the graphical presentation of $\mathcal F\K$ reviewed in Section~\ref{subsubsec:bicats}, the 2-category $\mathcal F_{\textrm{d}} \K$ has the following description:\footnote{As for the free 2-category $\mathcal F\K$ we denote horizontal composition vertically in $\mathcal F_{\textrm{d}} \K$, for compatibility with our constructions in Sections~\ref{subsec:tricatfromZ} and~\ref{subsec:GraycatwithdualsfromZ}.}
\begin{itemize}
\item
Objects of $\mathcal F_{\textrm{d}} \K$ are the same as objects of $\mathcal F\K$. 
\item 
1-morphisms~$X$ in $\mathcal F_{\textrm{d}} \K$ are 1-morphisms $(X_1, X_2, \dots, X_m)$ in $\mathcal F\K$ together with a choice of orientation of the 1-strata 
such that $\sigma_1(X_i)$ and $\tau_1(X_i)$ are the labels to the right and left, respectively, of the $X_i$-labelled line. 
The dual of~$X$ has the order of the $X_i$-lines and their orientations reversed, and we denote it $X^\#$. 
For example: 
\be\label{eq:Xhashfree}
X =  
%%%%%%%%%%%%%%%%%%%%%% 
\begin{tikzpicture}[thick,scale=3.0,color=black, baseline=1.4cm, rotate=90]
% 2-strata:
\fill [blue!40,opacity=0.4] (0,0) -- (0,1) -- (0.25,1) -- (0.25,0);
\fill [blue!50,opacity=0.5] (0.25,0) -- (0.25,1) -- (0.5,1) -- (0.5,0);
\fill [blue!30,opacity=0.4] (0.5,0) -- (0.5,1) -- (0.75,1) -- (0.75,0);
\fill [blue!50,opacity=0.4] (0.75,0) -- (0.75,1) -- (1,1) -- (1,0);
% 2-strata labels:
\draw[line width=1] (0.125,0.4) node[line width=0pt] (alpha) {{\tiny $\;\;\;\alpha = \tau_1(X_1)$}};
\draw[line width=1] (0.375,0.4) node[line width=0pt] (alpha) {{\tiny $\;\;\;\sigma_1(X_1) = \sigma_1(X_2)$}};
\draw[line width=1] (0.625,0.4) node[line width=0pt] (alpha) {{\tiny $\;\;\;\tau_1(X_2) = \sigma_1(X_3)$}};
\draw[line width=1] (0.875,0.4) node[line width=0pt] (alpha) {{\tiny $\;\;\;\beta = \tau_1(X_3)$}};
%
% 1-strata_
\draw[
	color=red!60, 
	very thick, 
	 >=stealth, 
	postaction={decorate}, decoration={markings,mark=at position .85 with {\arrow[draw]{>}}}
	] 
 (0.25,0) --  (0.25,1);
\draw[
	color=red!90!black, 
	very thick,
	 >=stealth, 
	postaction={decorate}, decoration={markings,mark=at position .85 with {\arrow[draw]{<}}}
	] 
 (0.5,0) --  (0.5,1);
\draw[
	color=red!90, 
	very thick,
	 >=stealth, 
	postaction={decorate}, decoration={markings,mark=at position .85 with {\arrow[draw]{<}}}
	] 
 (0.75,0) --  (0.75,1);
% 1-strata labels:
\draw[line width=1] (0.33,0.8) node[line width=0pt] (alpha) {{\footnotesize $X_1$}};
\draw[line width=1] (0.58,0.8) node[line width=0pt] (alpha) {{\footnotesize $X_2$}};
\draw[line width=1] (0.83,0.8) node[line width=0pt] (alpha) {{\footnotesize $X_3$}};
\end{tikzpicture}
%%%%%%%%%%%%%%%%%%%%%% 
\quad \implies \quad 
X^\#
= 
%%%%%%%%%%%%%%%%%%%%%% 
\begin{tikzpicture}[thick,scale=3.0,color=black, baseline=1.4cm, rotate=90]
% 2-strata:
\fill [blue!50,opacity=0.4] (0,0) -- (0,1) -- (0.25,1) -- (0.25,0);
\fill [blue!30,opacity=0.5] (0.25,0) -- (0.25,1) -- (0.5,1) -- (0.5,0);
\fill [blue!50,opacity=0.4] (0.5,0) -- (0.5,1) -- (0.75,1) -- (0.75,0);
\fill [blue!40,opacity=0.4] (0.75,0) -- (0.75,1) -- (1,1) -- (1,0);
% 2-strata labels:
\draw[line width=1] (0.125,0.4) node[line width=0pt] (alpha) {{\tiny $\;\;\;\beta = \tau_1(X_3)$}};
\draw[line width=1] (0.375,0.4) node[line width=0pt] (alpha) {{\tiny $\;\;\;\tau_1(X_2) = \sigma_1(X_3)$}};
\draw[line width=1] (0.625,0.4) node[line width=0pt] (alpha) {{\tiny $\;\;\;\sigma_1(X_1) = \sigma_1(X_2)$}};
\draw[line width=1] (0.875,0.4) node[line width=0pt] (alpha) {{\tiny $\;\;\;\alpha = \tau_1(X_1)$}};
%
% 1-strata_
\draw[
	color=red!90, 
	very thick, 
	 >=stealth, 
	postaction={decorate}, decoration={markings,mark=at position .85 with {\arrow[draw]{>}}}
	] 
 (0.25,0) --  (0.25,1);
\draw[
	color=red!90!black, 
	very thick,
	 >=stealth, 
	postaction={decorate}, decoration={markings,mark=at position .85 with {\arrow[draw]{>}}}
	] 
 (0.5,0) --  (0.5,1);
\draw[
	color=red!60, 
	very thick,
	 >=stealth, 
	postaction={decorate}, decoration={markings,mark=at position .85 with {\arrow[draw]{<}}}
	] 
 (0.75,0) --  (0.75,1);
% 1-strata labels:
\draw[line width=1] (0.33,0.8) node[line width=0pt] (alpha) {{\footnotesize $X_3$}};
\draw[line width=1] (0.58,0.8) node[line width=0pt] (alpha) {{\footnotesize $X_2$}};
\draw[line width=1] (0.83,0.8) node[line width=0pt] (alpha) {{\footnotesize $X_1$}};
\end{tikzpicture}
%%%%%%%%%%%%%%%%%%%%%% 
\ee
%arXiv_v2: 
	% Note that in the special case $\K = \K^\D$ of a computad associated to defect data~$\D$, there are no more 1-morphisms in $\mathcal F_{\textrm{d}} \K$ than there are in $\mathcal F\K$, as then the~$X_i$ already come with their orientations~$\eps_i$ specified. 
\item 
2-morphisms in $\mathcal F_{\textrm{d}}\K$ are isotopy classes of not-necessarily progressive string diagrams, i.\,e.~planar 2-category diagrams in the sense of \cite[Sect.\,3.2]{BMS}. 
For example, the right adjunction maps for the 1-morphism~$X$ in~\eqref{eq:Xhashfree} are 
\be\label{eq:evcoevfree}
\ev_X = 
%%%%%%%%%%%%%%%%%%%%%% 
\begin{tikzpicture}[thick,scale=3.0,color=black, baseline=1.4cm]
% 2-strata:
\fill [blue!40,opacity=0.4] (1, 0.6) .. controls +(-0.2,0) and +(-0.2,0) .. (1, 0.4);
\fill [blue!50,opacity=0.4] (1, 0.25) .. controls +(-0.4,0) and +(-0.4,0) .. (1, 0.75) -- (1, 0.6) .. controls +(-0.2,0) and +(-0.2,0) .. (1, 0.4);
\fill [blue!30,opacity=0.4] (1, 0.1) .. controls +(-0.6,0) and +(-0.6,0) .. (1, 0.9) -- (1, 0.75) .. controls +(-0.4,0) and +(-0.4,0) .. (1, 0.25) -- (1, 0.1);
\fill [blue!40,opacity=0.4] (1, 0.1) .. controls +(-0.6,0) and +(-0.6,0) .. (1, 0.9) -- (1,1) -- (0,1) -- (0,0) -- (1,0);
%
% 1-strata: 
\draw[string=red!60, >=stealth, very thick] (1, 0.6) .. controls +(-0.2,0) and +(-0.2,0) .. (1, 0.4);
\draw[string=red!90, >=stealth, very thick] (1, 0.25) .. controls +(-0.4,0) and +(-0.4,0) .. (1, 0.75);
\draw[string=red!90, >=stealth, very thick] (1, 0.1) .. controls +(-0.6,0) and +(-0.6,0) .. (1, 0.9);
\end{tikzpicture}
%%%%%%%%%%%%%%%%%%%%%% 
\, , \quad 
\coev_X = 
%%%%%%%%%%%%%%%%%%%%%% 
\begin{tikzpicture}[thick,scale=3.0,color=black, baseline=-1.6cm, rotate=180]
% 2-strata:
\fill [blue!40,opacity=0.4] (1, 0.6) .. controls +(-0.2,0) and +(-0.2,0) .. (1, 0.4);
\fill [blue!30,opacity=0.4] (1, 0.25) .. controls +(-0.4,0) and +(-0.4,0) .. (1, 0.75) -- (1, 0.6) .. controls +(-0.2,0) and +(-0.2,0) .. (1, 0.4);
\fill [blue!50,opacity=0.4] (1, 0.1) .. controls +(-0.6,0) and +(-0.6,0) .. (1, 0.9) -- (1, 0.75) .. controls +(-0.4,0) and +(-0.4,0) .. (1, 0.25) -- (1, 0.1);
\fill [blue!40,opacity=0.4] (1, 0.1) .. controls +(-0.6,0) and +(-0.6,0) .. (1, 0.9) -- (1,1) -- (0,1) -- (0,0) -- (1,0);
%
% 1-strata: 
\draw[string=red!60, >=stealth, very thick] (1, 0.6) .. controls +(-0.2,0) and +(-0.2,0) .. (1, 0.4);
\draw[string=red!90, >=stealth, very thick] (1, 0.75) .. controls +(-0.4,0) and +(-0.4,0) .. (1, 0.25);
\draw[string=red!90, >=stealth, very thick] (1, 0.1) .. controls +(-0.6,0) and +(-0.6,0) .. (1, 0.9);
\end{tikzpicture}
%%%%%%%%%%%%%%%%%%%%%% 
\, . 
\ee
The Zorro moves hold manifestly due to isotopy invariance. 
\item Composition and identities in $\mathcal F_{\textrm{d}}\K$ are the same as in $\mathcal F\K$. 
\end{itemize}

\begin{corollary}
Every set of defect data~$\D$ gives rise to a free strictly pivotal 2-category $\mathcal F_{\textrm{d}} \K^{\D}$, 
a free strictly pivotal pre-2-category $\mathcal F_{\textrm{d}}^{\textrm{p}} \K^{\D}$, 
and a free strictly pivotal bicategory $\mathcal F_{\textrm{d}}^{\textrm{b}} \K^{\D}$. 
\end{corollary}

In the pre-2-category $\mathcal F_{\textrm{d}}^{\textrm{p}} \K^{\D}$ we include adjunction maps for 1-morphisms as in~\eqref{eq:evcoevfree}, but the singular point (corresponding to the apex) of the projection to the $x$-axis is treated as a vertex in the equivalence relation for 2-morphisms. 
Hence, there is no allowed isotopy that moves it past another vertex. 
Moreover, we do not require the Zorro moves to hold in $\mathcal F_{\textrm{d}}^{\textrm{p}} \K^{\D}$.

\subsubsection{Tricategories}
\label{subsubsec:tricatsduals}

Roughly, a Gray category with duals~$\G$ has duals $X\dagg: \beta \rightarrow \alpha$ for 2-morphisms $X: \alpha \rightarrow \beta$, but also suitably compatible duals $\alpha\hash \in\G(v,u)$ for 1-morphisms $\alpha \in \G(u,v)$. 
The $\dagger$-dual is familiar, namely that of a strictly pivotal structure on the 2-categories $\G(u,v)$. 
However, the 3-dimensional analogue of the Zorro move for the $\#$-dual only holds up to a 3-isomorphism. 
The following definition is equivalent to \cite[Def.\,3.10]{BMS}, although reformulated with slightly different conventions:

\begin{definition}
\label{def:Graycatduals}
A \textsl{Gray category with duals} is a Gray category~$\G$ together with the following structure: 
\begin{enumerate}
\item
for all $u,v \in \G$, the 2-categories $\G(u,v)$ are strictly pivotal 2-categories with duals denoted $(-)\dagg$, and for all $\alpha \in \G(u,v)$ and $w\in \G$, the 2-functors
$$
\alpha \sta(-): \G(w,u) \lra \G(w,v)
\, , \quad 
(-) \sta \alpha : \G(v,w) \lra \G(u,w) 
$$
are strictly pivotal 2-functors; 
\item
%arXiv_v2: 
	%for all $\alpha \in \G(u,v)$ there is $\alpha\hash \in \G(v,u)$ 
	for every 1-morphism $\alpha \in \G(u,v)$ there is a 1-morphism $\alpha\hash \in \G(v,u)$ 
together with a 2-morphism $\coev_\alpha \colon 1_v \rightarrow \alpha \sta \alpha\hash$, called \textsl{fold}, and an invertible 3-morphism 
\be\label{eq:triangulator}
\tau_\alpha 
\colon  
\big( 1_\alpha \sta 
%arXiv_v2: 
	%( \coev_{\alpha\hash} )\dagg
	\ev_\alpha 
 \big) \fus \big( \coev_\alpha \sta 1_\alpha \big) \lra 1_{\alpha} \, ,
\ee
called \textsl{triangulator} (or \textsl{Zorro movie}), 
%arXiv_v2: 
	where $\ev_\alpha := ( \coev_{\alpha\hash} )\dagg \colon \alpha\hash \sta \alpha \to 1_u$, 
subject to the conditions
\begin{enumerate}
\item
$\alpha^{\#\#} = \alpha$ for all 1-morphisms~$\alpha$, 
\item 
$1_u\hash = 1_u$, 
$\coev_{1_u} = 1_{1_u}$, 
and $\tau_{1_u} = 1_{1_{1_u}}$ for all $u\in \G$, 
\item 
$( \alpha \sta \beta )\hash = \beta\hash \sta \alpha\hash$, 
$\coev_{\alpha\sta\beta} = (1_\alpha \sta \coev_\beta \sta 1_\alpha\hash ) \fus \coev_\alpha$, 
and 
\begin{align*}
\tau_{\alpha\sta\beta}
= \,
& \Big( \big( \tau_\alpha \sta 1_{1_\beta} \big) \fus \big( 1_{1_\alpha} \sta \tau_\beta \big) \Big) \\
& \quad \circ \Big( 1_{1_\alpha \sta 1_\beta \sta (\coev_{\beta\hash})\dagg} \otimes \sigma^{-1}_{1_\alpha \sta \coev_\beta, (\coev_{\alpha\hash})\dagg \sta 1_\beta} \fus 1_{\coev_\alpha \sta 1_\alpha \sta 1_\beta} \Big)
\end{align*}
for all $\alpha \in \G(u,v)$ and $\beta \in \G(w,u)$, where~$\sigma$ is the tensorator for~$\G$, 
\item
for all 1-morphisms $\alpha$: 
\begin{align*}
1_{\coev_\alpha} 
& = \Big(  \big( 
1_{1_\alpha} 
\Box (\tau_{\alpha^\#}^{-1})^\dag  
\big) \fus 1_{\coev_\alpha}
 \Big) 
\circ \Big( 1_{1_\alpha \sta (\coev_{\alpha\hash})\dagg \sta 1_{\alpha\hash}} \fus \sigma_{\coev_\alpha, \coev_\alpha} \Big) \\
& \quad \quad
\circ \Big(  (\tau_\alpha^\inv  \sta 1_{1_{\alpha^\#}})  \fus 1_{\coev_\alpha} \Big)
\, . 
\end{align*}
\end{enumerate}
\end{enumerate}
\end{definition}

We observe that the conditions in (ii) above are analogous to those of the 2-dimensional case in \eqref{eq:evfromcoev}. 
Expressing them graphically below will make them much more transparent. 

\begin{definition}
\label{def:strictFGraydual}
A \textsl{strict functor of Gray categories with duals} $F: \Cat{G} \rightarrow \Cat{G}'$ is a strict functor of Gray categories which is compatible with the duals in the sense that for all $u,v \in \Cat{G}$, the 2-functors $\Cat{G}(u,v)\rightarrow \Cat{G}'(F_{0}(u),F_{0}(v))$ are pivotal, for all 1-morphisms $\alpha$ in $\Cat{G}$ we have $F(\alpha^{\#})=F(\alpha)^{\#}$ and $F(\coev_{\alpha})=\coev_{F(\alpha)}$, and $F(\tau)=\tau'$ holds for the triangulators. 
$F$ is called an \textsl{equivalence of Gray categories with duals} if in addition it is an equivalence of Gray categories as in Definition~\ref{def:Grayequiv} .
\end{definition}

The graphical calculus for Gray categories with duals~$\G$ was developed in \cite[Sect.\,3.4]{BMS}. 
As in the 2-dimensional case one allows non-progressive diagrams in which surfaces may be folded. 
Hence locally, diagrams for~$\G$ either look like 3-dimensional progressive diagrams as in Section~\ref{subsubsec:tricats}, or like $\alpha$-folds 
$$
%arXiv_v2: 
	%\coev_\alpha = 
	1_{\coev_\alpha} = 
%%%%%%%%%%%%%%%%%%%%%% 
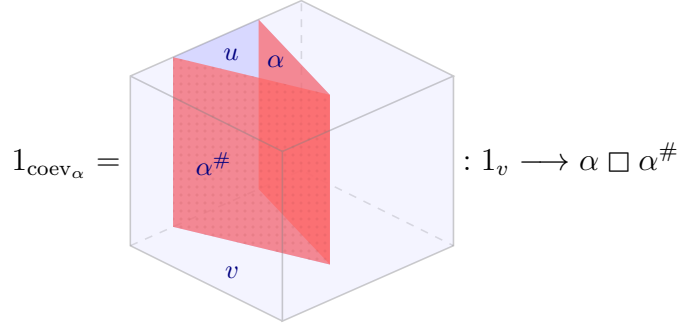
\begin{tikzpicture}[thick,scale=2.5,color=blue!50!black, baseline=0.0cm, >=stealth, 
				style={x={(-0.9cm,-0.4cm)},y={(0.8cm,-0.4cm)},z={(0cm,0.9cm)}}]
% invisible edges of cube: 
\draw[
	 color=gray, 
	 opacity=0.3, 
	 semithick,
	 dashed
	 ] 
	 (1,0,0) -- (0,0,0) -- (0,1,0)
	 (0,0,0) -- (0,0,1);
%%%
% for 3-strata: 
\coordinate (f0) at (0.5, 0.75, 0);
\coordinate (f1) at (0.5, 0.75, 1);
\coordinate (front0) at (0.75, 0, 0);
\coordinate (front1) at (0.75, 0, 1);
\coordinate (back0) at (0.25, 0, 0);
\coordinate (back1) at (0.25, 0, 1);
% 3-stratum v: 
\fill [blue!70,opacity=0.2] (back0) -- (back1) -- (f1) -- (f0) -- (front0) -- (front1) -- (back1);
% 3-stratum u: 
\fill [blue!20,opacity=0.2] (1,0,0) -- (1,1,0) -- (0,1,0) -- (0,1,1) -- (0,0,1) -- (1,0,1);
%
%
% alpha-plane: 
\fill [red!60, opacity=0.7] (f0) -- (f1) -- (back1) -- (back0);
% alpha-hash label: 
\draw[line width=1] (0.3, 0.17, 0.85) node[line width=0pt] (beta) {{\footnotesize $\alpha$}};
%
% alpha-hash-plane: 
\fill [pattern=dots, opacity=0.4] (f0) -- (f1) -- (front1) -- (front0);
\fill [red!60,opacity=0.7] (f0) -- (f1) -- (front1) -- (front0);
% alpha label: 
\draw[line width=1] (0.6, 0.1, 0.35) node[line width=0pt] (beta) {{\footnotesize $\alpha\hash$}};
%
% object labels: 
\draw[line width=1] (0.85, 0.5, 0) node[line width=0pt] (beta) {{\footnotesize $v$}};
\draw[line width=1] (0.5, 0.1, 0.95) node[line width=0pt] (beta) {{\footnotesize $u$}};
%
% visible edges of cube
\draw[
	 color=gray, 
	 opacity=0.4, 
	 semithick
	 ] 
	 (0,1,1) -- (0,1,0) -- (1,1,0) -- (1,1,1) -- (0,1,1) -- (0,0,1) -- (1,0,1) -- (1,0,0) -- (1,1,0)
	 (1,0,1) -- (1,1,1);
\end{tikzpicture}
%%%%%%%%%%%%%%%%%%%%%%
: 1_v \lra \alpha \sta \alpha\hash
$$
or like the $\dagger$-dual of the $\alpha\hash$-fold, 
$$
%arXiv_v2: 
	%\ev_\alpha := ( \coev_{\alpha\hash} )\dagg = 
	1_{\ev_\alpha} = 1_{( \coev_{\alpha\hash} )\dagg} = 
%%%%%%%%%%%%%%%%%%%%%% 
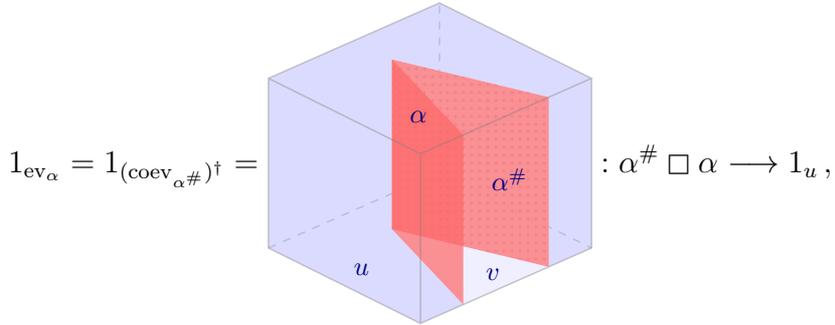
\begin{tikzpicture}[thick,scale=2.5,color=blue!50!black, baseline=0.0cm, >=stealth, 
				style={x={(-0.9cm,-0.4cm)},y={(0.8cm,-0.4cm)},z={(0cm,0.9cm)}}]
% invisible edges of cube: 
\draw[
	 color=gray, 
	 opacity=0.3, 
	 semithick,
	 dashed
	 ] 
	 (1,0,0) -- (0,0,0) -- (0,1,0)
	 (0,0,0) -- (0,0,1);
%%%
% for 3-strata: 
\coordinate (f0) at (0.5, 0.25, 0);
\coordinate (f1) at (0.5, 0.25, 1);
\coordinate (front0) at (0.75, 1, 0);
\coordinate (front1) at (0.75, 1, 1);
\coordinate (back0) at (0.25, 1, 0);
\coordinate (back1) at (0.25, 1, 1);
% 3-stratum v: 
%\fill [blue!70,opacity=0.2] (f0) -- (f1) -- (1,0,1) -- (1,0,0) -- (1,1,0) -- (front0);
\fill [blue!70,opacity=0.2] (1,0,0) -- (1,1,0) -- (0,1,0) -- (0,1,1) -- (0,0,1) -- (1,0,1);
% 3-stratum u: 
\fill [blue!5,opacity=0.9] (back0) -- (back1) -- (f1) -- (f0) -- (front0) -- (back0);
%
%
% alpha-hash-plane: 
\fill [pattern=dots, opacity=0.4] (f0) -- (f1) -- (back1) -- (back0);
\fill [red!60,opacity=0.7] (f0) -- (f1) -- (back1) -- (back0);
% alpha label: 
\draw[line width=1] (0.3, 0.8, 0.45) node[line width=0pt] (beta) {{\footnotesize $\alpha\hash$}};
%
% alpha-plane: 
\fill [red!60, opacity=0.7] (f0) -- (f1) -- (front1) -- (front0);
% alpha-hash label: 
\draw[line width=1] (0.3, 0.2, 0.55) node[line width=0pt] (beta) {{\footnotesize $\alpha$}};
%
% object labels: 
\draw[line width=1] (0.47, 0.88, 0) node[line width=0pt] (beta) {{\footnotesize $v$}};
\draw[line width=1] (0.9, 0.5, 0.05) node[line width=0pt] (beta) {{\footnotesize $u$}};
%
% visible edges of cube
\draw[
	 color=gray, 
	 opacity=0.4, 
	 semithick
	 ] 
	 (0,1,1) -- (0,1,0) -- (1,1,0) -- (1,1,1) -- (0,1,1) -- (0,0,1) -- (1,0,1) -- (1,0,0) -- (1,1,0)
	 (1,0,1) -- (1,1,1);
\end{tikzpicture}
%%%%%%%%%%%%%%%%%%%%%%
: \alpha\hash \sta \alpha \lra 1_u
\, , 
$$
%arXiv_v2: 
	or like the diagrams for~$\tau_\alpha$ and $\tau_\alpha^{-1}$ shown in~\eqref{eq:tautauinverse} below. 
As is implicit in these diagrams, the convention is that taking the $\dagger$-dual and the $\#$-dual corresponds to a rotation by~$\pi$ along the lines $y=z = \tfrac{1}{2}$ and $x=y = \tfrac{1}{2}$, respectively. 

Condition (i) in Definition~\ref{def:Graycatduals} states that $\sta$-composition is compatible with $\dagger$-duals. 
Analogously to the 2-dimensional case, this is built into the graphical calculus, e.\,g.
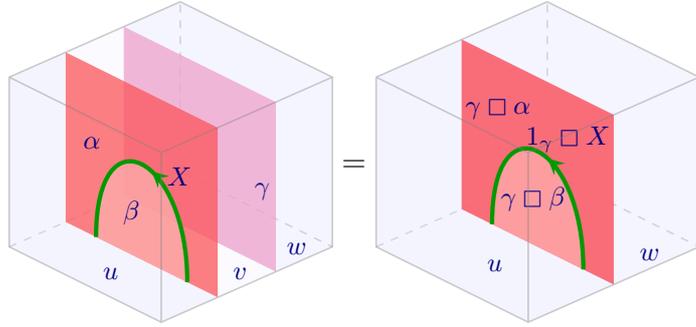
$$
%%%%%%%%%%%%%%%%%%%%%% 
\begin{tikzpicture}[thick,scale=2.5,color=blue!50!black, baseline=0.0cm, >=stealth, 
				style={x={(-0.9cm,-0.4cm)},y={(0.8cm,-0.4cm)},z={(0cm,0.9cm)}}]
% invisible edges of cube: 
\draw[
	 color=gray, 
	 opacity=0.3, 
	 semithick,
	 dashed
	 ] 
	 (1,0,0) -- (0,0,0) -- (0,1,0)
	 (0,0,0) -- (0,0,1);
%%%
% for 3-strata: 
\coordinate (d) at (-0.34, 0, 0);
\coordinate (d2) at (-0.67, 0, 0);
% 3-stratum v at back: 
\fill [blue!40,opacity=0.1] ($(d2) + (1,0,0)$) -- ($(d2) + (1,1,0)$) -- ($(d2) + (0.67,1,0)$) -- ($(d2) + (0.67,1,1)$) -- ($(d2) + (0.67,0,1)$) -- ($(d2) + (1,0,1)$);
%
% lower endpoint of X: 
\coordinate (X0) at (0.67, 0.8, 0);
% upper endpoint of X: 
\coordinate (X1) at (0.67, 0.2, 0);
%
%
% gamma-plane: 
\fill [magenta!50,opacity=0.7] (0.33, 0, 0) -- (0.33, 1, 0) -- (0.33, 1, 1) -- (0.33, 0, 1) -- (0.33, 0, 0);
% gamma label: 
\draw[line width=1] (0.4, 0.99, 0.5) node[line width=0pt] (beta) {{\footnotesize $\gamma$}};
%
% 3-stratum in middle: 
\fill [blue!20,opacity=0.1] ($(d) + (1,0,0)$) -- ($(d) + (1,1,0)$) -- ($(d) + (0.67,1,0)$) -- ($(d) + (0.67,1,1)$) -- ($(d) + (0.67,0,1)$) -- ($(d) + (1,0,1)$);
%
% 3-stratum u in front: 
\fill [blue!20,opacity=0.2] (1,0,0) -- (1,1,0) -- (0.67,1,0) -- (0.67,1,1) -- (0.67,0,1) -- (1,0,1);
%
% alpha-plane: 
\fill [red!70,opacity=0.7] (0.67, 0, 0) -- (0.67, 1, 0) -- (0.67, 1, 1) -- (0.67, 0, 1);
% beta-plane: 
\fill [red!30,opacity=0.7] (X0) .. controls +(0,0,0.75) and +(0,0,0.75) .. (X1);
%: X-line: 
\draw[string=green!60!black, ultra thick] (X0) .. controls +(0,0,0.75) and +(0,0,0.75) .. (X1);
% X label: 
\draw[line width=1] (0.67, 0.75, 0.6) node[line width=0pt] (beta) {{\footnotesize $X$}};
% alpha label: 
\draw[line width=1] (0.73, 0.24, 0.6) node[line width=0pt] (beta) {{\footnotesize $\alpha$}};
% beta label: 
\draw[line width=1] (0.73, 0.5, 0.3) node[line width=0pt] (beta) {{\footnotesize $\beta$}};
%
% object labels: 
\draw[line width=1] (0.85, 0.5, 0) node[line width=0pt] (beta) {{\footnotesize $u$}};
\draw[line width=1] (0.45, 0.9, 0) node[line width=0pt] (beta) {{\footnotesize $v$}};
\draw[line width=1] (0.12, 0.9, 0) node[line width=0pt] (beta) {{\footnotesize $w$}};
%
% visible edges of cube
\draw[
	 color=gray, 
	 opacity=0.4, 
	 semithick
	 ] 
	 (0,1,1) -- (0,1,0) -- (1,1,0) -- (1,1,1) -- (0,1,1) -- (0,0,1) -- (1,0,1) -- (1,0,0) -- (1,1,0)
	 (1,0,1) -- (1,1,1);
\end{tikzpicture}
%%%%%%%%%%%%%%%%%%%%%%
=
%%%%%%%%%%%%%%%%%%%%%% 
\begin{tikzpicture}[thick,scale=2.5,color=blue!50!black, baseline=0.0cm, >=stealth, 
				style={x={(-0.9cm,-0.4cm)},y={(0.8cm,-0.4cm)},z={(0cm,0.9cm)}}]
% invisible edges of cube: 
\draw[
	 color=gray, 
	 opacity=0.3, 
	 semithick,
	 dashed
	 ] 
	 (1,0,0) -- (0,0,0) -- (0,1,0)
	 (0,0,0) -- (0,0,1);
%%%
% for 3-strata: 
\coordinate (d2) at (-0.5, 0, 0);
% 3-stratum v at back: 
\fill [blue!40,opacity=0.1] ($(d2) + (1,0,0)$) -- ($(d2) + (1,1,0)$) -- ($(d2) + (0.5,1,0)$) -- ($(d2) + (0.5,1,1)$) -- ($(d2) + (0.5,0,1)$) -- ($(d2) + (1,0,1)$);
%
% lower endpoint of X: 
\coordinate (X0) at (0.5, 0.8, 0);
% upper endpoint of X: 
\coordinate (X1) at (0.5, 0.2, 0);
%
%
% gamma-plane: 
\fill [magenta!50,opacity=0.7] (0.5, 0, 0) -- (0.5, 1, 0) -- (0.5, 1, 1) -- (0.5, 0, 1) -- (0.5, 0, 0);
%
%
% 3-stratum u in front: 
\fill [blue!20,opacity=0.2] (1,0,0) -- (1,1,0) -- (0.5,1,0) -- (0.5,1,1) -- (0.5,0,1) -- (1,0,1);
%
% alpha-plane: 
\fill [red!70,opacity=0.7] (0.5, 0, 0) -- (0.5, 1, 0) -- (0.5, 1, 1) -- (0.5, 0, 1);
% beta-plane: 
\fill [red!30,opacity=0.7] (X0) .. controls +(0,0,0.75) and +(0,0,0.75) .. (X1);
%: X-line: 
\draw[string=green!60!black, ultra thick] (X0) .. controls +(0,0,0.75) and +(0,0,0.75) .. (X1);
% X label: 
\draw[line width=1] (0.5, 0.7, 0.72) node[line width=0pt] (beta) {{\footnotesize $1_\gamma \sta X$}};
% alpha label: 
\draw[line width=1] (0.5, 0.24, 0.7) node[line width=0pt] (beta) {{\footnotesize $\gamma \sta \alpha$}};
% beta label: 
\draw[line width=1] (0.5, 0.47, 0.27) node[line width=0pt] (beta) {{\footnotesize $\gamma \sta \beta$}};
%
%
% object labels: 
\draw[line width=1] (0.75, 0.5, 0) node[line width=0pt] (beta) {{\footnotesize $u$}};
\draw[line width=1] (0.2, 0.9, 0) node[line width=0pt] (beta) {{\footnotesize $w$}};
%
% visible edges of cube
\draw[
	 color=gray, 
	 opacity=0.4, 
	 semithick
	 ] 
	 (0,1,1) -- (0,1,0) -- (1,1,0) -- (1,1,1) -- (0,1,1) -- (0,0,1) -- (1,0,1) -- (1,0,0) -- (1,1,0)
	 (1,0,1) -- (1,1,1);
\end{tikzpicture}
%%%%%%%%%%%%%%%%%%%%%%
\, . 
$$
%arXiv_v2: 
	Here and below we are using the convention (consistent with the convention in~\eqref{eq:evcoevin2d}) that a line representing a 2-morphism $X\colon \alpha \to \beta$ is oriented such that the surface representing the 1-morphism~$\alpha$ is to the right of the line. Diagrams for Gray categories with duals are again evaluated by projecting to the back of the cube (i.\,e.~the face with $x=0$) as in Section~\ref{subsubsec:tricats}, inducing the orientations of lines in the projected 2-dimensional diagram. For further details, we refer to \cite[Sect.\,3.4]{BMS}. 

The triangulator~\eqref{eq:triangulator} is a 3-morphism that `morphs Zorro's~Z, lying at the bottom of the cube, into a straight line at the top'. 
In pictures, the triangulator and its inverse are given by:
%arXiv_v2: [added name to equation]
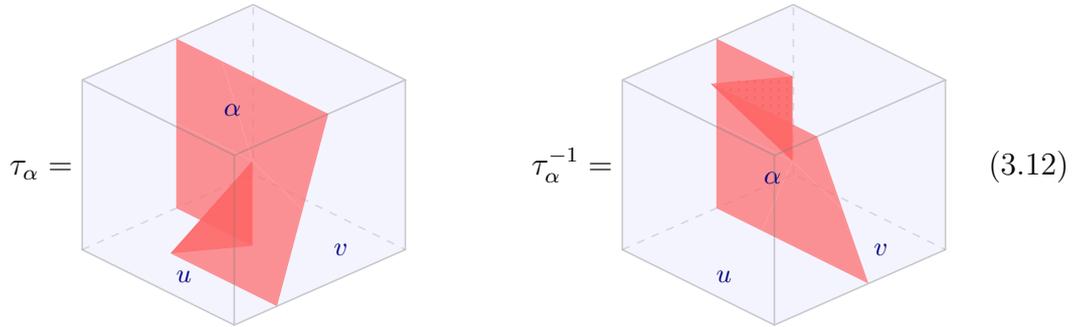
\be
\label{eq:tautauinverse}
\tau_\alpha = 
%%%%%%%%%%%%%%%%%%%%%% 
\begin{tikzpicture}[thick,scale=2.5,color=blue!50!black, baseline=0.0cm, >=stealth, 
				style={x={(-0.9cm,-0.4cm)},y={(0.8cm,-0.4cm)},z={(0cm,0.9cm)}}]
% invisible edges of cube: 
\draw[
	 color=gray, 
	 opacity=0.3, 
	 semithick,
	 dashed
	 ] 
	 (1,0,0) -- (0,0,0) -- (0,1,0)
	 (0,0,0) -- (0,0,1);
%%%
% for 3-strata: 
\coordinate (b1) at (0.75, 0.3, 0);
\coordinate (b2) at (0.45, 0.5, 0);
\coordinate (cusp) at (0.45, 0.5, 0.5);
\coordinate (abovecusp) at (0.45, 0.3, 1);
\coordinate (left0) at (0.45, 0, 0);
\coordinate (leftmid) at (0.45, 0, 0.5);
\coordinate (left1) at (0.45, 0, 1);
\coordinate (right0) at (0.75, 1, 0);
\coordinate (rightmid) at (0.6, 1, 0.5);
\coordinate (right1) at (0.45, 1, 1);
%
% 3-stratum: 
\fill [blue!20,opacity=0.2] (1,0,0) -- (1,1,0) -- (0,1,0) -- (0,1,1) -- (0,0,1) -- (1,0,1);
%
% alpha-plane bottom left: 
\fill [red!60, opacity=0.7] (b2) -- (left0) -- (leftmid) -- (cusp);
% alpha-plane bottom left: 
\fill [red!60, opacity=0.7] (leftmid) -- (cusp) -- (abovecusp) -- (left1);
% alpha-hash-plane: 
\fill [pattern=dots, opacity=0.4] (cusp) -- (b1) -- (b2);
\fill [red!60, opacity=0.7] (cusp) -- (b1) -- (b2);
% alpha-plane bottom right: 
\fill [red!60, opacity=0.7] (cusp) -- (b1) -- (right0) -- (rightmid);
% alpha-plane top right: 
\fill [red!60, opacity=0.7] (cusp) -- (rightmid) -- (right1) -- (abovecusp);
% alpha label: 
\draw[line width=1] (0.3, 0.2, 0.6) node[line width=0pt] (beta) {{\footnotesize $\alpha$}};
%
% object labels: 
\draw[line width=1] (0.85, 0.5, 0) node[line width=0pt] (beta) {{\footnotesize $u$}};
\draw[line width=1] (0.2, 0.8, 0) node[line width=0pt] (beta) {{\footnotesize $v$}};
%
% visible edges of cube
\draw[
	 color=gray, 
	 opacity=0.4, 
	 semithick
	 ] 
	 (0,1,1) -- (0,1,0) -- (1,1,0) -- (1,1,1) -- (0,1,1) -- (0,0,1) -- (1,0,1) -- (1,0,0) -- (1,1,0)
	 (1,0,1) -- (1,1,1);
\end{tikzpicture}
%%%%%%%%%%%%%%%%%%%%%%%%
\qquad\qquad
\tau_\alpha^\inv = 
%%%%%%%%%%%%%%%%%%%%%% 
\begin{tikzpicture}[thick,scale=2.5,color=blue!50!black, baseline=0.0cm, >=stealth, 
				style={x={(-0.9cm,-0.4cm)},y={(0.8cm,-0.4cm)},z={(0cm,0.9cm)}}]
% invisible edges of cube: 
\draw[
	 color=gray, 
	 opacity=0.3, 
	 semithick,
	 dashed
	 ] 
	 (1,0,0) -- (0,0,0) -- (0,1,0)
	 (0,0,0) -- (0,0,1);
%%%
% for 3-strata: 
\coordinate (b1) at (0.75, 0.3, 1);
\coordinate (b2) at (0.45, 0.5, 1);
\coordinate (cusp) at (0.45, 0.5, 0.5);
\coordinate (belcusp) at (0.45, 0.3, 0);
\coordinate (left0) at (0.45, 0, 1);
\coordinate (leftmid) at (0.45, 0, 0.5);
\coordinate (left1) at (0.45, 0, 0);
\coordinate (right0) at (0.75, 1, 1);
\coordinate (rightmid) at (0.6, 1, 0.5);
\coordinate (right1) at (0.45, 1, 0);
%
% 3-stratum: 
\fill [blue!20,opacity=0.2] (1,0,0) -- (1,1,0) -- (0,1,0) -- (0,1,1) -- (0,0,1) -- (1,0,1);
%
% alpha-plane bottom left: 
\fill [red!60, opacity=0.7] (b2) -- (left0) -- (leftmid) -- (cusp);
% alpha-plane bottom left: 
\fill [red!60, opacity=0.7] (leftmid) -- (cusp) -- (belcusp) -- (left1);
% alpha-hash-plane: 
\fill [pattern=dots, opacity=0.4] (cusp) -- (b1) -- (b2);
\fill [red!60, opacity=0.7] (cusp) -- (b1) -- (b2);
% alpha-plane bottom right: 
\fill [red!60, opacity=0.7] (cusp) -- (b1) -- (right0) -- (rightmid);
% alpha-plane top right: 
\fill [red!60, opacity=0.7] (cusp) -- (rightmid) -- (right1) -- (belcusp);
% alpha label: 
\draw[line width=1] (0.3, 0.2, 0.2) node[line width=0pt] (beta) {{\footnotesize $\alpha$}};
%
% object labels: 
\draw[line width=1] (0.85, 0.5, 0) node[line width=0pt] (beta) {{\footnotesize $u$}};
\draw[line width=1] (0.2, 0.8, 0) node[line width=0pt] (beta) {{\footnotesize $v$}};
%
% visible edges of cube
\draw[
	 color=gray, 
	 opacity=0.4, 
	 semithick
	 ] 
	 (0,1,1) -- (0,1,0) -- (1,1,0) -- (1,1,1) -- (0,1,1) -- (0,0,1) -- (1,0,1) -- (1,0,0) -- (1,1,0)
	 (1,0,1) -- (1,1,1);
\end{tikzpicture}
\ee
and the condition that  $\tau_\alpha^\inv$ is inverse to $\tau_\alpha$ reads
\begin{align}
&\tau_\alpha^\inv\circ\tau_\alpha =
%%%%%%%%%%%%%%%%%%%%%% 
\begin{tikzpicture}[thick,scale=2.5,color=blue!50!black, baseline=0.0cm, >=stealth, 
				style={x={(-0.9cm,-0.4cm)},y={(0.8cm,-0.4cm)},z={(0cm,0.9cm)}}]
% invisible edges of cube: 
\draw[
	 color=gray, 
	 opacity=0.3, 
	 semithick,
	 dashed
	 ] 
	 (1,0,0) -- (0,0,0) -- (0,1,0)
	 (0,0,0) -- (0,0,1);
%%%
% for 3-strata: 
\coordinate (b1) at (0.75, 0.3, 1);
\coordinate (b2) at (0.45, 0.5, 1);
\coordinate (c1) at (0.75, 0.3, 0);
\coordinate (c2) at (0.45, 0.5, 0);
\coordinate (cusp) at (0.45, 0.5, 0.55);
\coordinate (cusp2) at (0.45, 0.5, 0.45);
\coordinate (belcusp) at (0.45, 0.3, 0);
\coordinate (abcusp) at (0.45, 0.3, 1);
\coordinate (left0) at (0.45, 0, 1);
\coordinate (leftmid) at (0.45, 0, 0.5);
\coordinate (left1) at (0.45, 0, 0);
\coordinate (right0) at (0.75, 1, 1);
\coordinate (rightmid) at (0.6, 1, 0.5);
\coordinate (right1) at (0.75, 1, 0);
%
% 3-stratum: 
\fill [blue!20,opacity=0.2] (1,0,0) -- (1,1,0) -- (0,1,0) -- (0,1,1) -- (0,0,1) -- (1,0,1);
%
% alpha-plane bottom left: 
\fill [red!60, opacity=0.7] (b2) -- (left0) -- (left1) -- (c2);
% alpha-plane right middle: 
\fill [red!60, opacity=0.7] (rightmid) -- (cusp2) -- (cusp);
% alpha-hash-plane: 
\fill [pattern=dots, opacity=0.4] (cusp) -- (b1) -- (b2);
\fill [red!60, opacity=0.7] (cusp) -- (b1) -- (b2);
% alpha-hash plane 2
\fill [pattern=dots, opacity=0.4] (cusp2) -- (c1) -- (c2);
\fill [red!60, opacity=0.7] (cusp2) -- (c1) -- (c2);
% alpha-plane bottom right: 
\fill [red!60, opacity=0.7] (cusp) -- (b1) -- (right0) -- (rightmid);
% alpha-plane top right: 
\fill [red!60, opacity=0.7] (cusp2) -- (rightmid) -- (right1) -- (c1);
% alpha label: 
\draw[line width=1] (0.1, 0.2, 0.2) node[line width=0pt] (beta) {{\footnotesize $\alpha$}};
%
% object labels: 
\draw[line width=1] (0.85, 0.5, 0) node[line width=0pt] (beta) {{\footnotesize $u$}};
\draw[line width=1] (0.2, 0.8, 0) node[line width=0pt] (beta) {{\footnotesize $v$}};
%
% visible edges of cube
\draw[
	 color=gray, 
	 opacity=0.4, 
	 semithick
	 ] 
	 (0,1,1) -- (0,1,0) -- (1,1,0) -- (1,1,1) -- (0,1,1) -- (0,0,1) -- (1,0,1) -- (1,0,0) -- (1,1,0)
	 (1,0,1) -- (1,1,1);
\end{tikzpicture}
=
\begin{tikzpicture}[thick,scale=2.5,color=blue!50!black, baseline=0.0cm, >=stealth, 
				style={x={(-0.9cm,-0.4cm)},y={(0.8cm,-0.4cm)},z={(0cm,0.9cm)}}]
% invisible edges of cube: 
\draw[
	 color=gray, 
	 opacity=0.3, 
	 semithick,
	 dashed
	 ] 
	 (1,0,0) -- (0,0,0) -- (0,1,0)
	 (0,0,0) -- (0,0,1);
%%%
% for 3-strata: 
\coordinate (b1) at (0.75, 0.3, 1);
\coordinate (b2) at (0.45, 0.5, 1);
\coordinate (c1) at (0.75, 0.3, 0);
\coordinate (c2) at (0.45, 0.5, 0);
\coordinate (left0) at (0.45, 0, 1);
\coordinate (left1) at (0.45, 0, 0);
\coordinate (right0) at (0.75, 1, 1);
\coordinate (right1) at (0.75, 1, 0);
%
% 3-stratum: 
\fill [blue!20,opacity=0.2] (1,0,0) -- (1,1,0) -- (0,1,0) -- (0,1,1) -- (0,0,1) -- (1,0,1);
%
% alpha-plane  left: 
\fill [red!60, opacity=0.7] (b2) -- (left0) -- (left1) -- (c2);
% alpha-hash-plane: 
\fill [pattern=dots, opacity=0.4]  (b1) -- (b2)--(c2)--(c1);
\fill [red!60, opacity=0.7]  (b1) -- (b2)--(c2)--(c1);
% alpha-plane  right: 
\fill [red!60, opacity=0.7] (b1) -- (right0) -- (right1) -- (c1);
% alpha label: 
\draw[line width=1] (0.1, 0.2, 0.2) node[line width=0pt] (beta) {{\footnotesize $\alpha$}};
%
% object labels: 
\draw[line width=1] (0.85, 0.5, 0) node[line width=0pt] (beta) {{\footnotesize $u$}};
\draw[line width=1] (0.1, 0.85, 0) node[line width=0pt] (beta) {{\footnotesize $v$}};
%
% visible edges of cube
\draw[
	 color=gray, 
	 opacity=0.4, 
	 semithick
	 ] 
	 (0,1,1) -- (0,1,0) -- (1,1,0) -- (1,1,1) -- (0,1,1) -- (0,0,1) -- (1,0,1) -- (1,0,0) -- (1,1,0)
	 (1,0,1) -- (1,1,1);
\end{tikzpicture}
= 1_{\big(1_\alpha\Box(\mathrm{coev}_{\alpha^\#})^\dag\big)\otimes \big(\mathrm{coev}_\alpha\Box 1_\alpha\big)}\nonumber\\
&\tau_\alpha\circ \tau_\alpha^\inv=
%%%%%%%%%%%%%%%%%%%%%% 
\begin{tikzpicture}[thick,scale=2.5,color=blue!50!black, baseline=0.0cm, >=stealth, 
				style={x={(-0.9cm,-0.4cm)},y={(0.8cm,-0.4cm)},z={(0cm,0.9cm)}}]
% invisible edges of cube: 
\draw[
	 color=gray, 
	 opacity=0.3, 
	 semithick,
	 dashed
	 ] 
	 (1,0,0) -- (0,0,0) -- (0,1,0)
	 (0,0,0) -- (0,0,1);
%%%
% for 3-strata: 
\coordinate (b1) at (0.65, 0.3, 0.5);
\coordinate (b2) at (0.45, 0.5, 0.5);
\coordinate (cusp) at (0.45, 0.5, 0.9);
\coordinate (cusp2) at (0.45, 0.5, 0.1);
\coordinate (abovecusp) at (0.45, 0.5, 1);
\coordinate (belcusp) at (0.45, 0.5, 0);
\coordinate (left0) at (0.45, 0, 0);
\coordinate (leftmid) at (0.45, 0, 0.5);
\coordinate (left1) at (0.45, 0, 1);
\coordinate (right0) at (0.65, 1, 0.5);
\coordinate (rightmid) at (0.5, 1, 0.8);
\coordinate (rightmid2) at (0.5, 1, 0.2);
\coordinate (right1) at (0.45, 1, 1);
\coordinate (right2) at (0.45, 1, 0);
%
% 3-stratum: 
\fill [blue!20,opacity=0.2] (1,0,0) -- (1,1,0) -- (0,1,0) -- (0,1,1) -- (0,0,1) -- (1,0,1);
% alpha-plane  left: 
\fill [red!60, opacity=0.7] (left0) -- (belcusp) -- (abovecusp) -- (left1);
% alpha-hash-plane: 
\fill [pattern=dots, opacity=0.4] (cusp) --  (b1)--(b2);
\fill [red!60, opacity=0.7] (cusp) -- (b1) -- (b2);
\fill [pattern=dots, opacity=0.4] (cusp2) --  (b1)--(b2);
\fill [red!60, opacity=0.7] (cusp2) -- (b1) -- (b2);
% alpha-plane bottom right: 
\fill [red!60, opacity=0.7] (cusp) -- (b1) -- (right0) -- (rightmid);
\fill [red!60, opacity=0.7] (cusp2) -- (b1) -- (right0) -- (rightmid2);
% alpha-plane top right: 
\fill [red!60, opacity=0.7] (cusp) -- (rightmid) -- (right1) -- (abovecusp);
\fill [red!60, opacity=0.7] (cusp2) -- (rightmid2) -- (right2) -- (belcusp);
% alpha label: 
\draw[line width=1] (0, 0.2, 0.4) node[line width=0pt] (beta) {{\footnotesize $\alpha$}};
%
% object labels: 
\draw[line width=1] (0.85, 0.5, 0) node[line width=0pt] (beta) {{\footnotesize $u$}};
\draw[line width=1] (0.2, 0.8, 0) node[line width=0pt] (beta) {{\footnotesize $v$}};
%
% visible edges of cube
\draw[
	 color=gray, 
	 opacity=0.4, 
	 semithick
	 ] 
	 (0,1,1) -- (0,1,0) -- (1,1,0) -- (1,1,1) -- (0,1,1) -- (0,0,1) -- (1,0,1) -- (1,0,0) -- (1,1,0)
	 (1,0,1) -- (1,1,1);
\end{tikzpicture}
=
\begin{tikzpicture}[thick,scale=2.5,color=blue!50!black, baseline=0.0cm, >=stealth, 
				style={x={(-0.9cm,-0.4cm)},y={(0.8cm,-0.4cm)},z={(0cm,0.9cm)}}]
% invisible edges of cube: 
\draw[
	 color=gray, 
	 opacity=0.3, 
	 semithick,
	 dashed
	 ] 
	 (1,0,0) -- (0,0,0) -- (0,1,0)
	 (0,0,0) -- (0,0,1);
%%%
% for 3-strata: 
\coordinate (left0) at (0.45, 0, 0);
\coordinate (left1) at (0.45, 0, 1);
\coordinate (right0) at (0.45, 1, 0);
\coordinate (right1) at (0.45, 1, 1);
%
% 3-stratum: 
\fill [blue!20,opacity=0.2] (1,0,0) -- (1,1,0) -- (0,1,0) -- (0,1,1) -- (0,0,1) -- (1,0,1);
% alpha-plane  left: 
\fill [red!60, opacity=0.7] (left0) -- (left1) -- (right1) -- (right0);
% alpha label: 
\draw[line width=1] (0, 0.2, 0.4) node[line width=0pt] (beta) {{\footnotesize $\alpha$}};
%
% object labels: 
\draw[line width=1] (0.85, 0.5, 0) node[line width=0pt] (beta) {{\footnotesize $u$}};
\draw[line width=1] (0.1, 0.85, 0) node[line width=0pt] (beta) {{\footnotesize $v$}};
%
% visible edges of cube
\draw[
	 color=gray, 
	 opacity=0.4, 
	 semithick
	 ] 
	 (0,1,1) -- (0,1,0) -- (1,1,0) -- (1,1,1) -- (0,1,1) -- (0,0,1) -- (1,0,1) -- (1,0,0) -- (1,1,0)
	 (1,0,1) -- (1,1,1);
\end{tikzpicture}
=1_{1_\alpha}\label{eq:invbility}
\end{align}
The compatibility it has to satisfy with the $\dagger$-adjunction maps according to condition (ii.c) then reads: 
\be\label{eq:tautautau}
%%%%%%%%%%%%%%%%%%%%%% 
\begin{tikzpicture}[thick,scale=2.5,color=blue!50!black, baseline=0.0cm, >=stealth, 
				style={x={(-0.9cm,-0.4cm)},y={(0.8cm,-0.4cm)},z={(0cm,0.9cm)}}]
% invisible edges of cube: 
\draw[
	 color=gray, 
	 opacity=0.3, 
	 semithick,
	 dashed
	 ] 
	 (1,0,0) -- (0,0,0) -- (0,1,0)
	 (0,0,0) -- (0,0,1);
%%%
% for 3-strata: 
\coordinate (b1) at (0.75, 0.3, 0);
\coordinate (b2) at (0.45, 0.5, 0);
\coordinate (cusp) at (0.45, 0.5, 0.5);
\coordinate (abovecusp) at (0.45, 0.3, 1);
\coordinate (left0) at (0.45, 0, 0);
\coordinate (leftmid) at (0.45, 0, 0.5);
\coordinate (left1) at (0.45, 0, 1);
\coordinate (right0) at (0.75, 1, 0);
\coordinate (rightmid) at (0.6, 1, 0.5);
\coordinate (right1) at (0.45, 1, 1);
%
% 3-stratum: 
\fill [blue!20,opacity=0.2] (1,0,0) -- (1,1,0) -- (0,1,0) -- (0,1,1) -- (0,0,1) -- (1,0,1);
%
% alpha-plane bottom left: 
\fill [magenta!50!red, opacity=0.6] (b2) -- (left0) -- (leftmid) -- (cusp);
% alpha-plane bottom left: 
\fill [magenta!50!red, opacity=0.6] (leftmid) -- (cusp) -- (abovecusp) -- (left1);
% alpha-hash-plane: 
\fill [pattern=dots, opacity=0.4] (cusp) -- (b1) -- (b2);
\fill [magenta!50!red, opacity=0.6] (cusp) -- (b1) -- (b2);
% alpha-plane bottom right: 
\fill [magenta!50!red, opacity=0.6] (cusp) -- (b1) -- (right0) -- (rightmid);
% alpha-plane top right: 
\fill [magenta!50!red, opacity=0.6] (cusp) -- (rightmid) -- (right1) -- (abovecusp);
% alpha label: 
\draw[line width=1] (0.3, 0.2, 0.6) node[line width=0pt] (beta) {{\footnotesize $\alpha \sta \beta$}};
%
% object labels: 
\draw[line width=1] (0.85, 0.5, 0) node[line width=0pt] (beta) {{\footnotesize $w$}};
\draw[line width=1] (0.2, 0.8, 0) node[line width=0pt] (beta) {{\footnotesize $v$}};
%
% visible edges of cube
\draw[
	 color=gray, 
	 opacity=0.4, 
	 semithick
	 ] 
	 (0,1,1) -- (0,1,0) -- (1,1,0) -- (1,1,1) -- (0,1,1) -- (0,0,1) -- (1,0,1) -- (1,0,0) -- (1,1,0)
	 (1,0,1) -- (1,1,1);
\end{tikzpicture}
%%%%%%%%%%%%%%%%%%%%%%
=
%%%%%%%%%%%%%%%%%%%%%% 
\begin{tikzpicture}[thick,scale=2.5,color=blue!50!black, baseline=0.0cm, >=stealth, 
				style={x={(-0.9cm,-0.4cm)},y={(0.8cm,-0.4cm)},z={(0cm,0.9cm)}}]
% invisible edges of cube: 
\draw[
	 color=gray, 
	 opacity=0.3, 
	 semithick,
	 dashed
	 ] 
	 (1,0,0) -- (0,0,0) -- (0,1,0)
	 (0,0,0) -- (0,0,1);
% 
% points for beta: 
\coordinate (b1) at (0.75, 0.3, 0);
\coordinate (b2) at (0.45, 0.5, 0);
\coordinate (cusp) at (0.45, 0.5, 0.5);
\coordinate (abovecusp) at (0.45, 0.3, 1);
\coordinate (left0) at (0.45, 0, 0);
\coordinate (leftmid) at (0.45, 0, 0.5);
\coordinate (left1) at (0.45, 0, 1);
\coordinate (right0) at (0.75, 1, 0);
\coordinate (rightmid) at (0.6, 1, 0.5);
\coordinate (right1) at (0.45, 1, 1);
%
% points for beta: 
\coordinate (shift) at (-0.2, 0.2, 0);
\coordinate (Yshift) at (-0.2, 0, 0);
\coordinate (ab1) at ($(0.75, 0.3, 0) + (shift)$);
\coordinate (ab2) at ($(0.45, 0.5, 0) + (shift)$);
\coordinate (acusp) at ($(0.45, 0.5, 0.5) + (shift)$);
\coordinate (aabovecusp) at ($(0.45, 0.3, 1) + (shift)$);
\coordinate (aleft0) at ($(0.45, 0, 0) + (Yshift)$);
\coordinate (aleftmid) at ($(0.45, 0, 0.5) + (Yshift)$);
\coordinate (aleft1) at ($(0.45, 0, 1) + (Yshift)$);
\coordinate (aright0) at ($(0.75, 1, 0) + (Yshift)$);
\coordinate (arightmid) at ($(0.6, 1, 0.5) + (Yshift)$);
\coordinate (aright1) at ($(0.45, 1, 1) + (Yshift)$);
%
% 3-stratum: 
\fill [blue!20,opacity=0.2] (1,0,0) -- (1,1,0) -- (0,1,0) -- (0,1,1) -- (0,0,1) -- (1,0,1);
%
% alpha-plane bottom left: 
\fill [red!60, opacity=0.7] (ab2) -- (aleft0) -- (aleftmid) -- (acusp);
% alpha-plane bottom left: 
\fill [red!60, opacity=0.7] (aleftmid) -- (acusp) -- (aabovecusp) -- (aleft1);
% beta-plane bottom left: 
\fill [magenta!60, opacity=0.7] (b2) -- (left0) -- (leftmid) -- (cusp);
% beta-hash-plane: 
\fill [pattern=dots, opacity=0.4] (cusp) -- (b1) -- (b2);
\fill [magenta!60, opacity=0.7] (cusp) -- (b1) -- (b2);
% alpha-hash-plane: 
\fill [pattern=dots, opacity=0.4] (acusp) -- (ab1) -- (ab2);
\fill [red!60, opacity=0.7] (acusp) -- (ab1) -- (ab2);
% alpha-plane bottom right: 
\fill [red!60, opacity=0.7] (acusp) -- (ab1) -- (aright0) -- (arightmid);
% alpha-plane top right: 
\fill [red!60, opacity=0.7] (acusp) -- (arightmid) -- (aright1) -- (aabovecusp);
% alpha label: 
\draw[line width=1] (0.15, 0.7, 0.6) node[line width=0pt] (beta) {{\footnotesize $\alpha$}};
%
% beta-plane top left: 
\fill [magenta!60, opacity=0.7] (leftmid) -- (cusp) -- (abovecusp) -- (left1);
% beta-plane bottom right: 
\fill [magenta!60, opacity=0.7] (cusp) -- (b1) -- (right0) -- (rightmid);
% beta-plane top right: 
\fill [magenta!60, opacity=0.7] (cusp) -- (rightmid) -- (right1) -- (abovecusp);
% beta label: 
\draw[line width=1] (0.35, 0.01, 0.6) node[line width=0pt] (beta) {{\footnotesize $\beta$}};
%
% object labels: 
\draw[line width=1] (0.61, 0.93, 0) node[line width=0pt] (beta) {{\footnotesize $u$}};
\draw[line width=1] (0.85, 0.5, 0) node[line width=0pt] (beta) {{\footnotesize $w$}};
\draw[line width=1] (0.2, 0.8, 0) node[line width=0pt] (beta) {{\footnotesize $v$}};
%
% visible edges of cube
\draw[
	 color=gray, 
	 opacity=0.4, 
	 semithick
	 ] 
	 (0,1,1) -- (0,1,0) -- (1,1,0) -- (1,1,1) -- (0,1,1) -- (0,0,1) -- (1,0,1) -- (1,0,0) -- (1,1,0)
	 (1,0,1) -- (1,1,1);
\end{tikzpicture}
%%%%%%%%%%%%%%%%%%%%%%
\ee
Finally, condition (ii.d) states that the fold does not suffer from twisting it: 
\be\label{eq:twistcoev}
%%%%%%%%%%%%%%%%%%%%%% 
\begin{tikzpicture}[thick,scale=2.5,color=blue!50!black, baseline=0.0cm, >=stealth, 
				style={x={(-0.9cm,-0.4cm)},y={(0.8cm,-0.4cm)},z={(0cm,0.9cm)}}]
% invisible edges of cube: 
\draw[
	 color=gray, 
	 opacity=0.3, 
	 semithick,
	 dashed
	 ] 
	 (1,0,0) -- (0,0,0) -- (0,1,0)
	 (0,0,0) -- (0,0,1);
%%%
% for 3-strata: 
\coordinate (f0) at (0.5, 0.75, 0);
\coordinate (f1) at (0.5, 0.75, 1);
\coordinate (front0) at (0.75, 0, 0);
\coordinate (front1) at (0.75, 0, 1);
\coordinate (back0) at (0.25, 0, 0);
\coordinate (back1) at (0.25, 0, 1);
% 3-stratum: 
\fill [blue!20,opacity=0.2] (1,0,0) -- (1,1,0) -- (0,1,0) -- (0,1,1) -- (0,0,1) -- (1,0,1);
%
% alpha-plane: 
\fill [red!60, opacity=0.7] (f0) -- (f1) -- (back1) -- (back0);
% alpha label: 
\draw[line width=1] (0.3, 0.17, 0.85) node[line width=0pt] (beta) {{\footnotesize $\alpha$}};
%
% alpha-hash-plane: 
\fill [pattern=dots, opacity=0.4] (f0) -- (f1) -- (front1) -- (front0);
\fill [red!60,opacity=0.7] (f0) -- (f1) -- (front1) -- (front0);
% alpha label: 
\draw[line width=1] (0.6, 0.1, 0.35) node[line width=0pt] (beta) {{\footnotesize $\alpha\hash$}};
%
% object labels: 
\draw[line width=1] (0.85, 0.5, 0) node[line width=0pt] (beta) {{\footnotesize $v$}};
\draw[line width=1] (0.5, 0.1, 0.95) node[line width=0pt] (beta) {{\footnotesize $u$}};
%
% visible edges of cube
\draw[
	 color=gray, 
	 opacity=0.4, 
	 semithick
	 ] 
	 (0,1,1) -- (0,1,0) -- (1,1,0) -- (1,1,1) -- (0,1,1) -- (0,0,1) -- (1,0,1) -- (1,0,0) -- (1,1,0)
	 (1,0,1) -- (1,1,1);
\end{tikzpicture}
%%%%%%%%%%%%%%%%%%%%%%
=
%%%%%%%%%%%%%%%%%%%%%% 
\begin{tikzpicture}[thick,scale=2.5,color=blue!50!black, baseline=0.0cm, >=stealth, 
				style={x={(-0.9cm,-0.4cm)},y={(0.8cm,-0.4cm)},z={(0cm,0.9cm)}}]
% invisible edges of cube: 
\draw[
	 color=gray, 
	 opacity=0.3, 
	 semithick,
	 dashed
	 ] 
	 (1,0,0) -- (0,0,0) -- (0,1,0)
	 (0,0,0) -- (0,0,1);
%%%
% for 3-strata: 
\coordinate (b1) at (0.75, 0.3, 0);
\coordinate (b2) at (0.45, 0.5, 0);
\coordinate (cusp) at (0.45, 0.5, 0.5);
\coordinate (abovecusp) at (0.45, 0.3, 1);
\coordinate (left0) at (0.45, 0, 0);
\coordinate (leftmid) at (0.45, 0, 0.5);
\coordinate (left1) at (0.45, 0, 1);
\coordinate (right0) at (0.75, 1, 0);
\coordinate (rightmid) at (0.6, 1, 0.5);
\coordinate (right1) at (0.45, 1, 1);
\coordinate (left-front0) at (0.8, 0, 0);
\coordinate (left-front1) at (0.8, 0, 1);
\coordinate (left-back0) at (0.2, 0, 0);
\coordinate (left-back1) at (0.2, 0, 1);
\coordinate (fold-bottom0) at (0.7, 0.9, 0);
\coordinate (bottom-mid0) at (0.75, 0.45, 0); % in between (fold-bottom0) and (left-front0)
\coordinate (bottom-mid2) at (0.45, 0.45, 0); % in between (fold-bottom0) and (left-back0)
\coordinate (fold-bottom1) at (0.7, 0.9, 0.4);
\coordinate (fold-top0) at (0.3, 0.9, 0.6);
\coordinate (fold-top1) at (0.3, 0.9, 1);
\coordinate (top-mid1) at (0.55, 0.45, 1); % in between (fold-top1) and (left-front1)
\coordinate (top-mid2) at (0.25, 0.45, 1); % in between (fold-top1) and (left-back1)
\coordinate (cusp0) at (0.5, 0.3, 0.3);
\coordinate (cusp1) at (0.5, 0.3, 0.7);
%
% 3-stratum: 
\fill [blue!20,opacity=0.2] (1,0,0) -- (1,1,0) -- (0,1,0) -- (0,1,1) -- (0,0,1) -- (1,0,1);
%
% plane BLA: 
%\fill [pattern=dots, opacity=0.4] (cusp1) -- (fold-bottom1) -- (bottom-mid2);
%\fill [pattern=dots, opacity=0.4] (fold-bottom1) -- (fold-bottom0) -- (bottom-mid2);
%\fill [pattern=dots, opacity=0.4] (bottom-mid2) -- (cusp1) -- (left-back0);
%\fill [pattern=dots, opacity=0.4] (cusp1) -- (fold-top0) -- (top-mid2);
%\fill [pattern=dots, opacity=0.4] (fold-top1) -- (fold-top0) -- (top-mid2);
%\fill [pattern=dots, opacity=0.4] (left-back0) -- (left-back1) -- (top-mid2);
%\fill [pattern=dots, opacity=0.4] (fold-top0) -- (cusp0) -- (cusp1);
\fill [red!60, opacity=0.7] (cusp1) -- (fold-bottom1) -- (bottom-mid2);
\fill [red!60, opacity=0.7] (fold-bottom1) -- (fold-bottom0) -- (bottom-mid2);
\fill [red!60, opacity=0.7] (bottom-mid2) -- (cusp1) -- (left-back0);
\fill [red!60, opacity=0.7] (cusp1) -- (fold-top0) -- (fold-top1);
\fill [red!60, opacity=0.7] (fold-top1) -- (cusp1) -- (top-mid2);
%\fill [red!60, opacity=0.7] (left-back0) -- (left-back1) -- (top-mid2);
\fill [red!60, opacity=0.7] (0.464, 0.3, 0.715) -- (left-back1) -- (top-mid2);
\fill [red!60, opacity=0.7] (fold-top0) -- (cusp0) -- (cusp1);
\fill [pattern=dots, opacity=0.4] (cusp0) -- (cusp1) -- (fold-bottom1);
\fill [pattern=dots, opacity=0.4] (cusp0) -- (fold-bottom1) -- (fold-bottom0) -- (bottom-mid0);
\fill [pattern=dots, opacity=0.4] (bottom-mid0) -- (cusp0) -- (left-front0);
\fill [pattern=dots, opacity=0.4] (top-mid1) -- (fold-top1) -- (fold-top0) -- (cusp0);
\fill [pattern=dots, opacity=0.4] (top-mid1) -- (cusp0) -- (left-front0) -- (left-front1);
\fill [red!60, opacity=0.7] (cusp0) -- (cusp1) -- (fold-bottom1);
\fill [red!60, opacity=0.7] (cusp0) -- (fold-bottom1) -- (fold-bottom0) -- (bottom-mid0);
\fill [red!60, opacity=0.7] (bottom-mid0) -- (cusp0) -- (left-front0);
\fill [red!60, opacity=0.7] (top-mid1) -- (fold-top1) -- (fold-top0) -- (cusp0);
\fill [red!60, opacity=0.7] (top-mid1) -- (cusp0) -- (left-front0) -- (left-front1);
%
% alpha labels: 
\draw[line width=1] (0.8, 0.2, 0.5) node[line width=0pt] (beta) {{\footnotesize $\alpha\hash$}};
\draw[line width=1] (0.23, 0.2, 0.88) node[line width=0pt] (beta) {{\footnotesize $\alpha$}};
%
% object labels: 
\draw[line width=1] (0.5, 0.15, 1) node[line width=0pt] (beta) {{\footnotesize $u$}};
\draw[line width=1] (0.25, 0.7, 0) node[line width=0pt] (beta) {{\footnotesize $v$}};
%
% visible edges of cube
\draw[
	 color=gray, 
	 opacity=0.4, 
	 semithick
	 ] 
	 (0,1,1) -- (0,1,0) -- (1,1,0) -- (1,1,1) -- (0,1,1) -- (0,0,1) -- (1,0,1) -- (1,0,0) -- (1,1,0)
	 (1,0,1) -- (1,1,1);
\end{tikzpicture}
%%%%%%%%%%%%%%%%%%%%%%
\ee
A movie-slicing of this condition is given by:
$$
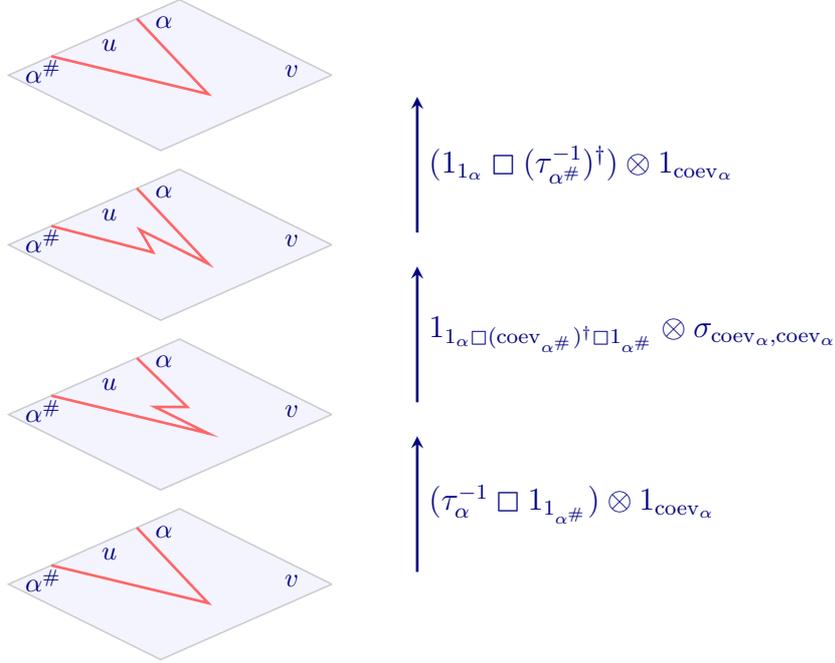
\begin{tikzpicture}[thick,scale=2.5,color=blue!50!black, baseline=0.0cm, >=stealth, 				style={x={(-0.9cm,-0.4cm)},y={(0.8cm,-0.4cm)},z={(0cm,0.9cm)}}]
%top_plane
%
%plane boundary
\draw[color=gray, opacity=0.4, semithick] 
	 (0,1,1) -- (1,1,1) -- (1,0,1) -- (0,0,1)--(0,1,1);	
%plane filling
\fill [blue!20,opacity=0.2] (1,0,1) -- (1,1,1) -- (0,1,1)--(0,0,1);	
%alphal_line
\draw[line width=1, red!60] (.25,0,1)--(.5,.75,1)--(.75,0,1);	
% labels: 
\node at (.75,.2,1) [anchor=east]  {{\footnotesize $\alpha^\#$}\;};
\node at (.1,.2,1) [anchor=east]  {{\footnotesize $\alpha$}\;};
\draw[line width=1] (0.5, 0.1, 1) node[line width=0pt] (beta) {{\footnotesize $u$}};
\draw[line width=1] (0.1, 0.85, 1) node[line width=0pt] (beta) {{\footnotesize $v$}};
%second_from_top_plane
%
\
%plane boundary
\draw[color=gray, opacity=0.4, semithick] 
	 (0,1,0) -- (1,1,0) -- (1,0,0) -- (0,0,0)--(0,1,0);	
%plane filling
\fill [blue!20,opacity=0.2] (1,0,0) -- (1,1,0) -- (0,1,0)--(0,0,0);	
%alphal_line
\draw[line width=1, red!60] (.25,0,0)--(.5,.75,0)--(.5,.3,0)--(.6,.5,0)--(.75,0,0);	
% labels: 
\node at (.75,.2,0) [anchor=east]  {{\footnotesize $\alpha^\#$}\;};
\node at (.1,.2,0) [anchor=east]  {{\footnotesize $\alpha$}\;};
\draw[line width=1] (0.5, 0.1, 0) node[line width=0pt] (beta) {{\footnotesize $u$}};
\draw[line width=1] (0.1, 0.85, 0) node[line width=0pt] (beta) {{\footnotesize $v$}};
%third_from_top_plane
%
%plane boundary
\draw[color=gray, opacity=0.4, semithick] 
	 (0,1,-1) -- (1,1,-1) -- (1,0,-1) -- (0,0,-1)--(0,1,-1);	
%plane filling
\fill [blue!20,opacity=0.2] (1,0,-1) -- (1,1,-1) -- (0,1,-1)--(0,0,-1);	
%alphal_line
\draw[line width=1, red!60] (.25,0,-1)--(.4,.5,-1)--(.5,.4,-1)--(.5,.75,-1)--(.75,0,-1);	
% labels: 
\node at (.75,.2,-1) [anchor=east]  {{\footnotesize $\alpha^\#$}\;};
\node at (.1,.2,-1) [anchor=east]  {{\footnotesize $\alpha$}\;};
\draw[line width=1] (0.5, 0.1, -1) node[line width=0pt] (beta) {{\footnotesize $u$}};
\draw[line width=1] (0.1, 0.85, -1) node[line width=0pt] (beta) {{\footnotesize $v$}};
%bottom_plane
%
%plane boundary
\draw[color=gray, opacity=0.4, semithick] 
	 (0,1,-2) -- (1,1,-2) -- (1,0,-2) -- (0,0,-2)--(0,1,-2);	
%plane filling
\fill [blue!20,opacity=0.2] (1,0,-2) -- (1,1,-2) -- (0,1,-2)--(0,0,-2);	
%alphal_line
\draw[line width=1, red!60] (.25,0,-2)--(.5,.75,-2)--(.75,0,-2);	
% labels: 
\node at (.75,.2,-2) [anchor=east]  {{\footnotesize $\alpha^\#$}\;};
\node at (.1,.2,-2) [anchor=east]  {{\footnotesize $\alpha$}\;};
\draw[line width=1] (0.5, 0.1, -2) node[line width=0pt] (beta) {{\footnotesize $u$}};
\draw[line width=1] (0.1, 0.85, -2) node[line width=0pt] (beta) {{\footnotesize $v$}};
%arrows
\draw[line width=1,->] (-.5,1,-2.15)--(-.5, 1, -1.75) node[anchor=west] {$( \tau_\alpha^\inv\sta 1_{1_{\alpha^\#}})\otimes 1_{\mathrm{coev}_\alpha}$} -- (-.5,1,-1.35) ;
\draw[line width=1,->] (-.5,1,-1.15)--(-.5, 1, -.75) node[anchor=west] {$1_{1_{\alpha} \sta (\coev_{\alpha^\#})\dagg \sta 1_{\alpha^\#}} \fus \sigma_{\coev_\alpha, \coev_\alpha}$} -- (-.5,1,-.35) ;
\draw[line width=1,->] (-.5,1,-.15)--(-.5, 1, .25) node[anchor=west] {$(1_{1_\alpha}\sta (\tau_{\alpha^\#}^\inv)\dagg )\otimes 1_{\mathrm{coev}_\alpha}$} -- (-.5,1,.65) ;
\end{tikzpicture}
$$

Note that there exists a (non-progressive) isotopy relative to the boundary between the two diagrams 
in~\eqref{eq:twistcoev} and the two pairs of diagrams in \eqref{eq:invbility}.
In this sense, the invertibility of the triangulator and condition (ii.d) on the triangulator in Definition~\ref{def:Graycatduals} are topologically motivated.

\subsection{Gray categories from defect TQFTs}
\label{subsec:tricatfromZ}

From now on we fix a set of defect data $\D = (D_3, D_2, D_1, s,t,f)$, and a defect TQFT
$$
\zz: \Bordd \lra \Vect_\Bbbk \, . 
$$
In this section we shall construct the tricategory naturally associated to~$\zz$. 
We proceed in two steps: objects, 1- and 2-morphisms are basically those of the free pre-2-category $\mathcal F^{\textrm{p}} \K^\D$ of Section~\ref{subsubsec:bicats} (independent of the functor~$\zz$), while 3-morphisms are obtained by evaluating~$\zz$ on appropriate decorated spheres. 

\medskip

%arXiv_v2: [Now the tricategory in Section 3.3 has a prime.]
	%In detail, we will build a four-layered structure $\tz$ 
	We will build a Gray category $\tz'$ 
from~$\zz$. 
%arXiv_v2: 
	%Its first three layers (presciently called \textsl{objects}, \textsl{1-} and \textsl{2-morphisms}) are cylinders over the objects, 1- and 2-morphisms of $\mathcal F^{\textrm{p}} \K^\D$ viewed as $\D$-decorated stratified manifolds. 
	Roughly, objects, 1- and 2-morphisms are cylinders over the objects, 1- and 2-morphisms of $\mathcal F^{\textrm{p}} \K^\D$ viewed as $\D$-decorated stratified manifolds. In detail: An object in~$\tz'$ is a unit cube $[0,1]^3$ (equipped with the standard orientation) labelled with an element in~$D_3$. A 1-morphism $\alpha \in \tz'(u,v)$ is an equivalence class of cylinders $\alpha_{\textrm{2d}} \times [0,1]$, where $\alpha_{\textrm{2d}}$ is a representative of a 1-morphism in $\mathcal F^{\textrm{p}} \K^\D(u,v)$ as in~\eqref{eq:example1minFK}, and the (standard-oriented) 3-strata and the 2-strata (with a prescribed orientation) of $\alpha_{\textrm{2d}} \times [0,1]$ carry the labels of their boundaries in~$\alpha_{\textrm{2d}}$; two cylinders $\alpha_{\textrm{2d}} \times [0,1]$ and $\alpha'_{\textrm{2d}} \times [0,1]$ are equivalent if $\alpha_{\textrm{2d}}$ and $\alpha'_{\textrm{2d}}$ represent the same 1-morphism in $\mathcal F^{\textrm{p}} \K^\D$. Analogously, a 2-morphism in~$\tz'(\alpha,\beta)$ is an equivalence class of decorated cylinders~$X$ over representatives~$X_{\textrm{2d}}$ of 2-morphisms in $\mathcal F^{\textrm{p}} \K^\D$, together with a choice of orientation for each line $p\times [0,1]$ and each surface $l\times [0,1]$ over labelled points~$p$ and lines~$l$ in~$X_{\textrm{2d}}$, such that the orientations of 2-strata in~$X$ ending on the right or left face agree with those of the corresponding strata in~$\alpha$ and~$\beta$, respectively. The identity $1_u$ on an object $u\in\tz'$ is represented by (the diagram for)~$u$ (viewed as a cube with no labelled planes), and the identity $1_\alpha$ on a 1-morphism~$\alpha$ in~$\tz'$ is represented by a representative of~$\alpha$ (viewed as a cube with no labelled lines). It follows that for a representative~$r$ of an object, 1- or 2-morphisms in~$\tz'$, $((0,1)^2 \times [0,1]) \cap r$ is a $\mathds D$-decorated 3-manifold. 

%arXiv_v2: 
	Usually we will use the same symbol for a 1- or 2-morphism~$Y$ and the cube diagrams representing it. However, when we want to highlight a choice of representing diagram for~$Y$, we will denote the representative as~$D_Y$ or similarly. 

%arXiv_v2: [start new paragraph]
%arXiv_v2: 
	%Hence 
	Examples of (representatives of) objects, 1- and 2-morphisms in~$\tz'$ are 
$$
%%%%%%%%%%%%%%%%%%%%%% 
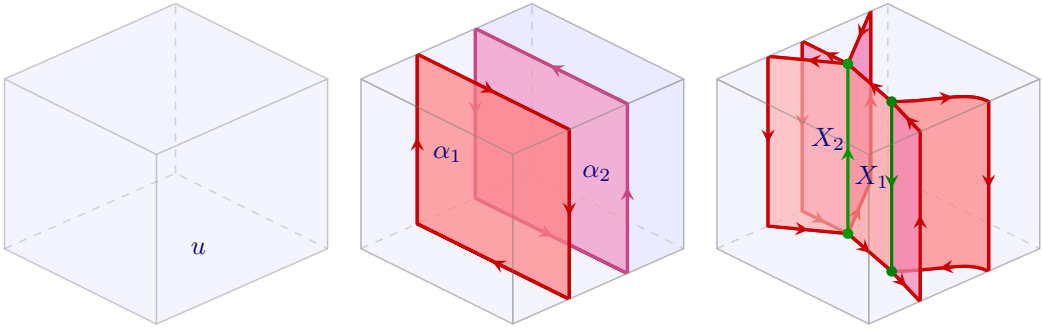
\begin{tikzpicture}[thick,scale=2.5,color=blue!50!black, baseline=0.0cm, >=stealth, 
				style={x={(-0.9cm,-0.4cm)},y={(0.8cm,-0.4cm)},z={(0cm,0.9cm)}}]
% 3-stratum: 
\fill [blue!20,opacity=0.2] (1,0,0) -- (1,1,0) -- (0,1,0) -- (0,1,1) -- (0,0,1) -- (1,0,1);
% invisible edges of cube: 
\draw[
	 color=gray, 
	 opacity=0.3, 
	 semithick,
	 dashed
	 ] 
	 (1,0,0) -- (0,0,0) -- (0,1,0)
	 (0,0,0) -- (0,0,1);
% object label: 
\draw[line width=1] (0.4, 0.6, 0) node[line width=0pt] (beta) {{\footnotesize $u$}};
% visible edges of cube
\draw[
	 color=gray, 
	 opacity=0.4, 
	 semithick
	 ] 
	 (0,1,1) -- (0,1,0) -- (1,1,0) -- (1,1,1) -- (0,1,1) -- (0,0,1) -- (1,0,1) -- (1,0,0) -- (1,1,0)
	 (1,0,1) -- (1,1,1);
\end{tikzpicture}
%%%%%%%%%%%%%%%%%%%%%% 
\quad 
%%%%%%%%%%%%%%%%%%%%%% 
\begin{tikzpicture}[thick,scale=2.5,color=blue!50!black, baseline=0.0cm, >=stealth, 
				style={x={(-0.9cm,-0.4cm)},y={(0.8cm,-0.4cm)},z={(0cm,0.9cm)}}]
% visible edges of cube
\draw[
	 color=gray, 
	 opacity=0.4, 
	 semithick
	 ] 
	 (0,1,1) -- (0,1,0) -- (1,1,0) -- (1,1,1) -- (0,1,1) -- (0,0,1) -- (1,0,1) -- (1,0,0) -- (1,1,0)
	 (1,0,1) -- (1,1,1);
% to cut off unwanted edges: 
\clip (1,0,0) -- (1,1,0) -- (0,1,0) -- (0,1,1) -- (0,0,1) -- (1,0,1);
% 3-stratum: 
\fill [blue!20,opacity=0.2] (1,0,0) -- (1,1,0) -- (0,1,0) -- (0,1,1) -- (0,0,1) -- (1,0,1);
% invisible edges of cube: 
\draw[
	 color=gray, 
	 opacity=0.3, 
	 semithick,
	 dashed
	 ] 
	 (1,0,0) -- (0,0,0) -- (0,1,0)
	 (0,0,0) -- (0,0,1);
%%%
% for 3-strata: 
\coordinate (d) at (-0.34, 0, 0);
\coordinate (d2) at (-0.67, 0, 0);
% 3-stratum v at back: 
\fill [blue!40,opacity=0.1] ($(d2) + (1,0,0)$) -- ($(d2) + (1,1,0)$) -- ($(d2) + (0.67,1,0)$) -- ($(d2) + (0.67,1,1)$) -- ($(d2) + (0.67,0,1)$) -- ($(d2) + (1,0,1)$);
%
% lower endpoint of X: 
\coordinate (X0) at (0.33, 0.1, 0);
% upper endpoint of X: 
\coordinate (X1) at (0.33, 0.9, 1);
% lower endpoint of Y: 
\coordinate (Y0) at (0.67, 0.9, 0);
% upper endpoint of Y: 
\coordinate (Y1) at (0.67, 0.1, 1);
%
% planes and side lines: 
\fill [magenta!50,opacity=0.7] (0.33, 0, 0) -- (0.33, 1, 0) -- (0.33, 1, 1) -- (0.33, 0, 1);
\draw[string=magenta!80!black, very thick] (0.33, 0, 0) -- (0.33, 1, 0);
\draw[string=magenta!80!black, very thick] (0.33, 1, 0) -- (0.33, 1, 1);
\draw[string=magenta!80!black, very thick] (0.33, 1, 1) -- (0.33, 0, 1);
\draw[string=magenta!80!black, very thick] (0.33, 0, 1) -- (0.33, 0, 0);
\draw[color=magenta!80!black, very thick] (0.33, 0, 0) -- (0.33, 1, 0) -- (0.33, 1, 1) -- (0.33, 0, 1) -- (0.33, 0, 0) -- (0.33, 1, 0);
% alpha2 label: 
\draw[line width=1] (0.33, 0.8, 0.5) node[line width=0pt] (beta) {{\footnotesize $\alpha_2$}};
\fill [red!50,opacity=0.7] (0.67, 0, 0) -- (0.67, 1, 0) -- (0.67, 1, 1) -- (0.67, 0, 1);
\draw[string=red!80!black, very thick] (0.67, 1, 0) -- (0.67, 0, 0);
\draw[string=red!80!black, very thick] (0.67, 0, 0) -- (0.67, 0, 1);
\draw[string=red!80!black, very thick] (0.67, 0, 1) -- (0.67, 1, 1);
\draw[string=red!80!black, very thick] (0.67, 1, 1) -- (0.67, 1, 0);
\draw[color=red!80!black, very thick] (0.67, 0, 0) -- (0.67, 1, 0) -- (0.67, 1, 1) -- (0.67, 0, 1) -- (0.67, 0, 0) -- (0.67, 1, 0);
% alpha1 label: 
\draw[line width=1] (0.67, 0.2, 0.5) node[line width=0pt] (beta) {{\footnotesize $\alpha_1$}};
%
% visible edges of cube
\draw[
	 color=gray, 
	 opacity=0.4, 
	 semithick
	 ] 
	 (0,1,1) -- (0,1,0) -- (1,1,0) -- (1,1,1) -- (0,1,1) -- (0,0,1) -- (1,0,1) -- (1,0,0) -- (1,1,0)
	 (1,0,1) -- (1,1,1);
\end{tikzpicture}
%%%%%%%%%%%%%%%%%%%%%% 
\quad 
%%%%%%%%%%%%%%%%%%%%%% 
\begin{tikzpicture}[thick,scale=2.5,color=blue!50!black, baseline=0.0cm, >=stealth, 
				style={x={(-0.9cm,-0.4cm)},y={(0.8cm,-0.4cm)},z={(0cm,0.9cm)}}]
% visible edges of cube
\draw[
	 color=gray, 
	 opacity=0.4, 
	 semithick
	 ] 
	 (0,1,1) -- (0,1,0) -- (1,1,0) -- (1,1,1) -- (0,1,1) -- (0,0,1) -- (1,0,1) -- (1,0,0) -- (1,1,0)
	 (1,0,1) -- (1,1,1);
% to cut off unwanted edges: 
\clip (1,0,0) -- (1,1,0) -- (0,1,0) -- (0,1,1) -- (0,0,1) -- (1,0,1);
% 3-stratum: 
\fill [blue!20,opacity=0.2] (1,0,0) -- (1,1,0) -- (0,1,0) -- (0,1,1) -- (0,0,1) -- (1,0,1);
% invisible edges of cube: 
\draw[
	 color=gray, 
	 opacity=0.3, 
	 semithick,
	 dashed
	 ] 
	 (1,0,0) -- (0,0,0) -- (0,1,0)
	 (0,0,0) -- (0,0,1);
%%%
% bottom vertices: 
\coordinate (b1) at (0.1, 0, 0);
\coordinate (b2) at (0.5, 0, 0);
\coordinate (b3) at (0.7, 0, 0);
\coordinate (b4) at (0.5, 0.3, 0);
\coordinate (b5) at (0.6, 0.7, 0);
\coordinate (b6) at (0.3, 1, 0);
\coordinate (b7) at (0.7, 1, 0);
% top vertices: 
\coordinate (t1) at (0.1, 0, 1);
\coordinate (t2) at (0.5, 0, 1);
\coordinate (t3) at (0.7, 0, 1);
\coordinate (t4) at (0.5, 0.3, 1);
\coordinate (t5) at (0.6, 0.7, 1);
\coordinate (t6) at (0.3, 1, 1);
\coordinate (t7) at (0.7, 1, 1);
%
% "frames":
\draw[color=red!80!black, very thick] (b4) -- (b1) -- (t1) -- (t4);
\draw[color=red!80!black, very thick] (b4) -- (b2) -- (t2) -- (t4);
\draw[color=red!80!black, very thick] (b4) -- (b3) -- (t3) -- (t4);
\draw[color=red!80!black, very thick] (b5) -- (b7) -- (t7) -- (t5);
\draw[color=red!80!black, very thick] (t6) -- (b6) .. controls +(0, -0.15, 0) and +(-0.15, 0.15, 0) .. (b5);
\draw[color=red!80!black, very thick] (b6) -- (t6) .. controls +(0, -0.15, 0) and +(-0.15, 0.15, 0) .. (t5);
% planes and some strings lines: 
\fill [magenta!70,opacity=0.7] (b1) -- (b4) -- (t4) -- (t1);
\draw[string=red!80!black, very thick] (b4) -- (b1);
\draw[string=red!80!black, very thick] (b1) -- (t1);
\fill [magenta!50,opacity=0.7] (b2) -- (b4) -- (t4) -- (t2);
\draw[string=red!80!black, very thick] (b2) -- (b4);
\draw[string=red!80!black, very thick] (t2) -- (b2);
\fill [red!30,opacity=0.7] (b3) -- (b4) -- (t4) -- (t3);
\draw[string=red!80!black, very thick] (t3) -- (b3);
\fill [red!40,opacity=0.7] (b4) -- (t4) -- (t5) -- (b5);
\draw[string=red!80!black, very thick] (t5) -- (t4);
\draw[string=red!80!black, very thick] (b4) -- (b5);
\fill [red!50,opacity=0.7] (b6) .. controls +(0, -0.15, 0) and +(-0.15, 0.15, 0) .. (b5) -- (t5) .. controls +(-0.15, 0.15, 0)  and +(0, -0.15, 0) .. (t6);
\draw[draw=red!80!black, postaction={decorate}, decoration={markings,mark=at position .51 with {\arrow[color=red!80!black]{>}}}, very thick] (b6) .. controls +(0, -0.15, 0) and +(-0.15, 0.15, 0) .. (b5);
\draw[string=red!80!black, very thick] (t6) -- (b6);
\fill [magenta!60,opacity=0.7] (b7) -- (b5) -- (t5) -- (t7);
\draw[string=red!80!black, very thick] (b7) -- (t7);
\draw[string=red!80!black, very thick] (b5) -- (b7);
\draw[string=red!80!black, very thick] (t7) -- (t5);
%
% remaining bottom string diagram: 
\draw[color=red!80!black, very thick] (b7) -- (b5);
\draw[color=red!80!black, very thick] (b5) -- (b4);
\draw[string=red!80!black, very thick] (b3) -- (b4);
% top string diagram: 
\draw[color=red!80!black, very thick] (t7) -- (t5);
\draw[draw=red!80!black, postaction={decorate}, decoration={markings,mark=at position .51 with {\arrow[color=red!80!black]{<}}}, very thick] (t6) .. controls +(0, -0.15, 0) and +(-0.15, 0.15, 0) .. (t5);
\draw[color=red!80!black, very thick] (t5) -- (t4);
\draw[string=red!80!black, very thick] (t1) -- (t4);
\draw[string=red!80!black, very thick] (t4) -- (t2);
\draw[string=red!80!black, very thick] (t4) -- (t3);
% X-lines: 
\draw[string=green!60!black, very thick] (b4) -- (t4);
\draw[string=green!50!black, very thick] (t5) -- (b5);
%
% bottom 0-strata: 
\fill[color=green!60!black] (b4) circle (0.8pt) node[left] (0up) {};
\fill[color=green!50!black] (b5) circle (0.8pt) node[left] (0up) {};
% top 0-strata: 
\fill[color=green!60!black] (t4) circle (0.8pt) node[left] (0up) {};
\fill[color=green!50!black] (t5) circle (0.8pt) node[left] (0up) {};
% Xlabels: 
\draw[line width=1] (0.5, 0.17, 0.5) node[line width=0pt] (beta) {{\footnotesize $X_2$}};
\draw[line width=1] (0.6, 0.57, 0.5) node[line width=0pt] (beta) {{\footnotesize $X_1$}};
%
% visible edges of cube
\draw[
	 color=gray, 
	 opacity=0.4, 
	 semithick
	 ] 
	 (0,1,1) -- (0,1,0) -- (1,1,0) -- (1,1,1) -- (0,1,1) -- (0,0,1) -- (1,0,1) -- (1,0,0) -- (1,1,0)
	 (1,0,1) -- (1,1,1);
\end{tikzpicture}
%%%%%%%%%%%%%%%%%%%%%% 
$$
%arXiv_v2: 
	%are examples of elements of the first three layers in $\tz$, respectively. Here 
	where here 
and also in~\eqref{eq:XXs} below we indicate the orientation of a 2-stratum $\alpha_i$ by arrows on the intersection of its boundary with the surface of the cube; the orientation of 1-strata $X_i$ can be chosen independently of the orientation of their adjacent 2-strata. 
To avoid clutter we suppress all non-essential decorations. 

%arXiv_v2: 
	%The fourth layer of $\tz$ is made of \textsl{3-morphisms} which 
	The 3-morphisms of $\tz'$ 
are constructed with the help of~$\zz$. 
Roughly, we obtain the set of 3-morphisms between two \textsl{parallel} 2-morphisms~$X$ and~$X'$, i.\,e.~two 2-morphisms between the same source and target, by 
%arXiv_v2: 
	choosing representing diagrams $D_{X}$ and $D_{X'}$ for~$X$ and~$X'$, respectively, and 
gluing the bottom and top of 
%arXiv_v2: 
	%~$X$ and~$X'$ 
	$D_{X}$ and~$D_{X'}$ 
to a sphere
%arXiv_v2: 
	%$S_{X,X'}$
	$S_{D_{X},D_{X'}}$, and then evaluating~$\zz$ on 
%arXiv_v2: 
	%$S_{X,X'}$. 
	$S_{D_{X},D_{X'}}$.
%arXiv_v2: 
	In the following, we will make this idea precise. 

%arXiv_v2: 
	%More precisely, for  
	For 
two parallel 2-morphisms $X,X': \alpha \rightarrow \beta$ 
%arXiv_v2: 
	we choose representing diagrams $D_{X}$ and $D_{X'}$, respectively, whose non-horizontal faces coincide. Then 
we 
%arXiv_v2: 
	%first 
consider the 
%arXiv_v2: 
	% empty 
cube 
%arXiv_v2
	%$C_{X,X'}$ 
	$C_{D_{X},D_{X'}}$ 
%arXiv_v2: 
	which has no strata in the interior and 
whose bottom and top coincide with the bottom and top of 
%arXiv_v2: 
	%$X$ and~$X'$, 
	$D_{X}$ and $D_{X'}$,
respectively, and whose remaining faces are 
%arXiv_v2: 
	those of~$D_{X}$ or~$D_{X'}$ (which are identical). 
For instance, in the example 
\be\label{eq:XXs}
%arXiv_v2: 
	%X =
	D_{X}= 
%arXiv_v2: [made cubes slightly larger]
%%%%%%%%%%%%%%%%%%%%%% 
\begin{tikzpicture}[thick,scale=3.0,color=blue!50!black, baseline=0.0cm, >=stealth, 
				style={x={(-0.9cm,-0.4cm)},y={(0.8cm,-0.4cm)},z={(0cm,0.9cm)}}]
% visible edges of cube
\draw[
	 color=gray, 
	 opacity=0.4, 
	 semithick
	 ] 
	 (0,1,1) -- (0,1,0) -- (1,1,0) -- (1,1,1) -- (0,1,1) -- (0,0,1) -- (1,0,1) -- (1,0,0) -- (1,1,0)
	 (1,0,1) -- (1,1,1);
% to cut off unwanted edges: 
\clip (1,0,0) -- (1,1,0) -- (0,1,0) -- (0,1,1) -- (0,0,1) -- (1,0,1);
% 3-stratum: 
\fill [blue!20,opacity=0.2] (1,0,0) -- (1,1,0) -- (0,1,0) -- (0,1,1) -- (0,0,1) -- (1,0,1);
% invisible edges of cube: 
\draw[
	 color=gray, 
	 opacity=0.3, 
	 semithick,
	 dashed
	 ] 
	 (1,0,0) -- (0,0,0) -- (0,1,0)
	 (0,0,0) -- (0,0,1);
%%%
% bottom points: 
\coordinate (b1) at (0.2, 0, 0);
\coordinate (b2) at (0.8, 0, 0);
\coordinate (b3) at (0.5, 0.5, 0);
\coordinate (b4) at (0.5, 1, 0);
% top points: 
\coordinate (t1) at (0.2, 0, 1);
\coordinate (t2) at (0.8, 0, 1);
\coordinate (t3) at (0.5, 0.5, 1);
\coordinate (t4) at (0.5, 1, 1);
%
%% planes: 
\fill [red!70,opacity=0.7] (b1) -- (b3) -- (t3) -- (t1);
\fill [red!50,opacity=0.7] (b2) -- (b3) -- (t3) -- (t2);
\fill [magenta!60,opacity=0.7] (b4) -- (b3) -- (t3) -- (t4);
\draw[draw=red!80!black, 
	 postaction={decorate}, 
	 decoration={markings,mark=at position .61 with {\arrow[color=red!80!black]{>}}}, 
	 very thick] 
	(t1) -- (b1);
\draw[draw=red!80!black, 
	 postaction={decorate}, 
	 decoration={markings,mark=at position .61 with {\arrow[color=red!80!black]{>}}}, 
	 very thick] 
	(t2) -- (b2);
\draw[draw=red!80!black, 
	 postaction={decorate}, 
	 decoration={markings,mark=at position .51 with {\arrow[color=red!80!black]{>}}}, 
	 very thick] 
	(b4) -- (t4);
%% "frames":
\draw[color=red!80!black, very thick] (t3) -- (t1) -- (b1) -- (b3);
\draw[color=red!80!black, very thick] (t3) -- (t2) -- (b2) -- (b3);
\draw[color=red!80!black, very thick] (t3) -- (t4) -- (b4) -- (b3);
% X-line: 
\draw[string=green!60!black, very thick] (t3) -- (b3);
%
% bottom 0-stratum: 
\fill[color=green!60!black] (b3) circle (0.8pt) node[left] (0up) {};
% top 0-strata: 
\fill[color=green!60!black] (t3) circle (0.8pt) node[left] (0up) {};
%
% visible edges of cube
\draw[
	 color=gray, 
	 opacity=0.4, 
	 semithick
	 ] 
	 (0,1,1) -- (0,1,0) -- (1,1,0) -- (1,1,1) -- (0,1,1) -- (0,0,1) -- (1,0,1) -- (1,0,0) -- (1,1,0)
	 (1,0,1) -- (1,1,1);
%arXiv_v2: [added labels for strata]
	\draw[line width=1] (0.4, 0.7, 0.4) node[line width=0pt] (beta) {{\footnotesize $\alpha$}};
	\draw[line width=1] (0.7, 0.2, 0.5) node[line width=0pt] (beta) {{\footnotesize $\gamma$}};
	\draw[line width=1] (0.3, 0.2, 0.5) node[line width=0pt] (beta) {{\footnotesize $\beta$}};
	\draw[line width=1] (0.3, 0.4, 0.4) node[line width=0pt] (beta) {{\footnotesize $X$}};
	\draw[line width=1] (0.8, 0.6, 0) node[line width=0pt] (beta) {{\footnotesize $u$}};
	\draw[line width=1] (0.2, 0.85, 0) node[line width=0pt] (beta) {{\footnotesize $u$}};
	\draw[line width=1] (0.5, 0.15, 1) node[line width=0pt] (beta) {{\footnotesize $u$}};
\end{tikzpicture}
%%%%%%%%%%%%%%%%%%%%%% 
\, , \quad 
%arXiv_v2: 
	%X' = 
	D_{X'}= 
%%%%%%%%%%%%%%%%%%%%%% 
\begin{tikzpicture}[thick,scale=3.0,color=blue!50!black, baseline=0.0cm, >=stealth, 
				style={x={(-0.9cm,-0.4cm)},y={(0.8cm,-0.4cm)},z={(0cm,0.9cm)}}]
% visible edges of cube
\draw[
	 color=gray, 
	 opacity=0.4, 
	 semithick
	 ] 
	 (0,1,1) -- (0,1,0) -- (1,1,0) -- (1,1,1) -- (0,1,1) -- (0,0,1) -- (1,0,1) -- (1,0,0) -- (1,1,0)
	 (1,0,1) -- (1,1,1);
% to cut off unwanted edges: 
\clip (1,0,0) -- (1,1,0) -- (0,1,0) -- (0,1,1) -- (0,0,1) -- (1,0,1);
% 3-stratum: 
\fill [blue!20,opacity=0.2] (1,0,0) -- (1,1,0) -- (0,1,0) -- (0,1,1) -- (0,0,1) -- (1,0,1);
% invisible edges of cube: 
\draw[
	 color=gray, 
	 opacity=0.3, 
	 semithick,
	 dashed
	 ] 
	 (1,0,0) -- (0,0,0) -- (0,1,0)
	 (0,0,0) -- (0,0,1);
%%%
\tdplotsetcoord{t3}{0.2}{90}{240}
\tdplotsetcoord{t4}{0.2}{90}{300}
%
% bottom points: 
\coordinate (b1) at (0.2, 0, 0);
\coordinate (b2) at (0.8, 0, 0);
\coordinate (b31) at ($(t3) + (0.5,0.5,0)$);
\coordinate (b32) at ($(t4) + (0.5,0.5,0)$);
\coordinate (b33) at (0.5, 0.7, 0);
\coordinate (b4) at (0.5, 1, 0);
% top points: 
\coordinate (t1) at (0.2, 0, 1);
\coordinate (t2) at (0.8, 0, 1);
\coordinate (t31) at ($(t3) + (0.5,0.5,1)$);
\coordinate (t32) at ($(t4) + (0.5,0.5,1)$);
\coordinate (t33) at (0.5, 0.7, 1);
\coordinate (t4) at (0.5, 1, 1);
%%
%% planes: 
\fill [red!70,opacity=0.7] (b1) -- (b31) -- (t31) -- (t1);
\fill [red!50,opacity=0.7] (b2) -- (b32) -- (t32) -- (t2);
\fill [magenta!60,opacity=0.7] (b4) -- (b33) -- (t33) -- (t4);
%
%% "frames":
\draw[color=red!80!black, very thick] (b31) -- (b1) -- (t1) -- (t31);
\draw[color=red!80!black, very thick] (b32) -- (b2) -- (t2) -- (t32);
\draw[color=red!80!black, very thick] (b33) -- (b4) -- (t4) -- (t33);
% orientations: 
\draw[draw=red!80!black, 
	 postaction={decorate}, 
	 decoration={markings,mark=at position .61 with {\arrow[color=red!80!black]{>}}}, 
	 very thick] 
	(t1) -- (b1);
\draw[draw=red!80!black, 
	 postaction={decorate}, 
	 decoration={markings,mark=at position .61 with {\arrow[color=red!80!black]{>}}}, 
	 very thick] 
	(t2) -- (b2);
\draw[draw=red!80!black, 
	 postaction={decorate}, 
	 decoration={markings,mark=at position .51 with {\arrow[color=red!80!black]{>}}}, 
	 very thick] 
	(b4) -- (t4);
% X-lines: 
\draw[string=green!60!black, very thick] (t31) -- (b31);
\draw[string=green!50!black, very thick] (b32) -- (t32);
\draw[string=green!70!black, very thick] (t33) -- (b33);
%
% bottom 0-stratum: 
\fill[color=green!60!black] (b31) circle (0.8pt) node[left] (0up) {};
\fill[color=green!50!black] (b32) circle (0.8pt) node[left] (0up) {};
\fill[color=green!70!black] (b33) circle (0.8pt) node[left] (0up) {};
% top 0-strata: 
\fill[color=green!60!black] (t31) circle (0.8pt) node[left] (0up) {};
\fill[color=green!50!black] (t32) circle (0.8pt) node[left] (0up) {};
\fill[color=green!70!black] (t33) circle (0.8pt) node[left] (0up) {};
%
% visible edges of cube
\draw[
	 color=gray, 
	 opacity=0.4, 
	 semithick
	 ] 
	 (0,1,1) -- (0,1,0) -- (1,1,0) -- (1,1,1) -- (0,1,1) -- (0,0,1) -- (1,0,1) -- (1,0,0) -- (1,1,0)
	 (1,0,1) -- (1,1,1);
	%arXiv_v2: [added labels for strata]
	\draw[line width=1] (0.4, 0.75, 0.4) node[line width=0pt] (beta) {{\footnotesize $\alpha$}};
	\draw[line width=1] (0.75, 0.2, 0.6) node[line width=0pt] (beta) {{\footnotesize $\gamma$}};
	\draw[line width=1] (0.33, 0.2, 0.25) node[line width=0pt] (beta) {{\footnotesize $\beta$}};
	\draw[line width=1] (0.3, 0.4, 0.3) node[line width=0pt] (beta) {{\footnotesize $Y_1$}};
	\draw[line width=1] (0.3, 0.31, 0.48) node[line width=0pt] (beta) {{\footnotesize $Y_2$}};
	\draw[line width=1] (0.75, 0.4, 0.55) node[line width=0pt] (beta) {{\footnotesize $Y_3$}};
	\draw[line width=1] (0.8, 0.6, 0) node[line width=0pt] (beta) {{\footnotesize $u$}};
%%%%%%%%%%%%%%%%%%%%%% 
\end{tikzpicture}
\, , 
\ee
%arXiv_v2: 
	%where all decorations are suppressed for legibility,  
we thus have 
\be\label{eq:CXXcube}
%arXiv_v2_ 
	% C_{X,X'} = 
	C_{D_{X},D_{X'}} = 
%%%%%%%%%%%%%%%%%%%%%% 
\begin{tikzpicture}[thick,scale=2.5,color=blue!50!black, baseline=0.0cm, >=stealth, 
				style={x={(-0.9cm,-0.4cm)},y={(0.8cm,-0.4cm)},z={(0cm,0.9cm)}}]
% visible edges of cube
\draw[
	 color=gray, 
	 opacity=0.4, 
	 semithick
	 ] 
	 (0,1,1) -- (0,1,0) -- (1,1,0) -- (1,1,1) -- (0,1,1) -- (0,0,1) -- (1,0,1) -- (1,0,0) -- (1,1,0)
	 (1,0,1) -- (1,1,1);
% to cut off unwanted edges: 
\clip (1,0,0) -- (1,1,0) -- (0,1,0) -- (0,1,1) -- (0,0,1) -- (1,0,1);
% 3-stratum: 
\fill [blue!20,opacity=0.2] (1,0,0) -- (1,1,0) -- (0,1,0) -- (0,1,1) -- (0,0,1) -- (1,0,1);
% invisible edges of cube: 
\draw[
	 color=gray, 
	 opacity=0.3, 
	 semithick,
	 dashed
	 ] 
	 (1,0,0) -- (0,0,0) -- (0,1,0)
	 (0,0,0) -- (0,0,1);
%%%
% bottom points: 
\coordinate (b1) at (0.2, 0, 0);
\coordinate (b2) at (0.8, 0, 0);
\coordinate (b3) at (0.5, 0.5, 0);
\coordinate (b4) at (0.5, 1, 0);
% top points: 
\coordinate (t1) at (0.2, 0, 1);
\coordinate (t2) at (0.8, 0, 1);
\tdplotsetcoord{t3}{0.2}{90}{240}
\tdplotsetcoord{t4}{0.2}{90}{300}
\tdplotsetcoord{t5}{0.2}{90}{90}
\coordinate (t6) at (0.5, 1, 1);
%
%% "frames":
\draw[color=red!80!black, very thick] ($(t3) + (0.5,0.5,1)$) -- (t1) -- (b1) -- (b3);
\draw[color=red!80!black, very thick] ($(t4) + (0.5,0.5,1)$) -- (t2) -- (b2) -- (b3);
\draw[color=red!80!black, very thick] ($(t5) + (0.5,0.5,1)$) -- (t6) -- (b4) -- (b3);
\draw[draw=red!80!black, 
	 postaction={decorate}, 
	 decoration={markings,mark=at position .61 with {\arrow[color=red!80!black]{>}}}, 
	 very thick] 
	(t1) -- (b1);
\draw[draw=red!80!black, 
	 postaction={decorate}, 
	 decoration={markings,mark=at position .61 with {\arrow[color=red!80!black]{>}}}, 
	 very thick] 
	(t2) -- (b2);
\draw[draw=red!80!black, 
	 postaction={decorate}, 
	 decoration={markings,mark=at position .51 with {\arrow[color=red!80!black]{>}}}, 
	 very thick] 
	(b4) -- (t6);
%
% bottom 0-stratum: 
\fill[color=green!60!black] (b3) circle (0.8pt) node[left] (0up) {};
% top 0-strata: 
\fill[color=green!60!black] ($(t3) + (0.5,0.5,1)$) circle (0.8pt) node[left] (0up) {};
\fill[color=green!60!black] ($(t4) + (0.5,0.5,1)$) circle (0.8pt) node[left] (0up) {};
\fill[color=green!60!black] ($(t5) + (0.5,0.5,1)$) circle (0.8pt) node[left] (0up) {};
%
	%arXiv_v2: [added labels for strata]
	\draw[line width=1] (0.4, 1, 0.4) node[line width=0pt] (beta) {{\footnotesize $\alpha$}};
	\draw[line width=1] (0.45, 0.2, 0.6) node[line width=0pt] (beta) {{\footnotesize $\beta$}};
	\draw[line width=1] (1, 0.12, 0.55) node[line width=0pt] (beta) {{\footnotesize $\gamma$}};
	\draw[line width=1] (0.62, 0.62, 0) node[line width=0pt] (beta) {{\footnotesize $X$}};
	\draw[line width=1] (0.37, 0.6, 1) node[line width=0pt] (beta) {{\footnotesize $Y_1$}};
	\draw[line width=1] (0.28, 0.31, 1) node[line width=0pt] (beta) {{\footnotesize $Y_2$}};
	\draw[line width=1] (0.5, 0.15, 1) node[line width=0pt] (beta) {{\footnotesize $Y_3$}};
	\draw[line width=1] (0.1, 0.85, 0.65) node[line width=0pt] (beta) {{\footnotesize $u$}};
%
% visible edges of cube
\draw[
	 color=gray, 
	 opacity=0.4, 
	 semithick
	 ] 
	 (0,1,1) -- (0,1,0) -- (1,1,0) -- (1,1,1) -- (0,1,1) -- (0,0,1) -- (1,0,1) -- (1,0,0) -- (1,1,0)
	 (1,0,1) -- (1,1,1);
\end{tikzpicture}
%%%%%%%%%%%%%%%%%%%%%%  
\, . 
\ee
We then enclose this decorated cube 
%arXiv_v2: 
	%surface 
%arXiv_v2: 
	%$C_{X,X'}$
	$C_{D_{X},D_{X'}}$ 
in the smallest possible sphere, i.\,e.~the sphere of radius $\frac{\sqrt{3}}{2}$ centred at $(\frac{1}{2},\frac{1}{2},\frac{1}{2})$. 
This sphere is made into a decorated stratified manifold
%arXiv_v2: 
	%$S_{X,X'}$ 
	$S_{D_{X},D_{X'}}$ 
by radially projecting the strata of 
%arXiv_v2: 
	%$C_{X,X'}$ 
	$C_{D_{X},D_{X'}}$ 
from the mutual centre $(\frac{1}{2},\frac{1}{2},\frac{1}{2})$. 
For the example in~\eqref{eq:XXs}, \eqref{eq:CXXcube} this yields
\be\label{eq:SXXprimesphere}
%arXiv_v2: 
	%S_{X,X'}
	S_{D_{X},D_{X'}}
= 
%%%%%%%%%%%%%%%%%%%%% 
 \tdplotsetmaincoords{50}{135}
 \begin{tikzpicture}[scale=2.5,opacity=0.3,tdplot_main_coords,fill opacity=0.4, baseline=-0.2cm, >=stealth]
 %
 % SPHERE START
 \pgfsetlinewidth{.1pt}
 \tdplotsetpolarplotrange{0}{180}{0}{360}
 \tdplotsphericalsurfaceplot{60}{60}%[parametricfill]
 {1}%{sqrt(15/2)*sin(\tdplottheta)*cos(\tdplottheta)}%
 {gray}%{transparent!0}%
 {blue!14}%{\tdplotphi}%
     {}%
     {}%
     {}%
     %	\coordinate (O) at (0 ,0 ,0);
 % SPHERE END	
 %
 \coordinate (O) at (0 ,0 ,0);
 %
 % right longitude: 
 	\foreach \angle in { 280 } % 45 is right at the back
 	{
 		%calculate the sine and cosine of the angle
 		\tdplotsinandcos{\sintheta}{\costheta}{\angle}%
%
 		%define the rotated coordinate frame based on the angle
 		\tdplotsetthetaplanecoords{\angle}
%		
 		%draw the circle in the rotated frame
 		\tdplotdrawarc[->, 
 					red!80!black, 
 					ultra thick, 
 					opacity=0.5, 
 					tdplot_rotated_coords]%
 			{(O)}{1}{180}{290}{}{}%
 		\tdplotdrawarc[%
 					red!80!black, 
 					ultra thick, 
 					opacity=0.5, 
 					tdplot_rotated_coords]%
 			{(O)}{1}{286.5}{345}{}{}%
 	}
 %
 % left longitudes: 
 	\foreach \angle in { 245, 315 } % 45 is right at the back
 	{
 		%calculate the sine and cosine of the angle
 		\tdplotsinandcos{\sintheta}{\costheta}{\angle}%

 		%define the rotated coordinate frame based on the angle
 		\tdplotsetthetaplanecoords{\angle}
		
 		%draw the circle in the rotated frame
 		\tdplotdrawarc[->, red!80!black,ultra thick,opacity=0.5,tdplot_rotated_coords]%
 			{(O)}{1}{15}{123.5}{}{}
 		\tdplotdrawarc[red!80!black,ultra thick,opacity=0.5,tdplot_rotated_coords]%
 			{(O)}{1}{120}{180}{}{}
 	}
 \tdplotsetcoord{P1}{1}{15}{245}
 \tdplotsetcoord{P2}{1}{15}{315}
 \tdplotsetcoord{P3}{1}{15}{100}
 \tdplotsetcoord{P4}{1}{180}{245}
 \tdplotsetcoord{Q1}{1}{75}{92}
 \tdplotsetcoord{Q3}{1}{116}{305}
 %
%arXiv_v2: [added labels for strata]
 \fill[color=green!50!black, opacity=1] (P1) circle (0.8pt) node[right] (0up) {{\footnotesize $Y_2$}};
 \fill[color=green!50!black, opacity=1] (P2) circle (0.8pt) node[right] (0up) {{\footnotesize $Y_3$}};
 \fill[color=green!50!black, opacity=1] (P3) circle (0.8pt) node[right] (0up) {{\footnotesize $Y_1$}};
 \fill[color=green!50!black, opacity=1] (P4) circle (0.8pt) node[below] (0up) {{\footnotesize $X$}};
	\draw[color=red!80!black, line width=1] (Q1) node[line width=0pt] (beta) {{\footnotesize $\alpha$}};
	\draw[color=red!80!black, line width=1] (0.45, 0.2, 0.6) node[line width=0pt] (beta) {{\footnotesize $\beta$}};
	\draw[color=red!80!black, line width=1] (Q3) node[line width=0pt] (beta) {{\footnotesize $\gamma$}};
 \end{tikzpicture}
%%%%%%%%%%%%%%%%%%%%% 
\ee
%arXiv_v2: 
	%Via the outward projection, the stratified manifold $S_{X,X'}$ equips the cube surface $C_{X,X'}$ with the structure of a stratified manifold. In the following, we assume that $C_{X,X'}$ is equipped with this stratified manifold structure, and use the terms ``sphere'' and ``cube surface'' interchangeably. 
	In the following, especially in the definition of composition of 3-morphism below, for typographical reasons we will sometimes depict a defect sphere 
%arXiv_v2: 
	%$S_{X,X'}$ 
	$S_{D_{X},D_{X'}}$ 
as the boundary of the decorated stratified cube (with corners) 
%arXiv_v2: 
	%$C_{X,X'}$.
	$C_{D_{X},D_{X'}}$. 
Similarly, if a decorated stratified cube (with corners) is a representative of an object, 1- or 2-morphism in $\tz'$, we can identify it with a defect ball (i.\,e.\ a defect bordism whose underlying manifold is a 3-ball) by projecting as above. 

%arXiv_v2: 
	We finally turn to the definition of 3-morphisms in $\tz'$. Given two parallel 2-morphisms~$X$ and~$X'$, we would like to set the vector space of 3-morphisms froms~$X$ to~$X'$ to be
\be\label{eq:3HomTZalmost}
\zz \big( S_{D_{X},D_{X'}} \big) \, , 
\ee
for a choice of representing diagrams $D_{X}$ and $D_{X'}$ for $X$ and $X'$, respectively.%
%arXiv_v2: [Made this discussion into a footnote:]
\footnote{The idea behind this definition is the following. 
In an $n$-dimensional defect TQFT $\zz_n$, there are `local operators' which are to be thought of as inserted at intersection points of 1-dimensional `line operators'. 
In the associated $n$-category these local operators are precisely the $n$-morphisms. 
To determine which operators can be inserted at a given intersection point~$\nu$, one removes a tiny $n$-ball around~$\nu$ from the defect bordism and evaluates $\zz_n$ on its decorated surface to obtain the `state space'  of $n$-morphisms. 
This was first explained for $n=2$ in \cite[Sect.\,2.4]{dkr1107.0495} (see also \cite[Eq.\,(3.7)]{cr1210.6363}); \eqref{eq:3HomTZalmost} is the 3-dimensional case.}
%
%arXiv_v2: 
	To obtain the actual space of 3-morphisms that is independent of choices, we note that for different representing diagrams $D'_{X}$ and $D'_{X'}$ for~$X$ and~$X'$, there are isotopies~$f$ from $D_{X}$ to $D'_{X}$ and~$f'$ from $D_{X'}$ to $D'_{X'}$, respectively. From these we obtain a cylinder $Z_{f,f'}$ which gives an isomorphism from $ S_{D_{X},D_{X'}}$ to $ S_{D'_{X},D'_{X'}}$ in $\Bordd$. Evaluating with~$\zz$ produces an isomorphism $\zz(Z_{f,f'}) \colon\zz ( S_{D_{X},D_{X'}} ) \rightarrow \zz (S_{D'_{X},D'_{X'}} ) $ which does not depend on the choice of isotopies~$f$ and~$f'$, and hence we may denote it $\zz(Z_{f,f'})=:F_{D,D'}$. Moreover, for yet another choice of representing diagrams $D''_{X}$ and $D''_{X'}$ for~$X$ and~$X'$, respectively, we have $F_{D',D''} \circ F_{D,D'}= F_{D,D''}$. In short, the isomorphisms~$F$ form a direct system relating all choices of representing diagrams for the 2-morphisms. We finally set
\be\label{eq:3HomTZ}
\Hom_{\tz'} \!\big(X,X'\big) = \lim_{F}  \zz \big( S_{D_{X},D_{X'}} \big) \, .
\ee
%arXiv_v2: 
	Note that since all morphisms in the direct system are isomorphisms, for every choice of representing diagrams there is an isomorphism $ \zz \big( S_{D_{X},D_{X'}} \big)  \cong \Hom_{\tz'} \!\big(X,X'\big)$. Moreover, we can invert the cone isomorphisms to find that the limit also serves as a colimit of the direct system. 
In the following, since all constructions will be compatible with the direct system above and thus factor through the limit, we will abuse notation and write $\Hom_{\tz'} \!(X,X') = \zz (S_{X,X'})$. 

\medskip
% %arXiv_v2: 
	% To conclude the construction of the 3-morphisms in $\tz'$,  we define the vector space of 3-morphisms from~$X$ to~$X'$ in $\tz'$ 
% 	(up to isomorphy\footnote{For different choices of representatives of the 2-morphisms $X,X'$, the vector spaces $\zz(S_{X,X'})$ may differ, but they are all isomorphic. To assign a definite vector space to the pair $(X,X')$, one option is to choose a distinguished defect sphere $S_{X,X'} \in \Bordd$ for every such pair. Instead, by a standard construction, we define $\Hom_{\tz'} \!(X,X')$ to be the limit of the diagram consisting of the vector spaces $\zz(S_{X,X'})$ and the isomorphisms between them.\label{foot:limit}}) 
% as
% \be\label{eq:3HomTZ}
% \Hom_{\tz'} \!\big(X,X'\big) = \zz \big( S_{X,X'} \big) \, . 
% \ee

%arXiv_v2: 
	%With the four layers of $\tz'$ now at hand, we wish to endow the top three layers with compatible composition maps. 
	With the objects, 1-, 2- and 3-morphisms of $\tz'$ at hand, we will now describe its the composition maps. 
Firstly, the Gray product of 1-morphisms, denoted~$\sta$, simply amounts to stacking planes in cubes together (up to linear isotopy). 
If $\alpha: u \rightarrow v$ and $\beta: v \rightarrow w$ are 1-morphisms, then their Gray product $\beta \sta \alpha$ is obtained by concatenating the two cubes in negative $x$-direction, followed by rescaling the thusly obtained cuboid to a cube: 
$$
%%%%%%%%%%%%%%%%%%%%%% 
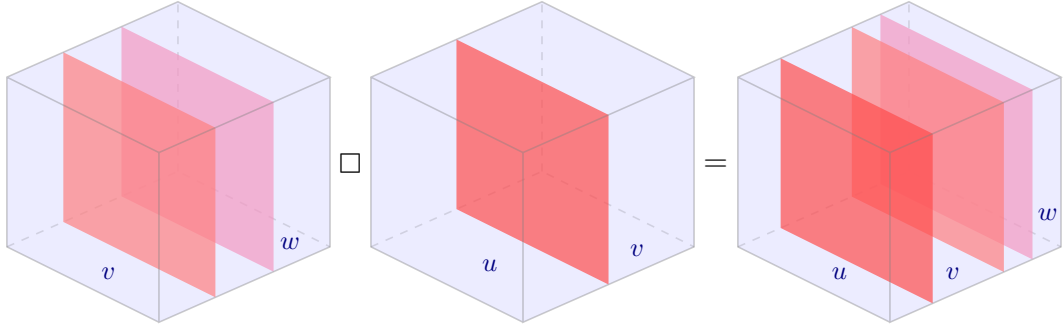
\begin{tikzpicture}[thick,scale=2.5,color=blue!50!black, baseline=0.0cm, >=stealth, 
				style={x={(-0.9cm,-0.4cm)},y={(0.8cm,-0.4cm)},z={(0cm,0.9cm)}}]
% 3-stratum: 
\fill [blue!20,opacity=0.2] (1,0,0) -- (1,1,0) -- (0,1,0) -- (0,1,1) -- (0,0,1) -- (1,0,1);
% invisible edges of cube: 
\draw[
	 color=gray, 
	 opacity=0.3, 
	 semithick,
	 dashed
	 ] 
	 (1,0,0) -- (0,0,0) -- (0,1,0)
	 (0,0,0) -- (0,0,1);
%%%
% 3-stratum: 
\fill [blue!20,opacity=0.2] (1,0,0) -- (1,1,0) -- (0,1,0) -- (0,1,1) -- (0,0,1) -- (1,0,1);
%
% lower endpoint of X: 
\coordinate (X0) at (0.33, 0.1, 0);
% upper endpoint of X: 
\coordinate (X1) at (0.33, 0.9, 1);
% lower endpoint of Y: 
\coordinate (Y0) at (0.67, 0.9, 0);
% upper endpoint of Y: 
\coordinate (Y1) at (0.67, 0.1, 1);
%
%
% planes and side lines: 
\fill [magenta!50,opacity=0.7] (0.33, 0, 0) -- (0.33, 1, 0) -- (0.33, 1, 1) -- (0.33, 0, 1);
\fill [red!50,opacity=0.7] (0.67, 0, 0) -- (0.67, 1, 0) -- (0.67, 1, 1) -- (0.67, 0, 1);
%
% 3-strata: 
\draw[line width=1] (0.85, 0.5, 0) node[line width=0pt] (beta) {{\footnotesize $v$}};
\draw[line width=1] (0.1, 0.85, 0) node[line width=0pt] (beta) {{\footnotesize $w$}};
%
% visible edges of cube
\draw[
	 color=gray, 
	 opacity=0.4, 
	 semithick
	 ] 
	 (0,1,1) -- (0,1,0) -- (1,1,0) -- (1,1,1) -- (0,1,1) -- (0,0,1) -- (1,0,1) -- (1,0,0) -- (1,1,0)
	 (1,0,1) -- (1,1,1);
\end{tikzpicture}
%%%%%%%%%%%%%%%%%%%%%% 
\sta
%%%%%%%%%%%%%%%%%%%%%% 
\begin{tikzpicture}[thick,scale=2.5,color=blue!50!black, baseline=0.0cm, >=stealth, 
				style={x={(-0.9cm,-0.4cm)},y={(0.8cm,-0.4cm)},z={(0cm,0.9cm)}}]
% 3-stratum: 
\fill [blue!20,opacity=0.2] (1,0,0) -- (1,1,0) -- (0,1,0) -- (0,1,1) -- (0,0,1) -- (1,0,1);
% invisible edges of cube: 
\draw[
	 color=gray, 
	 opacity=0.3, 
	 semithick,
	 dashed
	 ] 
	 (1,0,0) -- (0,0,0) -- (0,1,0)
	 (0,0,0) -- (0,0,1);
%%%
% 3-stratum: 
\fill [blue!20,opacity=0.2] (1,0,0) -- (1,1,0) -- (0,1,0) -- (0,1,1) -- (0,0,1) -- (1,0,1);
%
% planes and side lines: 
\fill [red!70,opacity=0.7] (0.5, 0, 0) -- (0.5, 1, 0) -- (0.5, 1, 1) -- (0.5, 0, 1);
%
% 3-strata: 
\draw[line width=1] (0.75, 0.5, 0) node[line width=0pt] (beta) {{\footnotesize $u$}};
\draw[line width=1] (0.2, 0.85, 0) node[line width=0pt] (beta) {{\footnotesize $v$}};
%
% visible edges of cube
\draw[
	 color=gray, 
	 opacity=0.4, 
	 semithick
	 ] 
	 (0,1,1) -- (0,1,0) -- (1,1,0) -- (1,1,1) -- (0,1,1) -- (0,0,1) -- (1,0,1) -- (1,0,0) -- (1,1,0)
	 (1,0,1) -- (1,1,1);
\end{tikzpicture}
%%%%%%%%%%%%%%%%%%%%%% 
= 
%%%%%%%%%%%%%%%%%%%%%% 
\begin{tikzpicture}[thick,scale=2.5,color=blue!50!black, baseline=0.0cm, >=stealth, 
				style={x={(-0.9cm,-0.4cm)},y={(0.8cm,-0.4cm)},z={(0cm,0.9cm)}}]
% 3-stratum: 
\fill [blue!20,opacity=0.2] (1,0,0) -- (1,1,0) -- (0,1,0) -- (0,1,1) -- (0,0,1) -- (1,0,1);
% invisible edges of cube: 
\draw[
	 color=gray, 
	 opacity=0.3, 
	 semithick,
	 dashed
	 ] 
	 (1,0,0) -- (0,0,0) -- (0,1,0)
	 (0,0,0) -- (0,0,1);
%%%
% 3-stratum: 
\fill [blue!20,opacity=0.2] (1,0,0) -- (1,1,0) -- (0,1,0) -- (0,1,1) -- (0,0,1) -- (1,0,1);
%
% lower endpoint of X: 
\coordinate (X0) at (0.33, 0.1, 0);
% upper endpoint of X: 
\coordinate (X1) at (0.33, 0.9, 1);
% lower endpoint of Y: 
\coordinate (Y0) at (0.67, 0.9, 0);
% upper endpoint of Y: 
\coordinate (Y1) at (0.67, 0.1, 1);
%
%
% planes and side lines: 
\fill [magenta!50,opacity=0.7] (0.1667, 0, 0) -- (0.1667, 1, 0) -- (0.1667, 1, 1) -- (0.1667, 0, 1);
\fill [red!50,opacity=0.7] (0.334, 0, 0) -- (0.334, 1, 0) -- (0.334, 1, 1) -- (0.334, 0, 1);
\fill [red!70,opacity=0.7] (0.75, 0, 0) -- (0.75, 1, 0) -- (0.75, 1, 1) -- (0.75, 0, 1);
%
% 3-strata: 
\draw[line width=1] (0.85, 0.5, 0) node[line width=0pt] (beta) {{\footnotesize $u$}};
\draw[line width=1] (0.5, 0.85, 0) node[line width=0pt] (beta) {{\footnotesize $v$}};
\draw[line width=1] (0.05, 0.97, 0.2) node[line width=0pt] (beta) {{\footnotesize $w$}};
%
% visible edges of cube
\draw[
	 color=gray, 
	 opacity=0.4, 
	 semithick
	 ] 
	 (0,1,1) -- (0,1,0) -- (1,1,0) -- (1,1,1) -- (0,1,1) -- (0,0,1) -- (1,0,1) -- (1,0,0) -- (1,1,0)
	 (1,0,1) -- (1,1,1);
\end{tikzpicture}
%%%%%%%%%%%%%%%%%%%%%% 
$$

Secondly, horizontal composition of 2-morphisms, denoted $\fus$, is concatenation in negative $y$-direction and rescaling: 
$$
%%%%%%%%%%%%%%%%%%%%%% 
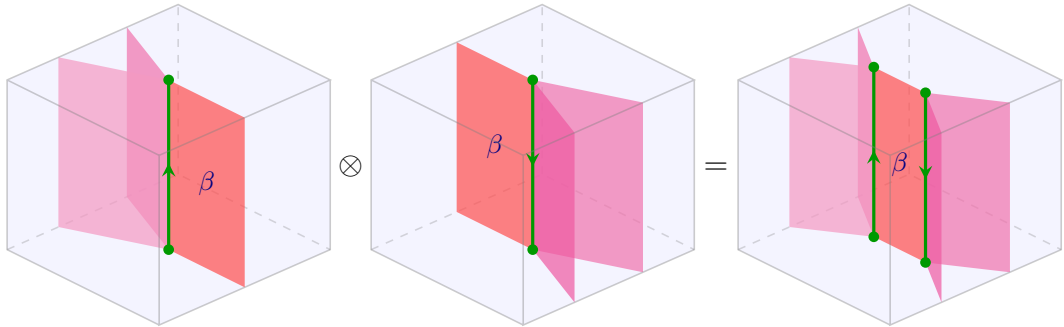
\begin{tikzpicture}[thick,scale=2.5,color=blue!50!black, baseline=0.0cm, >=stealth, 
				style={x={(-0.9cm,-0.4cm)},y={(0.8cm,-0.4cm)},z={(0cm,0.9cm)}}]
% 3-stratum: 
\fill [blue!20,opacity=0.2] (1,0,0) -- (1,1,0) -- (0,1,0) -- (0,1,1) -- (0,0,1) -- (1,0,1);
% invisible edges of cube: 
\draw[
	 color=gray, 
	 opacity=0.3, 
	 semithick,
	 dashed
	 ] 
	 (1,0,0) -- (0,0,0) -- (0,1,0)
	 (0,0,0) -- (0,0,1);
%%%
% bottom vertices: 
\coordinate (b1) at (0.3, 0, 0);
\coordinate (b2) at (0.7, 0, 0);
\coordinate (b4) at (0.5, 0.5, 0);
\coordinate (b7) at (0.5, 1, 0);
% top vertices: 
\coordinate (t1) at (0.3, 0, 1);
\coordinate (t2) at (0.7, 0, 1);
\coordinate (t4) at (0.5, 0.5, 1);
\coordinate (t7) at (0.5, 1, 1);
%
% planes and some strings lines: 
\fill [magenta!70,opacity=0.7] (b1) -- (b4) -- (t4) -- (t1);
\fill [magenta!50,opacity=0.7] (b2) -- (b4) -- (t4) -- (t2);
\fill [red!70,opacity=0.7] (b7) -- (b4) -- (t4) -- (t7);
%
% X-lines: 
\draw[string=green!60!black, very thick] (b4) -- (t4);
%
% bottom 0-strata: 
\fill[color=green!60!black] (b4) circle (0.8pt) node[left] (0up) {};
% top 0-strata: 
\fill[color=green!60!black] (t4) circle (0.8pt) node[left] (0up) {};
%
% 2-strata: 
\draw[line width=1] (0.5, 0.75, 0.5) node[line width=0pt] (beta) {{\footnotesize $\beta$}};
%
% visible edges of cube
\draw[
	 color=gray, 
	 opacity=0.4, 
	 semithick
	 ] 
	 (0,1,1) -- (0,1,0) -- (1,1,0) -- (1,1,1) -- (0,1,1) -- (0,0,1) -- (1,0,1) -- (1,0,0) -- (1,1,0)
	 (1,0,1) -- (1,1,1);
\end{tikzpicture}
%%%%%%%%%%%%%%%%%%%%%% 
\fus
%%%%%%%%%%%%%%%%%%%%%% 
\begin{tikzpicture}[thick,scale=2.5,color=blue!50!black, baseline=0.0cm, >=stealth, 
				style={x={(-0.9cm,-0.4cm)},y={(0.8cm,-0.4cm)},z={(0cm,0.9cm)}}]
% 3-stratum: 
\fill [blue!20,opacity=0.2] (1,0,0) -- (1,1,0) -- (0,1,0) -- (0,1,1) -- (0,0,1) -- (1,0,1);
% invisible edges of cube: 
\draw[
	 color=gray, 
	 opacity=0.3, 
	 semithick,
	 dashed
	 ] 
	 (1,0,0) -- (0,0,0) -- (0,1,0)
	 (0,0,0) -- (0,0,1);
%%%
% bottom vertices: 
\coordinate (b1) at (0.3, 1, 0);
\coordinate (b2) at (0.7, 1, 0);
\coordinate (b4) at (0.5, 0.5, 0);
\coordinate (b7) at (0.5, 0, 0);
% top vertices: 
\coordinate (t1) at (0.3, 1, 1);
\coordinate (t2) at (0.7, 1, 1);
\coordinate (t4) at (0.5, 0.5, 1);
\coordinate (t7) at (0.5, 0, 1);
%
% planes and some strings lines: 
\fill [magenta!70,opacity=0.7] (b1) -- (b4) -- (t4) -- (t1);
\fill [magenta!80,opacity=0.7] (b2) -- (b4) -- (t4) -- (t2);
\fill [red!70,opacity=0.7] (b7) -- (b4) -- (t4) -- (t7);
% 
% X-lines: 
\draw[string=green!60!black, very thick] (t4) -- (b4);
%
% bottom 0-strata: 
\fill[color=green!60!black] (b4) circle (0.8pt) node[left] (0up) {};
% top 0-strata: 
\fill[color=green!60!black] (t4) circle (0.8pt) node[left] (0up) {};
%
% 2-strata: 
\draw[line width=1] (0.5, 0.25, 0.5) node[line width=0pt] (beta) {{\footnotesize $\beta$}};
%
% visible edges of cube
\draw[
	 color=gray, 
	 opacity=0.4, 
	 semithick
	 ] 
	 (0,1,1) -- (0,1,0) -- (1,1,0) -- (1,1,1) -- (0,1,1) -- (0,0,1) -- (1,0,1) -- (1,0,0) -- (1,1,0)
	 (1,0,1) -- (1,1,1);
\end{tikzpicture}
%%%%%%%%%%%%%%%%%%%%%% 
= 
%%%%%%%%%%%%%%%%%%%%%% 
\begin{tikzpicture}[thick,scale=2.5,color=blue!50!black, baseline=0.0cm, >=stealth, 
				style={x={(-0.9cm,-0.4cm)},y={(0.8cm,-0.4cm)},z={(0cm,0.9cm)}}]
% 3-stratum: 
\fill [blue!20,opacity=0.2] (1,0,0) -- (1,1,0) -- (0,1,0) -- (0,1,1) -- (0,0,1) -- (1,0,1);
% invisible edges of cube: 
\draw[
	 color=gray, 
	 opacity=0.3, 
	 semithick,
	 dashed
	 ] 
	 (1,0,0) -- (0,0,0) -- (0,1,0)
	 (0,0,0) -- (0,0,1);
%%%
% bottom vertices: 
\coordinate (b1) at (0.3, 1, 0);
\coordinate (b2) at (0.7, 1, 0);
\coordinate (b4) at (0.5, 0.67, 0);
%\coordinate (b7) at (0.5, 0, 0);
\coordinate (n1) at (0.3, 0, 0);
\coordinate (n2) at (0.7, 0, 0);
\coordinate (n3) at (0.5, 0.33, 0);
% top vertices: 
\coordinate (t1) at (0.3, 1, 1);
\coordinate (t2) at (0.7, 1, 1);
\coordinate (t4) at (0.5, 0.67, 1);
%\coordinate (t7) at (0.5, 0, 1);
\coordinate (m1) at (0.3, 0, 1);
\coordinate (m2) at (0.7, 0, 1);
\coordinate (m3) at (0.5, 0.33, 1);
%
% planes and some strings lines: 
\fill [magenta!70,opacity=0.7] (b1) -- (b4) -- (t4) -- (t1);
\fill [magenta!80,opacity=0.7] (b2) -- (b4) -- (t4) -- (t2);
\fill [red!70,opacity=0.7] (n3) -- (b4) -- (t4) -- (m3);
\fill [magenta!70,opacity=0.7] (n1) -- (n3) -- (m3) -- (m1);
\fill [magenta!50,opacity=0.7] (n2) -- (n3) -- (m3) -- (m2);
% 
% X-lines:
\draw[string=green!60!black, very thick] (t4) -- (b4);
\draw[string=green!60!black, very thick] (n3) -- (m3);
%
% bottom 0-strata: 
\fill[color=green!60!black] (b4) circle (0.8pt) node[left] (0up) {};
\fill[color=green!60!black] (m3) circle (0.8pt) node[left] (0up) {};
% top 0-strata: 
\fill[color=green!60!black] (t4) circle (0.8pt) node[left] (0up) {};
\fill[color=green!60!black] (n3) circle (0.8pt) node[left] (0up) {};
%
% 2-strata: 
\draw[line width=1] (0.5, 0.5, 0.5) node[line width=0pt] (beta) {{\footnotesize $\beta$}};
%
% visible edges of cube
\draw[
	 color=gray, 
	 opacity=0.4, 
	 semithick
	 ] 
	 (0,1,1) -- (0,1,0) -- (1,1,0) -- (1,1,1) -- (0,1,1) -- (0,0,1) -- (1,0,1) -- (1,0,0) -- (1,1,0)
	 (1,0,1) -- (1,1,1);
\end{tikzpicture}
%%%%%%%%%%%%%%%%%%%%%% 
$$
2-morphisms also have an induced $\sta$-composition. 
Here we have to make a choice (which determines whether $\tz'$ will turn out to be cubical or opcubical). 
Indeed, given 1-morphisms $\alpha, \alpha': u \rightarrow v$ and $\beta, \beta': v \rightarrow w$ as well as 2-morphisms $X: \alpha \rightarrow \alpha'$ and $Y: \beta \rightarrow \beta'$, we define $Y \sta X$ as follows. 
We compress the contents of the cube~$X$ (respectively~$Y$) into its right (respectively left) half,
 %arXiv_v2: 
	i.\,e.\ the subset with $x$-coordinate larger (respectively smaller) than $1/2$, 
concatenate the resulting cubes in negative $y$-direction, and finally rescale to end up with a cube again: 
$$
%%%%%%%%%%%%%%%%%%%%%% 
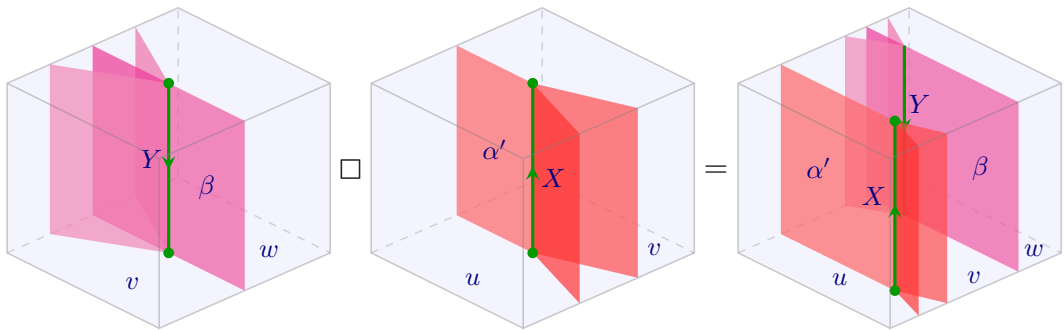
\begin{tikzpicture}[thick,scale=2.5,color=blue!50!black, baseline=0.0cm, >=stealth, 
				style={x={(-0.9cm,-0.4cm)},y={(0.8cm,-0.4cm)},z={(0cm,0.9cm)}}]
% 3-stratum: 
\fill [blue!20,opacity=0.2] (1,0,0) -- (1,1,0) -- (0,1,0) -- (0,1,1) -- (0,0,1) -- (1,0,1);
% invisible edges of cube: 
\draw[
	 color=gray, 
	 opacity=0.3, 
	 semithick,
	 dashed
	 ] 
	 (1,0,0) -- (0,0,0) -- (0,1,0)
	 (0,0,0) -- (0,0,1);
%%%
% bottom vertices: 
\coordinate (b1) at (0.25, 0, 0);
\coordinate (b2) at (0.5, 0, 0);
\coordinate (b3) at (0.75, 0, 0);
\coordinate (b4) at (0.5, 0.5, 0);
\coordinate (b5) at (0.5, 1, 0);
% top vertices: 
\coordinate (t1) at (0.25, 0, 1);
\coordinate (t2) at (0.5, 0, 1);
\coordinate (t3) at (0.75, 0, 1);
\coordinate (t4) at (0.5, 0.5, 1);
\coordinate (t5) at (0.5, 1, 1);
\fill [magenta!70,opacity=0.7] (b1) -- (b4) -- (t4) -- (t1);
\fill [magenta!90,opacity=0.7] (b2) -- (b4) -- (t4) -- (t2);
\fill [magenta!60,opacity=0.7] (b3) -- (b4) -- (t4) -- (t3);
\fill [magenta!80,opacity=0.7] (b5) -- (b4) -- (t4) -- (t5);
% X-lines: 
\draw[string=green!60!black, very thick] (t4) -- (b4);
%
% bottom 0-strata: 
\fill[color=green!60!black] (b4) circle (0.8pt) node[left] (0up) {};
% top 0-strata: 
\fill[color=green!60!black] (t4) circle (0.8pt) node[left] (0up) {};
%
% strata labels: 
\draw[line width=1] (0.8, 0.6, 0) node[line width=0pt] (beta) {{\footnotesize $v$}};
\draw[line width=1] (0.18, 0.8, 0) node[line width=0pt] (beta) {{\footnotesize $w$}};
\draw[line width=1] (0.5, 0.75, 0.5) node[line width=0pt] (beta) {{\footnotesize $\beta$}};
\draw[line width=1] (0.6, 0.5, 0.6) node[line width=0pt] (beta) {{\footnotesize $Y$}};
%
% visible edges of cube
\draw[
	 color=gray, 
	 opacity=0.4, 
	 semithick
	 ] 
	 (0,1,1) -- (0,1,0) -- (1,1,0) -- (1,1,1) -- (0,1,1) -- (0,0,1) -- (1,0,1) -- (1,0,0) -- (1,1,0)
	 (1,0,1) -- (1,1,1);
\end{tikzpicture}
%%%%%%%%%%%%%%%%%%%%%% 
\sta
%%%%%%%%%%%%%%%%%%%%%% 
\begin{tikzpicture}[thick,scale=2.5,color=blue!50!black, baseline=0.0cm, >=stealth, 
				style={x={(-0.9cm,-0.4cm)},y={(0.8cm,-0.4cm)},z={(0cm,0.9cm)}}]
% 3-stratum: 
\fill [blue!20,opacity=0.2] (1,0,0) -- (1,1,0) -- (0,1,0) -- (0,1,1) -- (0,0,1) -- (1,0,1);
% invisible edges of cube: 
\draw[
	 color=gray, 
	 opacity=0.3, 
	 semithick,
	 dashed
	 ] 
	 (1,0,0) -- (0,0,0) -- (0,1,0)
	 (0,0,0) -- (0,0,1);
%%%
% bottom vertices: 
\coordinate (a1) at (0.33, 1, 0);
\coordinate (a2) at (0.67, 1, 0);
\coordinate (a4) at (0.5, 0.5, 0);
\coordinate (a5) at (0.5, 0, 0);
% top vertices: 
\coordinate (c1) at (0.33, 1, 1);
\coordinate (c2) at (0.67, 1, 1);
\coordinate (c4) at (0.5, 0.5, 1);
\coordinate (c5) at (0.5, 0, 1);
\fill [red!70,opacity=0.7] (a1) -- (a4) -- (c4) -- (c1);
\fill [red!80,opacity=0.7] (a2) -- (a4) -- (c4) -- (c2);
\fill [red!60,opacity=0.7] (a5) -- (a4) -- (c4) -- (c5);
% X-lines: 
\draw[string=green!60!black, very thick] (a4) -- (c4);
%
% bottom 0-strata: 
\fill[color=green!60!black] (a4) circle (0.8pt) node[left] (0up) {};
% top 0-strata: 
\fill[color=green!60!black] (c4) circle (0.8pt) node[left] (0up) {};
%
% strata labels: 
\draw[line width=1] (0.85, 0.5, 0) node[line width=0pt] (beta) {{\footnotesize $u$}};
\draw[line width=1] (0.1, 0.85, 0) node[line width=0pt] (beta) {{\footnotesize $v$}};
\draw[line width=1] (0.5, 0.25, 0.5) node[line width=0pt] (beta) {{\footnotesize $\alpha'$}};
\draw[line width=1] (0.6, 0.75, 0.6) node[line width=0pt] (beta) {{\footnotesize $X$}};
%
% visible edges of cube
\draw[
	 color=gray, 
	 opacity=0.4, 
	 semithick
	 ] 
	 (0,1,1) -- (0,1,0) -- (1,1,0) -- (1,1,1) -- (0,1,1) -- (0,0,1) -- (1,0,1) -- (1,0,0) -- (1,1,0)
	 (1,0,1) -- (1,1,1);
\end{tikzpicture}
%%%%%%%%%%%%%%%%%%%%%% 
= 
%%%%%%%%%%%%%%%%%%%%%% 
\begin{tikzpicture}[thick,scale=2.5,color=blue!50!black, baseline=0.0cm, >=stealth, 
				style={x={(-0.9cm,-0.4cm)},y={(0.8cm,-0.4cm)},z={(0cm,0.9cm)}}]
% 3-stratum: 
\fill [blue!20,opacity=0.2] (1,0,0) -- (1,1,0) -- (0,1,0) -- (0,1,1) -- (0,0,1) -- (1,0,1);
% invisible edges of cube: 
\draw[
	 color=gray, 
	 opacity=0.3, 
	 semithick,
	 dashed
	 ] 
	 (1,0,0) -- (0,0,0) -- (0,1,0)
	 (0,0,0) -- (0,0,1);
%%%
%
% bottom vertices: 
\coordinate (b1) at (0.125, 0, 0);
\coordinate (b2) at (0.25, 0, 0);
\coordinate (b3) at (0.375, 0, 0);
\coordinate (b4) at (0.25, 0.25, 0);
\coordinate (b5) at (0.25, 1, 0);
% top vertices: 
\coordinate (t1) at (0.125, 0, 1);
\coordinate (t2) at (0.25, 0, 1);
\coordinate (t3) at (0.375, 0, 1);
\coordinate (t4) at (0.25, 0.25, 1);
\coordinate (t5) at (0.25, 1, 1);
\fill [magenta!70,opacity=0.7] (b1) -- (b4) -- (t4) -- (t1);
\fill [magenta!90,opacity=0.7] (b2) -- (b4) -- (t4) -- (t2);
\fill [magenta!60,opacity=0.7] (b3) -- (b4) -- (t4) -- (t3);
\fill [magenta!80,opacity=0.7] (b5) -- (b4) -- (t4) -- (t5);
%
% X-lines: 
\draw[string=green!60!black, very thick] (t4) -- (b4);
% bottom vertices: 
\coordinate (a1) at (0.6667, 1, 0);
\coordinate (a2) at (0.833, 1, 0);
\coordinate (a4) at (0.75, 0.75, 0);
\coordinate (a5) at (0.75, 0, 0);
% top vertices: 
\coordinate (c1) at (0.6667, 1, 1);
\coordinate (c2) at (0.833, 1, 1);
\coordinate (c4) at (0.75, 0.75, 1);
\coordinate (c5) at (0.75, 0, 1);
\fill [red!70,opacity=0.7] (a1) -- (a4) -- (c4) -- (c1);
\fill [red!80,opacity=0.7] (a2) -- (a4) -- (c4) -- (c2);
\fill [red!60,opacity=0.7] (a5) -- (a4) -- (c4) -- (c5);
%
%
% X-lines: 
\draw[string=green!60!black, very thick] (a4) -- (c4);
%
% bottom 0-strata: 
\fill[color=green!60!black] (a4) circle (0.8pt) node[left] (0up) {};
% top 0-strata: 
\fill[color=green!60!black] (c4) circle (0.8pt) node[left] (0up) {};
%
% strata labels: 
\draw[line width=1] (0.85, 0.5, 0) node[line width=0pt] (beta) {{\footnotesize $u$}};
\draw[line width=1] (0.42, 0.9, 0) node[line width=0pt] (beta) {{\footnotesize $v$}};
\draw[line width=1] (0.07, 0.9, 0) node[line width=0pt] (beta) {{\footnotesize $w$}};
\draw[line width=1] (0.75, 0.25, 0.5) node[line width=0pt] (beta) {{\footnotesize $\alpha'$}};
\draw[line width=1] (0.25, 0.75, 0.5) node[line width=0pt] (beta) {{\footnotesize $\beta$}};
\draw[line width=1] (0.65, 0.5, 0.4) node[line width=0pt] (beta) {{\footnotesize $X$}};
\draw[line width=1] (0.25, 0.35, 0.7) node[line width=0pt] (beta) {{\footnotesize $Y$}};
%
% visible edges of cube
\draw[
	 color=gray, 
	 opacity=0.4, 
	 semithick
	 ] 
	 (0,1,1) -- (0,1,0) -- (1,1,0) -- (1,1,1) -- (0,1,1) -- (0,0,1) -- (1,0,1) -- (1,0,0) -- (1,1,0)
	 (1,0,1) -- (1,1,1);
\end{tikzpicture}
%%%%%%%%%%%%%%%%%%%%%% 
$$

Thirdly, we have to provide compositions of 3-morphisms. 
These will be defined in terms of certain decorated 3d diagrams that we call `3d diagrams for~$\zz$'. 
Once we have established the nature of $\tz'$, we will identify such diagrams as Gray category diagrams for $\tz'$. 

As preparation for the precise definition, we note that in any progressive 3d diagram~$\Gamma$ as in Section~\ref{subsubsec:tricats}, each vertex~$\nu$ has a source $s_\nu$ and target $t_\nu$. 
%arXiv_v2
	%These are the bottom and top pre-2-category string diagrams of a small cube around~$\nu$, respectively.    
	These are defined by choosing a small cube around~$\nu$ whose faces are parallel to the faces of~$\Gamma$ and that does not contain any other vertex. The source and target of $\nu$ are by definition the bottom and top pre-2-category string diagrams of this cube, respectively.
Similarly, the source $s_\Gamma$ and target $t_\Gamma$ of the 3d diagram itself are given by the bottom and top of~$\Gamma$, respectively. 
We usually identify $s_\Gamma$ and $t_\Gamma$ with their cylinders $s_\Gamma \times [0,1]$ and $t_\Gamma \times [0,1]$; hence we can view them as 2-morphisms in $\tz'$ if~$\Gamma$ is decorated by~$\D$. 

\begin{definition}
A \textsl{3d diagram for $\zz: \Bordd \rightarrow \Vect_\Bbbk$} is a (progressive) 3d diagram with $j$-strata decorated by elements of $D_j$ as allowed by the maps $s,t,f$ for $j \in \{ 1,2,3 \}$, and vertices~$\nu$ decorated by 
%arXiv_v2: 
	%compatible 3-morphisms in $\tz'$. 
	compatible 3-morphisms in $\tz'$, i.\,e.\ by elements of $\zz(S_{s_\nu,t_\nu})$. 
\end{definition}

Next we define the notion of \textsl{evaluation $\zz(\Gamma)$ of a 3d diagram~$\Gamma$ for~$\zz$}, which will be crucial in the following. 
%arXiv_v2: 
	%For this we view~$\Gamma$ as a stratified bordism with boundary $\partial\Gamma = s_\Gamma^{\textrm{rev}} \sqcup t_\Gamma$, and we write $\Gamma^{\textrm{cut}}$ for the stratified bordism obtained from~$\Gamma$ by removing a small ball (or cube) $B_\nu$ around every vertex~$\nu$ in~$\Gamma$. 
	For this we view~$\Gamma$ as a defect ball, and we write $\Gamma^{\textrm{cut}}$ for the defect bordism obtained by removing a small ball (or cube) $B_\nu$ around every vertex~$\nu$ in~$\Gamma$. 
Viewing $\Gamma^{\textrm{cut}}$ as a bordism 
%arXiv_v2: 
	%$\bigsqcup_\nu \partial B_\nu \rightarrow \partial\Gamma$, 
	$\bigsqcup_\nu \partial B_\nu \to S_{s_\Gamma,t_\Gamma}$, 
we obtain a linear map 
%arXiv_v2: 
	%$\zz(\Gamma^{\textrm{cut}}): \bigotimes_\nu \zz(\partial B_\nu) \rightarrow \zz(\partial \Gamma)$, 
	$\zz(\Gamma^{\textrm{cut}}): \bigotimes_\nu \zz(\partial B_\nu) \to \zz(S_{s_\Gamma,t_\Gamma})$, 
which we can in turn apply to the vectors 
%arXiv_v2: 
	%$\Phi_\nu \in \Hom_{\tz'}(s_\nu, t_\nu) = \zz(\partial B_\nu)$ 
	$\Phi_\nu \in \Hom_{\tz'}(s_\nu, t_\nu) = \zz(S_{s_\nu,t_\nu})$ 
decorating the vertices~$\nu$.  

\begin{definition} 
\label{def:eval}
Let $\Gamma$ be a 3d diagram for $\zz$ with source $s_\Gamma$ and target $t_\Gamma$ and with vertices decorated by 3-morphisms $\Phi_\nu$. Then the 
 \textsl{evaluation} of~$\Gamma$ is the 3-morphism 
$$
\zz(\Gamma) = \zz (\Gamma^{\textrm{cut}}) \big( \bigotimes_\nu \Phi_\nu \big) \in \Hom_{\tz'}( s_\Gamma, t_\Gamma) \, . 
$$
 \end{definition}

\medskip
In particular, the evaluation of 3d diagrams for $\zz$ allows us to define the identity 3-morphism  $1_X: X\to X$ for each 2-morphism $X$ in $\tz'$. 
This corresponds to a 3d diagram for $\zz$ with no vertices in the interior. More precisely,
the 3d diagram for the 2-morphism $X$ defines a bordism $\Sigma_X: \emptyset\to S_{X,X}$. Its evaluation $\zz(\Sigma_X):\zz(\emptyset)\to \zz(S_{X,X})$ is a linear map from~$\Bbbk$ to $\zz(S_{X,X})=\Hom_{\tz'}(X,X)$, 
hence it is determined by a vector in this space of 3-morphisms.

\medskip

Now we are in a position to complete the discussion of compositions in $\tz'$. 
In a nutshell, composing 3-morphisms in $\tz'$ amounts to suitably arranging small cubes (or balls) in a unit cube and then evaluating with~$\zz$. 
Given two vertically composable 3-morphisms $\Phi: X\to Y$ and $\Psi: Y\to Z$, we consider the associated  3d diagrams whose bottom faces describe the 2-morphisms $X$ and $Y$, respectively, and whose top faces correspond to~$Y$ and~$Z$. 
Hence, the diagram for $\Psi$ can be stacked on top of the diagram for $\Phi$, and after a rescaling this yields a new 3d diagram $\Gamma$. 
The composite $\Psi \circ \Phi$ is 
%arXiv_v2: 
	%defined 
	defined\footnote{It is straightforward to verify that this definition in terms of representatives is compatible with the definition of 3-morphisms in terms of limits in the defining equation \eqref{eq:3HomTZ}: Since the limit and colimit of the relevant diagram coincide, and since all cone morphisms are isomorphisms, the composition is implicitly defined by picking representatives, evaluating with~$\zz$, and then mapping back to the limit. It is straightforward to see that the result is independent of the choice of representatives. Analogous remarks apply to horizontal and Gray composition of 3-morphisms.} 
as the evaluation $\Psi\circ\Phi=\zz(\Gamma)$.
 
Similarly, if $X,Y$ and $X',Y'$ are pairs of $\otimes$-composable 2-morphisms, the horizontal composite of  3-morphisms $\Phi: X\to X'$ and $\Psi: Y\to Y'$ is obtained by stacking the 3d diagrams for $\Phi$ and $\Psi$ next to each other along the $y$-axis. After a rescaling, this yields a new 3d diagram $\Gamma'$, and the 3-morphism  $\Psi\otimes \Phi: Y\otimes X\to Y'\otimes X'$ is defined to be $\zz(\Gamma')$. 
 
Finally, for two pairs $X,Y$ and $X',Y'$  of $\sta$-composable 2-morphisms, the $\sta$-composite of two  3-morphisms $\Phi: X\to X'$ and $\Psi: Y\to Y'$ is defined as the evaluation of the 3d diagram obtained by stacking the diagrams for $\Phi$ and $\Psi$ along the $x$-axis, rescaling the resulting diagram and then evaluating it with $\zz$.

We illustrate the definition of $\circ$-composition with an example.  
For the 2-morphisms
$$
X = Y = 
%%%%%%%%%%%%%%%%%%%%%% 
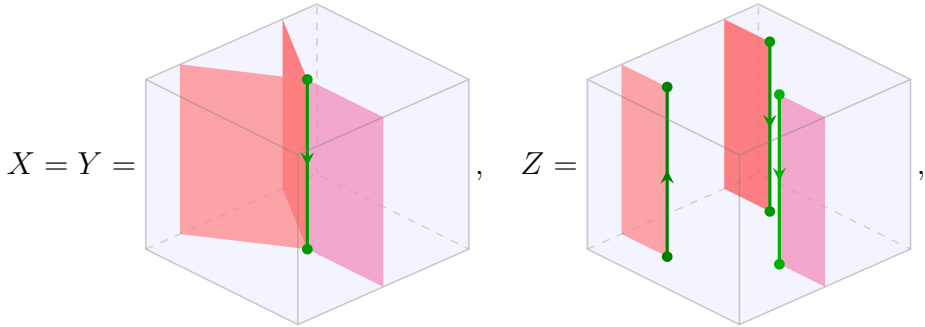
\begin{tikzpicture}[thick,scale=2.5,color=blue!50!black, baseline=0.0cm, >=stealth, 
				style={x={(-0.9cm,-0.4cm)},y={(0.8cm,-0.4cm)},z={(0cm,0.9cm)}}]
% 3-stratum: 
\fill [blue!20,opacity=0.2] (1,0,0) -- (1,1,0) -- (0,1,0) -- (0,1,1) -- (0,0,1) -- (1,0,1);
% invisible edges of cube: 
\draw[
	 color=gray, 
	 opacity=0.3, 
	 semithick,
	 dashed
	 ] 
	 (1,0,0) -- (0,0,0) -- (0,1,0)
	 (0,0,0) -- (0,0,1);
%%%
% bottom points: 
\coordinate (b1) at (0.2, 0, 0);
\coordinate (b2) at (0.8, 0, 0);
\coordinate (b3) at (0.5, 0.5, 0);
\coordinate (b4) at (0.5, 1, 0);
% top points: 
\coordinate (t1) at (0.2, 0, 1);
\coordinate (t2) at (0.8, 0, 1);
\coordinate (t3) at (0.5, 0.5, 1);
\coordinate (t4) at (0.5, 1, 1);
%%
%% planes: 
\fill [red!70,opacity=0.7] (b1) -- (b3) -- (t3) -- (t1);
\fill [red!50,opacity=0.7] (b2) -- (b3) -- (t3) -- (t2);
\fill [magenta!60,opacity=0.7] (b4) -- (b3) -- (t3) -- (t4);
% X-lines: 
\draw[string=green!60!black, very thick] (t3) -- (b3);
%
% bottom 0-stratum: 
\fill[color=green!60!black] (b3) circle (0.8pt) node[left] (0up) {};
% top 0-strata: 
\fill[color=green!60!black] (t3) circle (0.8pt) node[left] (0up) {};
%
% visible edges of cube
\draw[
	 color=gray, 
	 opacity=0.4, 
	 semithick
	 ] 
	 (0,1,1) -- (0,1,0) -- (1,1,0) -- (1,1,1) -- (0,1,1) -- (0,0,1) -- (1,0,1) -- (1,0,0) -- (1,1,0)
	 (1,0,1) -- (1,1,1);
\end{tikzpicture}
%%%%%%%%%%%%%%%%%%%%%%  
\, , \quad
Z = 
%%%%%%%%%%%%%%%%%%%%%% 
\begin{tikzpicture}[thick,scale=2.5,color=blue!50!black, baseline=0.0cm, >=stealth, 
				style={x={(-0.9cm,-0.4cm)},y={(0.8cm,-0.4cm)},z={(0cm,0.9cm)}}]
% 3-stratum: 
\fill [blue!20,opacity=0.2] (1,0,0) -- (1,1,0) -- (0,1,0) -- (0,1,1) -- (0,0,1) -- (1,0,1);
% invisible edges of cube: 
\draw[
	 color=gray, 
	 opacity=0.3, 
	 semithick,
	 dashed
	 ] 
	 (1,0,0) -- (0,0,0) -- (0,1,0)
	 (0,0,0) -- (0,0,1);
%%%
% bottom points: 
\coordinate (b1) at (0.2, 0, 0);
\coordinate (b2) at (0.8, 0, 0);
\coordinate (b31) at (0.2, 0.3, 0);
\coordinate (b32) at (0.8, 0.3, 0);
\coordinate (b33) at (0.5, 0.7, 0);
\coordinate (b4) at (0.5, 1, 0);
% top points: 
\coordinate (t1) at (0.2, 0, 1);
\coordinate (t2) at (0.8, 0, 1);
\coordinate (t31) at (0.2, 0.3, 1);
\coordinate (t32) at (0.8, 0.3, 1);
\coordinate (t33) at (0.5, 0.7, 1);
\coordinate (t4) at (0.5, 1, 1);
%%
%% planes: 
\fill [red!70,opacity=0.7] (b1) -- (b31) -- (t31) -- (t1);
\fill [red!50,opacity=0.7] (b2) -- (b32) -- (t32) -- (t2);
\fill [magenta!60,opacity=0.7] (b4) -- (b33) -- (t33) -- (t4);
% X-lines: 
\draw[string=green!60!black, very thick] (t31) -- (b31);
\draw[string=green!50!black, very thick] (b32) -- (t32);
\draw[string=green!70!black, very thick] (t33) -- (b33);
%
% bottom 0-stratum: 
\fill[color=green!60!black] (b31) circle (0.8pt) node[left] (0up) {};
\fill[color=green!50!black] (b32) circle (0.8pt) node[left] (0up) {};
\fill[color=green!70!black] (b33) circle (0.8pt) node[left] (0up) {};
% top 0-strata: 
\fill[color=green!60!black] (t31) circle (0.8pt) node[left] (0up) {};
\fill[color=green!50!black] (t32) circle (0.8pt) node[left] (0up) {};
\fill[color=green!70!black] (t33) circle (0.8pt) node[left] (0up) {};
%
% visible edges of cube
\draw[
	 color=gray, 
	 opacity=0.4, 
	 semithick
	 ] 
	 (0,1,1) -- (0,1,0) -- (1,1,0) -- (1,1,1) -- (0,1,1) -- (0,0,1) -- (1,0,1) -- (1,0,0) -- (1,1,0)
	 (1,0,1) -- (1,1,1);
\end{tikzpicture}
%%%%%%%%%%%%%%%%%%%%%%  
\, ,
$$
%arXiv_v2: 
	all with the same source and target 1-morphisms as depicted on the $x=1$ and $x=0$ faces of the cubes, respectively,
and the 3-morphisms $\Phi: X\to Y$, $\Psi: Y\to Z$ given by 
$$
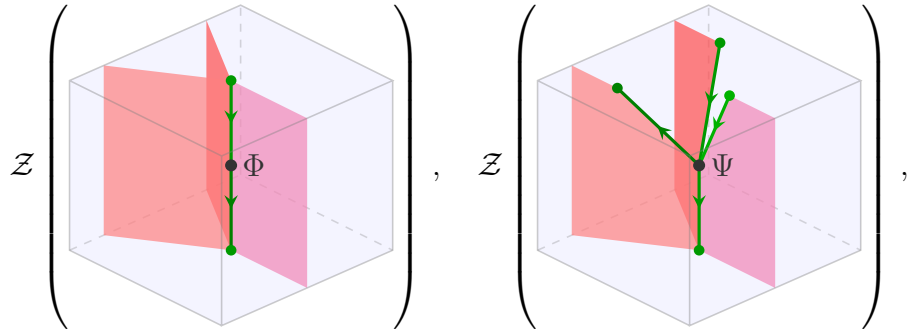
\zz \left( 
%%%%%%%%%%%%%%%%%%%%%% 
\begin{tikzpicture}[thick,scale=2.5,color=blue!50!black, baseline=0.0cm, >=stealth, 
				style={x={(-0.9cm,-0.4cm)},y={(0.8cm,-0.4cm)},z={(0cm,0.9cm)}}]
% 3-stratum: 
\fill [blue!20,opacity=0.2] (1,0,0) -- (1,1,0) -- (0,1,0) -- (0,1,1) -- (0,0,1) -- (1,0,1);
% invisible edges of cube: 
\draw[
	 color=gray, 
	 opacity=0.3, 
	 semithick,
	 dashed
	 ] 
	 (1,0,0) -- (0,0,0) -- (0,1,0)
	 (0,0,0) -- (0,0,1);
%%%
% bottom points: 
\coordinate (b1) at (0.2, 0, 0);
\coordinate (b2) at (0.8, 0, 0);
\coordinate (b3) at (0.5, 0.5, 0);
\coordinate (b4) at (0.5, 1, 0);
% top points: 
\coordinate (t1) at (0.2, 0, 1);
\coordinate (t2) at (0.8, 0, 1);
\coordinate (t31) at (0.5, 0.5, 1);
\coordinate (t32) at (0.5, 0.5, 1);
\coordinate (t33) at (0.5, 0.5, 1);
\coordinate (t4) at (0.5, 1, 1);
% Phi and Psi: 
\coordinate (Phi) at (0.5, 0.5, 0.5);
%%
%% planes: 
\fill [red!70,opacity=0.7] (b1) -- (b3) -- (Phi) -- (t31) -- (t1);
\fill [red!50,opacity=0.7] (b2) -- (b3) -- (Phi) -- (t32) -- (t2);
\fill [magenta!60,opacity=0.7] (b4) -- (b3) -- (Phi) -- (t33) -- (t4);
% X-lines: 
\draw[string=green!60!black, very thick] (t31) -- (Phi);
\draw[string=green!50!black, very thick] (Phi) -- (b3);
%
% bottom 0-stratum: 
\fill[color=green!60!black] (b3) circle (0.8pt) node[left] (0up) {};
% top 0-strata: 
\fill[color=green!60!black] (t31) circle (0.8pt) node[left] (0up) {};
% Phi and Psi vertices
\fill[color=black!80] (Phi) circle (0.9pt) node[right] (0up) {$\footnotesize \Phi$};
%
% visible edges of cube
\draw[
	 color=gray, 
	 opacity=0.4, 
	 semithick
	 ] 
	 (0,1,1) -- (0,1,0) -- (1,1,0) -- (1,1,1) -- (0,1,1) -- (0,0,1) -- (1,0,1) -- (1,0,0) -- (1,1,0)
	 (1,0,1) -- (1,1,1);
\end{tikzpicture}
%%%%%%%%%%%%%%%%%%%%%%  
\right)
\, , \quad 
\zz \left(
%%%%%%%%%%%%%%%%%%%%%% 
\begin{tikzpicture}[thick,scale=2.5,color=blue!50!black, baseline=0.0cm, >=stealth, 
				style={x={(-0.9cm,-0.4cm)},y={(0.8cm,-0.4cm)},z={(0cm,0.9cm)}}]
% 3-stratum: 
\fill [blue!20,opacity=0.2] (1,0,0) -- (1,1,0) -- (0,1,0) -- (0,1,1) -- (0,0,1) -- (1,0,1);
% invisible edges of cube: 
\draw[
	 color=gray, 
	 opacity=0.3, 
	 semithick,
	 dashed
	 ] 
	 (1,0,0) -- (0,0,0) -- (0,1,0)
	 (0,0,0) -- (0,0,1);
%%%
% bottom points: 
\coordinate (b1) at (0.2, 0, 0);
\coordinate (b2) at (0.8, 0, 0);
\coordinate (b3) at (0.5, 0.5, 0);
\coordinate (b4) at (0.5, 1, 0);
% top points: 
\coordinate (t1) at (0.2, 0, 1);
\coordinate (t2) at (0.8, 0, 1);
\coordinate (t31) at (0.2, 0.3, 1);
\coordinate (t32) at (0.8, 0.3, 1);
\coordinate (t33) at (0.5, 0.7, 1);
\coordinate (t4) at (0.5, 1, 1);
% Phi and Psi: 
\coordinate (Psi) at (0.5, 0.5, 0.5);
%%
%% planes: 
\fill [red!70,opacity=0.7] (b1) -- (b3) -- (Psi) -- (t31) -- (t1);
\fill [red!50,opacity=0.7] (b2) -- (b3) -- (Psi) -- (t32) -- (t2);
\fill [magenta!60,opacity=0.7] (b4) -- (b3) -- (Psi) -- (t33) -- (t4);
% X-lines: 
\draw[string=green!60!black, very thick] (t31) -- (Psi);
\draw[string=green!50!black, very thick] (Psi) -- (t32);
\draw[string=green!70!black, very thick] (t33) -- (Psi);
\draw[string=green!60!black, very thick] (Psi) -- (b3);
%
% bottom 0-stratum: 
\fill[color=green!60!black] (b3) circle (0.8pt) node[left] (0up) {};
% top 0-strata: 
\fill[color=green!60!black] (t31) circle (0.8pt) node[left] (0up) {};
\fill[color=green!50!black] (t32) circle (0.8pt) node[left] (0up) {};
\fill[color=green!70!black] (t33) circle (0.8pt) node[left] (0up) {};
% Phi and Psi vertices
\fill[color=black!80] (Psi) circle (0.9pt) node[right] (0up) {$\footnotesize \Psi$};
%
% visible edges of cube
\draw[
	 color=gray, 
	 opacity=0.4, 
	 semithick
	 ] 
	 (0,1,1) -- (0,1,0) -- (1,1,0) -- (1,1,1) -- (0,1,1) -- (0,0,1) -- (1,0,1) -- (1,0,0) -- (1,1,0)
	 (1,0,1) -- (1,1,1);
\end{tikzpicture}
%%%%%%%%%%%%%%%%%%%%%%  
\right)
\, ,
$$
the 3d diagram for the 3-morphism  $\Psi\circ\Phi: X\to Z$ 
%and the associated defect bordism are
is given by
$$
\Psi \circ \Phi 
= 
\zz 
\left( 
%%%%%%%%%%%%%%%%%%%%%% 
\begin{tikzpicture}[thick,scale=2.5,color=blue!50!black, baseline=0.0cm, >=stealth, 
				style={x={(-0.9cm,-0.4cm)},y={(0.8cm,-0.4cm)},z={(0cm,0.9cm)}}]
% 3-stratum: 
\fill [blue!20,opacity=0.2] (1,0,0) -- (1,1,0) -- (0,1,0) -- (0,1,1) -- (0,0,1) -- (1,0,1);
% invisible edges of cube: 
\draw[
	 color=gray, 
	 opacity=0.3, 
	 semithick,
	 dashed
	 ] 
	 (1,0,0) -- (0,0,0) -- (0,1,0)
	 (0,0,0) -- (0,0,1);
%%%
% bottom points: 
\coordinate (b1) at (0.2, 0, 0);
\coordinate (b2) at (0.8, 0, 0);
\coordinate (b3) at (0.5, 0.5, 0);
\coordinate (b4) at (0.5, 1, 0);
% top points: 
\coordinate (t1) at (0.2, 0, 1);
\coordinate (t2) at (0.8, 0, 1);
\coordinate (t31) at (0.2, 0.3, 1);
\coordinate (t32) at (0.8, 0.3, 1);
\coordinate (t33) at (0.5, 0.7, 1);
\coordinate (t4) at (0.5, 1, 1);
% Phi and Psi: 
\coordinate (Phi) at (0.5, 0.5, 0.33);
\coordinate (Psi) at (0.5, 0.5, 0.67);
%%
%% planes: 
\fill [red!70,opacity=0.7] (b1) -- (b3) -- (Psi) -- (t31) -- (t1);
\fill [red!50,opacity=0.7] (b2) -- (b3) -- (Psi) -- (t32) -- (t2);
\fill [magenta!60,opacity=0.7] (b4) -- (b3) -- (Psi) -- (t33) -- (t4);
% X-lines: 
\draw[string=green!60!black, very thick] (t31) -- (Psi);
\draw[string=green!50!black, very thick] (Psi) -- (t32);
\draw[string=green!70!black, very thick] (t33) -- (Psi);
\draw[string=green!60!black, very thick] (Psi) -- (Phi);
\draw[string=green!60!black, very thick] (Phi) -- (b3);
%
% bottom 0-stratum: 
\fill[color=green!60!black] (b3) circle (0.8pt) node[left] (0up) {};
% top 0-strata: 
\fill[color=green!60!black] (t31) circle (0.8pt) node[left] (0up) {};
\fill[color=green!50!black] (t32) circle (0.8pt) node[left] (0up) {};
\fill[color=green!70!black] (t33) circle (0.8pt) node[left] (0up) {};
% Phi and Psi vertices
\fill[color=black!80] (Psi) circle (0.9pt) node[right] (0up) {$\footnotesize \Psi$};
\fill[color=black!80] (Phi) circle (0.9pt) node[right] (0up) {$\footnotesize \Phi$};
%
% visible edges of cube
\draw[
	 color=gray, 
	 opacity=0.4, 
	 semithick
	 ] 
	 (0,1,1) -- (0,1,0) -- (1,1,0) -- (1,1,1) -- (0,1,1) -- (0,0,1) -- (1,0,1) -- (1,0,0) -- (1,1,0)
	 (1,0,1) -- (1,1,1);
\end{tikzpicture}
%%%%%%%%%%%%%%%%%%%%%%  
\right) 
=
\zz \left( 
%%%%%%%%%%%%%%%%%%%%%% 
\begin{tikzpicture}[thick,scale=2.5,color=blue!50!black, baseline=0.0cm, >=stealth, 
				style={x={(-0.9cm,-0.4cm)},y={(0.8cm,-0.4cm)},z={(0cm,0.9cm)}}]
% 3-stratum: 
\fill [blue!20,opacity=0.2] (1,0,0) -- (1,1,0) -- (0,1,0) -- (0,1,1) -- (0,0,1) -- (1,0,1);
% invisible edges of cube: 
\draw[
	 color=gray, 
	 opacity=0.3, 
	 semithick,
	 dashed
	 ] 
	 (1,0,0) -- (0,0,0) -- (0,1,0)
	 (0,0,0) -- (0,0,1);
%%%
% bottom points: 
\coordinate (b1) at (0.2, 0, 0);
\coordinate (b2) at (0.8, 0, 0);
\coordinate (b3) at (0.5, 0.5, 0);
\coordinate (b4) at (0.5, 1, 0);
% top points: 
\coordinate (t1) at (0.2, 0, 1);
\coordinate (t2) at (0.8, 0, 1);
\coordinate (t31) at (0.4, 0.4, 1);
\coordinate (t32) at (0.6, 0.4, 1);
\coordinate (t33) at (0.5, 0.6, 1);
\coordinate (t4) at (0.5, 1, 1);
% bottom point of lower cube: 
\coordinate (LCb) at (0.5, 0.5, 0.205);
% top point of lower cube: 
\coordinate (LCt) at (0.5, 0.5, 0.455);
% side points of lower cube: 
\coordinate (LCt1TL) at (0.5, 0.375, 0.455);
\coordinate (LCt1BL) at (0.5, 0.375, 0.205);
\coordinate (LCt1TR) at (0.5, 0.625, 0.455);
\coordinate (LCt1BR) at (0.5, 0.625, 0.205);
% bottom point of upper cube: 
\coordinate (UCb) at (0.5, 0.5, 0.545);
% top points of upper cube: 
\coordinate (UCt1) at (0.4, 0.4, 0.795);
\coordinate (UCt2) at (0.6, 0.4, 0.795);
\coordinate (UCt3) at (0.5, 0.6, 0.795);
% side points of upper cube: 
\coordinate (UCt1TL) at (0.4, 0.375, 0.795);
\coordinate (UCt1BL) at (0.4, 0.375, 0.545);
\coordinate (UCt2TL) at (0.6, 0.375, 0.795);
\coordinate (UCt2BL) at (0.6, 0.375, 0.545);
\coordinate (UCt3TL) at (0.5, 0.625, 0.795);
\coordinate (UCt3BL) at (0.5, 0.625, 0.545);
%%
%% planes: 
\fill [red!70,opacity=0.7] 
(b1) -- (b3) -- (LCb) -- (LCt1BL) -- (LCt1TL) -- (LCt) -- (UCb) -- (UCt1BL) -- (UCt1TL) -- (UCt1) -- (t31) -- (t1);
%(b1) -- (b3) -- (UCb) -- (UCt1) -- (t31) -- (t1);
\fill [red!50,opacity=0.7] 
(b2) -- (b3) -- (LCb) -- (LCt1BL) -- (LCt1TL) -- (LCt) -- (UCb) -- (UCt2BL) -- (UCt2TL) -- (UCt2) -- (t32) -- (t2);
%(b2) -- (b3) -- (UCb) -- (UCt2) -- (t32) -- (t2);
\fill [magenta!60,opacity=0.7] (b4) -- (b3) -- (LCb) -- (LCt1BR) -- (LCt1TR) -- (LCt) -- (UCb) -- (UCt3BL) -- (UCt3TL) -- (UCt3) -- (t33) -- (t4);
% X-lines: 
\draw[color=green!60!black, postaction={decorate}, decoration={markings,mark=at position .71 with {\arrow[color=green!60!black]{>}}}, very thick] (t31) -- (UCt1);
\draw[color=green!50!black, postaction={decorate}, decoration={markings,mark=at position .71 with {\arrow[color=green!50!black]{>}}}, very thick] (UCt2) -- (t32);
\draw[color=green!60!black, very thick] (UCb) -- (LCt);
\draw[color=green!60!black, postaction={decorate}, decoration={markings,mark=at position .71 with {\arrow[color=green!60!black]{>}}}, very thick] (LCb) -- (b3);
%
% 0-strata on small cubes: 
\fill[color=green!60!black] (LCb) circle (0.8pt) node[left] (0up) {};
% lower cube: 
\fill [white,opacity=0.7] ($(0.25,0,0) + (0.375, 0.375, 0.205)$) -- ($(0.25,0.25,0) + (0.375, 0.375, 0.205)$) -- ($(0,0.25,0) + (0.375, 0.375, 0.205)$) -- ($(0,0.25,0.25) + (0.375, 0.375, 0.205)$) -- ($(0,0,0.25) + (0.375, 0.375, 0.205)$) -- ($(0.25,0,0.25) + (0.375, 0.375, 0.205)$);
% 0-strata on small cubes: 
\fill[color=green!60!black] (UCb) circle (0.8pt) node[left] (0up) {};
\fill[color=green!60!black] (LCt) circle (0.8pt) node[left] (0up) {};
% upper cube: 
\fill [white,opacity=0.7] ($(0.25,0,0) + (0.375, 0.375, 0.545)$) -- ($(0.25,0.25,0) + (0.375, 0.375, 0.545)$) -- ($(0,0.25,0) + (0.375, 0.375, 0.545)$) -- ($(0,0.25,0.25) + (0.375, 0.375, 0.545)$) -- ($(0,0,0.25) + (0.375, 0.375, 0.545)$) -- ($(0.25,0,0.25) + (0.375, 0.375, 0.545)$);
% bottom 0-stratum: 
\fill[color=green!60!black] (b3) circle (0.8pt) node[left] (0up) {};
% top 0-strata: 
\fill[color=green!60!black] (t31) circle (0.8pt) node[left] (0up) {};
\fill[color=green!50!black] (t32) circle (0.8pt) node[left] (0up) {};
\fill[color=green!70!black] (t33) circle (0.8pt) node[left] (0up) {};
% 0-strata on small cubes: 
\fill[color=green!60!black] (UCt1) circle (0.8pt) node[left] (0up) {};
\fill[color=green!50!black] (UCt2) circle (0.8pt) node[left] (0up) {};
\fill[color=green!70!black] (UCt3) circle (0.8pt) node[left] (0up) {};
%
% X-line
\draw[color=green!70!black, postaction={decorate}, decoration={markings,mark=at position .71 with {\arrow[color=green!70!black]{>}}}, very thick] (t33) -- (UCt3);
%
% visible edges of cube
\draw[
	 color=gray, 
	 opacity=0.4, 
	 semithick
	 ] 
	 (0,1,1) -- (0,1,0) -- (1,1,0) -- (1,1,1) -- (0,1,1) -- (0,0,1) -- (1,0,1) -- (1,0,0) -- (1,1,0)
	 (1,0,1) -- (1,1,1);
\end{tikzpicture}
%%%%%%%%%%%%%%%%%%%%%%  
\right) 
\!(\Psi \otimes_\Bbbk \Phi) \, . 
$$
Here, the first cube is viewed as a 3d diagram for~$\zz$, while the second cube is 
%arXiv_v2: 
	%viewed as a defect bordism. 
	viewed as a defect bordism, namely a defect ball with two defect balls (corresponding to~$\Phi$ and~$\Psi$) removed, cf.\ the discussion after \eqref{eq:SXXprimesphere}. 
Note also that the definitions of vertical and horizontal composition of 3-morphisms in $\tz'$ have an analogue in 2-dimensional defect TQFT (cf.~\cite[Fig.\,6\,b\,\&\,c]{dkr1107.0495}), and the construction should generalise to higher-dimensional TQFTs as well. 

The only ingredient missing to state our first main result is the tensorator. 
This is straightforward: for 2-morphisms $Y: \alpha \rightarrow \alpha'$ in $\tz'(u,v)$ and $X: \beta \rightarrow \beta'$ in $\tz'(v,w)$, we define the 3-morphism  
$$
\sigma_{X,Y}: 
%%%%%%%%%%%%%%%%%%%%%% 
\begin{tikzpicture}[thick,scale=2.5,color=blue!50!black, baseline=0.0cm, >=stealth, 
				style={x={(-0.9cm,-0.4cm)},y={(0.8cm,-0.4cm)},z={(0cm,0.9cm)}}]
% invisible edges of cube: 
\draw[
	 color=gray, 
	 opacity=0.3, 
	 semithick,
	 dashed
	 ] 
	 (1,0,0) -- (0,0,0) -- (0,1,0)
	 (0,0,0) -- (0,0,1);
%%%
% for 3-strata: 
\coordinate (d) at (-0.34, 0, 0);
\coordinate (d2) at (-0.67, 0, 0);
% 3-stratum v at back: 
\fill [blue!40,opacity=0.1] ($(d2) + (1,0,0)$) -- ($(d2) + (1,1,0)$) -- ($(d2) + (0.67,1,0)$) -- ($(d2) + (0.67,1,1)$) -- ($(d2) + (0.67,0,1)$) -- ($(d2) + (1,0,1)$);
%
% lower endpoint of X: 
\coordinate (X0) at (0.33, 0.25, 0);
% upper endpoint of X: 
\coordinate (X1) at (0.33, 0.25, 1);
% lower endpoint of Y: 
\coordinate (Y0) at (0.67, 0.75, 0);
% upper endpoint of Y: 
\coordinate (Y1) at (0.67, 0.75, 1);
%
%
% beta'-plane: 
\fill [magenta!50,opacity=0.7] (0.33, 0, 0) -- (X0) -- (X1) -- (0.33, 1, 1) -- (0.33, 0, 1);
% beta-plane:
\fill [magenta!70,opacity=0.7] (X0) -- (X1) -- (0.33, 1, 1) -- (0.33, 1, 0) -- (0.33, 0.2, 0);
% X-line:
\draw[string=green!50!black, ultra thick] (X0) -- node[pos=0.82, color=blue!50!black, right] {$\footnotesize X$} (X1);  
% beta label: 
\draw[line width=1] (0.4, 0.99, 0.5) node[line width=0pt] (beta) {{\footnotesize $\beta$}};
% beta' label: 
\draw[line width=1] (0.38, 0.18, 0.88) node[line width=0pt] (beta) {{\footnotesize $\beta'$}};
%
% 3-stratum in middle: 
\fill [blue!20,opacity=0.1] ($(d) + (1,0,0)$) -- ($(d) + (1,1,0)$) -- ($(d) + (0.67,1,0)$) -- ($(d) + (0.67,1,1)$) -- ($(d) + (0.67,0,1)$) -- ($(d) + (1,0,1)$);
%
% 3-stratum u in front: 
\fill [blue!20,opacity=0.2] (1,0,0) -- (1,1,0) -- (0.67,1,0) -- (0.67,1,1) -- (0.67,0,1) -- (1,0,1);
%
% alpha'-plane: 
\fill [red!70,opacity=0.7] (0.67, 0, 0) -- (Y0) -- (Y1) -- (0.67, 0, 1);
% alpha-plane: 
\fill [red!30,opacity=0.7] (Y1) -- (0.67, 1, 1) -- (0.67, 1, 0) -- (0.67, 0.2, 0) -- (Y0);
% Y-line:
\draw[string=green!80!black, ultra thick] (Y0) -- node[pos=0.25, color=blue!50!black, left] {$\footnotesize Y$} (Y1);  
% alpha label: 
\draw[line width=1] (0.73, 0.99, 0.5) node[line width=0pt] (beta) {{\footnotesize $\alpha$}};
% alpha' label: 
\draw[line width=1] (0.73, 0.2, 0.5) node[line width=0pt] (beta) {{\footnotesize $\alpha'$}};
%
%
% object labels: 
\draw[line width=1] (0.85, 0.5, 0) node[line width=0pt] (beta) {{\footnotesize $u$}};
\draw[line width=1] (0.45, 0.9, 0) node[line width=0pt] (beta) {{\footnotesize $v$}};
\draw[line width=1] (0.12, 0.9, 0) node[line width=0pt] (beta) {{\footnotesize $w$}};
%
% visible edges of cube
\draw[
	 color=gray, 
	 opacity=0.4, 
	 semithick
	 ] 
	 (0,1,1) -- (0,1,0) -- (1,1,0) -- (1,1,1) -- (0,1,1) -- (0,0,1) -- (1,0,1) -- (1,0,0) -- (1,1,0)
	 (1,0,1) -- (1,1,1);
\end{tikzpicture}
%%%%%%%%%%%%%%%%%%%%%%
\lra
%%%%%%%%%%%%%%%%%%%%%% 
\begin{tikzpicture}[thick,scale=2.5,color=blue!50!black, baseline=0.0cm, >=stealth, 
				style={x={(-0.9cm,-0.4cm)},y={(0.8cm,-0.4cm)},z={(0cm,0.9cm)}}]
% invisible edges of cube: 
\draw[
	 color=gray, 
	 opacity=0.3, 
	 semithick,
	 dashed
	 ] 
	 (1,0,0) -- (0,0,0) -- (0,1,0)
	 (0,0,0) -- (0,0,1);
%%%
% for 3-strata: 
\coordinate (d) at (-0.34, 0, 0);
\coordinate (d2) at (-0.67, 0, 0);
% 3-stratum v at back: 
\fill [blue!40,opacity=0.1] ($(d2) + (1,0,0)$) -- ($(d2) + (1,1,0)$) -- ($(d2) + (0.67,1,0)$) -- ($(d2) + (0.67,1,1)$) -- ($(d2) + (0.67,0,1)$) -- ($(d2) + (1,0,1)$);
%
% lower endpoint of X: 
\coordinate (X0) at (0.33, 0.75, 0);
% upper endpoint of X: 
\coordinate (X1) at (0.33, 0.75, 1);
% lower endpoint of Y: 
\coordinate (Y0) at (0.67, 0.25, 0);
% upper endpoint of Y: 
\coordinate (Y1) at (0.67, 0.25, 1);
%
%
% beta'-plane: 
\fill [magenta!50,opacity=0.7] (0.33, 0, 0) -- (X0) -- (X1) -- (0.33, 1, 1) -- (0.33, 0, 1);
% beta-plane:
\fill [magenta!70,opacity=0.7] (X0) -- (X1) -- (0.33, 1, 1) -- (0.33, 1, 0) -- (0.33, 0.2, 0);
% X-line:
\draw[string=green!50!black, ultra thick] (X0) -- node[pos=0.82, color=blue!50!black, right] {$\footnotesize X$} (X1);  
% beta label: 
\draw[line width=1] (0.4, 0.99, 0.5) node[line width=0pt] (beta) {{\footnotesize $\beta$}};
% beta' label: 
\draw[line width=1] (0.38, 0.18, 0.88) node[line width=0pt] (beta) {{\footnotesize $\beta'$}};
%
% 3-stratum in middle: 
\fill [blue!20,opacity=0.1] ($(d) + (1,0,0)$) -- ($(d) + (1,1,0)$) -- ($(d) + (0.67,1,0)$) -- ($(d) + (0.67,1,1)$) -- ($(d) + (0.67,0,1)$) -- ($(d) + (1,0,1)$);
%
% 3-stratum u in front: 
\fill [blue!20,opacity=0.2] (1,0,0) -- (1,1,0) -- (0.67,1,0) -- (0.67,1,1) -- (0.67,0,1) -- (1,0,1);
%
% alpha'-plane: 
\fill [red!70,opacity=0.7] (0.67, 0, 0) -- (Y0) -- (Y1) -- (0.67, 0, 1);
% alpha-plane: 
\fill [red!30,opacity=0.7] (Y1) -- (0.67, 1, 1) -- (0.67, 1, 0) -- (0.67, 0.2, 0) -- (Y0);
% Y-line:
\draw[string=green!80!black, ultra thick] (Y0) -- node[pos=0.2, color=blue!50!black, right] {$\footnotesize Y$} (Y1);  
% alpha label: 
\draw[line width=1] (0.73, 0.99, 0.5) node[line width=0pt] (beta) {{\footnotesize $\alpha$}};
% alpha' label: 
\draw[line width=1] (0.73, 0.2, 0.5) node[line width=0pt] (beta) {{\footnotesize $\alpha'$}};
%
%
% object labels: 
\draw[line width=1] (0.85, 0.5, 0) node[line width=0pt] (beta) {{\footnotesize $u$}};
\draw[line width=1] (0.45, 0.9, 0) node[line width=0pt] (beta) {{\footnotesize $v$}};
\draw[line width=1] (0.12, 0.9, 0) node[line width=0pt] (beta) {{\footnotesize $w$}};
%
% visible edges of cube
\draw[
	 color=gray, 
	 opacity=0.4, 
	 semithick
	 ] 
	 (0,1,1) -- (0,1,0) -- (1,1,0) -- (1,1,1) -- (0,1,1) -- (0,0,1) -- (1,0,1) -- (1,0,0) -- (1,1,0)
	 (1,0,1) -- (1,1,1);
\end{tikzpicture}
%%%%%%%%%%%%%%%%%%%%%% 
$$
as the evaluation of a 3-cube with an $X$-$Y$-crossing: 
\be\label{eq:Ztensorator}
\sigma_{X,Y} = 
\zz \left(
%%%%%%%%%%%%%%%%%%%%%% 
\begin{tikzpicture}[thick,scale=2.5,color=blue!50!black, baseline=0.0cm, >=stealth, 
				style={x={(-0.9cm,-0.4cm)},y={(0.8cm,-0.4cm)},z={(0cm,0.9cm)}}]
% invisible edges of cube: 
\draw[
	 color=gray, 
	 opacity=0.3, 
	 semithick,
	 dashed
	 ] 
	 (1,0,0) -- (0,0,0) -- (0,1,0)
	 (0,0,0) -- (0,0,1);
%%%
% for 3-strata: 
\coordinate (d) at (-0.34, 0, 0);
\coordinate (d2) at (-0.67, 0, 0);
% 3-stratum v at back: 
\fill [blue!40,opacity=0.1] ($(d2) + (1,0,0)$) -- ($(d2) + (1,1,0)$) -- ($(d2) + (0.67,1,0)$) -- ($(d2) + (0.67,1,1)$) -- ($(d2) + (0.67,0,1)$) -- ($(d2) + (1,0,1)$);
%
% lower endpoint of X: 
\coordinate (X0) at (0.33, 0.1, 0);
% upper endpoint of X: 
\coordinate (X1) at (0.33, 0.9, 1);
% lower endpoint of Y: 
\coordinate (Y0) at (0.67, 0.9, 0);
% upper endpoint of Y: 
\coordinate (Y1) at (0.67, 0.1, 1);
%
%
% beta'-plane: 
\fill [magenta!50,opacity=0.7] (0.33, 0, 0) -- (X0) -- (X1) -- (0.33, 1, 1) -- (0.33, 0, 1);
% beta-plane:
\fill [magenta!70,opacity=0.7] (X0) -- (X1) -- (0.33, 1, 1) -- (0.33, 1, 0) -- (0.33, 0.2, 0);
% X-line:
\draw[string=green!50!black, ultra thick] (X0) -- node[pos=0.9, color=blue!50!black, below] {$\footnotesize X$} (X1);  
% beta label: 
\draw[line width=1] (0.4, 0.99, 0.5) node[line width=0pt] (beta) {{\footnotesize $\beta$}};
% beta' label: 
\draw[line width=1] (0.4, 0.2, 0.92) node[line width=0pt] (beta) {{\footnotesize $\beta'$}};
%
% 3-stratum in middle: 
\fill [blue!20,opacity=0.1] ($(d) + (1,0,0)$) -- ($(d) + (1,1,0)$) -- ($(d) + (0.67,1,0)$) -- ($(d) + (0.67,1,1)$) -- ($(d) + (0.67,0,1)$) -- ($(d) + (1,0,1)$);
%
% 3-stratum u in front: 
\fill [blue!20,opacity=0.2] (1,0,0) -- (1,1,0) -- (0.67,1,0) -- (0.67,1,1) -- (0.67,0,1) -- (1,0,1);
%
% alpha'-plane: 
\fill [red!70,opacity=0.7] (0.67, 0, 0) -- (Y0) -- (Y1) -- (0.67, 0, 1);
% alpha-plane: 
\fill [red!30,opacity=0.7] (Y1) -- (0.67, 1, 1) -- (0.67, 1, 0) -- (0.67, 0.2, 0) -- (Y0);
% Y-line:
\draw[string=green!80!black, ultra thick] (Y0) -- node[pos=0.2, color=blue!50!black, right] {$\footnotesize Y$} (Y1);  
% alpha label: 
\draw[line width=1] (0.73, 0.99, 0.5) node[line width=0pt] (beta) {{\footnotesize $\alpha$}};
% alpha' label: 
\draw[line width=1] (0.73, 0.2, 0.5) node[line width=0pt] (beta) {{\footnotesize $\alpha'$}};
%
%
% object labels: 
\draw[line width=1] (0.85, 0.5, 0) node[line width=0pt] (beta) {{\footnotesize $u$}};
\draw[line width=1] (0.45, 0.9, 0) node[line width=0pt] (beta) {{\footnotesize $v$}};
\draw[line width=1] (0.12, 0.9, 0) node[line width=0pt] (beta) {{\footnotesize $w$}};
%
% visible edges of cube
\draw[
	 color=gray, 
	 opacity=0.4, 
	 semithick
	 ] 
	 (0,1,1) -- (0,1,0) -- (1,1,0) -- (1,1,1) -- (0,1,1) -- (0,0,1) -- (1,0,1) -- (1,0,0) -- (1,1,0)
	 (1,0,1) -- (1,1,1);
\end{tikzpicture}
%%%%%%%%%%%%%%%%%%%%%%
\right)
(1) \, , 
\ee
where the argument of~$\zz$ is viewed as a bordism from the empty set to a decorated sphere. 

\medskip
\begin{theorem}
\label{thm:TZGray}
$\tz'$ as defined above is a $\Bbbk$-linear Gray category 
for every 3-dimensional defect TQFT~$\zz$. 
\end{theorem}

\begin{proof}
We verify that $\tz'$ has all the data and satisfies the axioms of a Gray category as in Definition~\ref{def:Graycat}: 
\begin{enumerate}
\item
The set of objects is $D_3$. 
\item 
$\tz'(u,v)$ is a 2-category for all $u,v \in D_3$ as by construction horizontal composition is strictly associative and unital. 
\item The identity 3-morphisms on a 2-morphism were introduced after Definition~\ref{def:eval}. 
%arXiv_v2: 
	%The identity 2-morphisms on a 1-morphism were given in Section~\ref{subsubsec:bicats}, and the identity 1-morphism  $1_u \in \tz(u,u)$ on an object $u$ is a $u$-decorated cube without strata, viewed as a 1-morphism in $\tz$. 
	The identity 2-morphisms on a 1-morphism and the identity 1-morphisms on an object were described in the second paragraph of Section~\ref{subsec:tricatfromZ}. 
\item 
$\sta$-composition is strictly associative and unital due to isotopy invariance and functoriality of~$\zz$. 
\item 
The tensorator \eqref{eq:Ztensorator} is a 3-morphism with inverse
$$
\sigma^{-1}_{X,Y} = 
\zz \left( % 
%%%%%%%%%%%%%%%%%%%%%% 
\begin{tikzpicture}[thick,scale=2.5,color=blue!50!black, baseline=0.0cm, >=stealth, 
				style={x={(-0.9cm,-0.4cm)},y={(0.8cm,-0.4cm)},z={(0cm,0.9cm)}}]
% invisible edges of cube: 
\draw[
	 color=gray, 
	 opacity=0.3, 
	 semithick,
	 dashed
	 ] 
	 (1,0,0) -- (0,0,0) -- (0,1,0)
	 (0,0,0) -- (0,0,1);
%%%
% for 3-strata: 
\coordinate (d) at (-0.34, 0, 0);
\coordinate (d2) at (-0.67, 0, 0);
% 3-stratum v at back: 
\fill [blue!40,opacity=0.1] ($(d2) + (1,0,0)$) -- ($(d2) + (1,1,0)$) -- ($(d2) + (0.67,1,0)$) -- ($(d2) + (0.67,1,1)$) -- ($(d2) + (0.67,0,1)$) -- ($(d2) + (1,0,1)$);
%
% lower endpoint of X: 
\coordinate (X0) at (0.33, 0.9, 0);
% upper endpoint of X: 
\coordinate (X1) at (0.33, 0.1, 1);
% lower endpoint of Y: 
\coordinate (Y0) at (0.67, 0.1, 0);
% upper endpoint of Y: 
\coordinate (Y1) at (0.67, 0.9, 1);
%
%
% beta'-plane: 
\fill [magenta!50,opacity=0.7] (0.33, 0, 0) -- (X0) -- (X1) -- (0.33, 1, 1) -- (0.33, 0, 1);
% beta-plane:
\fill [magenta!70,opacity=0.7] (X0) -- (X1) -- (0.33, 1, 1) -- (0.33, 1, 0) -- (0.33, 0.2, 0);
% X-line:
\draw[string=green!50!black, ultra thick] (X0) -- node[pos=0.82, color=blue!50!black, right] {$\footnotesize X$} (X1);  
% beta label: 
\draw[line width=1] (0.4, 0.99, 0.5) node[line width=0pt] (beta) {{\footnotesize $\beta$}};
% beta' label: 
\draw[line width=1] (0.4, 0.16, 0.86) node[line width=0pt] (beta) {{\footnotesize $\beta'$}};
%
% 3-stratum in middle: 
\fill [blue!20,opacity=0.1] ($(d) + (1,0,0)$) -- ($(d) + (1,1,0)$) -- ($(d) + (0.67,1,0)$) -- ($(d) + (0.67,1,1)$) -- ($(d) + (0.67,0,1)$) -- ($(d) + (1,0,1)$);
%
% 3-stratum u in front: 
\fill [blue!20,opacity=0.2] (1,0,0) -- (1,1,0) -- (0.67,1,0) -- (0.67,1,1) -- (0.67,0,1) -- (1,0,1);
%
% alpha'-plane: 
\fill [red!70,opacity=0.7] (0.67, 0, 0) -- (Y0) -- (Y1) -- (0.67, 0, 1);
% alpha-plane: 
\fill [red!30,opacity=0.7] (Y1) -- (0.67, 1, 1) -- (0.67, 1, 0) -- (0.67, 0.2, 0) -- (Y0);
% Y-line:
\draw[string=green!80!black, ultra thick] (Y0) -- node[pos=0.2, color=blue!50!black, right] {$\footnotesize Y$} (Y1);  
% alpha label: 
\draw[line width=1] (0.73, 0.99, 0.5) node[line width=0pt] (beta) {{\footnotesize $\alpha$}};
% alpha' label: 
\draw[line width=1] (0.73, 0.2, 0.5) node[line width=0pt] (beta) {{\footnotesize $\alpha'$}};
%
%
% object labels: 
\draw[line width=1] (0.85, 0.5, 0) node[line width=0pt] (beta) {{\footnotesize $u$}};
\draw[line width=1] (0.45, 0.9, 0) node[line width=0pt] (beta) {{\footnotesize $v$}};
\draw[line width=1] (0.12, 0.9, 0) node[line width=0pt] (beta) {{\footnotesize $w$}};
%
% visible edges of cube
\draw[
	 color=gray, 
	 opacity=0.4, 
	 semithick
	 ] 
	 (0,1,1) -- (0,1,0) -- (1,1,0) -- (1,1,1) -- (0,1,1) -- (0,0,1) -- (1,0,1) -- (1,0,0) -- (1,1,0)
	 (1,0,1) -- (1,1,1);
\end{tikzpicture}
%%%%%%%%%%%%%%%%%%%%%%
\right)
(1) \, . 
$$
\end{enumerate}
The axioms (vi)--(viii) follow immediately from the diagrammatic presentation. 
\end{proof}

\subsection{Gray categories with duals from defect TQFTs}
\label{subsec:GraycatwithdualsfromZ}

%arXiv_v2: 
	%In this section we show that the Gray category $\tz$ associated to a defect TQFT $\zz: \Bordd \rightarrow \Vect_\Bbbk$ naturally has the structure of a Gray category with duals. To establish this result we have to (i) provide for all 2-categories $\tz(u,v)$ pivotal structures which are compatible with the Gray product; (ii) for every 1-morphism~$\alpha$ provide its dual $\alpha\hash$ together with the fold $\coev_\alpha$ and the triangulator $\tau_\alpha$. 
	In this section we construct a Gray category with duals~$\tz$ for every defect TQFT $\zz: \Bordd \rightarrow \Vect_\Bbbk$, as a variant of the Gray category $\tz'$ of the previous section. To establish this result we in particular have to (i) provide for all 2-categories $\tz(u,v)$ pivotal structures which are compatible with the Gray product; 
(ii) for every 1-morphism~$\alpha$ provide its dual $\alpha\hash$ together with the fold $\coev_\alpha$ and the triangulator $\tau_\alpha$. 
That these data satisfy the axioms of Definition~\ref{def:Graycatduals} will be a direct consequence of functoriality for~$\zz$ and our diagrammatic presentation of~$\tz$.

\medskip

In order to have duals in the Gray category associated to~$\zz$, we need to enlarge the class of diagrams we considered in the previous section: 
we also have to allow 
%arXiv_v2: 
	%non-progressive diagrams as 2-morphisms in $\tz$, 
	non-progressive diagrams to represent 2-morphisms in $\tz$, 
in which surfaces may bend or fold. 
%arXiv_v2: 
	%Thus from now on by $\tz$ we mean the Gray category constructed in Section~\ref{subsec:tricatfromZ}, but with ``$\mathcal F^{\textrm{p}} \K^\D$'' replaced by ``$\mathcal F^{\textrm{p}}_{\textrm{d}} \K^\D$''. 
	Thus as a Gray category we define~$\tz$ as~$\tz'$ in Section~\ref{subsec:tricatfromZ}, but with ``$\mathcal F^{\textrm{p}} \K^\D$'' replaced by ``$\mathcal F^{\textrm{p}}_{\textrm{d}} \K^\D$''. More precisely, the objects of~$\tz$ and~$\tz'$ are the same; 1-morphisms in~$\tz$ are equivalence classes of cylinders over representatives~$\alpha$ of 1-morphisms in $\mathcal F^{\textrm{p}}_{\textrm{d}} \K^\D$, with prescribed orientations for 1-strata in $\alpha \times [0,1]$ and such that the orientations of 2-strata in $\alpha \times [0,1]$ induce the corresponding orientations of 1-strata in $\alpha \times \{1\}$; similarly, 2-morphisms in $\tz$ are equivalence classes of cylinders over representatives of 2-morphisms in $\mathcal F^{\textrm{p}}_{\textrm{d}} \K^\D$; 3-morphisms and the rest of the structure in~$\tz$ are then defined analogously to the case of~$\tz'$ in Section~\ref{subsec:tricatfromZ}. 
As a consequence the graphical presentation of~$\tz$ will be compatible with the graphical calculus for Gray categories with duals. 

\medskip

We begin with part (ii) of the above outline as most of the structure for $\#$-duals in $\tz$ is already contained in the free pre-2-category $\mathcal F^{\textrm{p}}_{\textrm{d}} \K^\D$ of Section~\ref{subsubsec:bicatsduals}. 
Indeed, taking the dual $\alpha\hash \in \tz(v,u)$ of a 1-morphism $\alpha \in \tz(u,v)$ simply amounts to `turning the stack of planes~$\alpha$ around'. 
More precisely, if~$\alpha$ is a cube with~$m$ oriented planes decorated by $\alpha_1,\alpha_2,\dots,\alpha_m$ from front to back, then $\alpha\hash$ is the cube whose~$m$ planes are labelled $\alpha_m,\alpha_{m-1},\dots,\alpha_1$ 
%arXiv_v2: 
	%from front to black, 
	from front to back, 
and every plane has reversed orientation. 
In other words, the $\#$-dual is obtained by a rotation by~$\pi$ along the line through the centre of the cube and parallel to the $y$-axis: 
\be\label{eq:alphahash}
\alpha = 
%%%%%%%%%%%%%%%%%%%%%% 
\begin{tikzpicture}[thick,scale=2.5,color=blue!50!black, baseline=0.0cm, >=stealth, 
				style={x={(-0.9cm,-0.4cm)},y={(0.8cm,-0.4cm)},z={(0cm,0.9cm)}}]
% 3-stratum: 
\fill [blue!20,opacity=0.2] (1,0,0) -- (1,1,0) -- (0,1,0) -- (0,1,1) -- (0,0,1) -- (1,0,1);
% invisible edges of cube: 
\draw[
	 color=gray, 
	 opacity=0.3, 
	 semithick,
	 dashed
	 ] 
	 (1,0,0) -- (0,0,0) -- (0,1,0)
	 (0,0,0) -- (0,0,1);
%%%
% 3-stratum: 
\fill [blue!20,opacity=0.2] (1,0,0) -- (1,1,0) -- (0,1,0) -- (0,1,1) -- (0,0,1) -- (1,0,1);
\coordinate (b1) at (0.25, 0, 0);
\coordinate (b2) at (0.5, 0, 0);
\coordinate (b3) at (0.75, 0, 0);
\coordinate (b4) at (0.25, 1, 0);
\coordinate (b5) at (0.5, 1, 0);
\coordinate (b6) at (0.75, 1, 0);
\coordinate (t1) at (0.25, 0, 1);
\coordinate (t2) at (0.5, 0, 1);
\coordinate (t3) at (0.75, 0, 1);
\coordinate (t4) at (0.25, 1, 1);
\coordinate (t5) at (0.5, 1, 1);
\coordinate (t6) at (0.75, 1, 1);
%
% planes: 
\fill [magenta!50,opacity=0.7] (b1) -- (b4) -- (t4) -- (t1);
\fill [red!50,opacity=0.7] (b2) -- (b5) -- (t5) -- (t2);
\fill [red!30,opacity=0.7] (b3) -- (b6) -- (t6) -- (t3);
%
% labels: 
\draw[line width=1] (0.88, 0.5, 0) node[line width=0pt] (beta) {{\footnotesize $u$}};
\draw[line width=1] (0.09, 0.92, 0) node[line width=0pt] (beta) {{\footnotesize $v$}};
\draw[line width=1] (0.75, 0.85, 0.5) node[line width=0pt] (beta) {{\footnotesize $\alpha_1$}};
\draw[line width=1] (0.5, 0.85, 0.5) node[line width=0pt] (beta) {{\footnotesize $\alpha_2$}};
\draw[line width=1] (0.25, 0.85, 0.5) node[line width=0pt] (beta) {{\footnotesize $\alpha_3$}};
%
% visible edges of cube
\draw[
	 color=gray, 
	 opacity=0.4, 
	 semithick
	 ] 
	 (0,1,1) -- (0,1,0) -- (1,1,0) -- (1,1,1) -- (0,1,1) -- (0,0,1) -- (1,0,1) -- (1,0,0) -- (1,1,0)
	 (1,0,1) -- (1,1,1);
\end{tikzpicture}
%%%%%%%%%%%%%%%%%%%%%% 
\quad\implies
\quad
\alpha\hash = 
%%%%%%%%%%%%%%%%%%%%%% 
\begin{tikzpicture}[thick,scale=2.5,color=blue!50!black, baseline=0.0cm, >=stealth, 
				style={x={(-0.9cm,-0.4cm)},y={(0.8cm,-0.4cm)},z={(0cm,0.9cm)}}]
% 3-stratum: 
\fill [blue!20,opacity=0.2] (1,0,0) -- (1,1,0) -- (0,1,0) -- (0,1,1) -- (0,0,1) -- (1,0,1);
% invisible edges of cube: 
\draw[
	 color=gray, 
	 opacity=0.3, 
	 semithick,
	 dashed
	 ] 
	 (1,0,0) -- (0,0,0) -- (0,1,0)
	 (0,0,0) -- (0,0,1);
%%%
% 3-stratum: 
\fill [blue!20,opacity=0.2] (1,0,0) -- (1,1,0) -- (0,1,0) -- (0,1,1) -- (0,0,1) -- (1,0,1);
\coordinate (b1) at (0.25, 0, 0);
\coordinate (b2) at (0.5, 0, 0);
\coordinate (b3) at (0.75, 0, 0);
\coordinate (b4) at (0.25, 1, 0);
\coordinate (b5) at (0.5, 1, 0);
\coordinate (b6) at (0.75, 1, 0);
\coordinate (t1) at (0.25, 0, 1);
\coordinate (t2) at (0.5, 0, 1);
\coordinate (t3) at (0.75, 0, 1);
\coordinate (t4) at (0.25, 1, 1);
\coordinate (t5) at (0.5, 1, 1);
\coordinate (t6) at (0.75, 1, 1);
%
% planes: 
\fill [pattern=dots, opacity=0.4] (b1) -- (b4) -- (t4) -- (t1);
\fill [red!30,opacity=0.7] (b1) -- (b4) -- (t4) -- (t1);
\fill [pattern=dots, opacity=0.4] (b2) -- (b5) -- (t5) -- (t2);
\fill [red!50,opacity=0.7] (b2) -- (b5) -- (t5) -- (t2);
\fill [pattern=dots, opacity=0.4] (b3) -- (b6) -- (t6) -- (t3);
\fill [magenta!50,opacity=0.7] (b3) -- (b6) -- (t6) -- (t3);
%
% labels: 
\draw[line width=1] (0.88, 0.5, 0) node[line width=0pt] (beta) {{\footnotesize $v$}};
\draw[line width=1] (0.09, 0.92, 0) node[line width=0pt] (beta) {{\footnotesize $u$}};
\draw[line width=1] (0.75, 0.85, 0.5) node[line width=0pt] (beta) {{\footnotesize $\alpha_3$}};
\draw[line width=1] (0.5, 0.85, 0.5) node[line width=0pt] (beta) {{\footnotesize $\alpha_2$}};
\draw[line width=1] (0.25, 0.85, 0.5) node[line width=0pt] (beta) {{\footnotesize $\alpha_1$}};
%
% visible edges of cube
\draw[
	 color=gray, 
	 opacity=0.4, 
	 semithick
	 ] 
	 (0,1,1) -- (0,1,0) -- (1,1,0) -- (1,1,1) -- (0,1,1) -- (0,0,1) -- (1,0,1) -- (1,0,0) -- (1,1,0)
	 (1,0,1) -- (1,1,1);
\end{tikzpicture}
%%%%%%%%%%%%%%%%%%%%%% 
\, . 
\ee
This is exactly the cylinder over the string diagrams in $\mathcal F^{\textrm{p}}_{\textrm{d}} \K^\D$ (compare~\eqref{eq:Xhashfree}, where however we use different notation, suitable for 2-categories). 

Similarly, the fold $\coev_\alpha$ in $\tz$ is the cylinder over the string diagram representing the coevaluation in the free pre-2-category $\mathcal F^{\textrm{p}}_{\textrm{d}} \K^\D$. 
Thus for~$\alpha$ as in \eqref{eq:alphahash} we have 
\be\label{eq:coevalpha}
\coev_\alpha = 
%%%%%%%%%%%%%%%%%%%%%% 
\begin{tikzpicture}[thick,scale=2.5,color=blue!50!black, baseline=0.0cm, >=stealth, 
				style={x={(-0.9cm,-0.4cm)},y={(0.8cm,-0.4cm)},z={(0cm,0.9cm)}}]
% visible edges of cube
\draw[
	 color=gray, 
	 opacity=0.4, 
	 semithick
	 ] 
	 (0,1,1) -- (0,1,0) -- (1,1,0) -- (1,1,1) -- (0,1,1) -- (0,0,1) -- (1,0,1) -- (1,0,0) -- (1,1,0)
	 (1,0,1) -- (1,1,1);
% to cut off unwanted edges: 
\clip (1,0,0) -- (1,1,0) -- (0,1,0) -- (0,1,1) -- (0,0,1) -- (1,0,1);
% 3-stratum: 
\fill [blue!20,opacity=0.2] (1,0,0) -- (1,1,0) -- (0,1,0) -- (0,1,1) -- (0,0,1) -- (1,0,1);
% invisible edges of cube: 
\draw[
	 color=gray, 
	 opacity=0.3, 
	 semithick,
	 dashed
	 ] 
	 (1,0,0) -- (0,0,0) -- (0,1,0)
	 (0,0,0) -- (0,0,1);
%%%
% alpha1-coordinates: 
\coordinate (p1) at (0.1, 0, 0);
\coordinate (p2) at (0.9, 0, 0);
\coordinate (q1) at (0.1, 0, 1);
\coordinate (q2) at (0.9, 0, 1);
% alpha2-coordinates: 
\coordinate (p12) at (0.2, 0, 0);
\coordinate (p22) at (0.8, 0, 0);
\coordinate (q12) at (0.2, 0, 1);
\coordinate (q22) at (0.8, 0, 1);
% alpha3-coordinates: 
\coordinate (p13) at (0.3, 0, 0);
\coordinate (p23) at (0.7, 0, 0);
\coordinate (q13) at (0.3, 0, 1);
\coordinate (q23) at (0.7, 0, 1);
\coordinate (S1) at (0.5, 0.9, 0);
\coordinate (S2) at (0.5, 0.7, 0);
\coordinate (S3) at (0.5, 0.5, 0);
%
% alpha1-fold back:
\foreach \z in {0, 0.0194, ..., 1.0}
{
	\draw[color=red!40, very thick, opacity=0.4]%
	($(S1) + (0,0,\z)$) .. controls +(-0.02,0,0) and +(0,0.9,0) .. ($(p1) + (0,0,\z)$);
}
% alpha2-fold back:
\foreach \z in {0, 0.0194, ..., 1.0}
{
	\draw[color=red!70, very thick, opacity=0.4]%
	($(S2) + (0,0,\z)$) .. controls +(-0.02,0,0) and +(0,0.7,0) .. ($(p12) + (0,0,\z)$);
}
% alpha3-fold:
\foreach \z in {0, 0.0194, ..., 1.0}
{
	\draw[color=magenta!80, very thick, opacity=0.4]%
	($(S3) + (0,0,\z)$) .. controls +(-0.01,0,0) and +(0,0.5,0) .. ($(p13) + (0,0,\z)$);
	\draw[color=magenta!80, very thick, opacity=0.4]%
	($(p23) + (0,0,\z)$) .. controls +(0,0.5,0) and +(0.01,0,0) .. ($(S3) + (0,0,\z)$);
}
%
% alpha2-fold front:
\foreach \z in {0, 0.0194, ..., 1.0}
{
	\draw[color=red!70, very thick, opacity=0.4]%
	($(p22) + (0,0,\z)$) .. controls +(0,0.7,0) and +(0.02,0,0) .. ($(S2) + (0,0,\z)$);
}
%
% alpha1-fold front:
\foreach \z in {0, 0.0194, ..., 1.0}
{
	\draw[color=red!40, very thick, opacity=0.4]%
	($(p2) + (0,0,\z)$) .. controls +(0,0.9,0) and +(0.02,0,0) .. ($(S1) + (0,0,\z)$);
}
%
% object labels: 
\draw[line width=1] (0.89, 0.75, 0) node[line width=0pt] (beta) {{\footnotesize $v$}};
\draw[line width=1] (0.5, 0.1, 0.95) node[line width=0pt] (beta) {{\footnotesize $u$}};
%
% visible edges of cube
\draw[
	 color=gray, 
	 opacity=0.4, 
	 semithick
	 ] 
	 (0,1,1) -- (0,1,0) -- (1,1,0) -- (1,1,1) -- (0,1,1) -- (0,0,1) -- (1,0,1) -- (1,0,0) -- (1,1,0)
	 (1,0,1) -- (1,1,1);
\end{tikzpicture}
%%%%%%%%%%%%%%%%%%%%%% 
\ee
which we will also write as 
\be\label{eq:coevalpha2}
\coev_\alpha = 
%%%%%%%%%%%%%%%%%%%%%% 
\begin{tikzpicture}[thick,scale=2.5,color=blue!50!black, baseline=0.0cm, >=stealth, 
				style={x={(-0.9cm,-0.4cm)},y={(0.8cm,-0.4cm)},z={(0cm,0.9cm)}}]
% invisible edges of cube: 
\draw[
	 color=gray, 
	 opacity=0.3, 
	 semithick,
	 dashed
	 ] 
	 (1,0,0) -- (0,0,0) -- (0,1,0)
	 (0,0,0) -- (0,0,1);
%%%
% for 3-strata: 
\coordinate (f0) at (0.5, 0.5, 0);
\coordinate (f1) at (0.5, 0.5, 1);
\coordinate (front0) at (0.7, 0, 0);
\coordinate (front1) at (0.7, 0, 1);
\coordinate (back0) at (0.3, 0, 0);
\coordinate (back1) at (0.3, 0, 1);
\coordinate (f02) at (0.5, 0.7, 0);
\coordinate (f12) at (0.5, 0.7, 1);
\coordinate (front02) at (0.8, 0, 0);
\coordinate (front12) at (0.8, 0, 1);
\coordinate (back02) at (0.2, 0, 0);
\coordinate (back12) at (0.2, 0, 1);
\coordinate (f01) at (0.5, 0.9, 0);
\coordinate (f11) at (0.5, 0.9, 1);
\coordinate (front01) at (0.9, 0, 0);
\coordinate (front11) at (0.9, 0, 1);
\coordinate (back01) at (0.1, 0, 0);
\coordinate (back11) at (0.1, 0, 1);

% 3-stratum
\fill [blue!20,opacity=0.2] (1,0,0) -- (1,1,0) -- (0,1,0) -- (0,1,1) -- (0,0,1) -- (1,0,1);
%
%
% alpha1-hash-plane: 
\fill [pattern=dots, opacity=0.4] (f01) -- (f11) -- (back11) -- (back01);
\fill [red!30, opacity=0.7] (f01) -- (f11) -- (back11) -- (back01);
%
% alpha2-hash-plane: 
\fill [pattern=dots, opacity=0.4] (f02) -- (f12) -- (back12) -- (back02);
\fill [red!50, opacity=0.7] (f02) -- (f12) -- (back12) -- (back02);
%
% alpha3-hash-plane: 
\fill [pattern=dots, opacity=0.4] (f0) -- (f1) -- (back1) -- (back0);
\fill [magenta!50, opacity=0.7] (f0) -- (f1) -- (back1) -- (back0);
%
% alpha3-plane: 
\fill [magenta!50,opacity=0.7] (f0) -- (f1) -- (front1) -- (front0);
%
% alpha2-plane: 
\fill [red!50,opacity=0.7] (f02) -- (f12) -- (front12) -- (front02);
%
% alpha1-plane: 
\fill [red!30,opacity=0.7] (f01) -- (f11) -- (front11) -- (front01);
%
% object labels: 
\draw[line width=1] (0.85, 0.5, 0) node[line width=0pt] (beta) {{\footnotesize $v$}};
\draw[line width=1] (0.5, 0.1, 0.98) node[line width=0pt] (beta) {{\footnotesize $u$}};
%
% visible edges of cube
\draw[
	 color=gray, 
	 opacity=0.4, 
	 semithick
	 ] 
	 (0,1,1) -- (0,1,0) -- (1,1,0) -- (1,1,1) -- (0,1,1) -- (0,0,1) -- (1,0,1) -- (1,0,0) -- (1,1,0)
	 (1,0,1) -- (1,1,1);
\end{tikzpicture}
%%%%%%%%%%%%%%%%%%%%%%
=
%%%%%%%%%%%%%%%%%%%%%% 
\begin{tikzpicture}[thick,scale=2.5,color=blue!50!black, baseline=0.0cm, >=stealth, 
				style={x={(-0.9cm,-0.4cm)},y={(0.8cm,-0.4cm)},z={(0cm,0.9cm)}}]
% invisible edges of cube: 
\draw[
	 color=gray, 
	 opacity=0.3, 
	 semithick,
	 dashed
	 ] 
	 (1,0,0) -- (0,0,0) -- (0,1,0)
	 (0,0,0) -- (0,0,1);
%%%
% for 3-strata: 
\coordinate (f0) at (0.5, 0.75, 0);
\coordinate (f1) at (0.5, 0.75, 1);
\coordinate (front0) at (0.75, 0, 0);
\coordinate (front1) at (0.75, 0, 1);
\coordinate (back0) at (0.25, 0, 0);
\coordinate (back1) at (0.25, 0, 1);
% 3-stratum
\fill [blue!20,opacity=0.2] (1,0,0) -- (1,1,0) -- (0,1,0) -- (0,1,1) -- (0,0,1) -- (1,0,1);
%
%
% alpha-hash-plane: 
\fill [pattern=dots, opacity=0.4] (f0) -- (f1) -- (back1) -- (back0);
\fill [red!60, opacity=0.7] (f0) -- (f1) -- (back1) -- (back0);
%
% alpha-plane: 
\fill [red!60,opacity=0.7] (f0) -- (f1) -- (front1) -- (front0);
% alpha label: 
\draw[line width=1] (0.6, 0.1, 0.35) node[line width=0pt] (beta) {{\footnotesize $\alpha$}};
%
% object labels: 
\draw[line width=1] (0.85, 0.5, 0) node[line width=0pt] (beta) {{\footnotesize $v$}};
\draw[line width=1] (0.5, 0.1, 0.95) node[line width=0pt] (beta) {{\footnotesize $u$}};
%
% visible edges of cube
\draw[
	 color=gray, 
	 opacity=0.4, 
	 semithick
	 ] 
	 (0,1,1) -- (0,1,0) -- (1,1,0) -- (1,1,1) -- (0,1,1) -- (0,0,1) -- (1,0,1) -- (1,0,0) -- (1,1,0)
	 (1,0,1) -- (1,1,1);
\end{tikzpicture}
%%%%%%%%%%%%%%%%%%%%%%
\ee

The only datum for the $\#$-dual in Definition~\ref{def:Graycatduals} which for $\tz$ does not originate in $\mathcal F^{\textrm{p}}_{\textrm{d}} \K^\D$ is the triangulator $\tau_\alpha$ in~\eqref{eq:triangulator}. 
Since $\tau_\alpha$ is a 3-morphism in $\tz$ it must be an element in the vector space which~$\zz$ associates to the `sphere with Zorro's~Z at the south pole and a straight line at the north pole'. 
More precisely, $\tau_\alpha$ is defined to be $\zz(-)(1)$ applied to the stratified 3-ball whose only 2-stratum~$\alpha$ (or list of 2-strata $\alpha_1,\alpha_2,\dots,\alpha_m$) intersects its southern hemisphere in Z-shape, while the remaining components of the intersection is a geodesic. 
%arXiv_v2: 
	%Hence by identifying cubes with spheres (for merely typographical reasons) 
	Hence by identifying cubes with balls (as explained after \eqref{eq:SXXprimesphere})
we set
$$
\tau_\alpha = 
\zz\left(
%%%%%%%%%%%%%%%%%%%%%% 
\begin{tikzpicture}[thick,scale=2.5,color=blue!50!black, baseline=0.0cm, >=stealth, 
				style={x={(-0.9cm,-0.4cm)},y={(0.8cm,-0.4cm)},z={(0cm,0.9cm)}}]
% invisible edges of cube: 
\draw[
	 color=gray, 
	 opacity=0.3, 
	 semithick,
	 dashed
	 ] 
	 (1,0,0) -- (0,0,0) -- (0,1,0)
	 (0,0,0) -- (0,0,1);
%%%
% for 3-strata: 
\coordinate (b1) at (0.75, 0.3, 0);
\coordinate (b2) at (0.45, 0.5, 0);
\coordinate (cusp) at (0.45, 0.5, 0.5);
\coordinate (abovecusp) at (0.45, 0.3, 1);
\coordinate (left0) at (0.45, 0, 0);
\coordinate (leftmid) at (0.45, 0, 0.5);
\coordinate (left1) at (0.45, 0, 1);
\coordinate (right0) at (0.75, 1, 0);
\coordinate (rightmid) at (0.6, 1, 0.5);
\coordinate (right1) at (0.45, 1, 1);
%
% 3-stratum: 
\fill [blue!20,opacity=0.2] (1,0,0) -- (1,1,0) -- (0,1,0) -- (0,1,1) -- (0,0,1) -- (1,0,1);
%
% alpha-plane bottom left: 
\fill [red!60, opacity=0.7] (b2) -- (left0) -- (leftmid) -- (cusp);
% alpha-plane bottom left: 
\fill [red!60, opacity=0.7] (leftmid) -- (cusp) -- (abovecusp) -- (left1);
% alpha-hash-plane: 
\fill [pattern=dots, opacity=0.4] (cusp) -- (b1) -- (b2);
\fill [red!60, opacity=0.7] (cusp) -- (b1) -- (b2);
% alpha-plane bottom right: 
\fill [red!60, opacity=0.7] (cusp) -- (b1) -- (right0) -- (rightmid);
% alpha-plane top right: 
\fill [red!60, opacity=0.7] (cusp) -- (rightmid) -- (right1) -- (abovecusp);
% alpha label: 
\draw[line width=1] (0.3, 0.2, 0.6) node[line width=0pt] (beta) {{\footnotesize $\alpha$}};
%
% object labels: 
\draw[line width=1] (0.85, 0.5, 0) node[line width=0pt] (beta) {{\footnotesize $u$}};
\draw[line width=1] (0.2, 0.8, 0) node[line width=0pt] (beta) {{\footnotesize $v$}};
%
% visible edges of cube
\draw[
	 color=gray, 
	 opacity=0.4, 
	 semithick
	 ] 
	 (0,1,1) -- (0,1,0) -- (1,1,0) -- (1,1,1) -- (0,1,1) -- (0,0,1) -- (1,0,1) -- (1,0,0) -- (1,1,0)
	 (1,0,1) -- (1,1,1);
\end{tikzpicture}
%%%%%%%%%%%%%%%%%%%%%%
\right)(1)
$$
for every 1-morphism $\alpha: u \rightarrow v$ in $\tz$. 
We note that here in the bordism on which~$\zz$ is evaluated, the cusp point is not singular; its appearance is only a matter of the graphical presentation.

\medskip

To endow the 2-categories $\tz(u,v)$ with a pivotal structure we need another kind of dual $X^\dagger: \beta \rightarrow \alpha$ for every 2-morphism $X: \alpha \rightarrow \beta$. 
By definition, $X^\dagger$ is obtained by rotating the cube~$X$ by~$\pi$ about the line through the centre of the cube and parallel to the $x$-axis. 
For example 
\be\label{eq:Xhash}
X = 
%%%%%%%%%%%%%%%%%%%%%% 
\begin{tikzpicture}[thick,scale=2.5,color=blue!50!black, baseline=0.0cm, >=stealth, 
				style={x={(-0.9cm,-0.4cm)},y={(0.8cm,-0.4cm)},z={(0cm,0.9cm)}}]
% visible edges of cube
\draw[
	 color=gray, 
	 opacity=0.4, 
	 semithick
	 ] 
	 (0,1,1) -- (0,1,0) -- (1,1,0) -- (1,1,1) -- (0,1,1) -- (0,0,1) -- (1,0,1) -- (1,0,0) -- (1,1,0)
	 (1,0,1) -- (1,1,1);
% to cut off unwanted edges: 
\clip (1,0,0) -- (1,1,0) -- (0,1,0) -- (0,1,1) -- (0,0,1) -- (1,0,1);
% 3-stratum: 
\fill [blue!20,opacity=0.2] (1,0,0) -- (1,1,0) -- (0,1,0) -- (0,1,1) -- (0,0,1) -- (1,0,1);
% invisible edges of cube: 
\draw[
	 color=gray, 
	 opacity=0.3, 
	 semithick,
	 dashed
	 ] 
	 (1,0,0) -- (0,0,0) -- (0,1,0)
	 (0,0,0) -- (0,0,1);
%%%
% top vertices: 
\coordinate (topleft2) at (0.5, 1, 1);
\coordinate (bottomleft2) at (0.5, 1, 0);
\coordinate (topright1) at (0.2, 0, 1);
\coordinate (topright2) at (0.5, 0, 1);
\coordinate (topright3) at (0.7, 0, 1);
\coordinate (bottomright1) at (0.2, 0, 0);
\coordinate (bottomright2) at (0.5, 0, 0);
\coordinate (bottomright3) at (0.7, 0, 0);
\coordinate (b1s) at (0.2, 0.4, 0);
\coordinate (b2s) at (0.5, 0.7, 0);
\coordinate (b3s) at (0.7, 0.4, 0);
\coordinate (t1s) at (0.2, 0.4, 1);
\coordinate (t2s) at (0.5, 0.7, 1);
\coordinate (t3s) at (0.7, 0.4, 1);
%
% rear plane: 
\fill [magenta!30,opacity=0.7] 
(topright1) -- (bottomright1) -- (b1s) -- (t1s);
%
% back tilted plane: 
\fill [red!60,opacity=0.7] 
(t1s) -- (b1s) -- (b2s) -- (t2s);
%
% X1:
\draw[color=green!60!black, 
	postaction={decorate}, 
	decoration={markings,mark=at position .51 with {\arrow[color=green!60!black]{>}}}, 
	very thick]%
(b1s) -- (t1s);
\fill[color=green!60!black] (b1s) circle (0.8pt) node[left] (0up) {};
%
% middle plane: 
\fill [magenta!45,opacity=0.7] 
(topright2) -- (bottomright2) -- (b2s) -- (t2s);
%
% front tilted plane: 
\fill [magenta!55,opacity=0.7] 
(t3s) -- (b3s) -- (b2s) -- (t2s);
%
% middle left plane: 
\fill [red!15,opacity=0.7] 
(topleft2) -- (bottomleft2) -- (b2s) -- (t2s);
%
% X2:
\draw[color=green!50!black, 
	postaction={decorate}, 
	decoration={markings,mark=at position .51 with {\arrow[color=green!50!black]{>}}}, 
	very thick]%
(b2s) -- (t2s);
\fill[color=green!50!black] (b2s) circle (0.8pt) node[left] (0up) {};
%
% front plane: 
\fill [red!25,opacity=0.7] 
(topright3) -- (bottomright3) -- (b3s) -- (t3s);
%
% X3:
\draw[color=green!60!black, 
	postaction={decorate}, 
	decoration={markings,mark=at position .51 with {\arrow[color=green!60!black]{>}}}, 
	very thick]%
(b3s) -- (t3s);
% top 0-strata: 
\fill[color=green!60!black] (t1s) circle (0.8pt) node[left] (0up) {};
\fill[color=green!50!black] (t2s) circle (0.8pt) node[left] (0up) {};
\fill[color=green!60!black] (t3s) circle (0.8pt) node[left] (0up) {};
\fill[color=green!60!black] (b3s) circle (0.8pt) node[left] (0up) {};
%
% visible edges of cube
\draw[
	 color=gray, 
	 opacity=0.4, 
	 semithick
	 ] 
	 (0,1,1) -- (0,1,0) -- (1,1,0) -- (1,1,1) -- (0,1,1) -- (0,0,1) -- (1,0,1) -- (1,0,0) -- (1,1,0)
	 (1,0,1) -- (1,1,1);
\end{tikzpicture}
%%%%%%%%%%%%%%%%%%%%%% 
\quad\implies\quad
X^\dagger = 
%%%%%%%%%%%%%%%%%%%%%% 
\begin{tikzpicture}[thick,scale=2.5,color=blue!50!black, baseline=0.0cm, >=stealth, 
				style={x={(-0.9cm,-0.4cm)},y={(0.8cm,-0.4cm)},z={(0cm,0.9cm)}}]
% visible edges of cube
\draw[
	 color=gray, 
	 opacity=0.4, 
	 semithick
	 ] 
	 (0,1,1) -- (0,1,0) -- (1,1,0) -- (1,1,1) -- (0,1,1) -- (0,0,1) -- (1,0,1) -- (1,0,0) -- (1,1,0)
	 (1,0,1) -- (1,1,1);
% to cut off unwanted edges: 
\clip (1,0,0) -- (1,1,0) -- (0,1,0) -- (0,1,1) -- (0,0,1) -- (1,0,1);
% 3-stratum: 
\fill [blue!20,opacity=0.2] (1,0,0) -- (1,1,0) -- (0,1,0) -- (0,1,1) -- (0,0,1) -- (1,0,1);
% invisible edges of cube: 
\draw[
	 color=gray, 
	 opacity=0.3, 
	 semithick,
	 dashed
	 ] 
	 (1,0,0) -- (0,0,0) -- (0,1,0)
	 (0,0,0) -- (0,0,1);
%%%
% top vertices: 
\coordinate (topleft1) at (0.2, 0, 1);
\coordinate (topleft2) at (0.5, 0, 1);
\coordinate (topleft3) at (0.7, 0, 1);
\coordinate (bottomleft1) at (0.2, 0, 0);
\coordinate (bottomleft2) at (0.5, 0, 0);
\coordinate (bottomleft3) at (0.7, 0, 0);
\coordinate (topright1) at (0.2, 1, 1);
\coordinate (topright2) at (0.5, 1, 1);
\coordinate (topright3) at (0.7, 1, 1);
\coordinate (bottomright1) at (0.2, 1, 0);
\coordinate (bottomright2) at (0.5, 1, 0);
\coordinate (bottomright3) at (0.7, 1, 0);
\coordinate (b1s) at (0.2, 0.6, 0);
\coordinate (b2s) at (0.5, 0.3, 0);
\coordinate (b3s) at (0.7, 0.6, 0);
\coordinate (t1s) at (0.2, 0.6, 1);
\coordinate (t2s) at (0.5, 0.3, 1);
\coordinate (t3s) at (0.7, 0.6, 1);
%
% rear plane: 
\fill [magenta!30,opacity=0.7] 
(topright1) -- (bottomright1) -- (b1s) -- (t1s);
%
% back tilted plane: 
\fill [red!60,opacity=0.7] 
(t1s) -- (b1s) -- (b2s) -- (t2s);
%
% X1:
\draw[color=green!60!black, 
	postaction={decorate}, 
	decoration={markings,mark=at position .51 with {\arrow[color=green!60!black]{<}}}, 
	very thick]%
(b1s) -- (t1s);
\fill[color=green!60!black] (b1s) circle (0.8pt) node[left] (0up) {};
%
% middle plane: 
\fill [magenta!45,opacity=0.7] 
(topright2) -- (bottomright2) -- (b2s) -- (t2s);
%
% front tilted plane: 
\fill [magenta!55,opacity=0.7] 
(t3s) -- (b3s) -- (b2s) -- (t2s);
%
% middle left plane: 
\fill [red!15,opacity=0.7] 
(topleft2) -- (bottomleft2) -- (b2s) -- (t2s);
%
% X2:
\draw[color=green!50!black, 
	postaction={decorate}, 
	decoration={markings,mark=at position .51 with {\arrow[color=green!50!black]{<}}}, 
	very thick]%
(b2s) -- (t2s);
\fill[color=green!50!black] (b2s) circle (0.8pt) node[left] (0up) {};
%
% front plane: 
\fill [red!25,opacity=0.7] 
(topright3) -- (bottomright3) -- (b3s) -- (t3s);
%
% X3:
\draw[color=green!60!black, 
	postaction={decorate}, 
	decoration={markings,mark=at position .51 with {\arrow[color=green!60!black]{<}}}, 
	very thick]%
(b3s) -- (t3s);
% top 0-strata: 
\fill[color=green!60!black] (t1s) circle (0.8pt) node[left] (0up) {};
\fill[color=green!50!black] (t2s) circle (0.8pt) node[left] (0up) {};
\fill[color=green!60!black] (t3s) circle (0.8pt) node[left] (0up) {};
\fill[color=green!60!black] (b3s) circle (0.8pt) node[left] (0up) {};
%
% visible edges of cube
\draw[
	 color=gray, 
	 opacity=0.4, 
	 semithick
	 ] 
	 (0,1,1) -- (0,1,0) -- (1,1,0) -- (1,1,1) -- (0,1,1) -- (0,0,1) -- (1,0,1) -- (1,0,0) -- (1,1,0)
	 (1,0,1) -- (1,1,1);
\end{tikzpicture}
%%%%%%%%%%%%%%%%%%%%%% 
\, . 
\ee
In particular, all 1-strata in~$X$ have opposite orientation in~$X^\dagger$. 

The 3-morphisms $\ev_X: X^\dagger \fus X \rightarrow 1_\alpha$ and $\coev_X: 1_\beta \rightarrow X \fus X^\dagger$ which exhibit~$X^\dagger$ as the dual of~$X$ are again obtained by applying~$\zz$  to an appropriately decorated ball. 
In the case of $\coev_X$ this is done by stretching the contents of the cube~$X$ in vertical direction, then bending the lower end back up to achieve a U-shape, and then viewing the result as a decorated 3-ball. 
Finally we define $\coev_X$ to be~$\zz$ applied to this bordism. 
For~$X$ as in \eqref{eq:Xhash} this means 
$$
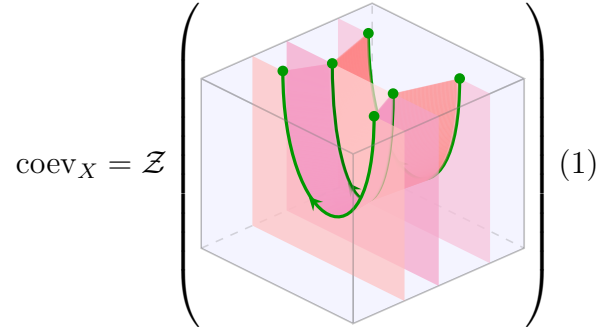
\coev_X = 
\zz\left(
%%%%%%%%%%%%%%%%%%%%%% 
\begin{tikzpicture}[thick,scale=2.5,color=blue!50!black, baseline=0.0cm, >=stealth, 
				style={x={(-0.9cm,-0.4cm)},y={(0.8cm,-0.4cm)},z={(0cm,0.9cm)}}]
% visible edges of cube
\draw[
	 color=gray, 
	 opacity=0.4, 
	 semithick
	 ] 
	 (0,1,1) -- (0,1,0) -- (1,1,0) -- (1,1,1) -- (0,1,1) -- (0,0,1) -- (1,0,1) -- (1,0,0) -- (1,1,0)
	 (1,0,1) -- (1,1,1);
% to cut off unwanted edges: 
\clip (1,0,0) -- (1,1,0) -- (0,1,0) -- (0,1,1) -- (0,0,1) -- (1,0,1);
% 3-stratum: 
\fill [blue!20,opacity=0.2] (1,0,0) -- (1,1,0) -- (0,1,0) -- (0,1,1) -- (0,0,1) -- (1,0,1);
% invisible edges of cube: 
\draw[
	 color=gray, 
	 opacity=0.3, 
	 semithick,
	 dashed
	 ] 
	 (1,0,0) -- (0,0,0) -- (0,1,0)
	 (0,0,0) -- (0,0,1);
%%%
% top vertices: 
\coordinate (topleft1) at (0.2, 0, 1);
\coordinate (topleft2) at (0.5, 0, 1);
\coordinate (topleft3) at (0.7, 0, 1);
\coordinate (bottomleft1) at (0.2, 0, 0);
\coordinate (bottomleft2) at (0.5, 0, 0);
\coordinate (bottomleft3) at (0.7, 0, 0);
\coordinate (topright1) at (0.2, 1, 1);
\coordinate (topright2) at (0.5, 1, 1);
\coordinate (topright3) at (0.7, 1, 1);
\coordinate (bottomright1) at (0.2, 1, 0);
\coordinate (bottomright2) at (0.5, 1, 0);
\coordinate (bottomright3) at (0.7, 1, 0);
\coordinate (t1) at (0.2, 0.2, 1);
\coordinate (t2) at (0.5, 0.3, 1);
\coordinate (t3) at (0.7, 0.2, 1);
\coordinate (t1s) at (0.2, 0.8, 1);
\coordinate (t2s) at (0.5, 0.7, 1);
\coordinate (t3s) at (0.7, 0.8, 1);
%
% rear plane: 
\fill [magenta!30,opacity=0.7] 
(topleft1) -- (t1) .. controls +(0,0,-0.95) and +(0,0,-0.95) .. (t1s) -- (topright1) -- (bottomright1) -- (bottomleft1);
% U in rear: 
\draw[color=green!60!black, 
	postaction={decorate}, 
	decoration={markings,mark=at position .51 with {\arrow[color=green!60!black]{<}}}, 
	very thick]%
(t1) .. controls +(0,0,-0.95) and +(0,0,-0.95) .. (t1s);
%
% curved surface in back:
\foreach \x in {0, 0.02, ..., 0.98}
{
	\draw[color=red!50, very thick, opacity=0.4]%
	($(t2) + \x*(t1) - \x*(t2)$) .. controls +(0,0,-0.95) and +(0,0,-0.95) .. ($(t2s) + \x*(t1s) - \x*(t2s)$);
}
%
% middle plane: 
\fill [magenta!45,opacity=0.7] 
(topleft2) -- (t2) .. controls +(0,0,-0.95) and +(0,0,-0.95) .. (t2s) -- (topright2) -- (bottomright2) -- (bottomleft2);
% middle U-plane: 
\fill [red!15,opacity=0.7] 
(t2) .. controls +(0,0,-0.95) and +(0,0,-0.95) .. (t2s);
% U in middle: 
\draw[color=green!50!black, 
	postaction={decorate}, 
	decoration={markings,mark=at position .51 with {\arrow[color=green!50!black]{<}}}, 
	very thick]%
(t2) .. controls +(0,0,-0.95) and +(0,0,-0.95) .. (t2s);
%
% curved surface in front:
\foreach \x in {0.06, 0.08, ..., 1.0}
{
	\draw[color=magenta!50, very thick, opacity=0.4]%
	($(t2) + \x*(t3) - \x*(t2)$) .. controls +(0,0,-0.95) and +(0,0,-0.95) .. ($(t2s) + \x*(t3s) - \x*(t2s)$);
}
%
% front plane: 
\fill [red!25,opacity=0.7] 
(topleft3) -- (t3) .. controls +(0,0,-0.95) and +(0,0,-0.95) .. (t3s) -- (topright3) -- (bottomright3) -- (bottomleft3);
% U in front: 
\draw[color=green!60!black, 
	postaction={decorate}, 
	decoration={markings,mark=at position .51 with {\arrow[color=green!60!black]{<}}}, 
	very thick]%
(t3) .. controls +(0,0,-0.95) and +(0,0,-0.95) .. (t3s);
%
% top 0-strata: 
\fill[color=green!60!black] (t1) circle (0.8pt) node[left] (0up) {};
\fill[color=green!60!black] (t2) circle (0.8pt) node[left] (0up) {};
\fill[color=green!60!black] (t3) circle (0.8pt) node[left] (0up) {};
\fill[color=green!60!black] (t1s) circle (0.8pt) node[left] (0up) {};
\fill[color=green!60!black] (t2s) circle (0.8pt) node[left] (0up) {};
\fill[color=green!60!black] (t3s) circle (0.8pt) node[left] (0up) {};
%
% visible edges of cube
\draw[
	 color=gray, 
	 opacity=0.4, 
	 semithick
	 ] 
	 (0,1,1) -- (0,1,0) -- (1,1,0) -- (1,1,1) -- (0,1,1) -- (0,0,1) -- (1,0,1) -- (1,0,0) -- (1,1,0)
	 (1,0,1) -- (1,1,1);
\end{tikzpicture}
%%%%%%%%%%%%%%%%%%%%%% 
\right)(1)
$$
and analogously 
$$
\ev_X = 
\zz\left(
%%%%%%%%%%%%%%%%%%%%%% 
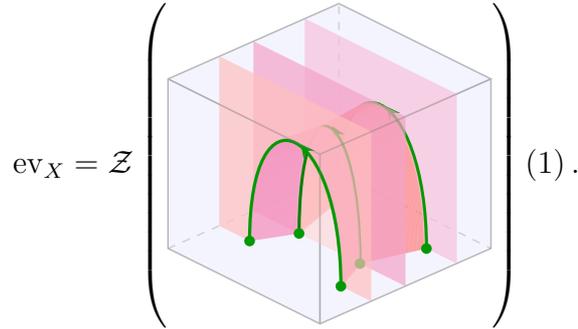
\begin{tikzpicture}[thick,scale=2.5,color=blue!50!black, baseline=0.0cm, >=stealth, 
				style={x={(-0.9cm,-0.4cm)},y={(0.8cm,-0.4cm)},z={(0cm,0.9cm)}}]
% visible edges of cube
\draw[
	 color=gray, 
	 opacity=0.4, 
	 semithick
	 ] 
	 (0,1,1) -- (0,1,0) -- (1,1,0) -- (1,1,1) -- (0,1,1) -- (0,0,1) -- (1,0,1) -- (1,0,0) -- (1,1,0)
	 (1,0,1) -- (1,1,1);
% to cut off unwanted edges: 
\clip (1,0,0) -- (1,1,0) -- (0,1,0) -- (0,1,1) -- (0,0,1) -- (1,0,1);
% 3-stratum: 
\fill [blue!20,opacity=0.2] (1,0,0) -- (1,1,0) -- (0,1,0) -- (0,1,1) -- (0,0,1) -- (1,0,1);
% invisible edges of cube: 
\draw[
	 color=gray, 
	 opacity=0.3, 
	 semithick,
	 dashed
	 ] 
	 (1,0,0) -- (0,0,0) -- (0,1,0)
	 (0,0,0) -- (0,0,1);
%%%
% top vertices: 
\coordinate (topleft1) at (0.2, 0, 1);
\coordinate (topleft2) at (0.5, 0, 1);
\coordinate (topleft3) at (0.7, 0, 1);
\coordinate (bottomleft1) at (0.2, 0, 0);
\coordinate (bottomleft2) at (0.5, 0, 0);
\coordinate (bottomleft3) at (0.7, 0, 0);
\coordinate (topright1) at (0.2, 1, 1);
\coordinate (topright2) at (0.5, 1, 1);
\coordinate (topright3) at (0.7, 1, 1);
\coordinate (bottomright1) at (0.2, 1, 0);
\coordinate (bottomright2) at (0.5, 1, 0);
\coordinate (bottomright3) at (0.7, 1, 0);
\coordinate (t1) at (0.2, 0.2, 0);
\coordinate (t2) at (0.5, 0.3, 0);
\coordinate (t3) at (0.7, 0.2, 0);
\coordinate (t1s) at (0.2, 0.8, 0);
\coordinate (t2s) at (0.5, 0.7, 0);
\coordinate (t3s) at (0.7, 0.8, 0);
%
% rear plane: 
\fill [magenta!30,opacity=0.7] 
(bottomleft1) -- (t1) .. controls +(0,0,0.95) and +(0,0,0.95) .. (t1s) -- (bottomright1) -- (topright1) -- (topleft1);
% U in rear: 
\draw[color=green!60!black, 
	postaction={decorate}, 
	decoration={markings,mark=at position .51 with {\arrow[color=green!60!black]{<}}}, 
	very thick]%
(t1) .. controls +(0,0,0.95) and +(0,0,0.95) .. (t1s);
\fill[color=green!60!black] (t1) circle (0.8pt) node[left] (0up) {};
%
% curved surface in back:
\foreach \x in {0, 0.02, ..., 0.98}
{
	\draw[color=red!50, very thick, opacity=0.4]%
	($(t2) + \x*(t1) - \x*(t2)$) .. controls +(0,0,0.95) and +(0,0,0.95) .. ($(t2s) + \x*(t1s) - \x*(t2s)$);
}
%
% middle plane: 
\fill [magenta!45,opacity=0.7] 
(bottomleft2) -- (t2) .. controls +(0,0,0.95) and +(0,0,0.95) .. (t2s) -- (bottomright2) -- (topright2) -- (topleft2);
% middle U-plane: 
\fill [red!15,opacity=0.7] 
(t2) .. controls +(0,0,0.95) and +(0,0,0.95) .. (t2s);
% U in middle: 
\draw[color=green!50!black, 
	postaction={decorate}, 
	decoration={markings,mark=at position .51 with {\arrow[color=green!50!black]{<}}}, 
	very thick]%
(t2) .. controls +(0,0,0.95) and +(0,0,0.95) .. (t2s);
%
% curved surface in front:
\foreach \x in {0.06, 0.08, ..., 1.0}
{
	\draw[color=magenta!50, very thick, opacity=0.4]%
	($(t2) + \x*(t3) - \x*(t2)$) .. controls +(0,0,0.95) and +(0,0,0.95) .. ($(t2s) + \x*(t3s) - \x*(t2s)$);
}
\fill[color=green!60!black] (t2s) circle (0.8pt) node[left] (0up) {};
%
% front plane: 
\fill [red!25,opacity=0.7] 
(bottomleft3) -- (t3) .. controls +(0,0,0.95) and +(0,0,0.95) .. (t3s) -- (bottomright3) -- (topright3) -- (topleft3);
% U in front: 
\draw[color=green!60!black, 
	postaction={decorate}, 
	decoration={markings,mark=at position .51 with {\arrow[color=green!60!black]{<}}}, 
	very thick]%
(t3) .. controls +(0,0,0.95) and +(0,0,0.95) .. (t3s);
%
% top 0-strata: 
\fill[color=green!60!black] (t2) circle (0.8pt) node[left] (0up) {};
\fill[color=green!60!black] (t3) circle (0.8pt) node[left] (0up) {};
\fill[color=green!60!black] (t1s) circle (0.8pt) node[left] (0up) {};
\fill[color=green!60!black] (t3s) circle (0.8pt) node[left] (0up) {};
%
% visible edges of cube
\draw[
	 color=gray, 
	 opacity=0.4, 
	 semithick
	 ] 
	 (0,1,1) -- (0,1,0) -- (1,1,0) -- (1,1,1) -- (0,1,1) -- (0,0,1) -- (1,0,1) -- (1,0,0) -- (1,1,0)
	 (1,0,1) -- (1,1,1);
\end{tikzpicture}
%%%%%%%%%%%%%%%%%%%%%% 
\right)(1) \, . 
$$

%We have now provided all the data to state the main result of the present section. 

\begin{theorem}
\label{thm:TZGrayduals}
$\tz$ as defined above is a $\Bbbk$-linear Gray category with duals for every 3-dimensional defect TQFT~$\zz$. 
\end{theorem}

\begin{proof}
The axioms of Definition~\ref{def:Graycatduals} can be straightforwardly verified: 
\begin{enumerate}
\item For every 2-morphism~$X$ in $\tz$, the adjunction maps $\ev_X$, $\coev_X$ satisfy the Zorro moves by isotopy invariance in $\Bordd$ and by functoriality of~$\zz$. 
Furthermore $(-)^{\dagger\dagger}$ is obviously the identity, hence $\tz(u,v)$ is strictly pivotal for all $u,v \in \tz$, and pre- and post-$\sta$-composition is strictly pivotal as well. 
\item 
The axioms $\alpha^{\#\#} = \alpha$, 
$1_u\hash = 1_u$, 
$\coev_{1_u} = 1_{1_u}$, 
and $\tau_{1_u} = 1_{1_{1_u}}$ in Definition~\ref{def:Graycatduals}\,(ii.a) and (ii.b) hold manifestly in $\tz$, and the same is true for $( \alpha \sta \beta )\hash = \beta\hash \sta \alpha\hash$ and 
$\coev_{\alpha\sta\beta} = (1_\alpha \sta \coev_\beta \sta 1_\alpha\hash ) \fus \coev_\alpha$ in part (ii.c), compare \eqref{eq:coevalpha} and \eqref{eq:coevalpha2}. 
The remaining axioms of (ii.c) and (ii.d) again follow from isotopy invariance in $\Bordd$ and from functoriality of~$\zz$, by applying~$\zz$ to~\eqref{eq:invbility}, \eqref{eq:tautautau} and~\eqref{eq:twistcoev}. 
\end{enumerate} 
\end{proof}

In the special case when there are no surface defects, i.\,e.~$D_2 = \emptyset$, then by construction the tricategory $\tz$ reduces to a product of braided monoidal categories $\End_{\tz}(1_\alpha)$ for all $\alpha \in D_3$. 
On the other hand, in the presence of surface defects, a defect TQFT `glues together' closed TQFTs $\zz_{\alpha}: \Bord_{3} \rightarrow \Vect_\Bbbk$ non-trivially; see also Definition~\ref{def:defectextension}. 

\begin{remark}
It is natural to ask whether we can determine functorial properties of the assignment 
\begin{align} 
	\Big\{ \text{defect TQFTs } \Bordd \rightarrow \Vect_{\Bbbk} \Big\} 
	& \lra 
	\Big\{ \text{linear Gray categories with duals} \Big\} \nonumber \\
	\zz
	& \longmapsto
	\tz \, . \label{eq:assignmentT}  
\end{align}
A first observation is that this assignment is functorial in the most basic sense: a monoidal natural isomorphism $\eta: \zz \rightarrow \zz'$ canonically induces an 
equivalence of Gray categories with duals $\mathcal T_{\eta}: \tz \rightarrow  \tzp$ by applying $\eta$ to 
the vector spaces of 3-morphisms. Furthermore for composable  isomorphisms~$\eta$ and~$\eta'$ 
there are identities of functors $\mathcal T_{\eta \circ \eta'}= \mathcal T_{\eta} \circ \mathcal T_{\eta'}$ and 
$\mathcal T_{\id}=1$.

We end this section by sketching a universal property for $\tz$ which shows that functors to $\tz$ are classified by the `cubical part of $\zz$'. 
To a linear Gray category with duals~$\Cat{G}$, we associate a symmetric monoidal category  $\Cube(\Catpre{G})$ of `stratified $\Cat{G}$-decorated cubes', which covertly was key to our construction of $\tz$ all along. 
If we denote by $\Catpre{G}$ the pre-2-category that consists of the  objects, 1- and 2-morphisms of~$\Cat{G}$, then   $\Cube(\Catpre{G})$ is a version of the stratified little-disc operad:
its objects are disjoint unions of decorated boundaries of cubes as in \eqref{eq:CXXcube}, i.\,e.~the cubes' bottom and top faces are pre-2-category diagrams for $\Catpre{G}$. 
Basic morphisms in $\Cube(\Catpre{G})$ are 
Gray category diagrams $\Gamma$ for $\Cat{G}$ with little cubes removed around all inner vertices. Such a diagram $\Gamma$ is regarded as a morphism from the disjoint union of all inner cubes to the outer cube boundary. 
General morphisms are disjoint unions of basic morphisms. Composition is insertion of diagrams in inner boundary components, and the monoidal structure is given by disjoint union. 

There is a symmetric monoidal functor 
$$
\eval_{\Cat{G}}: \Cube(\Catpre{G}) \lra \Vect_\Bbbk
$$ 
which sends a cube $X \in \Cube(\Catpre{G})$ to the vector space of 3-morphisms in $\Cat{G}$ which is 
	given by $\Cat{G}(s_X, t_X)$, where $s_X$ and $t_X$ are the 2-morphisms on the bottom and top of the cube~$X$, respectively. 
On the disjoint union of cubes, $\eval_{\Cat{G}}$ is defined by the tensor product, and on morphisms it acts as the evaluation of Gray category diagrams.

If we restrict to the case where $\Cat{G}=\tz$ for some defect TQFT $\zz: \Bordd \rightarrow \Vect_{\Bbbk}$, then there is a canonical functor $\iota: \Cube(\tz^{\mathrm{pre}}) \rightarrow \Bordd$ such that the diagram 
\begin{equation}
\label{eq:TZtriangle}
%%%%%%%%%%%%%%%%%%%%%%%
\begin{tikzpicture}[
			     baseline=(current bounding box.base), 
			     >=stealth,
			     descr/.style={fill=white,inner sep=3.5pt}, 
			     normal line/.style={->}
			     ] 
\matrix (m) [matrix of math nodes, row sep=3.5em, column sep=2.0em, text height=1.5ex, text depth=1ex] {%
\Cube(\tz^{\mathrm{pre}}) && 
\\
\Bordd && \Vect_\Bbbk
\\
};
\path[font=\footnotesize] (m-1-1) edge[->] node[auto] {$ \eval_{\tz} $} (m-2-3);
\path[font=\footnotesize] (m-1-1) edge[->] node[left] {$ \iota $} (m-2-1);
\path[font=\footnotesize] (m-2-1) edge[->] node[below] {$ \zz $} (m-2-3);
\end{tikzpicture}
%%%%%%%%%%%%%%%%%%%%%%% 
\end{equation}
commutes. 
The functor $\iota$ can be described as `gluing the diagrams for the morphisms in $\tz$ in the strata of the cube'. 
%arXiv_v2
	%This can be made precise 
	We expect that this can be made precise 
by using a canonical framing for the strata in the diagrams. 
Note that by definition of $\tz$, the diagram \eqref{eq:TZtriangle} commutes in the strict sense. 

The diagram~\eqref{eq:TZtriangle} may suggest that $\tz$ is a Kan extension of $\eval_{\tz}$ along~$\iota$, thus providing an inverse to the assignment~\eqref{eq:assignmentT}.
However, this is not true in general, and Turaev-Viro theory viewed as a defect TQFT~$\zz$ with trivial defect data is a counter-example: 
here, the tricategory $\tz$ simply consists of (multiples of) identities (cf.~Section~\ref{sec:examples}) and $\eval_{\tz}$ is the trivial functor, so for non-trivial $\zz$ the triangle \eqref{eq:TZtriangle} is not a Kan diagram. 

We expect instead that the universal property of $\tz$ is as follows:
For every linear Gray category with duals $\Cat{G}$ and every functor of pre-2-categories with duals $F: \Catpre{G} \rightarrow \Fdp \K^{\D}$, we have a functor $\iota \circ F: \Cube(\Catpre{G}) \rightarrow 
\Bordd$.  
There is a canonical bijection between natural transfomations $\eta: \eval_{\Cat{G}} \rightarrow \zz \circ \iota \circ F$ and extensions $\widehat{F}: \Cat{G} \rightarrow \tz$ of 
$F$ to a functor of Gray categories with duals ($\widehat{F}$ agrees with $F$ on objects, 1- and 2-morphisms).
However, proving a statement like this requires concepts beyond the scope of the present article and will be left for future work. 
\end{remark}

\section{Examples}
\label{sec:examples}

After the abstract structural analysis of 3-dimensional defect TQFTs, in this section we present three classes of examples.
As the first example we will find that in our setting even the trivial TQFT leads to an interesting tricategory: 
it is illustrative to see how the trivial defect functor adds relations which produce a 3-groupoid from the free pre-2-category $\Fdp \K^{\D}$ associated to given defect data~$\D$.

In the second example we consider the first known 3-dimensional TQFT, the Reshetikhin-Turaev theory \cite{ResTur} associated to a modular tensor category~$\mathcal C$.
It is expected from the original construction that this is automatically a defect TQFT with only 1-dimensional defects. 
We show that it indeed fits into our framework, and that the corresponding Gray category reproduces~$\mathcal C$ regarded as a tricategory. 

As a third source of examples we interpret the homotopy quantum field theories (HQFTs) of Turaev-Virelizier \cite{Hqft1} as defect TQFTs. 
The version of HQFT that we are interested in takes as input datum a spherical fusion category~$\mathcal A$ which is graded by a finite group~$G$ and produces a TQFT for 3-manifolds equipped with a fixed homotopy class of 
a map to $K(G,1)$.
We shall show that these HQFTs give rise to two different defect TQFTs: when restricted to bordisms without stratifications, one of our TQFTs reproduces the trivial TQFT, and the other one recovers the Turaev-Viro TQFT 
\cite{TurVir,BarWes} associated to the neutrally graded component $\mathcal A_{1}$. 
Furthermore we explicitly describe the tricategories $\tz$ associated to the above two types of defect TQFTs~$\zz$. 

\medskip

Recovering ordinary TQFTs when restricting a defect TQFT to trivially stratified bordisms is of course a more general phenomenon, which we conceptualise in the following

\begin{definition}
\label{def:defectextension}
A defect TQFT $\zzd: \Bordd \rightarrow \Vect_\Bbbk$ is called a \textsl{defect extension} of a set of TQFTs $\{ \zz_{\alpha}: \Bord_{3} \rightarrow \Vect_\Bbbk \}_{\alpha \in D_3}$, if for all $\alpha \in D_3$, $\zz_{\alpha}$ is isomorphic to the restriction of $\zzd$ to bordisms with only one stratum labelled by $\alpha \in D_3$. 
\end{definition}

\subsection{Tricategories from the trivial defect functor}
\label{subsec:case-trivial-defect}

If we start with arbitrary defect data $\D$, there is always a trivial defect TQFT 
\begin{equation}
\label{eq:triv}
  \zztriv_{\D}: \Bordd \lra \Vect_{\Bbbk}
\end{equation}
that associates to every object $\Sigma$ in $\Bordd$ the vector space $\zztriv_{\D}(\Sigma)=\Bbbk$ and to every bordism the identity map on~$\Bbbk$.
It is however illuminating to see how the 3-morphisms in $\mathcal T_{\zztriv_{\D}}$ relate the 1- and 2-morphisms of $\Fdp \K^{\D}$. 
We analyse the structure of the resulting tricategory $\mathcal T_{\zztriv_{\D}}$ from top to bottom in the degree of the morphisms with the following general fact:

\begin{lemma}
\label{lemma:structure-Gray-Invertible}
Let $\Cat{G}$ be a linear Gray category with duals with the property that between any two 2-morphisms~$X$ and~$Y$, the space of 3-morphisms 
$\Cat{G}(X,Y)$ is 1-dimensional and the vertical composition of 3-morphisms is never the zero map. 
Then all 
%arXiv_v2: 
	%morphisms
	1-morphisms 
in $\Cat{G}$ are (weakly) invertible, and all 2-morphisms are isomorphic. 
\end{lemma}

\begin{proof}
First we show that any non-zero 3-morphism in 
%arXiv_v2
	%$\mathcal T_{\zztriv_{\D}}$ 
	$\Cat{G}$ 
is an isomorphism: 
For any such $\Phi: X \rightarrow Y$ we can choose a non-zero 3-morphism $\Psi:Y \rightarrow X$ in the 1-dimensional vector space $\Cat{G}(Y,X)$. 
Since $\Cat{G}$ is linear and the composition is non-zero by assumption,  the composition $\Phi \circ \Psi$ is a non-zero multiple of the identity on $Y$. 
Similarly, 
$\Psi \circ \Phi$ is  the identity on $X$ up to a non-zero scalar.
 
We conclude that every 2-morphism in 
%arXiv_v2: 
	%$\mathcal T_{\zztriv_{\D}}$ 
	$\Cat{G}$ 
is an equivalence, since the duality morphisms relating a 2-morphism $X$ with its dual $X^{\dagger}$ are isomorphisms by the previous statement. 
Also all 1-morphisms in 
%arXiv_v2
	%$\mathcal T_{\zztriv_{\D}}$ 
	$\Cat{G}$ 
are equivalences, again because the duality morphisms relating a 1-morphism $\alpha$ with its dual $\alpha^{\#}$ are isomorphisms by the conclusion in the previous sentence. 
\end{proof}

In the case of~$\zztriv_{\D}$ we have $\zztriv_{\D}(\Sigma)=\Bbbk$ for every stratified 2-manifold $\Sigma$ 
with scalar multiplication as composition, hence the above lemma applies.  

\medskip

Next we describe the tricategory $\mathcal T_{\zztriv_{\D}}$ in terms of a Gray category $\mathcal T_{\widehat{\D}}$ associated more directly with $\D$. For this we first consider the following equivalence relations on $\Fdp \K^{\D}$. 
We call two parallel 1-morphisms in $\Fdp \K^{\D}$ equivalent if there exists some 2-morphism in $\Fdp \K^{\D}$ between them. 
In terms of this we define a Gray category with duals $\mathcal T_{\widehat{\D}}$ whose objects are the objects in $\Fdp \K^{\D}$, whose 1-morphisms are equivalences classes of 1-morphisms in $\Fdp \K^{\D}$, 
2-morphisms are only the identities, and 3-morphisms in $\mathcal T_{\widehat{\D}}$ are multiplies of identities. 

\begin{proposition}
There is a strict equivalence of Gray categories with duals 
$$
\mathcal T_{\zztriv_{\D}} \stackrel{\cong}{\lra} \mathcal T_{\widehat{\D}} \, . 
$$
\end{proposition}

\begin{proof}
We show that the projection to equivalence classes gives a strict functor of Gray categories  $\pi: \mathcal T_{\zztriv_{\D}} \rightarrow \mathcal T_{\widehat{\D}}$ which is an equivalence. 
On objects, $\pi$ is the identity. 
On 1-morphism $\alpha$, $\pi(\alpha)=[\alpha]$ is the projection to the corresponding class. 2-morphisms get mapped to the identity, and 3-morphisms $\lambda \in \Bbbk= \Hom(X,Y)$ get mapped to~$\lambda$ times the corresponding identity. 
By definition of the equivalence classes of 1-morphisms, this is a well-defined strict equivalence of Gray categories
as in Definition~\ref{def:Grayequiv}: 
$\pi$ is clearly fully faithful on 3-morphisms, essentially surjective on 2-morphisms, biessentially surjective on 1-morphisms, and triessentially surjective on objects. 
\end{proof}

\subsection{Reshetikhin-Turaev theory as 3-dimensional defect TQFT}
\label{subsec:resh-tura-as}

Let $\mathcal C$ be a modular tensor category, i.\,e.~a non-degenerate spherical braided fusion category, see e.\,g.~\cite{BakalovKirillov}. 
We can furthermore assume that $\mathcal C$ is strictly associative 
%arXiv_v2
	and for simplicity we assume that $\mathcal C$ is anomaly free, otherwise one would need to work with an extension of the bordism category.\footnote{We refer to \cite[Sect.\,IV.6]{tur} for the notion of anomaly and the relevant extension of the bordism category.} 
In this section we show that the associated Reshetikhin-Turaev (RT) TQFT $\zz^{\Cat{C}}_{\mathrm{RT}}$ in the formulation of \cite{BakalovKirillov} 
%arXiv_v2: 
	%is canonically 
	admits a formulation as a 
defect TQFT in our sense. 

\medskip

To fix our conventions, we write $\unit \in \Cat{C}$ for the monoidal unit, and we denote by $x^{\dagger}$ the right dual of an object $x \in \Cat{C}$. 
There are adjunction maps $\ev_{x}: x^{\dagger} \otimes x \rightarrow \unit$ and $\coev_{x}: \unit \rightarrow x \otimes x^{\dagger}$ satisfying the Zorro moves, and with the spherical structure $x^{\dagger}$ is also left dual to $x$. 

We first  recall  the construction of the value of  $\zz^{\Cat{C}}_{\textrm{RT}}$ on bordisms with embedded coloured ribbon graphs. The set of colours is  the set of objects of $\mathcal C$.
On the boundary of such a bordism, there are  \textsl{$\Cat{C}$-marked surfaces $\Sigma$}, i.\,e.~closed surfaces with marked points labelled by objects in $\Cat{C}$ together with a sign (to distinguish 
incoming and outgoing ribbons) and a choice of a 
%arXiv_v2
	%normal direction 
	non-zero tangent vector 
(to determine a ribbon). 
The state space $\zz^{\Cat{C}}_{\textrm{RT}}(\Sigma)$ of such a $\Cat{C}$-marked surface $\Sigma$ is defined by choosing a `standard marked surface' $\Sigma_{t}$ with an equivalence $\phi: \Sigma\rightarrow \Sigma_{t}$
of $\Cat{C}$-marked surfaces, i.\,e.~$\phi$ is a homeomorphism mapping marked points to marked points and preserving the labels.

In the following we will need the precise definition of $\Sigma_{t}$ only in the case where~$\Sigma$ is a sphere -- see e.\,g.~\cite[Sect.\,4.4]{BakalovKirillov} for the general definition. A \textsl{standard marked sphere $S_{t}$} is defined for a $\Cat{C}$-tuple $t=((c_{1},\eps_{1}), (c_{2},\eps_{2}) \ldots, (c_{m},\eps_{m}))$ for some $m \in \N$ with objects $c_{i} \in \Cat{C}$ and $\eps_{i} \in \{ \pm \}$. 
The sphere $S_{t}$ is constructed from a unit cube $C_t$ whose top boundary is decorated with~$m$ points $p_{i}$ along the line from $(\tfrac{1}{2}, 1, 1)$ to $(\tfrac{1}{2}, 0, 1)$ 
%arXiv_v2:   
	together with unit tangent vectors pointing in the positive $y$-axis. 
Each point $p_{i}$ is labelled with $(c_{i},\eps_{i})$ while the remaining faces of $C_t$ are undecorated. 
$S_{t}$ is the corresponding sphere obtained by projecting as in Section~\ref{subsec:tricatfromZ}.  
For example, if $t =((c_{1},\eps_{1}), (c_{2},\eps_{2}),(c_{3},\eps_{3}))$ we have
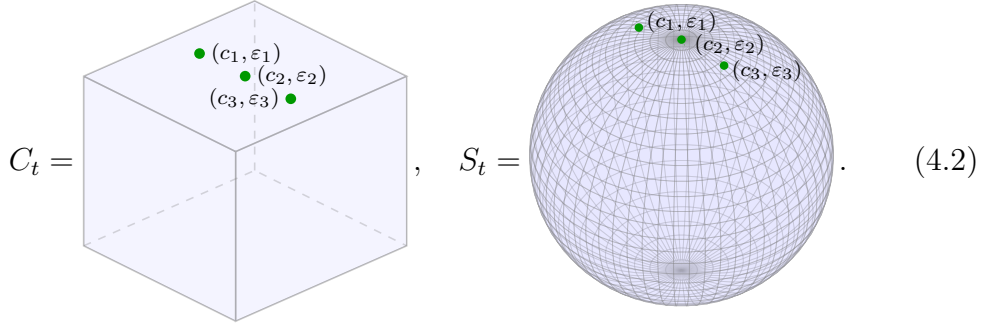
\be\label{eq:ExplRT}
C_t = 
%%%%%%%%%%%%%%%%%%%%%% 
\begin{tikzpicture}[thick,scale=2.5,color=blue!50!black, baseline=0.0cm, >=stealth, 
				style={x={(-0.9cm,-0.4cm)},y={(0.8cm,-0.4cm)},z={(0cm,0.9cm)}}]
% visible edges of cube
\draw[
	 color=gray, 
	 opacity=0.4, 
	 semithick
	 ] 
	 (0,1,1) -- (0,1,0) -- (1,1,0) -- (1,1,1) -- (0,1,1) -- (0,0,1) -- (1,0,1) -- (1,0,0) -- (1,1,0)
	 (1,0,1) -- (1,1,1);
% to cut off unwanted edges: 
\clip (1,0,0) -- (1,1,0) -- (0,1,0) -- (0,1,1) -- (0,0,1) -- (1,0,1);
% 3-stratum: 
\fill [blue!20,opacity=0.2] (1,0,0) -- (1,1,0) -- (0,1,0) -- (0,1,1) -- (0,0,1) -- (1,0,1);
% invisible edges of cube: 
\draw[
	 color=gray, 
	 opacity=0.3, 
	 semithick,
	 dashed
	 ] 
	 (1,0,0) -- (0,0,0) -- (0,1,0)
	 (0,0,0) -- (0,0,1);
%%%
% top points: 
\coordinate (t1) at (0.5, 0.2, 1);
\coordinate (t2) at (0.5, 0.5, 1);
\coordinate (t3) at (0.5, 0.8, 1);
%
% top 0-strata:
\fill[color=green!60!black] (t1) circle (0.8pt) node[right,black] (0up) {$\text{\scriptsize{$(c_{1},\eps_{1})$}}$};
\fill[color=green!60!black] (t2) circle (0.8pt) node[right,black] (0up) {$\text{\scriptsize{$(c_{2},\eps_{2})$}}$};
\fill[color=green!60!black] (t3) circle (0.8pt) node[left,black] (0up) {$\text{\scriptsize{$(c_{3},\eps_{3})$}}$};
%
% visible edges of cube
\draw[
	 color=gray, 
	 opacity=0.4, 
	 semithick
	 ] 
	 (0,1,1) -- (0,1,0) -- (1,1,0) -- (1,1,1) -- (0,1,1) -- (0,0,1) -- (1,0,1) -- (1,0,0) -- (1,1,0)
	 (1,0,1) -- (1,1,1);
\end{tikzpicture}
\, , \quad
S_t
=
%%%%%%%%%%%%%%%%%%%%% 
 \tdplotsetmaincoords{50}{135}
 \begin{tikzpicture}[scale=2.0,opacity=0.3,tdplot_main_coords,fill opacity=0.4, baseline=-0.2cm, >=stealth]
 %
 % SPHERE START
 \pgfsetlinewidth{.1pt}
 \tdplotsetpolarplotrange{0}{180}{0}{360}
 \tdplotsphericalsurfaceplot{60}{60}%[parametricfill]
 {1}%{sqrt(15/2)*sin(\tdplottheta)*cos(\tdplottheta)}%
 {gray}%{transparent!0}%
 {blue!14}%{\tdplotphi}%
     {}%
     {}%
     {}%
     %	\coordinate (O) at (0 ,0 ,0);
 % SPHERE END	
 %
 \tdplotsetcoord{P1}{1}{20}{280}
 \tdplotsetcoord{P2}{1}{0}{315}
 \tdplotsetcoord{P3}{1}{20}{100}
 \tdplotsetcoord{Q1}{1}{25}{270}
 \tdplotsetcoord{Q2}{0.95}{1}{150}
 \tdplotsetcoord{Q3}{1}{21}{92}
 \fill[color=green!60!black, opacity=1] (P1) circle (0.8pt) node[left] (0up) {};
 \fill[color=green!60!black, opacity=1] (P2) circle (0.8pt) node[left] (0up) {};
 \fill[color=green!60!black, opacity=1] (P3) circle (0.8pt) node[left] (0up) {};
 \fill[color=green!60!black, opacity=1] (Q1) circle (0pt) node[right,black] (0up) {$\text{\scriptsize{$(c_{1},\eps_{1})$}}$};
 \fill[color=green!60!black, opacity=1] (Q2) circle (0pt) node[right,black] (0up) {$\text{\scriptsize{$(c_{2},\eps_{2})$}}$};
 \fill[color=green!60!black, opacity=1] (Q3) circle (0pt) node[right,black] (0up) {$\text{\scriptsize{$(c_{3},\eps_{3})$}}$};
 \end{tikzpicture}
%%%%%%%%%%%%%%%%%%%%% 
\, . 
\ee

To a standard marked sphere $S_{t}$ with $t=((c_{1},\eps_{1}), (c_{2},\eps_{2}) \ldots, (c_{m},\eps_{m}))$, the RT construction associates the vector space $F(S_{t} )=\Hom_{\Cat{C}}(\unit, c_{t})$ with 
$$
c_{t}=c_{1}^{\eps_{1}} \otimes \dots \otimes c_{m}^{\eps_{m}}
\, , \quad \text{where} \quad
c^{+}=c \quad \text{and} \quad c^{-}=c^{\dagger} \, . 
$$
Moreover if we write $\cc{t}=((c_{m},-\eps_{m}), (c_{m-1},-\eps_{m-1}), \dots, (c_{1},-\eps_{1}))$ for the reversed tuple, then the evaluation maps in $\Cat{C}$ define a non-degenerate pairing $F(S_{t}) \otimes F(S_{\cc{t}}) \rightarrow \C$, hence an isomorphism $F(S_{\cc{t}})^{*} \cong F(S_{t})$. 

To an equivalence $f: S_{t} \rightarrow S_{t'}$ of $\Cat{C}$-marked spheres, 
%arXiv_v2
	i.\,e.~$f$ is a diffeomorphism that preserves the marked points and the tangent vectors up to a scalar, 
the RT construction assigns a linear isomorphism $F(f): F(S_{t}) \rightarrow F(S_{t'})$. 
The assignment is compatible with composition of equivalences and the identities, hence the set of equivalences forms a projective system. 
By definition, the RT state space of a $\Cat{C}$-marked sphere that is homeomorphic to $S_{t}$ is the projective limit of this projective system.

Finally, let $M$ be a ribbon tangle in $S^{2} \times [0,1]$ from 
$S_{t_{1}}$ to $S_{t_{2}}$ 
for some $\Cat{C}$-tuples $t_{1}$ and $t_{2}$. The RT invariant of ribbon tangles then 
defines a morphism $F(M): c_{t_{1}} \rightarrow c_{t_{2}}$ in $\Cat{C}$, and the TQFT construction provides an element $\zz_{\textrm{RT}}^{\Cat{C}}(M) \in \Hom_{\C}(F(S_{t_{1}}), F(S_{t_{2}}))$ which is given by 
\begin{equation}
  \label{eq:RT-Ribbon}
  \zz_{\textrm{RT}}^{\Cat{C}}(M)(f)=F(M) \circ f \quad \text{for} \quad f \in F(S_{t_{1}}) \, ,
\end{equation}
see e.\,g.~\cite[Ex.\,4.5.2]{BakalovKirillov}.

\medskip

To understand Reshetikhin-Turaev theory as a defect TQFT, 
we consider the following defect data $\D^{\mathcal C}$ associated to the modular tensor category~$\mathcal C$.
For $D^{\mathcal C}_{3}$ and $D^{\mathcal C}_{2}$ we take the singelton sets $\{\ast\}$. 
For any choice of ordered tuples of elements in $D^{\mathcal C}_{2} $, we want to allow all objects in~$\Cat{C}$ as preimages of the folding map. 
%arXiv_v2
	To single out the ribbon direction, we allow only for configurations with precisely one adjacent 2-stratum around the 1-stratum that is ordered positively relative to the orientation of the 1-stratum. Additionally there can be an arbitrary finite number of adjacent 2-strata with negative relative orientation. 
Thus we define 
$$
%arXiv_v2: 
	%D^{\mathcal C}_{1} = \Obj(\mathcal C) \sqcup \bigsqcup_{m\in\Z_{+}} \big( (\{\ast\} \times \{ \pm \}) \times \dots \times (\{\ast\} \times \{ \pm \}) \big)/C_m \, , 
	D^{\mathcal C}_{1} = \Obj(\mathcal C) \times  \bigsqcup_{m\in\Z_{+}} \Big( (\{\ast\} \times \{+ \}) \times  (\{\ast\} \times \{- \}) \times \dots \times (\{\ast\} \times \{ - \}) \Big)/C_m \, , 
$$
and the folding map 
%arXiv_v2: 
	% $$
	% f:  
	% \Obj(\mathcal C) \sqcup \bigsqcup_{m\in\Z_{+}}  \big(\{\ast\} \times \{ \pm \}\big)^{\times m}/C_m
	% \lra 
	% \bigsqcup_{m\in\Z_{+}} \big(\{\ast\} \times \{ \pm \}\big)^{\times m}/C_m 
	% $$ 
	$f$ 
%arXiv_v2: 
	%is simply set to be the projection to the second factor in the disjoint union above
	simply projects to the second factor, i.\,e.\ $f(x,\Sigma) = \Sigma$ for all $(x,\Sigma) \in D^{\mathcal C}_{1}$. 
	
We emphasise that there is no label in $D_{1}^{\Cat{C}}$ for a single 1-stratum that is detached from any 2-stratum. 
%arXiv_v2: 
	% This will be important when we next define a map from stratified bordisms in $\Bordd$ to bordisms with embedded coloured ribbon graphs, in order to interpret the RT construction as a defect TQFT in our sense. 
	%
	%For a bordism~$M$ in $\Bord_{3}^{\textrm{def}}(\D^{\Cat{C}})$ with stratification $M=M_3 \cup M_2 \cup M_1$ and a 1-stratum $e$ in the submanifold $M_{1}$, recall from~\eqref{eq:m1map} the cyclic set $m_1(e)$ of germs of oriented 2-strata incident on~$e$. The elements of $m_1(e)$ also agree with the defect labels under the folding map~$f$. Now we make a choice of  a 2-stratum $r_{e} \in m_1(e)$ and consider a normal vector field on~$e$ tangent to $r_{e}$. This defines a thickening of $e$ to a ribbon. We do this for all 1-strata in $M_{1}$ and then consider the embedding of the thusly obtained ribbons in~$M$ while forgetting the embedding of the 2-strata in $M_{2}$. Since every 1-stratum lies in the boundary of at least one 2-stratum in $M_{2}$, by definition of $\D^{\Cat{C}}$ we thus obtain an embedded coloured ribbon graph in~$M$ by keeping the colours and orientations from $M_{1}$.
	Since every 1-stratum lies in the boundary of precisely one 2-stratum in $M_{2}$ with positive relative orientation, thickening the 1-stratum in the direction of this 2-stratum we obtain an embedded coloured ribbon graph in~$M$ by keeping the colours and orientations from $M_{1}$.

We have thus arrived at our first non-trivial class of examples for defect TQFTs: 

\begin{proposition}
By evaluating the Reshetikhin-Turaev TQFT on bordisms with ribbon graphs constructed from stratified bordisms as above, one obtains for every modular tensor category $\mathcal C$ a 3-dimensional defect TQFT
  \begin{equation}
    \label{eq:RT}
    \zz^{\mathcal C}: \Bord_{3}(\D^{\mathcal C}) \lra \Vect_\C .
  \end{equation}
\end{proposition}

\medskip

Next we show that the Gray category with duals $\mathcal T_{\zz^{\mathcal C}}$ associated to $\zz^{\mathcal C}$ is equivalent  to 
$\mathcal C$ regarded as a Gray category with duals $\underline{\underline{\Cat{C}}}$.
Indeed, for any modular tensor category~$\Cat{C}$, the Gray category $\underline{\underline{\Cat{C}}}$ by definition has only one object and a single 1-morphism, while the 2- and 3-morphisms of $\underline{\underline{\Cat{C}}}$ are the objects and morphisms of $\Cat{C}$, respectively. 
Both horizontal composition~$\fus$ and the Gray product~$\sta$ in $\underline{\underline{\Cat{C}}}$ are the tensor product in~$\Cat{C}$ (recall that we assumed $\Cat{C}$ to be strictly associative), and the tensorator is given by the braiding of $\Cat{C}$. 

\begin{proposition}
There is a canonical strict equivalence of Gray categories with duals  
$$
\mathcal T_{\zz^{\mathcal C}} \stackrel{\cong}{\lra} \underline{\underline{\Cat{C}}} \, . 
$$
\end{proposition}
\begin{proof}
The equivalence $F:\mathcal T_{\zz^{\mathcal C}} \rightarrow \underline{\underline{\Cat{C}}}$ is defined as follows: it sends the object of $\mathcal T_{\zz^{\mathcal C}}$ to the object of $\underline{\underline{\Cat{C}}}$ and every 1-morphism of $\mathcal T_{\zz^{\mathcal C}}$ to the single 1-morphism in~$\underline{\underline{\Cat{C}}}$. 
Given a 2-morphism $X$ in $\mathcal T_{\zz^{\mathcal C}}$, which we identify with the string diagram on the bottom of the cube~$X$ as in Section \ref{subsec:tricatfromZ}, we forget the lines and project the string diagram to the $y$-axis. 
This gives a tuple of $\Cat{C}$-decorated points together with a 
sign, hence a $\Cat{C}$-tuple $X_{t}$. By definition the functor $F$ assigns the object $c_{X_{t}} \in \Cat{C}$ to~$X$. 

Consider next the space of 3-morphisms $\Hom_{\mathcal T_{\zz^{\mathcal C}}}(X,Y)$ between 2-morphisms $X,Y$ as defined in~\eqref{eq:3HomTZ}. 
After forgetting the 1-strata on the sphere $S_{X,Y}$ which is obtained from gluing $X$ and $Y$ together as in Section \ref{subsec:tricatfromZ}, the resulting $\Cat{C}$-marked sphere is homeomorphic to the standard $\Cat{C}$-marked sphere 
$S_{Y_{t} \cup \cc{X_{t}}}$ as  defined at the beginning of the present section. Here we use the symbol $\cup$ to denote the concatenation of tuples. 
Pick an arbitrary homeomorphism $\phi: S_{X,Y} \rightarrow
 S_{Y_{t} \cup \cc{X_{t}}}$. By definition, we have the associated vector space $F( S_{Y_{t} \cup \cc{X_{t}}})= \Hom_{\Cat{C}}(\unit, c_{Y} \otimes c_{X}^{\dagger})$, and the construction of the RT state spaces gives a chain of isomorphisms
\begin{equation}
  \label{eq:3-morphRT}
  \mathcal T_{\zz^{\mathcal C}}(\alpha,\beta)=\mathcal \zz^{\Cat{C}}(S_{X,Y}) \stackrel{\phi_{*}}{\longrightarrow} \Hom_{\Cat{C}}(\unit, c_{Y} \otimes c_{X}^{\dagger}) 
\cong \Hom_{\Cat{C}}(c_{X},c_{Y})
\end{equation}
where $\phi_{*}$ denotes the cone isomorphism from the projective limit, while the last isomorphism is obtained from duality in~$\Cat{C}$. 

The value of $F$ on 3-morphisms is defined by composing the isomorphisms in~\eqref{eq:3-morphRT}. By definition of 
$\zz^{\Cat{C}}_{\textrm{RT}}(S_{X,Y})$ as a projective limit, this is independent of the choice of the homeomorphism $\phi$. $F$ is functorial on the level of 3-morphisms due to \eqref{eq:RT-Ribbon}. By construction, it is functorial with respect to 2-morphisms and 1-morphisms. Hence it is a functor of Gray categories. It follows directly that it is strictly compatible with 
the duals on the level of 1- and 2-morphisms. By the isomorphism~\eqref{eq:3-morphRT}, it is fully faithful. 
The essential surjectivity for objects, 1- and 2- morphisms is also clear. Hence, $F$ is an equivalence of Gray categories with duals as in Definition~\ref{def:strictFGraydual}. 
\end{proof}

\subsection{HQFTs as defect TQFTs}
\label{subsec:htqfts-as-defect}

In this section we construct two defect extensions for every $G$-graded spherical fusion category $\Cat{A}$. 
On the one hand, $\Cat{A}$ allows us to define a non-trivial defect extension $\zzAtriv$ of the trivial TQFT~\eqref{eq:triv}. 
On the other hand, 
it yields a defect extension $\mathcal{Z}^{\Cat{A}}$ of the Turaev-Viro TQFT \cite{TurVir} associated with the neutrally graded component $\Cat{A}_{1}$. 
Both constructions rely on the HQFT construction of \cite{Hqft1}.

\medskip

Let $G$ be a finite group. A \textsl{spherical $G$-graded fusion category} $\Cat{A}$ is a spherical fusion category over $\C$ that is $G$-graded, $\Cat{A}= \bigoplus_{g \in G}\Cat{A}_{g}$, such that the tensor product is $G$-linear. 
Our conventions for duals  $x^*$ of objects $x \in \Cat{A}$ are as in Section~\ref{subsec:resh-tura-as}. 
Additionally we choose representatives $x_i$ for the isomorphism classes of simple objects, labelled  by a finite set $ I \ni i$. 
The $G$-grading induces a decomposition $I=\bigcup_{g \in G}I_g$ with finite index sets $I_g$ for simple objects in $\Cat{A}_g$.

For our construction of a defect extension of Turaev-Viro theory below, we will need to consider stratifications of 3-manifolds for which every 3-stratum is an open 3-ball. 
However, this is in general not possible without the presence of 0-strata in the interior -- contrary to our setup with stratified bordisms of standard form in Section~\ref{subsec:defectbord}. 
Hence we specify the type of \textsl{allowed stratification} for this section in case of  a closed 3-manifold $M$:
\begin{enumerate}
\item 
$M$ is a stratified manifold as in Definition~\ref{def:stratman}.
\item 
Every 1-stratum in~$M$ has a neigbourhood of standard form, cf.~\eqref{eq:starlikecylinder}.
\item 
Every 0-stratum in~$M$ has a neigbourhood whose image $N$ in $\R^{3}$ is a `linearly filled sphere': $N$ is a stratified ball whose stratification is induced by the stratification of the boundary $\partial N$ in the sense that each interior stratum is obtained by connecting the points in each boundary stratum with the centre of~$N$ by straight lines. 
For example
$$
N = 
%%%%%%%%%%%%%%%%%%%%% 
 \tdplotsetmaincoords{50}{135}
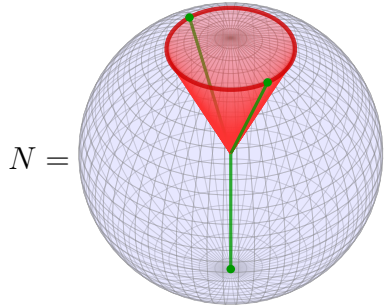
 \begin{tikzpicture}[scale=2.0,opacity=0.3,tdplot_main_coords,fill opacity=0.4, baseline=-0.2cm, >=stealth]
 %
 % SPHERE START
 \pgfsetlinewidth{.1pt}
 \tdplotsetpolarplotrange{0}{180}{0}{360}
 \tdplotsphericalsurfaceplot{60}{60}%[parametricfill]
 {1}%{sqrt(15/2)*sin(\tdplottheta)*cos(\tdplottheta)}%
 {gray}%{transparent!0}%
 {blue!14}%{\tdplotphi}%
     {}%
     {}%
     {}%
     %	\coordinate (O) at (0 ,0 ,0);
 % SPHERE END	
 %
 \coordinate (O) at (0 ,0 ,0);
 \tdplotsetcoord{P1}{1}{25}{265}
 \tdplotsetcoord{P3}{1}{25}{80}
 \tdplotsetcoord{P4}{1}{180}{245}
 %
 % green line to south pole
 \draw[green!60!black, very thick, opacity=0.8] (O) -- (P4); 
 %
 % green line in back
 \draw[green!60!black, very thick, opacity=0.8] (O) -- (P1); 
 %
 % red cone: 
 \foreach \q in { 1, 2.5, ..., 360 } % 45 is right at the back
 {
 	\tdplotsetcoord{Q}{1}{25}{\q}
 	\draw[red!80, semithick, opacity=0.3] (O) -- (Q); 
 }
 %
 % circle on top: 
 	\foreach \angle in { 65 } % 45 is right at the back
 	{
 		\tdplotsinandcos{\sintheta}{\costheta}{\angle}%
 		\coordinate (P) at (0,0,\sintheta);
 		%draw the circle in the main frame: 
 		\tdplotdrawarc[red!80!black, ultra thick, opacity=0.8]{(P)}{\costheta}{0}{360}{}{}
 	}
 %
 % green line in front
 \draw[green!60!black, very thick, opacity=0.8] (O) -- (P3); 
 %
 % green vertices: 
 \fill[color=green!60!black, opacity=1] (P1) circle (0.8pt) node[left] (0up) {};
 \fill[color=green!60!black, opacity=1] (P3) circle (0.8pt) node[left] (0up) {};
 \fill[color=green!60!black, opacity=1] (P4) circle (0.8pt) node[left] (0up) {};
 \end{tikzpicture}
%%%%%%%%%%%%%%%%%%%%% 
\, . 
$$
\end{enumerate}

Similary, if $M$ has a boundary, for~$M$ to have an \textsl{allowed stratification} we require that $(M, \partial M)$ is a stratified manifold according to Definition \ref{def:stratmanbdry} such that the conditions above are satisfied for the 1- and 0-strata in the interior of~$M$. 
Later we will show that in this case the stratification of~$M$ is of the type that is used in the construction of the HQFT in \cite{Hqft1}. 

\medskip

Next we describe a set of defect data $\D^{G}$ for every finite group $G$. 
From $\D^{G}$ we will construct the two defect extensions $\zzAtriv$ and $\mathcal{Z}^{\Cat{A}}$ mentioned at the beginning of this section. 

We define the three sets in $\D^G$ to be 
$$
D_3^G = \{\ast \} \, , \quad
D_2^G = G 
$$
and
\begin{align}
\label{eq:D1-group}
D_1^G = \bigsqcup_{m\in \Z_+} \Big\{   \big[   (g_{1}, \eps_{1}) , (g_{2}, \eps_{2}), \dots, (g_{m}, \eps_{m}) \big] \; \Big| \;  \prod_{i =1}^{m} g_{i}^{\eps_{i}}=1 \Big\}
\end{align}
where the equivalence classes $[ - ]$ of tuples in $G \times \{\pm \}$ are taken with respect to cyclic permutations. 
The source and target maps $s,t: D_2^G \to D_3^G$ are trivial, and the folding map~$f$ is defined to be the identity on the underlying cyclic product of group elements and signs.  
By construction, $f$ never maps to the set $D_{3}^G$.

The labelling of the strata of a stratified bordism~$M$ by $\D^{G}$ is as in the definition of a $\D^{G}$-labelled defect bordism in Section~\ref{subsec:decoratedbordisms} with no further labelling of the vertices in the interior. 
Concretely, a labelling consists of a decoration of all 2-strata~$r$ by elements $\ell(r) \in G$ such that for every 1-stratum~$e$ we have
  \begin{equation}
    \label{eq:condition-edge}
\ell(r_{1})^{\eps_{1}} \cdot \ell(r_{2})^{\eps_{2}} \cdot \ldots \cdot \ell(r_{m})^{\eps_{m}}=1 \, ,    
  \end{equation}
where $r_{i}$ denote the 2-strata that meet a small positive loop around $e$ starting with $r_{1}$, and $\eps_{i}=+$ if the orientation of the loop together with the orientation of $r_{i}$ is positive in~$M$, and $\eps_{i}=-$ otherwise. 
If condition \eqref{eq:condition-edge} is satisfied around each 1-stratum~$e$, then there exists a unique label map~$\ell$ of the 1-strata, otherwise~$\ell$ is not well-defined. 

In the following it will be convenient to consider  an oriented 1-stratum $e$ also with opposite orientation $\cc{e}$ and we denote by $E$ the set of 1-strata with both possible orientations. To avoid 
confusion, we call the elements  $e \in E$ \textsl{edges} (of which there are hence twice as many as 1-strata). 
As before~$M_2$  denotes the set of 2-strata of $M$.

The last ingredient we will need for our constructions of defect TQFTs is a type of `state sum labelling' \cite{TurVir}. 
For this, recall the index set $I=  \bigcup_{g \in G}I_g$ for the simple objects of the spherical $G$-graded fusion category~$\Cat{A}$. 
A \textsl{colouring}~$c$ of a decorated bordism $M$ consists of a function $c: M_2 \rightarrow I$, where for each 2-stratum $r \in M_2$ with label $\ell(r) \in G$ we  require that $c(r) \in I_{\ell(r)}$. 
Thus each 2-stratum gets assigned a simple object $x_{c(r)} \in \Cat{A}_{\ell(r)}$.
Such a colouring $c$ yields the following additional assignments for edges and vertices:

\begin{enumerate}
\item A small positive loop around an edge $e \in E$ determines 
the object $x_{c(r_{1})}^{\eps_{1}} \otimes \dots \otimes x_{c(r_{m})}^{\eps_{m}}$, 
where $r_{i}$ are the 2-strata adjacent to~$e$, and the signs $\eps_{i}$ are determined as above. 
As in Section~\ref{subsec:resh-tura-as}, we write~$\unit$ for the tensor unit of~$\Cat{A}$, and for $x \in \Cat{A}$ we set $x^{+}=x$ and $x^{-}=x^{\dagger}$.
For every cyclic permutation 
 $\sigma \in C_m$, 
there are isomorphisms 
\be\label{eq:Homisocyclic}
\Hom_{\Cat{A}} \big(\unit, x_{c(r_{1})}^{\eps_{1}} \otimes \dots \otimes  x_{c(r_{m})}^{\eps_{m}}\big)
\cong
\Hom_{\Cat{A}} \big(\unit, x_{c(r_{\sigma(1)})}^{\eps_{\sigma(1)}} \otimes \dots \otimes  x_{c(r_{\sigma(m)})}^{\eps_{\sigma(m)}}\big)
\ee
which follow from the spherical structure of $\Cat{A}$. 
The isomorphisms~\eqref{eq:Homisocyclic} define a projective system, and we assign its projective limit $H_{c}(e)$ to the edge~$e$. 

Furthermore, by the calculus for fusion categories, see e.\,g.~\cite{BakalovKirillov}, there is a non-degenerate pairing
\begin{equation}
  \label{eq:pairing-edg-spaces}
  \ev(e): H_{c}(e) \otimes H_{c}(\cc{e}) \lra \C
\end{equation}
which yields an isomorphism $H_{c}(\cc{e}) \cong H_{c}(e)^{*}$. 
We write $\coev(e)$ for the copairing $\C \rightarrow H_{c}(\cc{e}) \otimes H_{c}(e)$ associated to $\ev(e)$. 
\item 
For a vertex $\nu$ in a boundary component~$\Sigma$ of a stratified bordism~$M$ we define a vector space $H_{c}(\Sigma,\nu)$ analogously to step (i), using a small positive loop around $\nu$ in~$\Sigma$. 
Then we set 
\be\label{eq:HcSigma}
H_{c}(\Sigma)= \bigotimes_{\nu}H_{\nu}(\Sigma,\nu)
\ee
where the tensor product is over all vertices in~$\Sigma$. 
\item 
A vertex $\nu$ in the interior of $M$ defines a vector as follows. 
Consider the induced stratified manifold on a small ball $B_{\nu}$ around~$\nu$. 
The 1- and 0-strata on the boundary $\partial B_{\nu} \cong S^{2}$ determine an oriented coloured graph $\Gamma_{\nu}$ on the 2-sphere. 
The Reshetikhin-Turaev evaluation of graphs on 2-spheres \cite{ReshTurEval} for spherical tensor categories gives an element
 \begin{equation}
   \label{eq:RT-eval-sphere-graph}
   F_{\Cat{A}}(\Gamma_{\nu}) \in H_{c}(\partial B_{\nu})^{*} \, . 
 \end{equation}
\end{enumerate}

After these preliminaries we can move on to construct the two defect TQFTs $\zzAtriv$ and $\mathcal{Z}^{\Cat{A}}$.

\subsubsection[The defect TQFT $\zzAtriv$]{The defect TQFT $\boldsymbol{\zzAtriv}$}

In TQFTs of Turaev-Viro type one typically  considers `fine stratifications', where each 3-stratum is a 3-ball. 
We do not require this in our case, but using the construction of \cite{Hqft1} still leads to a defect TQFT 
$$
\zzAtriv: \Borddef(\D^G) \lra \Vect_\C
$$ 
which we describe in the following. 
It is particularly simple as we do not need to consider vertices in the interior of a bordism. 

We begin by defining the state spaces that $\zzAtriv$ assigns to objects in $\Borddef(\D^G)$. 
Note that each colouring~$c$ of the 2-strata of a defect bordism~$M$ induces a colouring of the 
1-strata of the boundary $\partial M$, since they are required to be boundaries of 2-strata.  
Hence we can set $H_{c}(\partial M)= \bigotimes_{\Sigma} H_{c}(\Sigma)$, where the tensor 
product is over all components $\Sigma$ of $\partial M$, taken with the induced orientations, and $H_{c}(\Sigma)$ as in~\eqref{eq:HcSigma}. 
With this we define 
\begin{equation}
  \label{eq:triv-state}
  \zzAtriv(\partial M)= \bigoplus_{c} H_{c}(\partial M) \, , 
\end{equation}
with the direct sum over all colourings $c$. 

Next we define a vector $\zzAtriv(M) \in \zzAtriv(\partial M)$. Note that the set of 1-strata is the union of the sets of open 1-strata which have boundary points in $\partial M$, and of closed 1-strata which form a loop. 
It follows from the definition of $H_{c}(\partial M)$ that there is a canonical isomorphism $H_{c}(\partial M) \cong \bigoplus_{\text{open} \; e}H_{c}(e) \otimes H_{c}(\cc{e})$, with the sum over all open 1-strata $e$. 
We use the copairing of  \eqref{eq:pairing-edg-spaces} to define 
\begin{equation}
  \label{eq:triv-cob}
  \zzAtriv(M)= \sum_{c} \Big(\prod_{\text{closed}\;  f}\dim_{\C} H_{c}(f)\Big) \bigotimes_{\text{open} \; e} \coev(e) 
  \; \in \; \zzAtriv(\partial M) \, .
\end{equation}
For   a bordism $M: \Sigma \rightarrow \Sigma'$, we can use that there is a canonical isomorphism 
 $H_{c}(\partial M) \cong H_c(\Sigma)^* \otimes H_c(\Sigma')$ to extract from~\eqref{eq:triv-cob} a linear map  $H_c(\Sigma) \to H_c(\Sigma')$. 

The factor of the dimensions of the spaces associated to loops in the interior is needed for the functoriality of the construction: 
If two bordisms are glued, two open 1-strata might be glued to a closed 1-stratum, which  yields a factor 
of the dimension of  the corresponding $\Hom$ space. 
It is clear that this construction assigns the identity to the cylinder over a decorated 2-manifold.
In total, it follows that we have defined a functor $\zzAtriv:\Borddef(\D^{G})\rightarrow \Vect_\C$ which is a defect extension of $\zztriv_\emptyset$.
By construction, the disjoint union of two bordisms is mapped to the tensor product of the corresponding linear maps. Hence, the functor is symmetric monoidal and thus a defect TQFT.

\paragraph{The associated tricategory $\boldsymbol{\mathcal T_{\zzAtriv}}$.}

We define a linear Gray category with duals $\mathcal G^{\Cat{A}}$ that is canonically associated to any $G$-graded fusion category $\Cat{A}$.
Then we argue that $\mathcal G^{\Cat{A}}$ is equivalent to $\mathcal T_{\zzAtriv}$.

 As objects, 1- and 2-morphisms of $\mathcal G^{\Cat{A}}$ we take the same as in  $\mathcal T_{\zzAtriv}$, i.\,e.~cylinders over the corresponding  diagrams in the free pre-2-category $\Fdp \K^{\D^{G}}$ of Section~\ref{subsubsec:bicatsduals}. 
Concretely, $\mathcal G^{\Cat{A}}$ has a single object, 1-morphisms are 3d diagrams with no lines and with surfaces labelled by elements of~$G$ (the surfaces are allowed to bend). 
2-morphisms are Gray category diagrams without vertices such that condition~\eqref{eq:condition-edge} is satisfied around every line. 

The space of 3-morphisms $\Hom_{\mathcal G^{\Cat{A}}}(X,Y)$ between parallel 2-morphisms~$X$ and~$Y$ is as follows. 
First we note that the notion of colouring carries over to the 2-strata of a 2-morphism in $\mathcal G^{\Cat{A}}$. 
Hence for every 1-stratum $X_j$ in the 2-morphism~$X$, and for every colouring $c$ of the 2-strata in~$X$, we have the vector space
$$
H_{c}(X_j) = \Hom_{\Cat{A}} \big(\unit, c(r_{1})^{\eps_{1}} \otimes \dots \otimes c(r_{m})^{\eps_{m}} \big)
$$
where $r_{i}$ are the 2-strata around $X_j$, and the signs $\eps_{i}$ are as in the previous section. 
Similarly, there are vector spaces $H_{c}(Y_k)$ for all 1-strata $Y_k$ in~$Y$, and we can define 
\begin{equation}
\label{eq:3-morph}
\Hom_{\mathcal G^{\Cat{A}}}(X,Y) := 
\bigoplus_{c} \Hom_{\C} \Big(\bigotimes_{j} H_{c}( X_{j}), \, \bigotimes_{k}H_{c}( Y_{k})\Big) \, . 
\end{equation}

Vertical composition of 3-morphisms in $\mathcal G^{\Cat{A}}$ is simply concatenation of linear maps, and horizontal composition is the tensor product $\otimes_\C$. 
The 2-functors $\alpha \sta (-)$ and $(-) \sta \alpha$ for 1-morphisms $\alpha$ are taken to act as the identity on 3-morphisms, and the tensorator is given by the symmetric flip in the tensor product. 

To complete the definition of $\mathcal G^{\Cat{A}}$, we have to specify its duals.
The duals of 1- and 2-morphisms of $\mathcal G^{\Cat{A}}$ are cylinders over the duals of $\Fdp \K^{\D^{G}}$, i.\,e.~taking duals corresponds to rotating diagrams. 
The evaluation map $\ev_X$ for a 2-morphism~$X$ is given by the canonical evaluation map in the space of 3-morphisms
\begin{equation}
\label{eq:3-morphi-eval}
\Hom_{\mathcal G^{\Cat{A}}} \big(X \otimes X^{\dagger},1\big) = 
\bigoplus_{c} \Hom_{\C} \Big(\bigotimes_{i} H_{c}( X_{i}) \otimes H_{c}( X_{i})^*  , \C \Big) \, . 
\end{equation}
Finally, the coevaluation $\coev_X$ is likewise obtained from the corresponding coevaluation map for vector spaces. 

\begin{proposition}
There is an equivalence of Gray categories with duals 
$$
\mathcal T_{\zzAtriv} \stackrel{\cong}{\lra} \mathcal G^{\Cat{A}} \, . 
$$
\end{proposition}

\begin{proof}
Objects, 1- and 2-morphisms agree for both Gray categories by definition. 
Thus we can define a functor $F: \mathcal T_{\zzAtriv} \rightarrow \mathcal G^{\Cat{A}}$ which is the identity on objects, 1- and 2-morphisms, and which relates the sets of 3-morphisms by the canonical isomorphism $V^{*}\otimes W \cong \Hom_{\C}(V,W)$, applied to all vector spaces~$V$ on the source and~$W$ on the target 2-morphism of a 3-morphism. 
With these isomorphisms on the 3-morphisms, we obtain a strict functor of Gray categories~$F$, which directly is an equivalence of Gray categories with duals. 
\end{proof}

\subsubsection[The defect TQFT $\zz^{\mathcal A}$]{The defect TQFT $\boldsymbol{\zz^{\mathcal A}}$}

%arXiv_v2
	%The defect TQFT that we construct in this section is a reformulation of the HQFT construction in \cite{Hqft1}. 
	In this section we formulate the HQFT construction in \cite{Hqft1} as a defect TQFT. 
The crucial differences to the defect extension $\zzAtriv$ above are as follows. 
Now we need to consider sufficiently fine stratifications of the bordisms. 
Generically this requires vertices~$\nu$ in the interior, to which we will assign the canonical vectors $F_{\Cat{A}}(\Gamma_{\nu})$ of~\eqref{eq:RT-eval-sphere-graph}. 
Finally, to obtain the vector spaces associated to stratified 2-manifolds, one performs a projection as is typical of state sum constructions. 

A stratification of a bordism $M$ is called \textsl{fine} if each 3-stratum of $M$ is diffeomorphic to an open 3-ball and each 2-stratum of $\partial M$ is diffeomorphic to an open $2$-disc. 
We shall need to  pass from a given bordism to a finer one: A \textsl{refinement} of a stratified bordism~$M$ is a stratification of~$M$ that is obtained 
by adding finitely many new strata to $(M,\partial M)$ such that~$M$ is still a stratified manifold with boundary with an allowed stratification as defined above. 
%arXiv_v2: 
	%In particular, vertices are allowed in the interior of~$M$. 
	In particular, vertices are allowed in the interior of the refinement of~$M$. Hence the refinement of a bordism in $\Bords$ may no longer be in this category, but it will be a bordism in $\Bordstrat_3$: 
	
\begin{lemma}
For every bordism in $\Bords$ there exists a fine 
%arXiv_v2: 
	%refinement
	refinement in $\Bordstrat_3$. 
\end{lemma}

\begin{proof}
We show that for any stratification of a bordism $M$, there is a  triangulation which is a refinement of the given stratification. 
It is clear that a triangulation is a fine stratification.
First  note that the standard form of the stratification of bordism in $\Bords$ implies that the closure of each stratum is a manifold (possibly with corners). 
Next we proceed inductively using the result that given an $n$-manifold~$M$ with $n \leqslant 3$ and a triangulation of its boundary, there exists a triangulation of~$M$ which restricts to the given triangulation on $\partial M$ \cite[Thm.\,10.6]{Munkres}.

To make use of this fact, we first pick a triangulation of the closure of all 1-strata of $M$ and $\partial M$ such that the $0$-strata are vertices of the triangulation. This defines a triangulation of 
the boundary of the closure of all 2-strata of $M$ and $\partial M$. Hence we can choose a compatible triangulation of the 2-strata. Finally, we can also triangulate the 3-strata accordingly with respect to the triangulations of their boundaries.   
\end{proof}

The above proof shows in particular that a bordism~$M$ with a fine stratification can be triangulated. 
It thus follows that the stratification satisfies all the requirements of a `skeleton' of~$M$ as defined in \cite{Hqft1}, so we can apply the HQFT construction from \cite{Hqft1} to such an~$M$.

Next we take~$M$ to be in $\Borddef(\D^{G})$, i.\,e.~$M$ is decorated with defect data in $\D^{G}$. 
%arXiv_v2: 
	%Any given refinement of~$M$ is also in $\Borddef(\D^{G})$ if we label every new 2-stratum by the identity $1 \in G$, and the condition~\eqref{eq:condition-edge} is still satisfied for all edges in the refined stratification.
	For any given refinement of~$M$ we label every new 2-stratum by the identity $1 \in G$, so the condition~\eqref{eq:condition-edge} is still satisfied for all edges in the refined stratification.
For such a fixed refinement of~$M$ we apply the construction of \cite{Hqft1} as follows: 

\begin{enumerate}
\item 
For any colouring $c$ of~$M$ we set $H_{c}(\partial M)= \bigotimes_{\nu}H_{c}(\nu)$ as in~\eqref{eq:HcSigma}, and we define $H(\partial M)= \bigoplus_{c} H_{c}(\partial M)$.  This is not yet the state space $\zz^{\mathcal A}(\partial M)$ of the our TQFT to-be. 
\item 
For any colouring~$c$ of $\partial M$, consider an arbitrary extension  $\widetilde{c}$ to a colouring of the 2-strata in the interior of $M$.
Hence $\widetilde{c}$ is fixed on all 2-strata which intersect $\partial M$, and it labels all remaining 2-strata~$r$ with simple objects in $\Cat{A}_{\ell(r)}$. 
We define a vector $|M,c| \in H_{c}(\partial M)^{*}$ as
\begin{equation}
  \label{eq:HQFT}
  |M,c|= (\dim \Cat{A}_{1})^{-|M_{3}|} \sum_{\widetilde{c}} \Big( \prod_{r \in M_2}\dim (\widetilde{c}(r))^{\chi(r)} \Big) \ev \Big(\bigotimes_{\nu}F_{\Cat{A}}(\Gamma_{\nu})\Big) \, ,
\end{equation}
where $|M_{3}|$ is the number of 3-strata in~$M$, 
%arXiv_v2
	%$\dim \Cat{A}_{1}= (\sum_{i \in I_1} \dim(x_{i})^{2})^{1/2}$ is the square root of the 
	$\dim \Cat{A}_{1}= \sum_{i \in I_1} \dim(x_{i})^{2}$ is the 
sum of squared dimensions of simple objects in $\Cat{A}_{1}$, $\chi(r)$ is the Euler characteristic of a 2-stratum $r \in M_2$, and $\ev:= \bigotimes _{e } \ev(e): \bigotimes_{e}(H_{c}(e) \otimes H_{c}(\cc{e}))\rightarrow \C$ is the tensor product over the evaluation maps associated to all 1-strata $e$ in the interior of $M$. 
The sum in~\eqref{eq:HQFT} is over all possible extensions $\widetilde{c}$ of the given colouring $c$ on the boundary, and $F_{\Cat{A}}$ was defined in \eqref{eq:RT-eval-sphere-graph}.
\item 
In case the boundary of~$M$ splits as $\partial M \cong \Sigma^{\textrm{rev}} \cup \widetilde{\Sigma}$, we want to assign to $M$ a 
linear map $ \zz^{\Cat{A}}_{c}(M)$ between the vector spaces associated to its boundary components. 
Writing $\Upsilon: H_{c}(\Sigma)^{*} \otimes H_{c}(\widetilde{\Sigma}) 
\rightarrow \Hom_{\C}(H_{c}(\Sigma), H_{c}(\widetilde{\Sigma}))$ for the canonical isomorphism, we define 
\begin{equation}
  \label{eq:inv-color}
  \zz^{\Cat{A}}_{c}(M)=\frac{\dim(\Cat{A}_{1})^{|\pi_0(\widetilde{\Sigma}_2)|}}{\prod_{e \in \pi_0(\widetilde{\Sigma}_1)}\dim c(e)} 
\Upsilon(|M,c|) \, ,
\end{equation}
where the product is over connected components of the  1-strata in $\pi_0(\widetilde{\Sigma}_1)$, and $|\pi_0(\widetilde{\Sigma}_2)|$ denotes the number of 2-strata of 
$\widetilde{\Sigma}$. 
Finally, we set 
$$
\zz^{\Cat{A}}(M)= \sum_{c}\zz_{c}^{\Cat{A}}(M) \, . 
$$
\item 
For a cylinder over $\Sigma$, the linear map $\zz^{\Cat{A}}(\Sigma \times [0,1])$ 
is a projector. 
We define $\zz^{\mathcal A}(\Sigma)$ to be the image under this projector. 
\end{enumerate}

\begin{theorem}
The construction above gives rise to a defect TQFT  $\zz^{\mathcal A}: \Borddef(\D^{G}) \rightarrow \Vect_\C$ which is 
a defect extension of the Turaev-Viro TQFT associated to $\Cat{A}_{1}$. 
\end{theorem}

\begin{proof}
According to \cite[Thm.\,8.4]{Hqft1}, $\zz^{\mathcal A}$ is an HQFT, and this implies that $\zz^{\mathcal A}$ is independent of the choice of a fine stratification in the following sense. It depends only on the topology of the underlying manifold $M$ and up to conjugation on the group homomorphism $\phi_{\ell}:\pi_{1}(M)\rightarrow G$ that is induced by the labelling~$\ell$: For a given path $\gamma$ in~$M$, denote by 
$\ell(r_{1})^{\eps_{1}} \cdot \ldots \cdot \ell(r_{m})^{\eps_{m}}$ the group element corresponding to the 2-strata $r_{i}$ that intersect the path, with $\eps_{i}=+1$ if the direction of the path agrees with the orientation of $r_i$, and $\eps_{i}=-1$ otherwise. This group element depends only on the homotopy class of $\gamma$ and on the starting point of the path up to conjugation, so it defines the homomorphism $\phi_{\ell}$ up to conjugation. 
It is clear that a refinement does not change $\phi_{\ell}$ and thus the statement follows from \cite[Thm.\,8.4]{Hqft1}.
 \end{proof}

\paragraph{The associated tricategory $\boldsymbol{\mathcal T_{\zz^{\mathcal A}}}$.}

We show that the tricategory associated to the   defect TQFT $\zz^{\mathcal A}$  is equivalent to the linear Gray category with duals $\underline{G}$ which is the group~$G$ considered as 1-morphisms in a monoidal bicategory with only identity 2-morphisms, and whose 3-morphisms are scalar multiples of the identity.  

\begin{lemma}
  \label{lemma:HTFT-spheres}
Let $S$ be a decorated stratified sphere in $\Borddef(\D^{G})$. Then the vector space $\zz^{\mathcal A}(S)$ is 1-dimensional. Moreover, $\zz^{\mathcal A}(S)$ has a distinguished element~$x$, so there is a canonical isomorphism $\zz^{\mathcal A}(S) \cong \C$ sending~$x$ to $1 \in \C$.
\end{lemma}

\begin{proof}
The HQFT invariants on surfaces depend only on the induced homomorphism from the fundamental group to $G$. Since on $S^{2}$ the fundamental group is trivial, the state space 
is the same as the one associated to the decoration for which all 1-strata are labelled with $1 \in G$. Evaluated on this  labelling,  the value of the HQFT of the stratified sphere is equal to 
the Turaev-Viro invariant for $\Cat{A}_{1}$, which is a 1-dimensional complex vector space, cf.~\cite[Sect.\,7.1]{Hqft1}.

To obtain the canonical element $x$ in this vector space, consider the stratification of the closed 3-ball $B$ that is obtained by linearly filling the stratified sphere~$S$. 
Then we can evaluate  $\zz^{\mathcal A}$ on~$B$ to obtain a linear map  $ \zz^{\mathcal A}(B): \C \rightarrow \zz^{\mathcal A}(S)$ which gives the distinguished element   $ x \in \zz^{\mathcal A}(S)$
when applied to $1 \in \C$.
\end{proof}

It follows that the space of 3-morphisms $\mathcal T_{\zz^{\mathcal A}}(X,Y)$ is always 1-dimensional for all parallel 2-morphisms $X$ and $Y$. 
Moreover, by indepence of the stratification, the composite of two distinguished elements is a
distinguished element, in particular non-zero. Hence Lemma~\ref{lemma:structure-Gray-Invertible} applies to $\mathcal T_{\zz^{\mathcal A}}$.
To analyse the structure of $ T_{\zz^{\mathcal A}}$ further, we note that 1-morphisms  $\alpha$ in $T_{\zz^{\mathcal A}}$ are in bijection with tuples of group elements with signs, $\alpha \equiv (g_{1}^{\eps_{1}}, \ldots, g_{m}^{\eps_{m}})$, by recording the labels and signs on the 2-strata in increasing $y$-direction. 
Hence the dual of~$\alpha$ is given by $\alpha\hash=
(g_{m}^{-\eps_{m}}, \ldots, g_{1}^{-\eps_{1}})$.

Let $x_{1}: X\rightarrow Y$ and $x_{2}:Y\rightarrow Z$ be two distinguished 3-morphisms as in  Lemma~\ref{lemma:HTFT-spheres}. According to the construction 
of $\mathcal T_{\zz^{\mathcal A}}$, the vertical  composition
of $x_{1}$ and $x_{2}$ is the value of the defect functor $\zz^{\mathcal A}$ on a stratified ball after inserting $x_{1}$ and $x_{2}$ as stratified balls. 
This yields again a distinguished 3-morphisms, since $\zz^{\mathcal A}$ is independent of the chosen stratification in the interior of the ball. 
Analogously if follows that the $\fus$- and $\sta$-compositions of distinguished 3-morphisms give the unique respective distinguished 3-morphism.

\begin{proposition}
There is an equivalence of Gray categories with duals 
$$
\mathcal T_{\zz^{\mathcal A}} \stackrel{\cong}{\lra} \underline{G} \, . 
$$
\end{proposition}

\begin{proof}
  We define a functor $F: \mathcal T_{\zz^{\mathcal A}}(X,Y) \rightarrow \underline{G}$ as follows:
On the single object, $F$  is the identity, and 1-morphisms  
 $\alpha=(g_{1}^{\eps_{1}}, g_{2}^{\eps_{2}}, \ldots, g_{m}^{\eps_{m}})$ are sent to $F(\alpha)= g_{1}^{\eps_{1}} \cdot g_{2}^{\eps_{2}} \cdot \ldots \cdot g_{m}^{\eps_{m}} \in G$. 
Every 2-morphism gets mapped to the identity. 
This is well defined since there exists a 2-morphism between two 1-morphisms $\alpha$ and $\beta$ iff $F(\alpha)=F(\beta)$ by definition of $\D^{G}$. 
By Lemma~\ref{lemma:HTFT-spheres}, every 3-morphism in $\mathcal T_{\zz^{\mathcal A}}$ can be uniquely written as $\lambda \cdot x$ with $\lambda \in \C$ and $x$ the distinguished element of Lemma~\ref{lemma:HTFT-spheres}. 
We define the functor on 3-morphism via $F(\lambda \cdot x)= \lambda \cdot 1$, where the identity is the identity on the identity 2-morphism of $\underline{G}$.

Since the composition of 3-morphisms maps distinguished 3-morphisms to distinguished 3-morphisms, it follows that $F$ defines a functor on the categories of 2- and 3-morphisms.
The compatibility of $F$ with the other two compositions of 3-morphisms follows analogously.
Clearly, $F$ is a strict functor on the level of 1- and 2-morphisms. Hence $F$ is a well-defined functor of Gray categories. 

It follows directly that $F(\alpha^{\#})=F(\alpha)^{\#}$ for all 1-morphisms $\alpha$. Since $F$ maps all 2-morphisms to the identity, it is compatible with the duality on the level of 2-morphisms as well. 
This means that $F$ is a functor of Gray category with duals. 

It remains to show that~$F$ is an equivalence. Clearly, it is fully faithful on 3-morphisms due to Lemma~\ref{lemma:HTFT-spheres}. Moreover it is obvious that $F$ is essentially surjective 
on 2-morphisms and triessentially surjective on the single object. 
To show that $F$ is biessentially surjective, we define for a 1-morphism $g \in \underline{G}$  the 1-morphism in $\mathcal T_{\zz^{\mathcal A}}$ consisting of a single 2-stratum labelled with~$g$. 
The functor~$F$ applied to this 1-morphism is~$g$, hence~$F$ is biessentially surjective, and the proof is complete. 
\end{proof}

\end{document}